\documentclass[11pt,reqno]{amsart}
\usepackage[pdfstartview=FitB]{hyperref}
\usepackage{amsmath,amssymb,amsthm}
\usepackage[dvips]{graphicx}
\usepackage{epsfig,epstopdf}
\usepackage{url}
\usepackage{colortbl,xr}

\renewcommand{\qed}{\hfill{\tiny \ensuremath{\blacksquare} }}%

\newcommand{\R}{\mathbb{R}}

\newcommand{\N}{\mathbb{N}}
\newcommand{\D}{\mathbb{D}}
\newcommand{\E}{\mathbb{E}}

\newcommand{\mA}{\mathcal{A}}

\newcommand{\mP}{\mathcal{P}}

\newcommand{\mW}{\mathcal{W}}
\newcommand{\mF}{\mathcal{F}}

\newcommand{\mH}{\mathcal{H}}

\newcommand{\mU}{\mathcal{U}}
\newcommand{\mZ}{\mathcal{Z}}
\newcommand{\mV}{\mathcal{V}}

\newcommand{\mQ}{\mathcal{Q}}

\newcommand{\bP}{\mathbb{P}}
\newcommand{\supp}{\mathrm{supp}}
\newcommand{\Id}{\mathrm{Id}}

\newcommand{\bG}{\mathbb{G}}

\newcommand{\G}{{\Lambda}}

\newcommand{\ci}{\perp\!\!\!\perp}
\renewcommand{\qed}{\hfill {\tiny {\ensuremath{\blacksquare}}}}

\newtheorem{theorem}{Theorem}[section]
\newtheorem{corollary}{Corollary}[section]
\newtheorem{lemma}{Lemma}[section]

\newtheorem{assumption}{Assumption}[section]
\newtheorem{algorithm}{Algorithm}[section]

\newtheorem{definition}{Definition}[section]

\theoremstyle{definition}

\newtheorem{remark}{Comment}[section]
\numberwithin{remark}{section}
\newtheorem{example}{Example}
\numberwithin{equation}{section}
\numberwithin{theorem}{section}

\newcommand{\eps}{\varepsilon}
\usepackage{url}
\usepackage{natbib}
\newcommand{\citen}{\citet}
\renewcommand{\citeyear}{\citeyearpar}

\newcommand{\Gn}{\mathbb{G}_n}

\newcommand{\Pn}{\mathbb{P}_n}

\newcommand{\F}{\mathcal{F}}

\newcommand{\Ep}{{\mathrm{E}}}
\newcommand{\barEp}{\bar \Ep}
\renewcommand{\Pr}{{\mathrm{P}}}

\renewcommand{\(}{\left(}
\renewcommand{\)}{\right)}
\renewcommand{\hat}{\widehat}
\newcommand{\En}{{\mathbb{E}_n}}

\renewcommand{\Pr}{{\mathrm{P}}}
\newcommand{\RR}{\mathbb{R}}

\newcommand{\ceil}[1]{\left\lceil #1 \right\rceil}
\newcommand{\semin}[1]{\phi_{{\rm min}}(#1)}
\newcommand{\semax}[1]{\phi_{{\rm max}}(#1)}
\renewcommand{\hat}{\widehat}
\renewcommand{\leq}{\leqslant}
\renewcommand{\geq}{\geqslant}
\newcommand{\sign}{ {\rm sign}}

\DeclareMathOperator{\Var}{Var}

\newcommand{\diag}{{\rm diag}}

\newcommand{\dn}{{d_u}}
\newcommand{\cc}{\mathbf{c}}
\renewcommand{\(}{\left(}
\renewcommand{\)}{\right)}
\renewcommand{\[}{\left[}
\renewcommand{\]}{\right]}

\begin{document}

\title[]{Program Evaluation and Causal Inference with High-Dimensional Data}

\author[]{A. Belloni, V. Chernozhukov,  I. Fern\'andez-Val, and C. Hansen}

\date{\tiny First version: April 2013.  This version: \today.  
We gratefully acknowledge research support from
the NSF. We are very grateful to the co-editor, three anonymous referees,  Alberto Abadie, Stephane Bonhomme, Matias Cattaneo, Jinyong Hahn, Michael Jansson, Toru Kitagawa, Roger Koenker, Simon Lee, Yuan Liao, Oliver Linton, Blaise Melly, Whitney Newey, Adam Rosen, and seminar participants at the Bristol Econometric Study Group, Cambridge, CEMFI,  Cornell-Princeton Conference on ``Inference on Non-Standard Problems", Winter 2014 ES meeting, Semi-Plenary Lecture at of ES Summer Meeting 2014, ES World Congress 2015, 2013 Summer NBER Institute, University of Montreal,  UIUC,  UCL, and University of Warwick for helpful comments. We are especially grateful to Andres Santos for many useful comments. Matlab code for the empirical illustration is available on the \textit{Econometrica} website.  R code for implementing the treatment effects estimators is available in the R package ``hdm'' from \citen{R:hdm}.  }
\keywords{\tiny machine learning, causality,  Neyman orthogonality,  heterogeneous treatment effects, endogeneity, local average and quantile treatment effects, instruments, local effects of treatment on the treated, propensity score, Lasso, inference after model selection, moment condition models, moment condition models with a continuum of target parameters, Lasso and Post-Lasso with functional response data, randomized control trials}
\begin{abstract}
\begin{footnotesize}

In this paper, we provide efficient estimators and honest confidence bands
for a variety of treatment effects including local average (LATE) and local quantile treatment effects (LQTE) in data-rich environments.    We can handle \textit{very many} control variables, \textit{endogenous} receipt of treatment, \textit{heterogeneous} treatment effects, and \textit{function-valued} outcomes.  Our framework covers the special case of exogenous receipt of treatment, either conditional on controls
or unconditionally as in randomized control trials. In the latter case, our approach produces efficient estimators and honest bands for (functional) average treatment effects (ATE) and quantile treatment effects (QTE).  To make informative inference possible, we assume that key reduced form predictive relationships are approximately sparse.  This assumption allows the use of regularization and selection methods to estimate those relations, and we provide methods for post-regularization and post-selection inference that are uniformly valid (honest) across a wide-range of models. We show that a key ingredient enabling honest inference  is the use of orthogonal or doubly robust moment conditions in estimating certain reduced form functional parameters. We illustrate the use of the proposed  methods with an application to estimating the effect of 401(k) eligibility and participation on accumulated assets.

The results on program evaluation are obtained as a consequence of more general results on honest inference in a
general moment condition framework, which arises from structural equation models in econometrics.    Here too the crucial ingredient is the use of orthogonal moment conditions, which
can be constructed from the initial moment conditions. We provide  results on honest inference for (function-valued) parameters within this general framework where \textit{any high-quality}, modern \textit{machine learning} methods (e.g., boosted trees, deep neural networks, random forests, and their aggregated and hybrid versions) 
can be used to learn the nonparametric/high-dimensional components of the model.  These include a number of supporting auxilliary results that are of major independent interest: namely, we (1) prove uniform validity of a  multiplier bootstrap,
(2)  offer a uniformly valid functional delta method, and (3) provide results for sparsity-based estimation of regression functions for function-valued outcomes.
\end{footnotesize}

\end{abstract}

\enlargethispage*{\baselineskip}

\maketitle
\thispagestyle{empty}

\section{Introduction}

The goal of many empirical analyses is to understand the causal effect of a treatment, such as participation in a training program or a government policy, on economic and other outcomes.  Such analyses are often complicated by the fact that treatments or policies are rarely randomly assigned.  The lack of true random assignment has led to the adoption of a variety of quasi-experimental approaches to estimating treatment effects that are based on observational data.  Such approaches include instrumental variable (IV) methods in cases where treatment is not randomly assigned but there is some other external variable, such as eligibility for receipt of a government program or service, that is either randomly assigned or the researcher is willing to take as exogenous conditional on the right set of control variables (or simply controls).  Another common approach is to assume that the treatment variable itself may be taken as exogenous after conditioning on the right set of controls which leads to regression or matching based methods, among others, for estimating treatment effects.\footnote{There is a large literature about estimation of treatment effects.  See, for example, the textbook treatments in \citen{AngristBook}, \citen{wooldridge:text}, and \citen{imbens:rubin:book}.}

A practical problem empirical researchers face when trying to estimate treatment effects is deciding what conditioning variables to include.  When the treatment variable or instrument is not randomly assigned, a researcher must choose what needs to be conditioned on to make the argument that the instrument or treatment is exogenous plausible.  Typically, economic intuition will suggest a set of variables that might be important to control for but will not identify exactly which variables are important or the functional form with which variables should enter the model. While less crucial to identifying treatment effects, the problem of selecting controls  also arises in situations where the key treatment or instrumental variables are randomly assigned.  In these cases, a researcher interested in obtaining precisely estimated policy effects will also typically consider including additional controls to help absorb residual variation.  As in the case where including controls  is motivated by a desire to make identification of the treatment effect more plausible, one rarely knows exactly which variables will be most useful for accounting for residual variation.  In either case, the lack of clear guidance about what variables to use presents the problem of selecting controls from a potentially large set  including raw variables available in the data as well as interactions and other transformations of these variables.

In this paper, we consider estimation of the effect of an \textit{endogenous} binary treatment, $D$, on an outcome, $Y$, in the presence of a binary instrumental variable, $Z$, in settings with very many potential controls, $f(X)$.  Allowing many potential controls expressly covers both the case where there are simply many controls (where $f(X) = X$)) and the case where there are many technical controls $f(X)$ generated as transformations such as powers, b-splines, or interactions of raw controls,\footnote{See, e.g., \citen{koenker:jappliedeconometircs},  \citen{newey:series}, \citen{wasserman:npbook},  \citen{chen:Chapter},  and  \citen{tsybakov:npbook}.} $X$, along with combinations of the two cases. The notation $f(X)$ naturally accommodates these cases, and we call $f(X)$ the \emph{controls} regardless of the case.  We allow for fully \textit{heterogeneous} treatment effects and thus focus on estimation of causal quantities that are appropriate in heterogeneous effects settings such as the local average treatment effect (LATE) or the local quantile treatment effect (LQTE).  We focus our discussion on the \textit{endogenous} case where identification is obtained through the use of an instrumental variable, but all results carry through to the \textit{exogenous} case where the treatment is taken as exogenous unconditionally or  after conditioning on sufficient controls  by simply replacing the instrument with the treatment variable in the estimation and inference methods and in the formal results.  In the latter case, LATE reduces
to the average treatment effect (ATE) and LQTE to the quantile treatment effect (QTE).

The methodology for estimating treatment effects we consider allows for cases where the number of potential controls, $p := \dim f(X)$, is much larger than the sample size, $n$.  Of course, informative inference about causal parameters cannot proceed allowing for $p \gg n$ without further restrictions.  We impose sufficient structure through the assumption that reduced form relationships such as the conditional expectations $\Ep_P[D|X]$, $\Ep_P[Z|X]$, and $\Ep_P[Y|X]$ are approximately sparse.  Intuitively, approximate sparsity imposes that these reduced form relationships can be represented up to a small approximation error as a linear combination, possibly inside of a known link function such as the logistic function, of a number $s \ll n$ of the variables in $f(X)$ whose identities are \emph{a priori} unknown to the researcher.  This assumption allows us to use methods for estimating models in high-dimensional sparse settings that are known to have good prediction properties to estimate the fundamental reduced form relationships.  We may then use these estimated reduced form quantities as inputs to estimating the causal parameters of interest.  Approaching the problem of estimating treatment effects within this framework allows us to accommodate the realistic scenario in which a researcher is unsure about exactly which confounding variables or transformations of these confounds are important and so must search among a broad set of controls.

Valid inference following model selection is non-trivial.  Direct application of usual inference procedures following model selection does not provide valid inference about causal parameters even in low-dimensional settings, such as when there is only a single control, unless one assumes sufficient structure on the model that perfect model selection is possible.  Such structure can be restrictive and seems unlikely to be satisfied in many economic applications.  For example, a typical condition that allows perfect model selection is the ``beta-min" condition, which requires that all but a small number of coefficients are exactly zero and that the non-zero coefficients are all large enough that they can be distinguished from zero with probability very near one in finite samples.  Such a condition rules out the possibility that there may be some variables which have moderate, but non-zero, partial effects.  Ignoring such variables may result in large omitted variables bias that has a substantive impact on estimation and inference regarding individual model parameters; see Leeb and P{\"o}tscher \citeyear{leeb:potscher:pms,leeb:potscher:review}; \citen{Potscher2009}; and Belloni, Chernozhukov, and Hansen \citeyear{BCH2011:InferenceGauss,BelloniChernozhukovHansen2011}.

The \textit{first main contribution} of this paper is providing inferential procedures for key parameters used in program evaluation that are theoretically valid within approximately sparse models allowing for imperfect model selection.  Our procedures build upon \citen{BellChernHans:Gauss} and  \citen{BellChenChernHans:nonGauss}, who were the first to demonstrate in a highly specialized context, that valid inference can proceed following model selection allowing for model selection mistakes under two conditions.
 We formulate and extend these two conditions to a rather general moment-condition framework (e.g., \citen{hansen_gmm} and \citen{hansen-singleton_gmm}) as follows.  First, estimation should be based upon ``orthogonal" moment conditions that are first-order insensitive to changes in the values of nuisance parameters that will be estimated using high-dimensional methods. Specifically, if the target parameter value $\alpha_0$ is identified via the moment condition
\begin{equation}
\Ep_{P} \psi(W, \alpha_0, h_0) =  0,
\end{equation}
where $h_0$ is a function-valued nuisance parameter estimated via a model-selection or regularization method, one needs to use a moment function, $\psi$, such that
the corresponding moment condition is orthogonal with respect to perturbations of $h$ around $h_0$.  More formally, the moment condition should satisfy the \textit{Neyman orthogonality condition}
\begin{equation}\label{LowBias}
\partial_h [\Ep_{P} \psi(W, \alpha_0, h)]_{h = h_0} =  0,
\end{equation}
where $\partial_h$ is a functional derivative operator with respect to $h$ restricted to directions of possible deviations of estimators of $h_0$ from $h_0$. Second, one needs to ensure that the model selection mistakes occurring in the estimation of nuisance parameters are uniformly ``moderately" small with respect to the underlying model.  Specifically, we will require that the nuisance parameter $h_0$ is estimated at the rate $o(n^{-1/4})$, which ensures small bias, and that the estimator takes values in a space whose entropy does not grow too fast, which ensures no overfitting. In this paper, we establish that building estimators based upon moment conditions with the orthogonality condition (\ref{LowBias}) holding ensures that crude estimation of $h_0$ via post-selection or other regularization methods has an asymptotically negligible effect on the estimation of $\alpha_0$ in general frameworks.  It then follows that we can form a regular, root-$n$ consistent estimator of $\alpha_0$, uniformly with respect to the underlying model.

In the endogenous treatment effects setting, we build moment conditions satisfying (\ref{LowBias}) from the efficient influence functions for certain reduced form parameters, building upon \citen{hahn-pp}. We illustrate how  orthogonal moment conditions coupled with methods developed for forecasting in high-dimensional approximately sparse models can be used to estimate and obtain valid inferential statements about a wide variety of structural/treatment effects.  We formally demonstrate the uniform validity of the resulting inference within a broad class of approximately sparse models including models where perfect model selection is theoretically impossible.
An important feature of our main theoretical results is that they cover the use of variable selection for \textit{functional response data} using $\ell_1$-penalized methods.  Functional response data arises, for example,  when one is interested in the LQTE at not just a single quantile but over a range of quantile indices.  Considering this case then necessitates looking  at the functional dependent variable $u \longmapsto 1(Y \leq u)$, where $u$ denotes various levels that $Y$ can cross.  Treating such functional response data allows us to provide a unified inference procedure for interesting quantities such as the (local) distributional  and quantile effects of the treatment, including simpler important parameters such as LQTE at a given quantile as a special case.

The \textit{second main contribution} of this paper  is providing a general set of results for uniformly valid estimation and inference methods in moment-condition problems, arising in structural analysis in econometrics and other data sciences. These results are useful not only for establishing the properties of treatment effects estimators developed here, but they are also useful for attacking a wide range of problems in structural econometrics.  For example, \citen{CHS:PnP} provide estimates of parameters characterizing a simple structural demand model based loosely on the analysis in \citen{BLP:Autos} using the framework developed here; see also \citen{CHS:AnnRev}.   A key element to our establishing uniform validity of post-regularization inference is again the use of Neyman orthogonal moment conditions.  In the general framework we consider, we may have (a continuum of) target parameters identified via (a continuum of) moment conditions that involve (a continuum of) nuisance functions that will be estimated via Lasso, Post-Lasso, or some other high-quality machine learning method.   Our general theory expressly allows for a wide variety of traditional and machine learning methods, including those that do not rely on approximate sparsity, as long as the methods \begin{itemize}
\item[\textbf{1)}] have good approximation ability and
\item[\textbf{2)}] do not overfit.
\end{itemize}
By ``not overfitting"  we mean that the entropy of the function classes containing  the realizations of the estimator of the nuisance function/parameter does not increase too rapidly with the sample size.  This second condition can  only be verified analytically, but can be avoided by the use of various data splitting methods.  For example, we can set aside a vanishing fraction of the data  to estimate the nuisance parameter, as in \citen{bickel:1982}, or employ cross-fitting, as in Belloni \textit{et al.} (2010, 2012) and \citen{CCDHM16}.  Either scheme ensures that there is no asymptotic efficiency loss from data-splitting. We refer the reader to \citen{CCDHM16} for a detailed discussion and analysis of cross-fitting in connection to inference on ATE and other causal parameters using machine learning methods for high-dimensional data.\footnote{Cross-fitting proceeds as follows: (1) split the sample into two equal parts, the auxiliary and main parts; (2) use the auxiliary part to estimate the nuisance parameter and the main part to estimate the target  parameter, obtaining  one estimator of the target parameter;  (3) by reversing the roles of the main and auxiliary parts,  obtain another estimator of the target parameter; and (4) average the two estimators of the target parameter to obtain the final estimator.  The theorems  established  in Section 5 yield the properties of the final estimator; see \citen{CCDHM16}.}

  These results contain the results on treatment effects relevant for program evaluation, particularly the results for distributional and quantile effects, as a leading special case.  These results are also immediately useful in other contexts such as nonseparable quantile models as in \citen{iqr:ema}, \citen{CH06}, \citen{C03}, and \citen{IN09}; semiparametric and partially identified models as in \citen{EZ2013}; and many others. In our results, we first establish a functional central limit theorem for the continuum of target parameters and show that this functional central limit theorem holds uniformly in a wide range of data-generating processes $P$ with approximately sparse continua of nuisance functions. Second, we establish a functional central limit theorem for the multiplier bootstrap that resamples the first order approximations to the standardized estimators and demonstrate its uniform-in-$P$ validity. These uniformity results build upon and complement those given in \citen{Romano:Shaikh:AoS} for the empirical bootstrap. Third, we establish a functional delta method for smooth functionals of the continuum of target parameters and a functional delta method for the multiplier bootstrap of these smooth functionals, both of which hold uniformly in $P$, using an appropriately strengthened notion of Hadamard differentiability.  All of these results are new and are of independent interest outside of the treatment effects focus of this paper.

We illustrate the use of our methods by estimating the effect of 401(k) eligibility and 401(k) participation on measures of accumulated assets as in \citen{CH401k}.\footnote{See also Poterba, Venti, and Wise \citeyear{pvw:94,pvw:95,pvw:nber96,pvw:01}, \citen{abadie:401k}, \citen{benjamin},  and \citen{ORR:401k} among others.}  Similar to \citen{CH401k}, we provide estimates
of ATE and QTE of 401(k) eligibility and of LATE and LQTE of 401(k) participation.  We differ from this previous work by using the high-dimensional methods developed in this paper to allow ourselves to consider a broader set of controls than has previously been considered.  We find that 401(k) participation has a moderate impact on accumulated financial assets at low quantiles while appearing to have a much larger impact at high quantiles.  Interpreting the quantile index as ``preference for savings'' as in \citen{CH401k}, this pattern suggests that 401(k) participation has little causal impact on the accumulated financial assets of those with low desire to save but a much larger impact on those with stronger preferences for saving.
It is interesting that these results are similar to those in \citen{CH401k} despite allowing for a much richer set of controls.

\subsection*{Links to the literature} The Neyman orthogonality condition embodied in (\ref{LowBias}) has a long history in statistics and econometrics.  For example, this type of orthogonality was used by \citen{Neyman1979} in low-dimensional settings to deal with crudely estimated parametric nuisance parameters.  See also \citen{newey90}, \citen{andrews94},  \citen{newey94}, \citen{robins:dr}, and \citen{linton96} for the use of this condition in semi-parametric problems.

To the best of our knowledge, \citen{BellChernHans:Gauss} and \citen{BellChenChernHans:nonGauss} were the first to use the orthogonality (\ref{LowBias})
to expressly address the question of the uniform post-selection inference without imposing ``beta-min" conditions, either in high-dimensional settings with $p \gg n$ or in low-dimensional settings with $p \ll n$. They applied it to the specific problem of the linear instrumental variables model with many instruments where the nuisance function $h_0$ is the optimal instrument estimated by Lasso or Post-Lasso methods and $\alpha_0$ is the coefficient of the endogenous regressor. \citen{BCH2011:InferenceGauss} and  \citen{BelloniChernozhukovHansen2011} also exploited this approach to develop a double-selection method 
that yields valid post-selection inference
 on the parameters of the linear part of a partially linear model and on average treatment effects when the  treatment is binary and \textit{exogenous} conditional on controls in both the $p \gg n$ and the $p \ll n$ setting.\footnote{Note that these results as well as results of this paper on the uniform post-selection inference in moment-condition problems are new for either $p\ll n$ or $p \gg n$ settings.  The results also apply to arbitrary model selection devices, such as the Dantzig selector, Square-Root-Lasso,  or Adaptive Lasso,  that are able to select good sparse approximating models; and ``moderate'' model selection errors are explicitly allowed in the paper.}
Subsequently, \citen{Farrell:JMP} extended the results of \citen{BCH2011:InferenceGauss} and  \citen{BelloniChernozhukovHansen2011} to estimation of ATE when the treatment is multivalued and exogenous conditional on controls using group penalization for selection.  Note that this previous work on treatment effects covers only the exogenous case and does not allow for functional responses which are necessary, for example, for working with distributional or quantile treatment effects.

Our work also contributes to the line of research on obtaining
$\sqrt{n}$-consistent and asymptotically normal estimates for low-dimensional components within traditional semiparametric frameworks as
in the important work by  \citen{bickel:1982},  \citen{robinson},  \citen{newey90}, \citen{vaart:1991}, \citen{andrews94},  \citen{newey94},  \citen{ai:chen, AC2012}, and \cite{CLK:EfficientSP}.    The major difference is that we allow for the use of modern high-dimensional methods, a.k.a. machine learning methods, for modeling and fitting the non-parametric (or high-dimensional) components of the model. In contrast to the former literature, we expressly allow for data-driven choice of the approximating model for the high-dimensional component, which addresses a crucial problem that arises in empirical work. 
Moreover, recent methods based on $\ell_1$-penalization, upon which we focus in this paper, allow for much more flexible modeling of the non-parametric/high-dimensional parts of the model.\footnote{See, for instance, \citen{BellChenChernHans:nonGauss} and \citen{BCW-SqLASSO2} for a formalization of this claim in terms of rearranged Sobolev spaces where it is shown that traditional methods can fail to be consistent while $\ell_1$-penalized methods remain consistent and have good rates of convergence.}  Our general theory in Section 5 also allows, in principle, for a wide variety of both traditional and machine learning methods.

The paper also generates a number of new results on sparse estimation with functional response data.   These results are of independent interest in themselves, and they build upon the work of \citen{BC-SparseQR} who provided rates of convergence for variable selection when one is interested in estimating the quantile regression process with exogenous variables.  More generally, this theoretical work complements and extends the rapidly growing set of results for $\ell_1$-penalized estimation methods; see, for example, \citen{FF:1993}; \citen{T1996}; \citen{FanLi2001}; \citen{Zou2006};  \citen{CandesTao2007};  \citen{vdGeer};  \citen{HHS2008};  \citen{BickelRitovTsybakov2009};  \citen{MY2007};  \citen{Bach2010}; \citen{horowitz:lasso};  \citen{BC-SparseQR};  \citen{kato};  \citen{BellChenChernHans:nonGauss}; \citen{BC-PostLASSO};  \citen{BCK-LAD}; \citen{BCY-honest}; \citen{CanerZhang:GMMEL}; and the references therein.

\subsection*{Plan of the Paper} Section \ref{sec: model} introduces the structural parameters for policy evaluation and relates these parameters to reduced form functions.
 Section \ref{EstimationSection} describes a three step procedure to estimate and make inference on the structural parameters and functionals of these parameters, and Section \ref{sec: asymptotics} provides asymptotic theory in the treatment effects setting. Section \ref{sec: general} generalizes the setting and results to moment-condition problems with a continuum of structural parameters and a continuum of reduced form functions. Section \ref{FunctionalLassoSection} derives general asymptotic theory for the Lasso and post-Lasso estimators for functional response data used in the estimation of the reduced form functions.
Section \ref{sec: 401k} presents the empirical application. We provide notation, proofs of key results, and details about implementation of the methods in the empirical example in Appendices \ref{subsec:notation}--\ref{subsec:ProofSection5}. An on-line Supplementary Appendix provides all remaining proofs, additional technical material, and results from a small Monte Carlo simulation  \citep{bcfh15sup}.  

\section{The Treatment Effects Setting and Target Parameters}\label{sec: model}

\subsection{Observables and Reduced Form Parameters}
The observed random variables consist of $( (Y_u)_{u \in \mU}, X,Z,D)$. The outcome variable of interest $Y_u$ is indexed by $u \in \mU$. We give examples of the index $u$ below.
The variable $D \in \mathcal{D}=\{0,1\}$ is a binary indicator of the  receipt of a treatment
or participation in a program.  It will typically be treated as endogenous; that is, we will typically view the treatment as assigned non-randomly with respect to the outcome.
The instrumental variable $Z \in \mathcal{Z}=\{0,1\}$ is a binary indicator, such as an offer of participation, that is assumed to be randomly assigned conditional on the observable covariates $X$ with support $\mathcal{X}$.\footnote{Of course, by ``randomly assigned" we mean independently of potential outcomes conditional on the covariates.}
For example, we argue that 401(k) eligibility can be considered exogenous only after conditioning on income and other individual characteristics in the empirical application. The notions of exogeneity and endogeneity we employ are standard and thus omitted.\footnote{For completeness, we provide a review of these conditions as well as restate standard conditions that are sufficient for a causal interpretation of the target parameters in the Supplementary Appendix.}

The indexing of the outcome $Y_u$ by $u$ is useful to analyze functional data.  For example, $Y_u$ could represent an outcome falling short of a threshold,
namely $Y_u = 1(Y \leq u)$, in the context of distributional analysis;   $Y_u $ could be a height indexed by age $u$  in growth charts analysis; or $Y_u$ could be a health outcome indexed by a dosage $u$ in dosage response studies.  Our framework
is tailored for such functional response data.  The special case with no index is included by simply considering $\mU$
to be a singleton set.


We make use of two key types of reduced form parameters for estimating the structural parameters of interest -- (local) treatment effects and related quantities.  These reduced form parameters are defined as
\begin{equation}\label{KeyRF1}
\alpha_V(z) := \Ep_P[ g_V(z,X)]  \ \text{ and }   \ \gamma_V:= \Ep_P[V],
\end{equation}
where $z = 0$ or $z =1$ are the fixed values of  $Z$.\footnote{The expectation that defines $\alpha_V(z)$ is well-defined under the standard support condition $0 < c < \Pr_P(Z=1 \mid X) < 1 - c < 1$ a.s. This condition is standard in treatment effects estimation; see, e.g., the supplementary appendix.  We impose this condition in 
Assumption \ref{assumption: basic}.}
The function $ g_{V}$ maps $\mathcal{ZX}$, the support  of the vector $(Z,X)$, to the real line $\mathbb{R}$ and is defined as
\begin{eqnarray}
&&  g_{V}(z,x): = \Ep_P[V|Z=z,X=x].
\end{eqnarray}
We use $V$ to denote a target variable whose identity may change depending on the context such as $V=\mathbf{1}_d(D) Y_u$  or $V=\mathbf{1}_d(D)$ where $\mathbf{1}_d(D) := 1 (D=d)$ is the indicator function.

All the structural parameters we consider are
smooth functionals of these reduced-form parameters.  In our approach to estimating treatment effects, we estimate the key reduced form parameter $\alpha_V(z)$ using modern methods to deal with high-dimensional data coupled with orthogonal estimating equations. The orthogonality property allows us to deal with the ``non-regular" nature of  penalized and post-selection
estimators which do not admit linearizations except under very restrictive conditions.  The use of regularization by model selection or penalization is in turn motivated by the desire to accommodate high-dimensional data.

\subsection{Target Structural Parameters -- Local Treatment Effects}\label{subset: lates}

The reduced form parameters defined in (\ref{KeyRF1}) are key because the structural parameters of interest are
functionals of these elementary objects.  The local average structural
function (LASF) defined as
\begin{equation}\label{define:LASF}
\theta_{Y_u}(d) =  \frac{\alpha_{\mathbf{1}_d(D)Y_u}(1)-\alpha_{\mathbf{1}_d(D)Y_u}(0)}{\alpha_{\mathbf{1}_d(D)}(1) - \alpha_{\mathbf{1}_d(D)}(0)},  \  \ d \in \{0,1\}
\end{equation}
underlies the formation of many commonly used treatment effects.  Under standard assumptions, the LASF
identifies average potential outcomes for the group of \textit{compliers}, individuals whose treatment status may be influenced by variation in the instrument, in the treated and non-treated states; see, e.g. Abadie \citeyear{abadie:bstest,abadie:401k}. The local average treatment
effect (LATE) of \citen{imbens:angrist:94} corresponds to the difference of the two values of the LASF:
\begin{equation}\label{define:LATE}
\theta_{Y_u}(1) -  \theta_{Y_u}(0).
\end{equation}
 The term local designates
 that this parameter does not measure the effect on the entire population
 but rather measures the effect on the subpopulation of compliers.\footnote{The methods of the paper can be extended to analyze the marginal treatment effects of \citen{heckman:vytlacil, hv2005}.}

When there is no endogeneity, formally when $D \equiv Z$, the LASF and LATE become the average structural
function (ASF) and average treatment effect (ATE) on the entire population.  Thus, our results cover this situation as a special case where
the ASF and ATE simplify to
\begin{equation}
\theta_{Y_u}(z) = \alpha_{Y_u}(z), \  \ \theta_{Y_u}(1)-\theta_{Y_u}(0) = \alpha_{Y_u}(1) -  \alpha_{Y_u}(0).
\end{equation}
We also note that the impact of the instrument $Z$ itself may be of interest since $Z$
often encodes an offer of participation in a program.  In this case, the parameters of interest are again simply
the reduced form parameters $$
\alpha_{Y_u}(z) ,  \ \  \alpha_{Y_u}(1) -  \alpha_{Y_u}(0).$$
Thus, the LASF and LATE are primary targets of interest in this paper, and the ASF and ATE are subsumed as special cases.

\subsubsection{Local Distribution and Quantile Treatment Effects}
Setting $Y_u = Y$ in (\ref{define:LASF}) and (\ref{define:LATE}) provides the conventional LASF and LATE.  An important generalization arises by letting $Y_u = 1(Y \leq u)$ be the indicator of the outcome of interest falling below a threshold $u \in \mathbb{R}$.   In this case, the family of effects
\begin{equation}\label{define:LDTE}
( \theta_{Y_u}(1) -  \theta_{Y_u}(0))_{ u \in \mathbb{R}},
\end{equation}
describe the local distribution treatment effects (LDTE).
Similarly, we can look at the quantile left-inverse transform of the
curve $u \longmapsto \theta_{Y_u}(d)$,
\begin{equation}{\label{define:LSFQ}}
\theta^{\leftarrow}_Y(\tau, d) := \inf\{ u \in \mathbb{R} :  \theta_{Y_u}(d) \geq \tau \},
\end{equation}
and examine the family of local quantile treatment effects  (LQTE):
\begin{equation}{\label{define:LQTE}}
(\theta^{\leftarrow}_Y(\tau, 1) - \theta^{\leftarrow}_Y(\tau, 0))_{\tau \in (0,1)}.
\end{equation}
The LQTE identify the differences of quantiles between the distribution of potential outcomes in the treated and non-treated states for compliers.

\subsection{Target Structural Parameters -- Local Treatment Effects on the Treated}\label{subset: latts}
We may also be interested in local treatment effects on the treated.  The key object in defining these effects  is the local average structural function on the treated (LASF-T) which is defined by its two values:
\begin{equation}\label{define:LASF-T}
\vartheta_{Y_u}(d) =  \frac{\gamma_{\mathbf{1}_d(D) Y_u} -  \alpha_{\mathbf{1}_d(D) Y_u}(0) }{\gamma_{\mathbf{1}_d(D) } -  \alpha_{\mathbf{1}_d(D)}(0)   }, \ \  d \in \{0,1\}.
 \end{equation}
The LASF-T identifies average potential outcomes for the group of \textit{treated compliers} in the treated and non-treated states under standard assumptions.  The local average treatment
effect on the treated (LATE-T) introduced in \citen{hong:nekipelov:2010} and \citen{frolich:melly} is the difference of two values of the LASF-T:
\begin{equation}\label{define:LATE-T}
\vartheta_{Y_u}(1) -  \vartheta_{Y_u}(0).
\end{equation}
The LATE-T may be of interest because it measures the average treatment effect for \textit{treated compliers}, namely the subgroup of compliers that actually receive the treatment.  

When the treatment is assigned randomly given controls so we can take $D=Z$, the LASF-T and LATE-T become the average structural
function on the treated (ASF-T) and average treatment effect on the treated (ATE-T).  In this special case, the ASF-T and ATE-T simplify to
\begin{equation}
\vartheta_{Y_u}(1) =   \frac{\gamma_{\mathbf{1}_1(D) Y_u}}{\gamma_{\mathbf{1}_1(D) }   } ,     \  \
\vartheta_{Y_u}(0) =   \frac{\gamma_{\mathbf{1}_0(D) Y_u} -  \alpha_{Y_u}(0) }{\gamma_{\mathbf{1}_0(D) } -  1   },  \ \
\vartheta_{Y_u}(1)-\vartheta_{Y_u}(0);
\end{equation}
and we can use our results to provide estimation and inference methods for these quantities.

\subsubsection{Local Distribution and Quantile Treatment Effects on the Treated}
Local distribution treatment effects on the treated (LDTE-T) and local quantile treatment effects on the treated (LQTE-T) can also be defined.  As in Section 2.2.1, we let $Y_u = 1(Y \leq u)$ be the indicator of the outcome of interest falling below a threshold $u$.  The family of treatment effects
\begin{equation}\label{define:LDTE-T}
(\vartheta_{Y_u}(1) -  \vartheta_{Y_u}(0))_{ u \in \mathbb{R}}
\end{equation}
then describes the LDTE-T.  We can also use the quantile left-inverse transform of the
curve $u \longmapsto \vartheta_{Y_u}(d)$, namely $
\vartheta^{\leftarrow}_Y(\tau, d) := \inf\{ u \in \mathbb{R} :  \vartheta_{Y_u}(d) \geq \tau \},$
and define the LQTE-T:
\begin{equation}{\label{define:LQTE-T}}
(\vartheta^{\leftarrow}_Y(\tau, 1) - \vartheta^{\leftarrow}_Y(\tau, 0))_{\tau \in (0,1)}.
\end{equation}
Under conditional exogeneity LQTE and LQTE-T reduce to the quantile treatment effects (QTE)
and quantile treatment effects on the treated (QTE-T) \cite[Chap. 2]{koenker:book}.

\section{Estimation of Reduced-Form and Structural Parameters in a Data-Rich Environment}\label{EstimationSection}

The key objects used to define the structural parameters in Section 2 are the expectations
\begin{equation}
\alpha_V(z) = \Ep_P[ g_V(z,X)] \ \textrm{and} \ \gamma_V= \Ep_P[V],
\end{equation}
where $g_V(z,X)= \Ep_P[V|Z=z,X]$ and $V$ denotes a variable whose identity will change with the context.  Specifically, we shall vary $V$  over the set $\mathcal{V}_u$:
\begin{equation}\label{define: V}
V \in \mathcal{V}_u:= \{V_{uj}\}_{j=1}^5:=\{Y_u, \mathbf{1}_0(D) Y_u,  \mathbf{1}_0(D),\mathbf{1}_1(D) Y_u,  \mathbf{1}_1(D) \}.
\end{equation}

It is clear that $g_V(z,X)$ will play an important role in estimating $\alpha_V(z)$.  A related function that will also play an important role in forming a robust estimation strategy is the propensity score $m_Z: \mathcal{ZX} \longmapsto \mathbb{R}$ defined by
\begin{eqnarray}
&& m_{Z}(z,x) := \Pr_P[Z=z|X=x].
\end{eqnarray}
We will denote other potential values for the functions $g_{V}$ and $m_Z$ by the parameters $g$ and $m$, respectively.
We can then estimate $\alpha_V(z)$ by estimating $g_V$ and $m_Z$ using high-dimensional modeling and estimation methods.\footnote{Note that there is an alternative approach based on decomposing $g_V$ as $g_{V}(z,x) = \sum_{d=0}^1 e_{V}(d,z,x)  l_{D}(d,z,x)$ where the regression functions $e_{V}$ and $l_{D}$ map the support of $(D,Z,X)$, $\mathcal{DZX}$, to the real line and are defined by $e_{V}(d,z,x): = \Ep_P[V|D=d, Z=z,X=x]$ and $l_{D}(d,z,x): =  \Pr_P[D=d|Z=z, X=x]$.  We provide some discussion of this approach in the supplementary appendix.}

In the rest of this section, we describe the estimation of the reduced-form and structural parameters. The estimation method consists of 3 steps:

\textbf{1)} Estimate the predictive relationships $m_Z$ and $g_V$ 
using high-dimensional nonparametric methods with model selection.

\textbf{2)} Estimate the reduced form parameters $\alpha_V$ and $\gamma_V$ using orthogonal estimating equations to immunize the reduced form estimators to imperfect model selection in the first step.

\textbf{3)} Estimate the structural parameters and effects via the plug-in rule.

\subsection{First Step: Modeling and Estimating $g_{V}$ and $m_{Z}$}\label{FirstStepSection}
In this section, we discuss estimation of the conditional expectation functions $g_{V}$ and $m_{Z}$. Since these functions are unknown and potentially complicated, we use a generalized linear combination of a large number of control terms
\begin{equation}
f(X)= (f_j(X))_{j=1}^p,
\end{equation}
to approximate $g_{V}$ and $m_{Z}$.  Specifically, we use
\begin{eqnarray}\label{eq: approximations}
& & g_{V}(z,x)  =: \Lambda_V [ f(z,x)'\beta_{V}] + r_{V}(z,x) , \\
&& f(z,x):=  ( (1-z) f(x)', z f(x)')',  \ \  \beta_V := (\beta_V(0)', \beta_V(1)')',\\
& & m_{Z}(1,x)  =: \Lambda_Z [f(x) '\beta_{Z}] + r_{Z}(x),  \ \   m_{Z}(0,x) = 1-  \Lambda_Z [f(x) '\beta_{Z}] - r_{Z}(x).\label{eq: approximations4}
\end{eqnarray}
In these equations,  $r_{V}(z,x)$ and $r_{Z}(x)$ are approximation errors, and the functions $ \Lambda_V (f(z,x) '\beta_{V})$ and $\Lambda_Z(f(x)'\beta_{Z})$ are generalized linear approximations to the target functions
 $g_{V}(z,x)$  and $m_{Z}(1,x)$.  The functions $\Lambda_V$ and $\Lambda_Z$ are taken to be known link functions $\Lambda$.  The most common example is the linear link $\Lambda(u) = u$. When the response variable is binary, we may also use the logistic link $\Lambda(u) = \Lambda_0 (u) = e^u/(1+ e^u)$ and its complement $1- \Lambda_0(u)$ or the probit
link $\Lambda(u) = \Phi(u) = (2\pi)^{-1/2}\int_{-\infty}^u e^{-z^2/2}dz$ and its complement  $1-\Phi(u)$.  For clarity, we use links from the finite set $\mathcal{L}=\{ \Id, \Phi, 1-\Phi,  \Lambda_0, 1-\Lambda_0\}$ where $\Id$ is the identity (linear) link.

As discussed in the Introduction,
 the dictionary of controls, denoted by $ f(X)$, can be ``rich" in the sense that its dimension $p=p_n$ may be large relative to the sample size. Specifically, our results require only that $\log p = o(n^{1/3})$ along with other technical conditions.
We also note that the functions $f$ forming the dictionary can depend on $n$, but we suppress this dependence.

Having very many controls $f(X)$ creates a challenge for estimation and inference.  A useful condition that makes it possible to perform constructive estimation and inference in such cases is termed approximate sparsity or simply sparsity.  Sparsity imposes that there exist approximations of the form given in (\ref{eq: approximations})-(\ref{eq: approximations4}) that require only a small number of non-zero coefficients to render the approximation errors small relative to estimation error.  More formally, sparsity relies on two conditions.  First, there must exist $\beta_{V}$ and $\beta_{Z}$ such  that, for all $V \in \mV := \{\mV_u : u \in \mU\},$
\begin{equation}\label{eq: sparsity1-s1}
 \| \beta_{V}\|_0  + \| \beta_{Z}\|_0 \leq s,
\end{equation}
where $\|x \|_0$ is the number of non-zero components of vector $x$ and all other norms we use are defined in Appendix A. That is, there are at most $s=s_n \ll n$ components of $f(Z,X)$ and $f(X)$ with nonzero coefficient in the approximations to $g_V$ and $m_Z$. Second, the sparsity condition requires that the size of the resulting approximation errors is small compared to the conjectured size of the estimation error; namely, for all $V \in \mV$,
\begin{equation}\label{eq: sparsity2-s1}
\{\Ep_P[r^{2}_{V}(Z,X)]\}^{1/2} + \{\Ep_P[r^{2}_{Z}(X)]\}^{1/2} \lesssim  \sqrt{s/n}.
\end{equation}
Note that the size of the approximating model $s=s_n$ can grow with $n$ just as in standard series estimation, subject to the rate condition $$s^2 \log^2 (p\vee n) \log^2 n/n \to 0.$$ These conditions ensure that the functions $g_V$ and $m_Z$
are estimable at a $o(n^{-1/4})$ rate and are used to derive asymptotic normality results for the structural
and reduced-form parameter estimators.  They could be relaxed through the use of sample splitting methods
as in \citen{BellChenChernHans:nonGauss}.


The high-dimensional-sparse-model framework outlined above extends the standard framework in the program evaluation literature which assumes both that the identities of the relevant controls are known and that the number of such controls $s$ is small relative to the sample size.\footnote{For example, one would select a set of basis functions, $\{f_{j}(X)\}_{j=1}^{\infty}$, such as power series or splines and then use only the first $s \ll n$ terms in the basis under the assumption that $s^C/n \rightarrow 0$ for some number $C$ whose value depends on the specific context in a standard nonparametric approach using series.}  Instead, we  assume that there are many, $p$, potential controls of which at most $s$ controls suffice to achieve a desirable approximation to the unknown functions $g_V$ and $m_{Z}$; and we allow the identity and number of these controls to be unknown.  Relying on this assumed sparsity, we use selection methods to choose approximately the right set of controls.

Current estimation methods that exploit approximate sparsity employ different types of regularization aimed at producing estimators that theoretically perform well in high-dimensional settings while remaining computationally tractable.  Many widely used methods are based on $\ell_1$-penalization.  The Lasso method is one such commonly used approach that adds a penalty for the weighted sum of the absolute values of the model parameters to the usual objective function of an M-estimator.
A related approach is the Post-Lasso method which performs re-estimation of the model after selection of variables by Lasso.   These methods are discussed at length in recent papers and review articles; see, for example, \citen{BCH2011:InferenceGauss}.

In the following, we outline the general features of the Lasso and Post-Lasso methods focusing on estimation of $g_V$. Given the data  $(\tilde Y_i, \tilde X_i )_{i=1}^n = (V_i, f(Z_i, X_i))_{i=1}^n$, the Lasso estimator $\hat \beta_{V}$ solves
\begin{equation}\label{Def:LASSOmain2}
 \hat \beta_V\in \arg \min_{\beta \in \mathbb{R}^{\dim(\widetilde X)}} \Big ( \En[M(\tilde Y, \tilde X'\beta)] +  \frac{\lambda}{n} \|\hat \Psi \beta\|_1 \Big),
\end{equation} where $\hat \Psi = {\rm diag}(\hat l_1,\ldots,\hat l_{\dim(\widetilde X)})$ is a diagonal matrix of data-dependent
penalty loadings, $M(y, t) =  (y-t)^2/2$ in the case of linear regression, and
$M(y, t) = -\{1(y=1) \log \Lambda_V(t)+ 1(y=0) \log(1- \Lambda_V(t))\}$ in the case of binary regression.  
The penalty level, $\lambda$,
and loadings, $\hat l_j, \ j = 1,...,\dim(\widetilde X)$, are selected to guarantee good theoretical properties of the method.
We provide further discussion of these methods for estimation of a continuum of functions in Section \ref{FunctionalLassoSection}, and we specify detailed implementation algorithms used in the empirical example in Appendix \ref{sec:Implementation}.
A key consideration in this paper is that the penalty level needs to be set to account for the fact that we will be simultaneously estimating potentially a \textit{continuum} of Lasso regressions since our $V$ varies over the list $\mV_u$ with $u$ varying over the index set
$\mU$.

The Post-Lasso method uses $\widehat \beta_V$ solely as a model selection device.  Specifically, it makes use of the labels of the regressors with non-zero estimated coefficients,
$
\widehat I_V := \supp(\widehat \beta_V).
$
The Post-Lasso estimator is then a solution to
\begin{equation}\label{Def:PostLASSOmain2}
 \tilde \beta_V\in \arg \min_{\beta \in \mathbb{R}^{\dim(\widetilde X)}} \Big ( \En[M(\tilde Y, \tilde X'\beta)] :  \beta_j = 0, j  \not \in \widehat I_V \Big).
\end{equation}
 A main contribution of this paper is establishing that the estimator $\widehat g_{V}(Z,X) = \Lambda_V(f(Z,X)'\bar \beta_V)$ of  the regression function $g_{V}(Z,X)$, where $\bar \beta_V= \hat \beta_V$ or $\bar \beta_V= \tilde \beta_V$,  achieves the near oracle rate of convergence $\sqrt{(s \log p)/n} $ and maintains desirable theoretic properties while allowing for a \textit{continuum} of response variables.

Estimation of $m_Z$ proceeds similarly.  The Lasso estimator $\hat \beta_Z$ and Post-Lasso estimator $\tilde \beta_Z$ are defined analogously to $\hat \beta_V$ and $\tilde \beta_V$ using the data  $(\tilde Y_i, \tilde X_i)_{i=1}^n$= $( Z_i, f(X_i))_{i=1}^n$. The estimator $\widehat m_Z(1, X) = \Lambda_Z(f(X)'\bar \beta_Z)$ of $m_{Z}(X)$, with $\bar \beta_Z= \hat \beta_Z$ or $\bar \beta_Z= \tilde \beta_Z$, also achieves the near oracle rate of convergence $\sqrt{(s \log p)/n}$ and has other good theoretic properties.  The estimator of $\widehat m_Z(0,X)$ is then formed as  $1 - \widehat m_Z(1,X)$.


\subsection{Second Step: Robust Estimation of  the Reduced-Form Parameters $\alpha_V(z)$ and $\gamma_V$}
 Estimation of the key quantities $\alpha_V(z)$ will make heavy use of orthogonal
moment functions as defined in (\ref{LowBias}).  These moment functions are closely tied to efficient influence functions, where efficiency is in the sense of locally minimax semi-parametric efficiency.  The use of these functions
will deliver robustness with respect to the non-regularity of the
post-selection and penalized estimators needed to manage high-dimensional data.  The use of these functions
also automatically delivers semi-parametric efficiency for estimating and performing inference on the reduced-form parameters and their smooth transformations -- the structural parameters.

The efficient influence function and orthogonal moment function for $\alpha_V(z)$,  $z \in \mathcal{Z} = \{0,1\}$,  are given
respectively by
\begin{align}
\label{EffInfluence}
\psi^{\alpha}_{V,z} (W) &:= \psi^{\alpha}_{V, z, g_{V},m_Z} (W, \alpha_V(z)) \ \ \textrm{and} \\
\label{LowBiasMoment}
\psi^{\alpha}_{V, z, g, m}(W,\alpha) &:= \frac{1(Z=z) (V - g(z,X))}{m(z,X)}  + g(z,X) - \alpha.
\end{align}
 This efficient influence function was derived by \citen{hahn-pp}; it has recently been used by \citen{cattaneo2010efficient} in the  series context (with $p \ll n$) and \citen{rothe_firpo2013} in the kernel context. The efficient influence function and the moment function for $\gamma_V$ are trivially given by
\begin{eqnarray}
&&\psi^{\gamma}_{V} (W) := \psi^{\gamma}_{V} (W,\gamma_V), \ \textrm{and} \  \psi^{\gamma}_{V} (W,\gamma) :=V - \gamma.
 \end{eqnarray}

We then define estimators of the reduced-form parameters $\alpha_V(z)$  and $\gamma_V(z)$ as solutions $\alpha = \hat \alpha_V(z)$ and $\gamma= \hat \gamma_V$
to the equations
\begin{eqnarray}
\En[\psi^{\alpha}_{V, z, \hat g_{V}, \hat m_Z} (W, \alpha) ] = 0, \ \ \En[\psi^{\gamma}_{V} (W, \gamma) ] = 0,
\end{eqnarray}
where $\hat g_V$ and $\hat m_Z$  are constructed as in Section \ref{FirstStepSection}.  
We apply this procedure
to each  variable name $V \in \mV_u$ and obtain
the estimator\footnote{By default notation, $(a_j)_{j \in \mathcal{J}}$ returns a column vector
produced by stacking components together in some consistent order. }
\begin{eqnarray}\label{define rho}
\hat \rho_u := \big (\{\hat \alpha_{V} (0), \hat \alpha_{V} (1),  \hat \gamma_V\} \big )_{V \in \mV_u} \  \ \text { of }
\ \   \rho_u := \big (\{ \alpha_{V} (0),  \alpha_{V} (1),   \gamma_V\} \big )_{V \in \mV_u}.\end{eqnarray}
The estimator and the parameter are vectors in $\mathbb{R}^{d_\rho}$ with dimension $d_{\rho} = 3 \times \dim \mV_u = 15$.

In the next section, we formally establish a principal result which shows that
\begin{align}\label{result:normal}
\sqrt{n} (\hat \rho_u - \rho_u)  \rightsquigarrow  {N}(0, \Var_{P} ( \psi^\rho_u)), \ \ \psi^\rho_{u}:= (\{\psi^\alpha_{V,0},
\psi^\alpha_{V,1}, \psi^\gamma_{V}\})_{ V \in \mV_u},   \ \text{  uniformly in $P \in \mP_n$, }
 \end{align}
where  $\mP_n$ is a rich set of data generating processes $P$ which includes
cases where perfect model selection is impossible theoretically. The notation ``$Z_{n,P}  \rightsquigarrow Z_P$  uniformly  in $P \in \mP_n$" is defined formally in Appendix \ref{subsec:notation} and can be read as ``$Z_{n,P}$ is approximately distributed as $Z_P$ uniformly  in $P \in \mP_n$." This usage corresponds  to the usual notion of asymptotic distribution extended to handle uniformity in $P$.

We then stack all the reduced form estimators and  parameters over $u \in \mU$ as
$$
\hat \rho = (\hat \rho_u)_{u \in \mU} \ \text{ and } \ \  \rho = (\rho_u)_{u \in \mU},
$$
giving rise to the empirical reduced-form process $\hat \rho$ and the reduced-form function-valued parameter $\rho$.
We establish that  $\sqrt{n}(\hat \rho - \rho)$ is asymptotically Gaussian: In $\ell^\infty(\mU)^{d_\rho}$,
\begin{eqnarray}
& &\sqrt{n}(\hat \rho - \rho)  \rightsquigarrow Z_P := (\mathbb{G}_{P}   \psi^\rho_{u})_{u \in \mathcal{U}},  \text{ uniformly in $P \in \mP_n$},
 \end{eqnarray}
where $\mathbb{G}_{P}$ denotes the $P$-Brownian bridge \cite[p. 81--82]{vdV-W}.
This result contains (\ref{result:normal}) as a special case and again allows $\mP_n$ to be a ``rich" set of data generating processes $P$ that includes cases where perfect model selection is impossible theoretically.  Importantly, this result verifies that the functional central limit theorem applies to the reduced-form estimators in the presence of possible model selection mistakes.

Since some of our objects of interest are complicated, inference can be facilitated by a multiplier bootstrap method as in \cite{GineZinn1984}. We define  $\hat \rho^* = (\hat \rho_u^*)_{u \in \mU}$, a bootstrap draw of $\hat \rho$, via
\begin{eqnarray}\label{BootstrapDraw}
&&  \hat \rho^*_u = \hat \rho_u + n^{-1} \sum_{i=1}^n  \xi_i \hat \psi^{\rho}_u(W_i).
\end{eqnarray}
Here  $(\xi_i)_{i=1}^n$ are i.i.d. copies of $\xi$ which are independently distributed
from the data $(W_i)_{i=1}^n$ and whose distribution $P_\xi$ does not depend on $P$.  We also impose that
\begin{equation}\label{eq: multipliers}
\Ep [\xi] = 0,  \ \ \Ep[\xi^2]=1, \ \ \Ep[\exp(|\xi|)] < \infty.
\end{equation}
Examples of $\xi$ include (a) $\xi = \mathcal{E}-1$, where
$\mathcal{E}$ is a standard exponential random variable,
(b) $\xi = \mathcal{N}$, where $\mathcal{N}$ is a standard
normal random variable, and (c) $\xi =  \mathcal{N}_1/\sqrt{2} +
(\mathcal{N}_2^2-1)/2$, where $\mathcal{N}_1$ and $\mathcal{N}_2$ are mutually independent  standard
normal random variables.\footnote{We do not consider the nonparametric bootstrap, which corresponds to using multinomial multipliers $\xi$, to reduce the length of the paper; but we note that the conditions and analysis could be extended to cover this case.}  The choices of (a), (b), and (c) correspond respectively to the Bayesian bootstrap (e.g.,  \citen{Hahn97} and
 \citen{chamberlain2003nonparametric}), the Gaussian multiplier method (e.g, \citen{GineZinn1984} and \citen[Chap. 3.6]{vdV-W}), and  the wild bootstrap method (\cite{mammen1993:bootstrap}).\footnote{ The motivation for method (c) is that it is able to match 3 moments since $\Ep[\xi^2] = \Ep[\xi^3]=1$.  Methods (a) and (b) do not satisfy this property since $\Ep[\xi^2] = 1$ but $\Ep[\xi^3]\neq 1$ for these approaches.
}
$\hat \psi^\rho_{u}$ in (\ref{BootstrapDraw})
is an estimator of the influence function $\psi^{\rho}_u$ defined via the plug-in rule:
\begin{equation}\label{estimate psi}
\hat \psi^\rho_{u}=    (\hat \psi^\rho_{V})_{ V \in \mV_u}, \ \
 \hat \psi^{\rho}_{V} (W) :=\{\psi^{\alpha}_{V, 0, \hat g_{V}, \hat m_{Z}}(W,\hat \alpha_{V}(0)), \psi^{\alpha}_{V, 1, \hat g_{V}, \hat m_{Z}}(W,\hat \alpha_{V}(1)), \psi^{\gamma}_{V}(W, \hat \gamma_{V})\}.
 \end{equation}
 Note that this bootstrap is computationally
efficient since it does not involve recomputing the influence
functions $\hat \psi_u^{\rho}$.\footnote{\citen{CH06} and \citen{HS2006} proposed related computationally efficient bootstrap schemes that resample the influence functions.} 
Each new draw of $(\xi_i)_{i=1}^n$ generates a new draw of $\hat \rho^*$
holding the data and the estimates of the influence functions fixed. This method
 simply amounts to resampling the first-order approximations to the estimators.
Here we build upon  prior uses of this or similar methods in low-dimensional settings such as \citen{Hansen96}  and \citen{KS12}.

We establish that the bootstrap law of $\sqrt{n} (\hat \rho^*- \hat \rho)$ is uniformly asymptotically consistent:
In the metric space $\ell^\infty(\mU)^{d_\rho}$, conditionally on the data,
\begin{eqnarray}
& &\sqrt{n} (\hat \rho^* - \hat \rho)\rightsquigarrow_B Z_P,  \text{ uniformly in $P \in \mP_n$}, \nonumber
 \end{eqnarray}
where $\rightsquigarrow_B$ denotes weak convergence of the bootstrap law in probability, as defined in Appendix \ref{app: B}.

\subsection{Third Step: Robust Estimation of the Structural Parameters}

All structural parameters we consider take the form of smooth transformations of the reduced-form parameters:
\begin{equation}  \label{eq: functional_of_interest}
\Delta := (\Delta_q)_{q \in \mathcal{Q}},  \text{ where }  \Delta_q := \phi ( \rho)(q),  \ q \in \mathcal{Q}.\end{equation}
The structural parameters may themselves carry an index $q \in \mathcal{Q}$ that can be different from $u$; for example,
the LQTE is indexed by a quantile index $q \in (0,1)$. This formulation includes as special
cases all the structural functions of Section \ref{sec: model}.   We estimate these quantities
by the plug-in rule.  We establish the asymptotic behavior of these estimators and the validity of the bootstrap as
a corollary from the results outlined in Section 3.2 and the functional delta method (extended to handle uniformity in $P$).

For the application of the functional delta method, we require that the functional $\rho \longmapsto \phi(\rho)$ be  Hadamard
differentiable \textit{uniformly} in $\rho \in \D_{\rho}$, where $\D_{\rho}$ is a set that contains the true values $\rho= \rho_P$ for all $P \in \mP_n$, tangentially to a subset that contains the realizations of $Z_P$ for all $P \in \mP_n$ with derivative map $h \longmapsto \phi'_\rho(h) = (\phi'_{\rho}(h)(q))_{q \in \mathcal{Q}}$.\footnote{We give the definition of  uniform Hadamard differentiability in Definition \ref{def:uhd} of Appendix \ref{app: B}.} We define the
estimators of the structural parameters and their bootstrap versions via the plug-in rule as
\begin{equation}  \label{eq: estimate functional_of_interest}
\hat \Delta := (\hat \Delta_q)_{q \in \mQ},  \ \ \hat \Delta_q:= \phi\left( \hat \rho\right)(q), \text{ and } \hat \Delta^*:=(\hat \Delta^*_q)_{q \in \mQ},   \ \ \hat \Delta^*_q:= \phi\left(\widehat \rho^*\right)(q).
\end{equation}
We establish that these estimators are asymptotically Gaussian
\begin{equation}
\sqrt{n} (\hat \Delta - \Delta)   \rightsquigarrow  \phi'_\rho (Z_P) ,  \text{  uniformly in $P \in \mP_n$},
\end{equation}
and that the bootstrap consistently estimates their large sample distribution:
\begin{equation}
\sqrt{n} (\hat \Delta^* - \widehat \Delta)  \rightsquigarrow_B \phi'_\rho (Z_P),  \text{  uniformly in $P \in \mP_n$. }
\end{equation}
These results can be used to construct simultaneous
confidence bands and test functional hypotheses on $\Delta$ using the methods described for example in   \citen{CF15} and \citen{CFM}.

\section{Theory: Estimation and Inference on Local Treatment Effects Functionals}\label{sec: asymptotics}

Consider fixed sequences of numbers $\delta_n \searrow 0$, $\epsilon_n \searrow 0$, $\Delta_n \searrow 0$, at a speed at most polynomial in $n$ (for example, $\delta_n \geq 1/n^c$ for some $c > 0$),  $\ell_n := \log n$, and positive constants $c$, $C$, and $c' <1/2$. These sequences and constants will not vary with $P$.    The probability $P$ can vary in the set $\mP_n$ of probability measures, termed ``data-generating processes", where $\mP_n$ is typically a set that is weakly increasing in $n$, i.e. $\mP_n \subseteq \mP_{n+1}$.  Other definitions and notation are collected in Appendix A.

\begin{assumption}[Basic Assumptions] \label{assumption: basic} (i) Consider a random element $W$  with values in a measure space $(\mathcal{W}, \mathcal{A}_\mathcal{W})$ and  law determined by a probability measure $P \in \mP_n$. The observed data $((W_{ui})_{u \in \mathcal{U}})_{i=1}^{n}$ consist of $n$ i.i.d. copies of a random element $(W_u)_{u \in \mathcal{U}}= ( (Y_u)_{u \in \mU}, D, Z, X)$,  where $\mU$ is a Polish space equipped with its Borel sigma-field and $(Y_u, D, Z, X)  \in \mathbb{R}^{3 +d_X}$.   Each $W_u$ is generated via a measurable transform
$t(W,u)$ of $W$ and $u$, namely the map $t: \mW \times \mU \longmapsto \mathbb{R}^{3+ d_X}$ is measurable,
and the map can possibly depend on $P$. Let  $$\mV_u:= \{V_{uj}\}_{j \in \mathcal{J}}:=\{Y_u, \mathbf{1}_0(D) Y_u,  \mathbf{1}_0(D), \mathbf{1}_1(D) Y_u,  \mathbf{1}_1(D) \}, \ \ \mathcal{V} := (\mV_u)_{u \in \mU}, $$ where $\mathcal{J} = \{1,...,5\}$.   (ii) For $\mP := \cup_{n=n_0}^{\infty} \mP_n$, the map $u \longmapsto Y_{u}$ obeys the uniform continuity property: $$\lim_{\epsilon \searrow 0}\sup_{P \in \mP} \sup_{ d_{\mU} (u , \bar u) \leq \epsilon } \| Y_{u} - Y_{\bar u}\|_{P,2} = 0,  \  \sup_{P \in \mP } \Ep_P \sup_{u \in \mU}|Y_{u}|^{2+c} < \infty,$$ 
where the second supremum in the first expression is taken over $u, \bar u \in \mU$, and $\mU$ is a totally bounded metric space equipped with a semi-metric $d_{\mU}$. The uniform covering entropy
of  the set  $\mathcal{F}_P= \{Y_u : u \in \mU\}$, viewed as a collection of maps $(\mW, \mA_{\mW}) \longmapsto \mathbb{R}$, obeys
$$\sup_Q  \log N(\epsilon \|F_P\|_{Q,2}, \mathcal{F}_P, \| \cdot \|_{Q,2}) \leq C  \log(\mathrm{e}/\epsilon) \vee 0 $$ for all $P \in \mP$, where  $F_P(W)= \sup_{u \in \mU} |Y_u|$, with the supremum taken over all finitely discrete probability measures $Q$ on $(\mW, \mA_{\mW})$. (iii) For each $P \in \mP$, the conditional probability of $Z=1$ given $X$ is bounded away from zero or one, namely  $ c' \leq m_{Z}(1,X) \leq 1-c'$ $P$-a.s.,  the instrument $Z$ has a non-trivial impact on $D$, namely
$c' \leq |\Pr_P[D=1|Z=1, X] -\Pr_P[D=1|Z=0, X]|$ $P$-a.s, and the regression function $g_V$ is bounded, $\|g_V\|_{P,\infty} < \infty$ for all $V \in \mV$.

\end{assumption}

Assumption \ref{assumption: basic}  is stated  to deal
with the measurability issues associated with functional response data. This assumption also implies that the set of functions $(\psi_u^{\rho})_{ u \in \mU}$, where $$\psi^\rho_{u}:= (\{\psi^\alpha_{V,0}, \psi^\alpha_{V,1}, \psi^\gamma_{V}\})_{ V \in \mV_u},$$ is $P$-Donsker uniformly in $\mP$.  That is, it implies
\begin{eqnarray}
&& Z_{n,P} \rightsquigarrow  Z_{P}  \ \  \text{ in $\ell^\infty(\mU)^{d_\rho}$,  uniformly in $P \in  \mP$},\\
&& \label{define: ZP} Z_{n, P} :=   (\mathbb{G}_{n}   \psi^\rho_{u})_{u \in \mathcal{U}}
\text{ and } Z_{P} :=   (\mathbb{G}_{P}   \psi^\rho_{u})_{u \in \mathcal{U}},
\end{eqnarray}
with $\mathbb{G}_{P}$ denoting the $P$-Brownian bridge \cite[p. 81--82]{vdV-W}
and with $Z_P$ having bounded, uniformly continuous paths uniformly in $P \in \mP$:
\begin{equation}\label{eq: nice paths}
\sup_{P \in  \mP} \Ep_P\sup_{u \in \mU} \|Z_P(u)\| < \infty, \ \ \lim_{\varepsilon \searrow 0} \sup_{P \in \mP}
\Ep_P \sup_{ d_{\mathcal{U}}(u, \tilde u) \leq \varepsilon}  \|Z_P(u) - Z_P(\tilde u)\| = 0.
\end{equation}


We work with the sequence of constants defined prior to Assumption \ref{assumption: basic}.

\begin{assumption}[Approximate Sparsity]\label{ass: sparse1}   Under each $P \in \mP_n$ and for each $n \geq n_0$, uniformly for all $V \in \mV $:  (i) The approximations (\ref{eq: approximations})-(\ref{eq: approximations4}) hold with the link
functions $\Lambda_V$ and $\Lambda_Z$ belonging to the set $\mathcal{L}$, the sparsity condition $\|\beta_{V}\|_0  + \|\beta_{Z} \|_0 \leq s$ holding, the approximation errors satisfying $\|r_{V}\|_{P,2} + \|r_{Z}\|_{P,2} \leq \delta_n n^{-1/4}$ and $\|r_{V}\|_{P,\infty} + \|r_{Z}\|_{P,\infty}  \leq \epsilon_n $, and the sparsity index $s$ and the number of terms $p$ in the vector $f(X)$ obeying $s^2 \log^2 (p\vee n) \log^2 n  \leq \delta_n n$. (ii) There are estimators $\bar \beta_V$ and $\bar \beta_{Z}$ such that, with probability  no less than $1- \Delta_n$, the estimation errors satisfy $\|f(Z,X)'(\bar \beta_{V} - \beta_{V})\|_{\Pn,2} + \|f(X)'(\bar \beta_{Z} - \beta_{Z})\|_{\Pn,2} \leq \delta_n n^{-1/4}$,  $K_n \|\bar \beta_{V} - \beta_{V}\|_1 + K_n \|\bar \beta_{Z} - \beta_{Z}\|_1 \leq \epsilon_n $; the
estimators  are sparse such that $\|\bar \beta_{V}\|_0  + \|\bar \beta_{Z}\|_0 \leq Cs$; and the empirical and population norms induced by the Gram matrix formed by $(f(X_i))_{i=1}^n$ are equivalent on sparse subsets, $\sup_{\|\delta\|_{0} \leq \ell_n s} \left | \| f(X) '\delta\|_{\Pn,2}/\|f(X)'\delta\|_{P,2} -1 \right | \leq \epsilon_n$. (iii)  The following boundedness conditions hold: $\|\|f(X)\|_\infty ||_{P, \infty} \leq K_n$ and   $\|V\|_{P,\infty} \leq C$. \end{assumption}

\begin{remark} Assumption \ref{ass: sparse1} imposes  simple intermediate-level conditions which encode
both the approximate sparsity of the models as well as some
reasonable behavior of the sparse estimators of $m_Z$ and $g_V$. These conditions significantly extend and generalize the conditions employed in the literature on adaptive estimation using series methods. The boundedness conditions are made to simplify arguments, and they could be removed at the cost of more complicated proofs and more stringent side conditions.  Sufficient conditions for the equivalence between empirical and population norms and primitive examples of functions admitting sparse approximations are given in \citen{BelloniChernozhukovHansen2011}.  We provide primitive conditions for Lasso estimators to satisfy the bounds above while addressing the problem of estimating continua of approximately sparse nuisance functions in Section 6.  We expect that other sparsity-based estimators, such as the Dantzig selector or adaptive Lasso, could be used in the present context as well. 
\qed \end{remark}

Under the stated assumptions, the empirical reduced form process  $ \hat Z_{n,P} = \sqrt{n}(\hat \rho - \rho)$ defined by (\ref{define rho})
obeys the following relations.  We recall definitions of convergence uniformly in $P \in \mathcal{P}_n$ in Appendix \ref{subsec:notation}.

\begin{theorem}[\textbf{Uniform Gaussianity of the Reduced-Form Parameter Process}]\label{theorem1}
 Under Assumptions \ref{assumption: basic} and \ref{ass: sparse1},  the reduced-form empirical process admits a linearization; namely,
\begin{eqnarray}
& & \hat Z_{n,P}:= \sqrt{n}(\hat \rho - \rho) = Z_{n,P} + o_{P}(1) \ \  \text{ in $\ell^\infty(\mU)^{d_\rho}$,  uniformly in $P \in \mP_n$}.
 \end{eqnarray}
 The process $\hat Z_{n,P}$ is asymptotically Gaussian, namely
 \begin{eqnarray}
\hat Z_{n,P} \rightsquigarrow  Z_{P}  \ \  \text{ in $\ell^\infty(\mU)^{d_\rho}$,  uniformly in $P \in \mP_n$},
\end{eqnarray}
where $Z_P$ is defined in (\ref{define: ZP}) and its paths obey the  property  (\ref{eq: nice paths}). \end{theorem}

Another main result of this section shows that the bootstrap law of the process
$$
\hat Z^*_{n,P}  := \sqrt{n} (\hat \rho^* - \hat \rho) := \frac{1}{\sqrt{n}} \sum_{i=1}^n  \xi_i \hat \psi^{\rho}_u(W_i),
$$
where $\hat \psi^\rho_u$ is defined in (\ref{estimate psi}), provides a valid approximation to the large sample law of $\sqrt{n}(\hat \rho - \rho)$.

\begin{theorem}[\textbf{Validity of Multiplier Bootstrap for
Inference on Reduced-Form  Parameters}]\label{theorem2}
 Under Assumptions \ref{assumption: basic} and \ref{ass: sparse1}, the bootstrap law consistently approximates the large sample
 law $Z_P$ of $Z_{n,P}$ uniformly in $P \in \mP_n$, namely,
 \begin{eqnarray}
\hat Z^*_{n,P} \rightsquigarrow_B  Z_{P}  \ \  \text{ in $\ell^\infty(\mU)^{d_\rho}$,  uniformly in $P \in \mP_n$}.
\end{eqnarray}
\end{theorem}

Next we consider inference on the structural functionals $\Delta$ defined in (\ref{eq: functional_of_interest}).
We derive the large sample distribution of the estimator $\hat \Delta$  in (\ref{eq: estimate functional_of_interest}),
 and show that the multiplier bootstrap  law of $\hat \Delta^*$  in (\ref{eq: estimate functional_of_interest})
 provides a consistent approximation to that distribution.   We rely on the functional delta method in our derivations, which we modify to handle uniformity with respect to the underlying d.g.p. $P$.  Our argument relies on the following assumption on the structural functionals.

\begin{assumption}[Uniform Hadamard Differentiability of Structural Functionals] \label{ass: uhd} Suppose that for each $P \in  \mP$,  $\rho= \rho_P \in \D_\rho$, a compact metric space.  Suppose $\varrho \longmapsto \phi(\varrho) $, a functional of interest mapping
 $\mathbb{D}_{\phi} \subset \mathbb{D}= \ell^{\infty}(\mU)^{d_\rho}$ to $\ell^{\infty}( \mathcal{Q})$,  where $\D_\rho \subset \D_\phi$,
 is  Hadamard differentiable in $\varrho$ tangentially to $\mathbb{D}_0 = UC(\mU)^{d_\rho}$ uniformly in $\varrho \in \mathbb{D}_{\rho}$, with the linear derivative map $\phi^{\prime}_{\varrho}: \D_0 \longmapsto \D$ such  that the mapping $(\varrho, h) \longmapsto \phi'_{\varrho}(h)$  from $\mathbb{D}_\rho \times \mathbb{D}_0$ to $\ell^{\infty}(\mQ)$ is continuous. \end{assumption}
The definition of  uniform Hadamard differentiability is given in Definition \ref{def:uhd} of Appendix \ref{app: B}. Assumption \ref{ass: uhd} holds for all examples of structural parameters  listed in Section 2.   \\

The following corollary gives the large sample law of $\sqrt{n} (\hat \Delta - \Delta)$,
the properly normalized structural estimator.
It also shows that the bootstrap law of
$
 \sqrt{n} (\hat \Delta^* - \hat \Delta),
$
computed conditionally on the data, approaches the large sample law $\sqrt{n} (\hat \Delta - \Delta)$.
It follows from the previous theorems as well as from a more general result contained in Theorem \ref{theorem3}.

\begin{corollary}[\textbf{Limit Theory and Validity of Multiplier Bootstrap for Smooth Structural Functionals}]
\label{corollary: bs LATE} Under Assumptions \ref{assumption: basic}, \ref{ass: sparse1}, and \ref{ass: uhd},
\begin{equation}  \label{eq:results on functionals}
\sqrt{n} (\hat \Delta - \Delta)   \rightsquigarrow T_P:= \phi'_{\rho_P} (Z_P),  \ \  \text{ in $\ell^\infty(\mathcal{Q})$,  uniformly in $P \in \mP_n$},
\end{equation}
where $T_P$ is a zero mean tight Gaussian process,  for each $P \in \mP$.  Moreover,
\begin{equation}  \label{eq:results on functionals}
\sqrt{n} (\hat \Delta^* - \hat \Delta)  \rightsquigarrow_B T_P 
,  \ \  \text{ in $\ell^\infty(\mathcal{Q})$,  uniformly in $P \in \mP_n$}.
\end{equation}

\end{corollary}









 \section{General Theory:  Honest Inference in General Moment Condition Problems
 with  Nuisance Functions Estimated by Machine Learning Methods}\label{sec: general}

 In this section, we consider a general moment condition framework, where possibly a continuum of target parameters is of interest and we use modern machine learning methods, with Lasso-type methods being a lead example, to estimate a continuum of high-dimensional nuisance functions.  This setting covers a rich variety of  modern moment-condition problems in econometrics including  the  treatment effects problem.  We establish a functional central limit theorem for the estimators of the continuum of target parameters that holds uniformly in $P \in \mathcal{P}$, where $\mathcal{P}$ includes a wide range of data-generating processes with well-approximable continuums of nuisance functions. We also derive a functional central limit theorem for the multiplier bootstrap that resamples the first order approximations to the standardized estimators of the continuum of target parameters and establish its uniform validity. Moreover, we establish the uniform validity of the functional delta method and the functional delta method for the multiplier bootstrap for smooth functionals of the continuum of target parameters using an appropriate strengthening of Hadamard differentiability.

\subsection{Setting}   We are interested in function-valued target parameters indexed  by $u \in \mathcal{U} \subset \mathbb{R}^{d_u}$.  We denote  the true value of the target parameter by   $$\theta^0 = (\theta_u)_{u \in \mathcal {U}}, \text{ where } \theta_u \in \Theta_u \subset \Theta \subset \mathbb{R}^{d_\theta}, \text{ for each }
 u \in \mathcal{U}. $$
 We assume that for each $u \in \mathcal{U},$ the true value $\theta_u$ is identified as the solution to the following moment condition:
 \begin{equation}\label{eq:ivequation}
 \Ep_P[ \psi_u(W_u, \theta_u, h_u(Z_u) ) ] = 0,
 \end{equation}
where   $W_u$ is a random vector that takes values in a Borel set $\mathcal{W}_u \subset \mathbb{R}^{d_w}$ and contains as a subcomponent the vector $Z_u$ taking values in a Borel set $\mathcal{Z}_u$, the moment function
\begin{equation}\label{def:psi_u}
\psi_u:  \mathcal{W}_u \times \Theta_u \times T_u \longmapsto \mathbb{R}^{d_\theta},  \ \  (w, \theta, t) \longmapsto \psi_u(w, \theta, t) = (\psi_{uj}(w, \theta, t))_{j=1}^{d_\theta}\end{equation}
 is a Borel measurable map, and the function
 \begin{equation}\label{def:h_u}
h_u:  \mathcal{Z}_u \longmapsto  \mathbb{R}^{d_t},  \  \ \ z \longmapsto h_u(z) = (h_{um}(z))_{m=1}^{d_t}  \in   T_u(z),
 \end{equation}
is another Borel measurable map that denotes the possibly infinite-dimensional nuisance parameter.  The sets $T_u(z)$ are assumed to be convex for each $u \in \mU$ and $z \in \mZ_u$. Finite-dimensional nuisance parameters that do not depend on $Z_u$ are treated as part of $h_u$ as well.

We assume that the continuum of nuisance functions $(h_{u})_{u \in \mathcal{U}}$ is well-approximable and can be well estimated by the modern generation of statistical and machine learning methods.  In particular, our regularity conditions allow for
approximately sparse nuisance functions, which can be modeled and estimated using methods such as Lasso and Post-Lasso.   We let
 $\hat h_u= (\hat h_{um})_{m=1}^{d_t}$ denote the estimator of $h_u$, which we assume obeys the conditions in Assumption \ref{ass: AS}.   The estimator $\hat \theta_u$ of $\theta_u$ is constructed as any approximate $\epsilon_n$-solution  in $\Theta_u$ to a sample analog of the moment condition (\ref{eq:ivequation}), i.e.,
\begin{equation}\label{eq:analog}
\|\En[ \psi_u(W_u, \hat \theta_u, \hat h_u(Z_u) ) ] \| \leq \inf_{\theta \in \Theta_u}\|\En[ \psi_u(W_u, \theta, \hat h_u(Z_u) ) ] \| + \epsilon_n,   \text{ where }  \epsilon_n = o(n^{-1/2}).
\end{equation}

\begin{remark}[Handling Over-identified Cases] We do not analyze
over-identified cases explicitly, but it is helpful to note that they can be handled
within the current framework. Let
 $\psi^o_u (W_u, \theta, h^o_u(Z_u))$ be the original over-identifying moment function. Let $A_u(Z_u)$ denote the pointwise optimal matrix of linear combinations of the moments, so that the final moment function $\psi_u (W_u, \theta, h(Z_u)) = A_u(Z_u) \psi^o_u  (W_u, \theta, h^o_u(Z_u))$ has the same dimension as $\theta_u$.  Here $h_u(Z_u) = (\text{vec}(A_u(Z_u))', $ $ {h^o}'_u(Z_u))'$; that is, we simply treat $A_u$ as part of the nuisance function $h_u$ being estimated.  We do not analyze the preliminary estimation of $A_u$ in the present paper in order to maintain the focus on exactly identified cases as in Section 4.   \qed
\end{remark}

\subsection{The Neyman Orthogonality or Immunization Condition} A key condition needed for regular estimation of $\theta_u$ is an orthogonality or immunization condition.
The simplest to explain, yet  strongest, form of this condition can be expressed as follows:
  \begin{equation}\label{eq:orthogonality0}
\partial_t  \Ep_P[  \psi_u(W_u, \theta_u, t)  |Z_u] \Big|_{t = h_u(Z_u)}= 0, \text{ a.s.,}
\end{equation}
subject to additional technical conditions such as continuity (\ref{eq:contin}) and dominance (\ref{eq:dom}) stated below,
where we use the symbol $\partial_t$ to abbreviate  $\frac{\partial}{\partial t'}$.
This condition holds in the previous setting of inference on treatment effects after interchanging the order of the derivative and expectation. The formulation here also  covers certain non-smooth cases such as  structural and instrumental quantile regression problems.

In the formal development, we use a more general form of the orthogonality condition.
\begin{definition}[\textbf{Neyman Orthogonality for Moment Condition Models, General Form}]  For each $u \in \mathcal{U}$,
 suppose that (\ref{eq:ivequation})--(\ref{def:h_u}) hold.   Consider $\mathcal{H}_u$, a  set of measurable functions $z \longmapsto h(z) \in T_u(z) $ from $\mathcal{Z}_u$ to $\mathbb{R}^{d_t}$ such that $ \| h (Z_u) - h_u(Z_u)\|_{P,2} < \infty$ for all $h \in \mH_u$.    Suppose also that the set $T_u(z)$ is a convex subset of $\mathbb{R}^{d_t}$ for each $z \in \mZ_u$.  We say that $\psi_u$ obeys a general form of orthogonality with respect to $\mH_u$ uniformly in $u \in \mathcal{U}$ if the following conditions hold: For each $u \in \mathcal{U}$, the derivative
\begin{equation}\label{eq:contin} t \longmapsto \partial_t  \Ep_P[  \psi_u(W_u, \theta_u, t)  |Z_u] \text{ is continuous on $t \in T_u(Z_u)$ $P$-a.s.;}
\end{equation}  is dominated, \begin{equation}\label{eq:dom} \left \|\sup_{t \in T_u(Z_u)}  \Big \|  \partial_t  \Ep_P[ \psi_u(W_u, \theta_u, t)  |Z_u] \Big \| \right \|_{P,2}< \infty;
\end{equation}
and obeys the orthogonality condition:
  \begin{equation}\label{eq:orthogonality}
\Ep_P \Big [  \partial_t  \Ep_P\big [  \psi_u(W_u, \theta_u, h_u(Z_u))  |Z_u\big ]  (h (Z_u) - h_u(Z_u))  \Big ]= 0 \ \  \text{ for all } h \in \mH_u.
\end{equation}
\end{definition}

\noindent The orthogonality condition (\ref{eq:orthogonality})  reduces to (\ref{eq:orthogonality0}) when $\mH_u$ can span all measurable functions $h: \mZ_u  \longmapsto T_u$  such that $\|h\|_{P,2} < \infty$ but is more general otherwise.

\begin{remark}[\textbf{An alternative formulation of the Neyman orthogonality condition}]  A slightly more general, though less primitive definition of the orthogonality condition is as follows.  For each $u \in \mathcal{U}$,
 suppose that (\ref{eq:ivequation})- (\ref{def:h_u}) hold.   Consider $\mathcal{H}_u$, a  set of measurable functions $z \mapsto h(z) \in T_u(z)$ from $\mathcal{Z}_u$ to $\R^{d_t}$ such that $ \| h (Z_u) - h_u(Z_u)\|_{P,2} < \infty$ for all $h \in \mH_u$,  where the set $T_u(z)$ is a convex subset of $\R^{d_t}$ for each $z \in \mZ_u$.  We say that $\psi_u$ obeys a general form of orthogonality with respect to $\mH_u$ uniformly in $u \in \mathcal{U}$, if the following conditions hold: The Gateaux derivative map
 $$
  \mathrm{D}_{u,t}[h - h_u]:=  \partial_t  \Ep_P \Bigg (  \psi_u \Big\{ W_u, \theta_u, h_u(Z_u)+ t \Big [h (Z_u) - h_u(Z_u)\Big] \Big\}   \Bigg )
  $$
  exists for all $t \in [0,1)$, $h \in \mH_u$, and $u \in \mU$ and vanishes at $t=0$ -- namely,
  \begin{equation}\label{eq:cont}
 \mathrm{D}_{u,0}[h - h_u] = 0  \ \  \text{ for all } h \in \mH_u.
\end{equation}
Definition 5.1 implies this definition by the mean-value expansion and the dominated convergence theorem. \qed
\end{remark}

\begin{remark}[\textbf{Orthogonalization typically expands the number of nuisance parameters}]
It is important to use a moment function $\psi_u$ that satisfies the orthogonality property given in (\ref{eq:orthogonality}); see examples given below.  Generally, if we  have a
moment function $\tilde \psi_u$ which identifies  $\theta_u$ but does not have this property,  we can construct a moment function $\psi_u$ that identifies $\theta_u$ and has the required orthogonality property by projecting the original function $\tilde \psi_u$ onto the orthocomplement of the tangent space for the original set of nuisance functions $h^o_u$; see, for example, \citen{vdV-W}, \citen[Chap. 25]{vdV}, \citen{kosorok:book}, \citen{BCK-LAD}, and \citen{BelloniChernozhukovHansen2011}.  This projection creates the semi-parametrically efficient score function.  There are other ways to create orthogonal nuisance functions, as illustrated by the second example below.

Note that the projection typically depends on $P$, which gives rise to additional nuisance parameters $h^n_u$, which are then incorporated together with the original nuisance parameters into the new parameter $h_u = (h_u^0, h^n_u)$. Note that this is a feature of all of the examples we consider.
For example, the orthogonal moment functions in the exogenous case of the treatment effects framework depend on both the regression function and the propensity score function.
 This point is clarified further by considering the classical linear model as demonstrated in the next remark.\qed
\end{remark}

\begin{example}[\textbf{Neyman Orthogonal Equations for Linear Regression}]  To illustrate the orthogonality condition in the simplest possible setting, let us consider the linear model:
 \begin{equation}
 Y = D\theta_0 +  X'\beta_0 + \epsilon,  \quad \Ep_P [\epsilon X] = 0, \quad   D = X'\pi_0 + \nu, \quad \Ep_P[\nu X] = 0,
 \end{equation}
 where $D$ is the treatment and $X$
 are the controls of high dimension $p \gg n$.  Call the first equation
 the regression equation, and the second equation the propensity score equation.
 The orthogonal moment
 condition that identifies the projection coefficient $\theta_0$ is  the Frisch-Waugh-Lovell partialling out interpretation of $\theta_0$:
 \begin{equation}
 \Ep_P ( U - \nu \theta_0) \nu = 0,  \ \
 \end{equation}
 where $U$ is the population residual left after projecting out the controls $X$ from the outcome, i.e.
  $Y = X'\delta_0 + U, \ \Ep_P U X = 0$;  and $\nu$ is the population residual left after
  projecting out controls from the treatment as defined in the propensity score equation.
  The high-dimensional nuisance function is    $h(Z) = (X'\delta, X'\pi)'$,
  for $Z =X$, with true value denoted by   $h_0(Z) = (X'\delta_0, X'\pi_0)'$. Now
   the moment function
 \begin{equation} \psi(W, \theta, h(Z)) = \{ ( Y - X'\delta) -
(D - X'\pi) \theta\} (D - X'\pi),
\end{equation} has the required orthogonality property (\ref{eq:orthogonality}), since
by the law of iterated expectations and some simple algebra:
\begin{eqnarray}
&&  \Ep_P \Big [  \partial_t  \Ep_P\big [  \psi(W, \theta_0, h_0(Z))  |Z\big ]  (h (Z) - h_0(Z))  \Big ] \\
&& =   \Big (
\Ep_P[
-(D - X'\pi_0) X'a],
\Ep_P[\{ -( Y - X'\delta_0) +
2(D - X'\pi_0) \theta_0\} X'b]   \Big) = 0, \nonumber
\end{eqnarray}
for $a = \delta - \delta_0$ and $b =  \pi-\pi_0$.
In fact, $ \psi(W, \theta_0, h_0)$ is the semi-parametrically efficient score for $\theta_0$.  The resulting estimator of $\theta_0$ is root-$n$ consistent and asymptotically normal,
uniformly within a class of approximately sparse models as follows from the general
results of this section, and is also semi-parametrically efficient. See also \citen{BelloniChernozhukovHansen2011} which deals with the partially linear model in detail and thus covers this linear example as a special case.

Note that the orthogonal moment function contains two nuisance functions --
the regression function and the propensity score -- $X'\delta_0$ and $X'\pi_0$.
 We could also identify $\theta_0$ through non-orthogonal moment conditions
containing single nuisance functions:
$$
 \Ep_P [\{Y - D\theta_0 -  X'\beta_0\} D ] =0 \quad \text{ or } \quad  \Ep_P [\{Y - D\theta_0 \} (D - X'\pi_0) ] =0.
$$
The first moment condition corresponds to the  regression method, while the second
 to the so-called covariate balancing method. Importantly, the use of these non-orthogonal moment conditions generally does not produce an estimator for $\theta_0$ that is $\sqrt{n}$-consistent
and asymptotically normal uniformly in the class of approximately sparse models.  This failure occurs because
we are forced to use highly non-regular estimators to estimate the nuisance functions $X'\delta_0$ and $X'\pi_0$ in the $p \gg n$ setting.   In fact, this failure would also occur with a low number of controls, including having only $p=1$, whenever selection procedures that exclude irrelevant variables with very high probability are used to estimate the regression parameter $\delta_0$ or the propensity score parameter $\pi_0$.   For more discussion
and documentation of this failure, see Leeb and P{\"o}tscher \citeyear{leeb:potscher:pms,leeb:potscher:review}; \citen{Potscher2009}; and Belloni, Chernozhukov, and Hansen \citeyear{BCH2011:InferenceGauss,BelloniChernozhukovHansen2011}. By contrast, constructing orthogonal moment conditions -- involving the projection of both the outcome and the treatment onto the controls and thereby combining the regression and covariate balancing methods -- makes it possible to achieve $\sqrt{n}$ consistency and asymptotic normality uniformly within a class of approximately sparse models.
  \qed
\end{example}

\begin{example}[\textbf{Neyman Orthogonal Equations for a Class of Conditional Moment Problems}]
Next, consider the conditional moment restrictions framework studied by \cite{C92}:
$$
\Ep_P [ \varphi(W, \theta_{0}, g_{0}(X)) \mid X] = 0,
$$
where $X$ and $W$ are random vectors with $X$ being a sub-vector of $W$, $\theta \in \Theta \subset \Bbb{R}^d$ is a finite-dimensional parameter whose true value $\theta_0$ is of interest, $g$ is a functional nuisance parameter mapping the support of $X$ into a convex set $V \subset \Bbb{R}^l$ whose true value is $g_0$, and $\varphi$ is a known function with values in $\mathbb R^k$ for $k\geq d + l$.

Here we would like to build a score function $(\theta, h)\mapsto \psi (W, \theta,  h)$ for estimating $\theta_{0}$, the true value of parameter $\theta$, where $h$ is a new nuisance parameter  with true value $h_{0}$ that obeys the strong form of the orthogonality condition
(\ref{eq:orthogonality0}) and thus also its weak form \eqref{eq:orthogonality}.
To this end, let $t\mapsto \Ep_P[\varphi(W,\theta_0,t)\mid X]$ be a function mapping $\mathbb R^l$ into $\mathbb R^k$ and let $\gamma(X,\theta_0,g_0) = \partial_{t'}\Ep_P[\varphi(W,\theta_0,t)\mid X] \vert_{t = g_0(X)}$ be a $k\times l$ matrix of its derivatives. We will set $Z=X$ and $h(X) = \text{vec}(g(X),\beta(X),\Sigma(X)),$ where $\beta$ is a function mapping the support of $X$ into the space of $d\times k$ matrices, $\mathbb R^{d\times k}$, and $\Sigma$ is the function mapping the support of $X$ into the space of $k\times k$ matrices, $\mathbb R^{k\times k}$.  Define the true value $\beta_0$  of $\beta$ as
$$
\beta_{0}(X)  =   A(X) (I - \Pi_0(X)),
$$
where $A(X)$ is a $d\times k$ matrix of measurable transformations of $X$, $I$ is the $k\times k$ identity matrix, and $\Pi_0(X) \neq I_{k\times k}$ is a $k\times k$ non-identity matrix with the property:
\begin{eqnarray}
\label{annihilate}
&& \Pi_0(X)  \Sigma_{0}(X)^{-1} \gamma(X,\theta_0,g_0) = \Sigma_{0}(X)^{-1}  \gamma(X,\theta_0,g_0),
\end{eqnarray}
where $\Sigma_0$ is the true value of parameter $\Sigma$.
For example, $\Pi_0(X)$ can be chosen to be the idempotent matrix:
 \begin{align*}
 \Pi_0(X) & =    \Sigma_{0}(X)^{-1} \gamma(X,\theta_0,g_0)\left ( \gamma(X,\theta_0,g_0)'\Sigma_{0}(X)^{-1}   \gamma(X,\theta_0,g_0) \right )^{-1}    \gamma(X,\theta_0,g_0) '.
  \end{align*}
Then an orthogonal score for the problem above can be constructed as
$$
\psi (W, \theta, h(X)) = {\beta(X)} { \Sigma(X)^{-1}}{\varphi(W, \theta, h(X))}, \quad h(X) = \text{vec}(g(X), \beta(X), \Sigma(X)).
$$
 It is straightforward to check that under mild regularity conditions that the score function $\psi$ satisfies $\Ep_P[\psi(W,\theta_0,h_0(X))] = 0$ for $h_0(X) = \text{vec}(g_0(X),\beta_0(X),\Sigma_0(X))$ and also obeys the orthogonality condition:
  \begin{equation}\label{eq:orthogonality01}
\partial_t  \Ep_P[  \psi(W, \theta_0, t)  |X] \Big|_{t = h_0(X)}
=  0, \text{ a.s.}
\end{equation}

 Furthermore, by setting
$$
A(X) = \Big(\partial_{\theta'} \Ep_P[\varphi(W,\theta,g_0(X)\mid X]\vert_{\theta = \theta_0}\Big)' ,\quad \Sigma_{0}(X)  =   \Ep_P\Big[ \varphi(W, \theta_{0}, g_0(X)) \varphi(W, \theta_{0}, g_0(X))' |X\Big],
$$
and using $\Pi_0(X)$ suggested above, we obtain  the efficient score $\psi$ that yields an estimator of $\theta_0$ achieving the semi-parametric efficiency bound provided in \cite{C92}.  

Here we would like to note that an analogous, though more involved, construction can be provided for the more general class of problems considered in \citen{ai:chen} where the nuisance functions depend on the endogenous variables. \qed \end{example}

\subsection{Regularity Conditions and Results}

In what follows, we shall denote by $\delta$, $c_0$, $c$, and $C$ some positive constants. For a positive integer $d$, $[d]$ denotes the set $\{1,\ldots, d\}.$  We shall impose the following regularity conditions.

\begin{assumption}[Moment condition problem]\label{ass: S1}
Consider a random element $W$,  taking values in a measure space $(\mathcal{W}, \mathcal{A}_\mathcal{W})$, with law determined by a probability measure $P \in \mP_n$.
The observed data $((W_{ui})_{u \in \mathcal{U}})_{i=1}^{n}$ consist of $n$ i.i.d. copies of a random element $(W_u)_{u \in \mathcal{U}}$ which is  generated as a suitably measurable transformation with respect to $W$ and $u$.  Uniformly for all $n \geq n_0$ and $P \in \mathcal{P}_n$, the following conditions hold: (i) The true parameter value $\theta_u$ obeys (\ref{eq:ivequation}) and is interior relative to $\Theta_u \subset \Theta \subset \mathbb{R}^{d_\theta}$, namely there is a ball
of radius $\delta$ centered at $\theta_u$ contained in $\Theta_u$ for all $u \in \mathcal{U}$, and $\Theta$ is compact.   (ii) For
$\nu := (\nu_k)_{k=1}^{d_\theta + d_t} = (\theta, t)$,  each $j \in [d_{\theta}]$ and $u \in \mathcal{U}$,  the map $ \Theta_u \times T_u(Z_u) \ni  \nu \longmapsto \Ep_P[\psi_{uj}(W_u, \nu)|Z_u]$ is twice continuously differentiable a.s. with derivatives obeying the integrability conditions specified in Assumption \ref{ass: S2}. (iii) For all $u \in \mU,$ the moment function $\psi_u$ obeys the orthogonality condition given in Definition 5.1 for the set $\mathcal{H}_u =\mathcal{H}_{un}$ specified in Assumption \ref{ass: AS}. (iv) The following identifiability condition holds: $\|\Ep_P[\psi_u(W_u, \theta, h_u(Z_u))]\| \geq 2^{-1} ( \|J_u (\theta- \theta_u)\| \wedge c_0)\ \text{  for all } \theta \in \Theta_u,$  where the singular values of $J_u :=  \partial_\theta \Ep[ \psi_u (W_u, \theta_u, h_u(Z_u))]$ lie between $c$  and $C$ for all $u \in \mathcal{U}$.
\end{assumption}

The conditions of Assumption \ref{ass: S1} are mild and standard in moment condition problems.   Assumption \ref{ass: S1}(iv) encodes sufficient global and local identifiability to obtain a rate result.  The suitably measurable condition, defined in Appendix \ref{subsec:notation}, is a mild condition satisfied in most practical cases.


\begin{assumption}[Entropy and smoothness]\label{ass: S2}
 The set $(\mathcal{U}, d_{\mathcal{U}})$ is a semi-metric space such that $\log N(\epsilon, \mathcal{U},  d_{\mathcal{U}}) \leq  C \log (\mathrm{e}/\epsilon) \vee 0$.  Let
$\alpha \in [1,2]$, and let $\alpha_1$ and $\alpha_2$ be some positive constants. Uniformly for all $n \geq n_0$ and $P \in  \mathcal{P}_n$,  the following conditions hold: (i) The set of functions $\mathcal{F}_0 = \{  \psi_{uj}(W_u, \theta_u, h_u(Z_u)):  j  \in [d_\theta], u \in \mathcal{U}\}$, viewed as functions of $W$ is suitably measurable;  has an envelope function $F_0(W)= \sup_{j\in [d_\theta], u \in \mathcal{U}, \nu \in \Theta_u\times  T_u(Z_u)}|\psi_{uj}(W_u, \nu)|$ that is measurable with respect to $W$ and obeys $\|F_0\|_{P, q} \leq C$, where $q\geq 4$ is a fixed constant; and has a uniform covering entropy obeying
$
\sup_Q  \log N(\epsilon \|F_0\|_{Q,2}, \mathcal{F}_0, \| \cdot \|_{Q,2}) \leq C  \log(\mathrm{e}/\epsilon) \vee 0.
$
(ii)   For all $j \in [d_\theta]$ and  $k,r \in [d_\theta+d_t]$, and $\psi_{uj}(W) := \psi_{uj}(W_u, \theta_u, h_u(Z_u) )$, \begin{itemize}
\item[(a)] $\sup_{u \in \mathcal{U}, (\nu, \bar \nu) \in (\Theta_u\times T_u(Z_u))^2} \Ep_P[  ( \psi_{uj}(W_u, \nu) - \psi_{uj}(W_u, \bar \nu))^2 |Z_u] / \| \nu - \bar \nu\|^{\alpha}\leq C$, $P$-a.s.,
 \item[(b)] $\sup_{d_\mathcal{U} (u, \bar u) \leq \delta } \Ep_P[  ( \psi_{uj}(W) - \psi_{\bar{u}j}(W))^2] \leq C \delta^{ \alpha_1},   \ \  \sup_{d_\mathcal{U}(u, \bar u) \leq \delta} \| J_u -  J_{\bar u} \| \leq C \delta^{\alpha_2}, $
     \item [(c)]
 $\Ep_P  \sup_{u \in \mathcal{U}, \nu \in \Theta_u\times T_u(Z_u)} |\partial_{\nu_r} \Ep_P \left [ \psi_{uj}(W_u, \nu)\mid Z_u \right ]|^2  \leq C$,
 \item[(d)] $\sup_{u \in \mathcal{U}, \nu \in \Theta_u\times T_u(Z_u)} |\partial_{\nu_k} \partial_{\nu_r}  \Ep_P[\psi_{uj}(W_u, \nu)|Z_u]| \leq  C,$ $P$-a.s.
    \end{itemize}
\end{assumption}

Assumption \ref{ass: S2} imposes smoothness and integrability conditions on various quantities derived from $\psi_u$.
It also imposes conditions on the complexity of the relevant function classes.

 In what follows, let $\Delta_n \searrow 0$, $\delta_n \searrow 0$, and $\tau_n \searrow 0$ be sequences of constants approaching zero from above at a speed at most polynomial in $n$ (for example, $\delta_n \geq 1/n^c$ for some $c > 0$).  \\

\begin{assumption}[Estimation of nuisance functions]\label{ass: AS}
The following conditions hold for each $n \geq n_0$ and all $P \in  \mathcal{P}_n$. The estimated functions $\hat h_u = (\hat h_{um})_{m=1}^{d_t} \in \mathcal{H}_{un}$ with probability at least $1- \Delta_n$, where
$\mathcal{H}_{un}$ is the set of measurable maps $\mathcal{Z}_u \ni z \longmapsto h = (h_m)_{m=1}^{d_t}(z) \in T_u(z)$ such that
$$
\| h_m - h_{um}\|_{P,2} \leq \tau_n,  \quad \tau_n^2 \sqrt{n} \leq \delta_n,
$$
and whose complexity does not grow too quickly in the sense that
$\mathcal{F}_1 = \{  \psi_{uj}(W_u, \theta, h(Z_u)):  j  \in [d_\theta], u \in \mathcal{U}, \theta \in \Theta_u, h \in \mathcal{H}_{un} \}$
is suitably measurable and its uniform covering entropy obeys
 $$
\sup_Q  \log N(\epsilon \|F_1\|_{Q,2}, \mathcal{F}_1, \| \cdot \|_{Q,2}) \leq s_n ( \log (a_n/\epsilon)) \vee 0,
$$
where $F_1(W)$ is an envelope for $\mathcal{F}_1$ which is measurable with respect to  $W$ and satisfies $F_1(W) \leq F_0(W)$ for $F_0$  defined in Assumption \ref{ass: S2}. The complexity characteristics $a_n \geq \max(n, \mathrm{e}) $ and $s_n \geq 1$ obey the growth conditions:
$$
n^{-1/2}  \left ( \sqrt{ s_n \log (a_n) } + n^{-1/2} s_n n^{\frac{1}{q}}  \log (a_n) \right) \leq \tau_n \text{ and }
\tau_n^{\alpha/2} \sqrt{ s_n \log (a_n)}  +  s_n n^{\frac{1}{q}-\frac{1}{2}}  \log (a_n) \log n  \leq \delta_n,
$$
where $q$ and $\alpha$ are defined in Assumption \ref{ass: S2}.
\end{assumption}

\begin{remark}[On Rate and Entropy Rate Conditions] Assumption \ref{ass: AS} imposes conditions on the estimation rate of the nuisance functions $h_{um}$ and on the complexity of the functions sets that contain the estimators $\hat h_{um}$.   This condition allows for a wide variety of modern modeling assumptions and regularization methods for function fitting, including both traditional methods and more recent statistical and machine learning methods. Within the approximately
sparse framework, the index $s_n$ corresponds to the maximum of the
dimension of the approximating models and of the size of the selected models; and $a_n = p \vee n$.    Under other frameworks, these parameters could be different; yet if they are well-behaved, then our results still apply. Thus, these results cover other frameworks, where structured assumptions other than approximate sparsity are used to make the estimation and modeling problem  manageable.  It is important  to point out that the class $\mathcal{F}_1$ generally will not be Donsker because its entropy is allowed to increase with $n$.  Allowing for non-Donsker classes is crucial for accommodating modern, high-dimensional estimation methods for the nuisance functions.  This feature makes the conditions imposed here very different from the conditions imposed in various classical references on dealing with nonparametrically estimated nuisance functions; see, for example, \citen{vdV-W}, \citen{vdV},  \citen{kosorok:book}, and other references listed in the introduction.
\end{remark}

\begin{remark}[Removing Entropy Rate Conditions by Sample-Splitting]  We can can set $s_n=1$ and $a_n = e$ in Assumption 5.3 if we employ data-splitting. That is, under data-splitting the entropy condition becomes very weak, akin to that in parametric problems, facilitating the application of modern statistical and machine learning methods (e.g. random forest, boosted trees, deep neural nets, and their aggregated and hybrid versions)
 to estimate the nuisance functions. Thus, with data-splitting  Assumption 5.3 only requires that the estimators of nuisance parameters attain sufficiently rapid rates of convergences $\tau_n$, in particular $\tau_n = o(n^{-1/4})$ in smooth problems.  Of course in practice we can not verify that these rates hold in a given problem, but the regularity conditions become \textit{more plausible} with data-splitting than without it.   \citen{bickel:1982} employs the idea of data-splitting, namely setting aside a vanishing fraction of the  sample to estimate the nuisance parameter, to set up adaptive estimators of the main parameter; see also \citen{vdV}. This ensures that there is no asymptotic efficiency loss from data-splitting.  Another method, which seems more practical, is to use the following cross-fitting approach: (1) split the sample into two equal parts, the auxiliary and main parts; (2) use the auxiliary part to estimate the nuisance parameter and the main part to estimate the target  parameter, obtaining one estimator of the target parameter;   (3) by reversing the roles of the main and auxiliary parts,  obtain another estimator of the target parameter; and (4) average the two estimators of the target parameter to obtain the final estimator.  Theorems  5.1 given below yields the properties of the final estimator. We refer to  \citen{CCDHM16} for further details, including the result that there is no asymptotic efficiency loss from data-splitting under cross-fitting.
\end{remark}

The following theorem is one of the main results of the paper:

\begin{theorem}[\textbf{Uniform Functional Central Limit Theorem for a Continuum of Target Parameters in Moment Condition Problems}] \label{theorem:semiparametric} Under  Assumptions \ref{ass: S1}, \ref{ass: S2}, and \ref{ass: AS}, for an estimator $(\hat \theta_u)_{u \in \mathcal{U}}$ that obeys equation (\ref{eq:analog}),
$$ \sqrt{n}(\hat \theta_u - \theta_u)_{u \in \mathcal{U}} =  ( \Gn    \bar \psi_u  )_{u \in \mathcal{U}} + o_P(1)$$  in  $\ell^\infty(\mathcal{U})^{d_\theta},$ uniformly in $P \in  \mathcal{P}_n$,
where  $\bar \psi_u(W):= - J^{-1}_u  \psi_u(W_u, \theta_u, h_u(Z_{u}))$, and
$$
 Z_{n,P} := ( \Gn    \bar \psi_u )_{u \in \mathcal{U}} \rightsquigarrow Z_P := ( \mathbb{G}_P \bar \psi_u )_{u \in \mathcal{U}} \text{ in } \ell^\infty(\mathcal{U})^{d_\theta}, \text{ uniformly in $P \in  \mathcal{P}_n$,}
$$
where  the paths of $u \longmapsto  \mathbb{G}_P \bar \psi_u$ are a.s. uniformly continuous on $(\mathcal{U}, d_{\mathcal{U}})$ and
$$\sup_{P \in  \mathcal{P}_n} \Ep_P \sup_{u \in \mathcal{U}}\|\mathbb{G}_P \bar \psi_u\|  < \infty \text{ and }
\displaystyle \lim_{\delta \to 0} \sup_{P \in  \mathcal{P}_n} \Ep_P \sup_{ d_{\mathcal{U}}(u,\bar u) \leq \delta  }\|\mathbb{G}_P \bar \psi_u - \mathbb{G}_P \bar \psi_{\bar u} \|  = 0.$$

\end{theorem}

\begin{remark} It is important to mention here that this result on a continuum of parameters solving a continuum of moment conditions is completely new.
The prior approaches dealing with continua of moment conditions with infinite-dimensional nuisance parameters, for example, the ones  given in
\citen{CH06} and \citen{EZ2013}, impose Donsker conditions on the class of functions, following \citen{andrews:emp}, that contain the values of the estimators of these nuisance functions.  This approach is precluded in our setting because the resulting class of functions in our case has entropy that grows with the sample size
and therefore is not Donsker.  Hence, we develop a new approach to establishing the results
which exploits the interplay between the rate of growth of entropy, the biases, and the size of the estimation error.
In addition, the new approach allows for obtaining results that are uniform in $P$. \qed
\end{remark}

We can estimate the law of $Z_P$ with the bootstrap law of
\begin{equation}\label{define: BS draw}
\hat Z^*_{n,P}  := \sqrt{n} (\hat \theta^*_u - \hat \theta_u)_{u \in \mU} := \left ( \frac{1}{\sqrt{n}} \sum_{i=1}^n \xi_i \hat \psi_u(W_i) \right)_{u \in \mU},
\end{equation}
where $(\xi_i)_{i=1}^n$ are i.i.d. multipliers as defined in equation (\ref{eq: multipliers}),  $ \hat \psi_u(W_i)$ is the estimated score
$$
 \hat \psi_u(W_i):= - \hat J_u^{-1} \psi_u(W_{ui}, \hat \theta_u, \hat h_u(Z_{ui})),
 $$
and $\hat J_u$ is a suitable estimator of $J_u$.\footnote{We do not discuss the estimation of $J_u$ since
 it is often a problem-specific matter. In Section 3, $J_u$ was equal to minus the identity matrix, so we did not need to estimate it.}  The bootstrap
law  is computed by drawing $(\xi_i)_{i=1}^n$  conditional on the data.

The following theorem shows that the multiplier bootstrap provides a valid approximation to the large sample law of $\sqrt{n}(\hat \theta_u- \theta_u)_{u \in \mathcal{U}}$.

\begin{theorem}[\textbf{Uniform Validity of Multiplier Bootstrap}]\label{theorem: general bs}
Suppose Assumptions \ref{ass: S1}, \ref{ass: S2}, and \ref{ass: AS} hold, the estimator $(\hat \theta_u)_{u \in \mathcal{U}}$ obeys equation (\ref{eq:analog}), and that
the estimator $(\hat J_u)_{u \in \mathcal{U}}$ obeys the following condition: uniformly in $P \in \mP_n$ with probability $1-\delta_n$,
$\sup_{u \in \mU }\| \widehat J_u - J_u \| \leq \Delta_n.$
Then, $$ \hat Z^*_{n,P} \rightsquigarrow_B  Z_{P} \text{ in }  \ell^\infty(\mU)^{d_\theta},
\text{ uniformly in $P \in \mP_n$}.$$
\end{theorem}

 We next derive the large sample distribution and validity of the multiplier bootstrap for the estimator $\hat \Delta := \phi(\hat \theta):= \phi( (\hat \theta_u)_{u \in \mU})$ of the functional $\Delta := \phi(\theta^0)= \phi( (\theta_u)_{u \in \mU} )$ using the functional delta method.
The functional $\theta^0 \longmapsto \phi(\theta^0)$ is defined as a uniformly Hadamard differentiable transform of $\theta^0 = (\theta_u)_{u \in \mU}$.
The following result gives the large sample law  of $\sqrt{n} (\hat \Delta - \Delta)$, the properly normalized estimator.
It also shows that the bootstrap law of
$ \sqrt{n} (\hat \Delta^* - \hat \Delta),$
computed conditionally on the data, is consistent for the large sample law of $\sqrt{n} (\hat \Delta - \Delta)$. Here $\hat \Delta^* := \phi(\hat \theta^*) = \phi ( (\hat \theta^*)_{u \in \mU})$ is the bootstrap version of $\hat \Delta$,
and  $\hat \theta^*_u =  \hat \theta_u + n^{-1} \sum_{i=1}^n \xi_i \hat \psi_u(W_i)$ is the multiplier bootstrap version of $\hat \theta_u$ defined via equation (\ref{define: BS draw}).

\begin{theorem}[\textbf{Uniform Limit Theory and Validity of Multiplier Bootstrap for Smooth Functionals of $\theta$}]
\label{theorem3}
Suppose that for each $P \in  \mP:= \cup_{n \geq n_0} \mP_n$,  $\theta^0= \theta^0_P$ is an element of a compact set
 $\mathbb{D}_{\theta}$.  Suppose $\theta \longmapsto \phi(\theta) $, a functional of interest mapping
 $\mathbb{D}_{\phi} \subset \mathbb{D}= \ell^{\infty}(\mU)^{d_\theta}$ to $\ell^{\infty}( \mathcal{Q})$,  where $\D_\theta \subset \D_\phi$,
 is  Hadamard differentiable in $\theta$ tangentially to $\mathbb{D}_0 = UC(\mU)^{d_\theta}$ uniformly in $\theta \in \mathbb{D}_{\theta}$, with the linear derivative map $\phi^{\prime}_{\theta}: \D_0 \longmapsto \D$ such that the mapping $(\theta, h) \longmapsto \phi'_{\theta}(h)$  from $\mathbb{D}_\theta \times \mathbb{D}_0$ to $\ell^{\infty}(\mQ)$ is continuous.
   Then,
\begin{equation}  \label{eq:results on functionals}
\sqrt{n} (\hat \Delta - \Delta)   \rightsquigarrow T_P:= \phi'_{\theta^0_P} (Z_P)  \ \  \text{ in $\ell^\infty(\mathcal{Q})$,  uniformly in $P \in \mP_n$},
\end{equation}
where $T_P$ is a zero mean tight Gaussian process,  for each $P \in \mP$.  Moreover,
\begin{equation}  \label{eq:results on functionals}
\sqrt{n} (\hat \Delta^* - \hat \Delta)  \rightsquigarrow_B T_P 
 \ \  \text{ in $\ell^\infty(\mathcal{Q})$,  uniformly in $P \in \mP_n$}.
\end{equation}
\end{theorem}
To derive Theorem \ref{theorem3}, we strengthen the usual notion of Hadamard differentiability to a uniform notion introduced in Definition \ref{def:uhd}.  Theorems  \ref{thm: delta-method} and \ref{theorem:delta-method-bootstrap} show that this uniform Hadamard differentiability  is sufficient to guarantee the  validity of the functional delta uniformly in $P$.  These new uniform functional delta method theorems may be of independent interest.

\section{Theory: Lasso and Post-Lasso for Functional Response Data}\label{FunctionalLassoSection}

In this section, we provide results for Lasso and Post-Lasso estimators with function-valued outcomes and linear or logistic links.  As these results are of interest beyond the context of estimation of nuisance functions for moment condition problems or treatment effects estimation,
we present this section in a way that leaves it autonomous with respect to the rest of the paper. 

\subsection{The generic setting with function-valued outcomes}

Consider a data generating process with a functional response variable $(Y_{u})_{u\in \mathcal{U}}$ and observable covariates $X$ satisfying for each $u\in \mathcal{U}$,
\begin{equation}\label{A:EqMainFunc}\Ep_P[ Y_{u} \mid X ] = \G(f(X)'\theta_u) + r_u(X), \end{equation}
 where $f:\mathcal{X}\to\mathbb{R}^p$ is a set of $p$ measurable transformations of the initial controls $X$, $\theta_u$ is a $p$-dimensional vector, $r_u$ is an approximation error, and $\G$ is a fixed known link function. The notation in this section differs from the rest of the paper with $Y_u$ and $X$ denoting a generic response and a generic vector of covariates to facilitate the application of these results to other contexts.  We only consider the linear link function, $\G(t) = t$, and the logistic link function, $\G(t)=\exp(t)/\{1+\exp(t)\}$, in detail.

 Considering the logistic link is useful when the functional response is binary, though the linear link can be used in that case as well under some conditions. For example, it is useful for estimating a high-dimensional generalization of the distributional regression models considered in \citen{CFM} where the response variable is the continuum $(Y_u = 1( Y \leq u))_{u \in \mU}$.   Even though we focus on these two cases we  note that the principles discussed here apply to many other $M$-estimators with convex (or approximately convex) criterion functions. In the remainder of the section, we discuss and establish results for $\ell_1$-penalized and post-model selection estimators of $(\theta_u)_{u \in \mU}$  that hold uniformly over $u\in\mathcal{U}$.

Throughout the section, we assume that $u\in \mathcal{U} \subset [0,1]^\dn$ and that we have $n$ i.i.d. observations from d.g.p.'s where (\ref{A:EqMainFunc}) holds, $\{( Y_{ui})_{u\in\mathcal{U}},  X_i)\}_{i=1}^n$, available for estimating $(\theta_u)_{u\in\mathcal{U}}$. For each $u\in\mathcal{U}$, penalty level $\lambda$, and diagonal matrix of penalty loadings $\hat\Psi_u,$ we define the Lasso estimator as
 \begin{equation}\label{Adef:LassoFunc} \hat\theta_u \in \arg\min_{\theta \in \mathbb{R}^{p}} \En[M(Y_{u},f(X)'\theta)] + \frac{\lambda}{n}\|\hat\Psi_u\theta\|_1 \end{equation} where $M(y,t) = \frac{1}{2}(y-\G(t))^2$ in the case of linear regression, and $M(y,t) = -\{1(y=1)\log \G(t) + 1(y=0)\log(1-\G(t))\}$ in the case of the logistic link function for binary response data.
 For each $u\in\mathcal{U}$, the Post-Lasso estimator based on a set of covariates $\widetilde T_u$ is then defined as
 \begin{equation}\label{Adef:PostFunc}\widetilde\theta_u \in \arg\min_{\theta \in \mathbb{R}^{p}} \En[M(Y_{u},f(X)'\theta)] \ \ : \ \ \supp(\theta)\subseteq \widetilde T_u\end{equation}
where the set $\widetilde T_u$ contains $\supp(\hat\theta_u)$ and may also contain additional variables deemed as important.\footnote{The total number of additional variables $\hat s_a$ should also obey the same growth conditions that $s$ obeys. For example, if the additional variables are chosen so that $\hat s_a \lesssim \|\hat\theta_u\|_0$ the growth condition is satisfied with probability going to one for the designs covered by Assumptions \ref{ass: linear} and \ref{ass: logistic}. See also \citen{BelloniChernozhukovHansen2011} for a discussion on choosing additional variables.} We will set $\widetilde T_u = \supp(\hat\theta_u)$ unless otherwise noted.

The chief departure between the analysis when $\mathcal{U}$ is a singleton and the functional response case is that the penalty level needs to be set to control selection errors
uniformly over $u\in\mathcal{U}$. To do so, we will set
$\lambda$ so that with high probability
\begin{equation}\label{Eq:reg} \frac{\lambda}{n} \geq c \sup_{u\in \mathcal{U}}\left\|\hat \Psi^{-1}_u \En\left[\partial_\theta M(Y_{u},f(X)'\theta_u) \right] \right\|_\infty,
\end{equation}where $c>1$ is a fixed constant. When $\mathcal{U}$ is a singleton  the strategy above is similar to \citen{BickelRitovTsybakov2009}, \citen{BC-PostLASSO}, and \citen{BCW-SqLASSO}, who use an analog of (\ref{Eq:reg}) to derive the properties of Lasso and Post-Lasso.   When $\mathcal{U}$ is not a  singleton, this strategy was first employed in the context of $\ell_1$-penalized quantile regression processes by \citen{BC-SparseQR}.

To implement (\ref{Eq:reg}), we propose setting the penalty level as
\begin{equation}\label{Eq:Def-lambda}\lambda =   c \sqrt{n} \Phi^{-1}(1-\gamma/\{2pn^\dn\}),
 \end{equation}
where $\dn$ is the dimension of $\mathcal{U}$, $1-\gamma$ with $\gamma = o(1)$ is a confidence level associated with the probability of event (\ref{Eq:reg}),  and $c>1$ is a slack constant.\footnote{When the set $\mathcal{U}$ is a singleton, one can use the penalty level in (\ref{Eq:Def-lambda}) with $\dn = 0$.  This choice corresponds to that used in \citen{BelloniChernozhukovHansen2011}.}  When implementing the estimators, we set $c=1.1.$ and $\gamma = .1/\log(n)$, which is theoretically motivated and practically tested in an extensive set of simulation experiments in \citen{BelloniChernozhukovHansen2011}.
In addition to the penalty parameter $\lambda$, we also need to construct a penalty loading matrix $\widehat\Psi_u = \diag(\{\hat l_{u j}, j=1,\ldots,p\})$.  This loading matrix can be formed according to the following iterative algorithm.

\begin{algorithm}[Estimation of Penalty Loadings]\label{AlgFunc} Choose $\gamma \in [1/n, \min\{1/\log n, pn^{\dn-1}\}]$ and $c > 1$ to form $\lambda$ as defined in (\ref{Eq:Def-lambda}), and choose a constant $K \geq 1$ as an upper bound on the number of iterations. (0) Set $k=0$,  and initialize $\hat l_{uj,0}$ for each $j=1,\ldots,p$. For the linear link function, set $\hat l_{uj,0} = \{\En[f_{j}^2(X)(Y_{u}-\bar Y_{u})^2]\}^{1/2}$ with $\bar Y_{u}=\En[Y_{u}]$. For the logistic link function, set $\hat l_{uj,0} = \frac{1}{2}\{\En[f_{j}^2(X)]\}^{1/2}$. (1)  Compute the Lasso and Post-Lasso estimators, $\hat \theta_u$ and $\widetilde \theta_u$, based on $\widehat\Psi_u = \diag(\{\hat l_{uj,k}, j=1,\ldots,p\})$. (2) Set $\hat l_{uj,k+1} := \{\En[ f_{j}^2(X)(Y_{u}- \G(f(X)'\widetilde\theta_u))^2]\}^{1/2}.$ (3) If $k> K$, stop; otherwise set $k \leftarrow k+1 $ and go to step (1). \end{algorithm}

\subsection{Properties of a Continuum of Lasso and Post-Lasso: Linear Link}

We provide sufficient conditions for establishing good performance of the estimators discussed above when the linear link function is used. 
In the statement of the following assumption, $\delta_n\searrow 0$ and $\Delta_n\searrow 0$ are fixed sequences approaching zero from above at a speed at most polynomial in $n$ (for example, $\delta_n \geq 1/n^c$ for some $c > 0$), $\ell_n := \log n$,  and $c, C, \kappa', \kappa''$ and $\nu \in (0,1]$ are positive finite constants.


\begin{assumption}\label{ass: linear}  Consider  a random element $W$  taking values in a measure space $(\mathcal{W}, \mathcal{A}_\mathcal{W})$, with law determined by a probability measure $P \in \mP_n$.
The observed data  $((Y_{ui})_{u \in \mathcal{U}}, X_i)_{i=1}^{n}$ consist of $n$ i.i.d. copies of random element $((Y_{u})_{u \in \mathcal{U}}, X)$, which is  generated as a suitably measurable transformation of  $W$ and $u$.  The model (\ref{A:EqMainFunc}) holds with linear link $t \longmapsto \Lambda(t) = t$ for all  $ u \in  \mathcal{U}\subset [0,1]^\dn$, where $\dn$ is fixed and $\mathcal{U}$ is equipped with the semi-metric $d_\mathcal{U}$. Uniformly for all $n \geq n_0$ and $P \in \mathcal{P}_n$, the following conditions hold. (i) The model (\ref{A:EqMainFunc}) is approximately sparse with sparsity index obeying $\sup_{u\in\mathcal{U}}\|\theta_u\|_0\leq s$ and the growth restriction  $\log (p \vee n) \leq \delta_n n^{1/3}$. (ii) The set $\mathcal{U}$ has uniform covering entropy obeying
 $\log N(\epsilon,\mathcal{U},d_\mathcal{U}) \leq \dn \log (1/\epsilon)\vee 0$, and the collection $(\zeta_u=Y_u-\Ep_P[Y_{u}\mid X], r_u)_{u\in \mathcal{U}}$ are suitably measurable transformations of $W$ and $u$. (iii) Uniformly over $u\in\mathcal{U}$, the moments of the model are boundedly heteroscedastic, namely $c \leq \Ep_P[\zeta_{u}^2\mid X] \leq C $ a.s., and ${ \max_{j\leq p} } \Ep_P[|f_{j}(X)\zeta_{u}|^3+|f_{j}(X)Y_{u}|^3] \leq C.$ (iv) For a fixed $\nu>0$ and a sequence $K_n$, the dictionary functions, approximation errors, and empirical errors obey the following regularity conditions: (a) $c\leq \Ep_P [f_j^2(X)] \leq C$, $j=1,\ldots,p$; $\max_{j \leq p}|f_j(X)| \leq K_{n}$ a.s.; $K_{n}^2s\log(p\vee n)\leq \delta_n n$. (b) With probability $1-\Delta_n$,   $ \sup_{u\in\mathcal{U}} \En[ r_{u}^2(X)] \leq C  s\log (p \vee n) / n$;
$  \sup_{u\in\mathcal{U}}\max_{j\leq p} |(\En-\Ep_P)[f_{j}^2(X)\zeta_{u}^2]| \vee  |(\En-\Ep_P)[f_{j}^2(X)Y_{u}^2]| \leq \delta_n$; ${  \log^{1/2}(p\vee n ) \sup_{d_\mathcal{U}(u,u')\leq 1/n} } \max_{j\leq p}\{\En[f_j(X)^2(\zeta_{u}-\zeta_{u'})^2]\}^{1/2} \leq \delta_n$, and ${\sup_{d_\mathcal{U}(u,u')\leq 1/n}} \| \En[ f(X)(\zeta_{u}- \zeta_{u'}) ]\|_\infty\leq \delta_n n^{-1/2}$. (c)
With probability $1-\Delta_n$, the empirical minimum and maximum sparse eigenvalues are bounded from zero and above, namely $ \kappa' \leq \inf_{\|\delta\|_0\leq s \ell_n, \|\delta\|=1}\|f(X)'\delta\|_{\Pn,2} \leq \sup_{\|\delta\|_0\leq s \ell_n, \|\delta\|=1}\|f(X)'\delta\|_{\Pn,2} \leq \kappa''$. \end{assumption}

Assumption \ref{ass: linear} is only a set of sufficient conditions. The finite sample results in the Supplementary  Appendix allow for more general conditions (for example, $\dn$ can grow with the sample size). We verify that the more technical conditions in Assumption \ref{ass: linear}(iv)(b) hold in a variety of cases, see Lemma \ref{PrimitiveWL} in Appendix \ref{subsec: lasso} in the Supplementary Appendix. Under Assumption \ref{ass: linear}, we establish results on the performance of the estimators (\ref{Adef:LassoFunc}) and (\ref{Adef:PostFunc}) for the linear link function case that hold uniformly over $u \in \mathcal{U}$ and $P \in \mP_n$.

\begin{theorem}[Rates and Sparsity for Functional Responses under Linear Link]\label{Thm:RateEstimatedLassoLinear}
 Under Assumption \ref{ass: linear} and setting the penalty and loadings as in Algorithm \ref{AlgFunc}, for all $n$ large enough, uniformly for all $P \in \mathcal{P}_n$ with $\mathrm{P}_P$ probability $1-o(1)$,  for some constant $\bar C$, the Lasso estimator $\hat\theta_u$ is uniformly sparse, $\sup_{u\in \mathcal{U}}\|\hat \theta_u \|_0  \leq \bar C  s$, and the following performance bounds hold:
$$\begin{array}{l}
 \displaystyle\sup_{u\in\mathcal{U}} \| f(X)'(\hat\theta_u - \theta_{u})\|_{\Pn,2} \leq \bar C \sqrt{\frac{s\log (p\vee n)}{n}} \ \mbox{and} \ \
 \displaystyle \sup_{u\in\mathcal{U}}\|\hat\theta_u-\theta_{u}\|_1  \leq \bar C   \sqrt{\frac{s^2\log (p\vee n)}{n}}.\end{array}$$
For all $n$ large enough, uniformly for all $P \in \mathcal{P}_n$, with $\mathrm{P}_P$ probability $1-o(1)$,  the Post-Lasso estimator corresponding to $\hat\theta_u$ obeys
$$ \sup_{u\in \mathcal{U}} \| f(X)'(\widetilde \theta_u -\theta_u)\|_{\Pn,2} \leq \bar C  \sqrt{\frac{s \log (p \vee n)}{n}}, \text{ and }   \ \
 \displaystyle \sup_{u\in\mathcal{U}}\|\widetilde \theta_u-\theta_{u}\|_1  \leq \bar C \sqrt{\frac{s^2\log (p\vee n)}{n}}. $$\end{theorem}

We note that the performance bounds are exactly of the type used in Assumption \ref{ass: sparse1} (see also Assumption \ref{ass: sparse2}  in the Supplementary Appendix). Indeed, under the condition $s^2\log^2(p\vee n) \log^2 n \leq \delta_n n$, the rate of convergence established in Theorem \ref{Thm:RateEstimatedLassoLinear} yields $\sqrt{s\log(p\vee n)/n} \leq o( n^{-1/4})$.

\subsection{Properties of Lasso and Post-Lasso Estimators: Logistic Link}

We provide sufficient conditions to state results on the performance of the estimators discussed above for the logistic link function. 
Consider the fixed sequences $\delta_n\searrow 0$ and $\Delta_n\searrow 0$  approaching zero from above at a speed at most polynomial in $n$, $\ell_n := \log n$, and the positive finite constants $c$, $C$, $\kappa'$, $\kappa''$, and $\underline{c} \leq 1/2$.

\begin{assumption}\label{ass: logistic}   Consider  a random element $W$  taking values in a measure space $(\mathcal{W}, \mathcal{A}_\mathcal{W})$, with law determined by a probability measure $P \in \mP_n$.
The observed data  $((Y_{ui})_{u \in \mathcal{U}}, X_i)_{i=1}^{n}$ consist of $n$ i.i.d. copies of random element $((Y_{u})_{u \in \mathcal{U}}, X)$, which is  generated as a suitably measurable transformation of  $W$ and $u$.   The model (\ref{A:EqMainFunc}) holds with $Y_{ui} \in \{0,1\}$ with the logistic link $t \longmapsto \Lambda(t) = \exp(t)/\{1+\exp(t)\}$ for each  $u \in \mathcal{U}\subset [0,1]^{\dn}$, where $\dn$ is fixed and $\mathcal{U}$ is equipped with the semi-metric $d_\mathcal{U}$.  Uniformly for all $n \geq n_0$ and $P \in \mathcal{P}_n$, the following conditions hold.
 (i) The model (\ref{A:EqMainFunc}) is approximately sparse with sparsity index obeying $\sup_{u\in\mathcal{U}}\|\theta_u\|_0\leq s$ and the growth restriction  $\log (p \vee n) \leq \delta_n n^{1/3}$. (ii) The set $\mathcal{U}$ has uniform covering entropy obeying
 $\log N(\epsilon,\mathcal{U},d_\mathcal{U}) \leq \dn \log (1/\epsilon)\vee 0$, and the collection $(\zeta_{u}=Y_u-\Ep_P[Y_{u}\mid X], r_{u})_{u\in \mathcal{U}}$ is a suitably measurable transformation of $W$ and $u$. (iii) Uniformly over $u\in\mathcal{U}$ the moments of the model satisfy  ${ \max_{j\leq p} } \Ep_P[|f_{j}(X)|^3] \leq C,$ and $\underline{c}\leq \Ep_P[Y_{u}\mid X ] \leq 1-\underline{c}$ a.s. (iv) For a sequence $K_n$, the dictionary functions, approximation errors, and empirical errors obey the following boundedness and empirical regularity conditions: (a) $\sup_{u\in\mathcal{U}}|r_{u}(X)|\leq \delta_n$ a.s.; $c\leq \Ep_P [f_j^2(X)] \leq C$, $j=1,\ldots,p$; $\max_{j \leq p}|f_j(X)| \leq K_{n}$ a.s.;  and $K_n^2s^2\log^2(p\vee n) \leq \delta_n n$.  (b) With probability $1-\Delta_n$,   $ \sup_{u\in\mathcal{U}} \En[ r_{u}^2(X)] \leq C  s\log (p \vee n) / n$;
$  \sup_{u\in\mathcal{U}}\max_{j\leq p} |(\En-\Ep_P)[f_{j}^2(X)\zeta_{u}^2]| \leq \delta_n;$  $\sup_{u,u'\in\mathcal{U},d_\mathcal{U}(u,u')\leq 1/n}  \max_{j\leq p}\{\En[f_j(X)^2(\zeta_{u}-\zeta_{u'})^2]\}^{1/2} \leq \delta_n$, and ${\sup_{u,u'\in\mathcal{U},d_\mathcal{U}(u,u')\leq 1/n}} \| \En[ f(X)(\zeta_{u}- \zeta_{u'}) ]\|_\infty\leq \delta_n n^{-1/2}$.
(c)  With probability $1-\Delta_n$,  the empirical minimum and maximum sparse eigenvalues are bounded from zero and above: $ \kappa' \leq \inf_{\|\delta\|_0\leq s \ell_n, \|\delta\|=1}\|f(X)'\delta\|_{\Pn,2} \leq \sup_{\|\delta\|_0\leq s \ell_n, \|\delta\|=1}\|f(X)'\delta\|_{\Pn,2} \leq \kappa''$. \end{assumption}

The following result characterizes the performance of the estimators (\ref{Adef:LassoFunc}) and (\ref{Adef:PostFunc}) for the logistic link function case under Assumption \ref{ass: logistic}.

\begin{theorem}[Rates and Sparsity for Functional Response under Logistic Link]\label{Thm:RateEstimatedLassoLogistic}
Under Assumption \ref{ass: logistic} and setting the penalty and loadings as in Algorithm \ref{AlgFunc}, for all $n$ large enough, uniformly for all $P \in \mathcal{P}_n$ with $\mathrm{P}_P$ probability $1-o(1)$, the following performance bounds hold for some constant $\bar C$:
$$\begin{array}{l}
 \displaystyle\sup_{u\in\mathcal{U}} \| f(X)'(\hat\theta_u - \theta_{u})\|_{\Pn,2} \leq \bar C \sqrt{\frac{s\log (p\vee n)}{n}} \ \mbox{and} \ \
 \displaystyle \sup_{u\in\mathcal{U}}\|\hat\theta_u-\theta_{u}\|_1  \leq \bar C   \sqrt{\frac{s^2\log (p\vee n)}{n}}.\end{array}$$
and the estimator is uniformly sparse: $\sup_{u\in \mathcal{U}}\|\hat \theta_u \|_0  \leq \bar C  s$. For all $n$ large enough, uniformly for all $P \in \mathcal{P}_n$, with $\mathrm{P}_P$ probability $1-o(1)$,  the Post-Lasso estimator corresponding to $\hat\theta_u$ obeys
$$ \sup_{u\in \mathcal{U}} \| f(X)'(\widetilde \theta_u -\theta_u)\|_{\Pn,2} \leq \bar C  \sqrt{\frac{s \log (p \vee n)}{n}}, \text{ and }   \ \
 \displaystyle \sup_{u\in\mathcal{U}}\|\widetilde \theta_u-\theta_{u}\|_1  \leq \bar C \sqrt{\frac{s^2\log (p\vee n)}{n}}. $$\end{theorem}

\begin{remark} The performance bounds derived in Theorem \ref{Thm:RateEstimatedLassoLogistic} satisfy the conditions of Assumption \ref{ass: sparse1} (see also Assumption \ref{ass: sparse2}  in the Supplementary Material). Moreover, since the link function is $1$-Lipschitz in the logistic case  and the approximation errors are assumed to be small, the results above establish the same rates of convergence for estimators of the conditional probabilities; for example,
$$ \sup_{u\in\mathcal{U}}\| \Ep_P[Y_u\mid X] - \G(f(X)'\widehat\theta_u)\|_{\Pn,2} \leq \bar{C}\sqrt{\frac{s\log(p\vee n)}{n}}.$$
\end{remark}

\section{Application:  the Effect of 401(k) Participation on Asset Holdings}\label{sec: 401k}

As a practical illustration of the methods developed in this paper, we consider estimation of the effect of 401(k) eligibility and participation on accumulated assets as in \citen{abadie:401k} and \citen{CH401k}.  Our goal here is to illustrate the estimation results and inference statements and to make the following points that underscore our theoretical findings:   1) In a \textit{low-dimensional setting}, where the number of controls is low and therefore there is no need for selection,  our robust post-selection inference methods perform well. That is, the results of our methods agree with the results of standard methods that do not employ any selection.   2) In a \textit{high-dimensional} setting, where there are (moderately) many controls, our post-selection inference methods perform well, producing well-behaved estimates and confidence intervals compared to the erratic estimates and confidence intervals produced by standard methods that do not employ selection as a means of regularization.   3)  Finally, in a \textit{very high-dimensional} setting, where the number of controls is comparable to the sample size,  the standard methods break down completely, while our methods still produce well-behaved estimates and confidence intervals. These findings are in line with our theoretical results about uniform validity of our inference methods.

The key problem in determining the effect of participation in 401(k) plans on accumulated assets is saver heterogeneity coupled with the fact that the decision  to enroll in a 401(k) is non-random.  It is generally recognized that some people have a higher preference for saving than others. It also seems likely that those individuals with high unobserved preference for saving would be most likely to choose to participate in tax-advantaged retirement savings plans and would tend to have otherwise high amounts of accumulated assets.  The presence of unobserved savings preferences with these properties then implies that conventional estimates that do not account for saver heterogeneity and endogeneity of participation will be biased upward, tending to overstate the savings effects of 401(k) participation.

To overcome the endogeneity of 401(k) participation, \citen{abadie:401k} and \citen{CH401k} adopt the strategy detailed in Poterba, Venti, and Wise \citeyear{pvw:94,pvw:95,pvw:nber96,pvw:01} and \citen{benjamin}, who used data from the 1991 Survey of Income and Program Participation and argue that eligibility for enrolling in a 401(k) plan in this data can be taken as exogenous after conditioning on a few observables of which the most important for their argument is income.  The basic idea of their argument is that, at least around the time 401(k)'s initially became available, people were unlikely to be basing their employment decisions on whether an employer offered a 401(k) but would instead focus on income.  Thus, eligibility for a 401(k) could be taken as exogenous conditional on income, and the causal effect of 401(k) eligibility could be directly estimated by appropriate comparison across eligible and ineligible individuals.\footnote{Poterba, Venti, and Wise \citeyear{pvw:94,pvw:95,pvw:nber96,pvw:01} and \citen{benjamin} all focus on estimating the effect of 401(k) eligibility, the intention to treat parameter.  Also note that there are arguments that eligibility should not be taken as exogenous given income; see, for example, \citen{engen} and \citen{engen:gale}.}  {\citen{abadie:401k}, \citen{CH401k}, and \citen{ORR:401k}} use this argument for the exogeneity of eligibility conditional on controls to argue that 401(k) eligibility provides a valid instrument for 401(k) participation and employ IV methods to estimate the effect of 401(k) participation on accumulated assets.

As a complement to the work cited above, we estimate various treatment effects of 401(k) participation on financial  wealth using high-dimensional methods.  A key component of the argument underlying the exogeneity of 401(k) eligibility is that eligibility may only be taken as exogenous after conditioning on income.  Both \citen{abadie:401k} and \citen{CH401k} adopt this argument but control only for a small number of terms.  One might wonder whether the small number of terms considered is sufficient to adequately control for income and other related confounds.  At the same time, the power to learn anything about the effect of 401(k) participation decreases as one controls more flexibly for confounds.  The methods developed in this paper offer one resolution to this tension by allowing us to consider a very broad set of controls and functional forms under the assumption that among the set of variables we consider there is a relatively low-dimensional set that adequately captures the effect of confounds.  This approach is more general than that pursued in previous research which implicitly assumes that confounding effects can adequately be controlled for by a small number of variables chosen \textit{ex ante} by the researcher.

We use the same data as \citen{CH401k}.  The data consist of 9,915 observations at the household level drawn from the 1991 SIPP.  We use net financial assets
  as the outcome variable, $Y$,  in our analysis.  Our treatment variable, $D$, is an indicator for having positive 401(k) balances; and our instrument, $Z$, is an indicator for being eligible to enroll in a 401(k) plan.  The vector of raw covariates, $X$, consists of age, income, family size, years of education, a married indicator, a two-earner status indicator, a defined benefit pension status indicator, an IRA participation indicator, and a home ownership indicator.  Further details can be found in \citen{CH401k}.

We present detailed results for three different sets of controls $f(X)$.  The first specification uses indicators of marital status, two-earner status, defined benefit pension status, IRA participation status, and home ownership status, second order polynomials in family size and education, a third order polynomial in age, and a quadratic spline in income with six break points\footnote{Specifically, we allow for income, income-squared, and then interact these two variables with seven dummies for the categories formed by the cut points.} (Quadratic Spline specification).  The second specification augments the Quadratic Spline specification by interacting all the non-income variables with each term in the income spline (Quadratic Spline Plus Interactions specification).  The final specification forms a larger set of potential controls by starting with all of the variables from the Quadratic Spline specification and forming all two-way interactions between all of the non-income variables.  The set of main effects and interactions of all non-income variables is then fully interacted with all of the income terms (Quadratic  Spline Plus Many  Interactions specification).\footnote{The specifications are motivated by the original specification used in \citen{abadie:401k}, \citen{benjamin}, and \citen{CH401k} allowing for data-dependent accommodation of nonlinearity.  We report results based on the exact specification used in previous papers in the Supplementary Appendix.}  The dimensions of the set of controls are thus 35, 311, and 1756 for the Quadratic Spline, Quadratic Spline Plus Interactions, and Quadratic Spline Plus Many Interactions specification, respectively. For methods that do not use variable selection, we use 32, 272, and 1526 variables resulting from removing terms that are perfectly collinear. We refer to the specification without interactions as \textit{low-$p$}, to the specification with only income interactions as \textit{high-$p$}, and to the specification with all two-way interactions further interacted with income as \textit{very-high-$p$}.

We report a variety of results for each specification.  Under the maintained assumption that 401(k) eligibility may be taken as exogenous after controlling for the variables defined in the preceding paragraph, we can use the methods of this paper to estimate intention to treat effects of 401(k) eligibility by setting 401(k) eligibility as $D = Z$.  We report the estimated average intention to treat and average intention to treat on the treated as the ATE and ATE-T, and we report estimates of quantile intention to treat and quantile intention to treat on the treated effects as QTE and QTE-T.  We also directly apply the results of this paper to estimate effects of 401(k) participation, reporting estimates of the LATE, LATE-T, LQTE, and LQTE-T for each specification.\footnote{We note that because of one-sided compliance the local effects for the treated actually coincide with population effects for the treated; see
 \citen{frolich:melly}.}  For comparison, we also report estimates of the eligibility effect from the linear model without selection and with selection using the approach of \citen{BelloniChernozhukovHansen2011} and estimates of the participation effect from linear instrumental variables estimation without selection and with selection as in \citen{CHS:PnP}.

Estimation of all  these treatment effects depends on first-stage estimates of reduced form functions as detailed in Section \ref{EstimationSection}.  We estimate reduced form functions where the outcome is continuous using ordinary least squares when no model selection is used or Post-Lasso when selection is used.  We estimate reduced form functions where the outcome is binary by logistic regression when no model selection is used or Post-$\ell_1$-penalized logistic regression when selection is used.  We only report selection-based estimates in the very-high-$p$ setting.\footnote{The estimated propensity score shows up in the denominator of the efficient moment conditions.  As is conventional, we use trimming to keep the denominator bounded away from zero with trimming set to $10^{-12}$.  Trimming occurs in the Quadratic Spline Plus Interactions (12 observations trimmed) and Quadratic Spline Plus Many Interactions specifications (9915 observations trimmed) when selection is not done.  Trimming never occurs in the selection-based estimates in this example.  We choose not to report unregularized estimates in the very-high-$p$ specification since all observations are trimmed and, in fact, have estimated propensity scores of either 0 or 1.}  We refer to Appendix \ref{sec:Implementation} for detailed discussion of implementing our approach in this example.

Estimates of the ATE, ATE-T, LATE and LATE-T as well as the coefficient on 401(k) eligibility from the linear model and coefficient on 401(k) participation in the linear IV model are given in Table 1.  In this table, we provide point estimates for each of the three sets of controls with and without variable selection.  We report conventional heteroscedasticity consistent standard error estimates for the linear model and linear IV coefficient.  For the ATE, ATE-T, LATE, and LATE-T, we report both analytic and multiplier bootstrap standard errors.  The bootstrap standard errors are based on 500 bootstrap replications with \citen{mammen1993:bootstrap} weights as multipliers.

Looking first at the two sets of standard error estimates for the average treatment effect estimates, we see that the bootstrap and analytic standard errors are quite similar and that one would not draw substantively different conclusions from using one versus the other.  We also see that estimates of the effect of 401(k) eligibility using the linear model and estimates of the effect of 401(k) participation using the linear IV model are broadly consistent with each other across all specifications and regardless of whether or not variable selection is done.  We also have that the estimates of the ATE, ATE-T, LATE, and LATE-T are very similar regardless of whether selection is used in the low-p Quadratic Spline specification.  The ATE and ATE-T both indicate a positive and significant average effect of 401(k) eligibility; and the LATE and LATE-T suggest positive and significant effects of 401(k) participation for compliers.  The similarity in the low-p case is reassuring as it illustrates that there is little impact of variable selection relative to simply including everything in a low-dimensional setting.\footnote{In the low-dimensional setting, using all available controls is semi-parametrically efficient and allows uniformly valid inference.  Thus, the similarity between the results in this case is an important feature of our method which results from our reliance on low-bias moment functions and sensible variable selection devices to produce semi-parametrically efficient estimators and uniformly valid inference statements \emph{following} model selection.}

We observe somewhat different results in the Quadratic Spline Plus Interactions specification.  For both the ATE and the LATE in the Quadratic Spline Plus Interactions case, we see a substantially larger point estimate without selection than with selection, with the selection results being similar to those obtained in the low-p case.  Along with the larger point estimate, we also see that the estimated standard errors in the no selection case for the ATE and LATE are roughly three times larger than the standard errors in the selection case.  For the ATE-T and LATE-T in the Quadratic Spline Plus Interactions case, point estimates following selection are notably smaller than without selection but estimated standard errors after selection are somewhat larger.  We note that one might suspect estimated standard errors for all of the estimators without selection to be substantially downward biased in this case due to the use of many control variables without regularization as in \citen{CJN:PLMStandardError}. Finally, we see a large difference in the Orthogonal Polynomials Plus Many Interactions Specifications as estimates cannot even be computed reliably without selection due to severe overfitting:
The estimated propensity score is either 0 or 1 for every observation.


We provide estimates of the QTE and QTE-T in Figure 1 and estimates of the LQTE and LQTE-T in Figure 2. The left column of Figure 1 gives results for the QTE, and the right column displays the results for the QTE-T.  Similarly, the left and right columns of Figure 2 provide the LQTE and LQTE-T respectively.  We give the results for the Quadratic Spline, Quadratic Spline Plus Interactions, and Quadratic Spline Plus Many Interactions specification in the top row, middle row, and bottom row respectively.  In each graphic, we use solid lines for point estimates and report uniform 95\% confidence intervals with dashed lines.

Looking across the figures, we see a similar pattern to that seen for the estimates of the average effects in that the selection-based estimates are stable across all specifications and are very similar to the estimates obtained without selection from the baseline low-$p$ Quadratic Spline specification.  In the more flexible Quadratic Spline plus Interactions specification, the estimates that do not make use of selection behave somewhat erratically.  This erratic behavior is especially apparent in the estimated LQTE of 401(k) participation where we observe that small changes in the quantile index may result in large swings in the point estimate of the LQTE and estimated standard errors are quite large.  Again, this erratic behavior is likely due to overfitting due to the large set of variables considered.  As with the average effects, estimated quantile effects without selection in the Quadratic Spline Plus Many Interactions specification are not reported as the estimated propensity score is always 0 or 1.

If we focus on the LQTE and LQTE-T estimated from variable selection methods, we find that 401(k) participation has a small impact on accumulated net total financial assets at low quantiles while appearing to have a larger impact at high quantiles.  Looking at the uniform confidence intervals, we can see that this pattern is statistically significant at the 5\% level and that we would reject the hypothesis that 401(k) participation has no effect and reject the hypothesis of a constant treatment effect more generally.

It is also worth discussing the results of the variable selection briefly as well.  Due to the number of models and variable selection steps taken, especially in computing quantile effects, it is not practical to give a complete accounting of the selected variables here.  Rather, we note that for the linear model, linear IV, ATE, and LATE results, we select between two and 22 variables depending on the specification of controls and left-hand-side variable.  The median  number of variables selected for the QTE and LQTE results, where the median is taken across index values $u$, across the different specifications of controls and left-hand-side variables varies between one and 11.  There is considerable variability in the number of variables selected across $u$ though, ranging from a minimum of no variables selected to a maximum of 237 selected variables.\footnote{Having more than 100 variables selected occurs in the very high dimensional setting when the outcome in the penalized regression is $\mathbf{1}_0(D) Y_u$ for the six lowest values of $u$ among the subset of households eligible for 401(k)'s and for the six highest values of $u$ among the subset of households that are not eligible for 401(k)'s.}  The selected variables themselves mostly correspond to capturing the effect of income.  For example, the union of the variables selected in forming each of the reduced form quantities used for estimating the LATE in the Quadratic Spline Plus Many Interactions specification consists of 36 variables, only four of which do not include income.\footnote{Let $i_1$ be the indicator for income in the first income category, and define $i_2-i_7$ similarly.  Let $db$ be the defined benefit dummy, $ira$ be the IRA dummy, $hown$ be the home ownership dummy, $mar$ be the married dummy, $te$ be the two-earner household dummy, $ed$ be years of schooling, and $fsize$ be family size.  The exact identities of the variables selected for modeling any reduced form quantity used in estimating the LATE in the very-high-dimensional case are $i_1$, $i_2$, $i_3$, $income*i_3$, $income^2*i_6$, $db$, $ira*hown$, $age*ira$, $ed*ira$, $i_1*fsize$, $i_1*fsize^2*db$, $i_2*fsize$, $i_2*fsize^2*db$, $i_3*age^3$, $i_3*fsize$, $i_3*mar$, $i_3*fsize*te$, $i_3*fsize*mar$, $i_4*fsize$, $i_4*te$, $i_4*fsize*te$, $i_4*mar*te$, $i_4*ed*fsize$, $i_5*ed^2*te$, $i_5*fsize*te$, $income*ira$, $income*hown$, $income*mar*hown$, $income*te*hown$, $income*fsize*ira$, $income*fsize^2*ira$, $income*i_1*fsize$, $income^2*ira*hown$, $income^2*ed^2*te$, $income^2*i_3*fsize$, and $income^2*i_6*hown$.} This pattern of largely selecting terms that are direct income effects or interactions of income with other variables holds up across the specifications considered.

It is interesting that our results are similar to those in \citen{CH401k} despite allowing for a much richer set of controls.  The fact that we allow for a rich set of controls but produce similar results to those previously available lends further credibility to the claim that previous work controlled adequately for the available observables.\footnote{Of course, the estimates are still not valid causal estimates if one does not believe that 401(k) eligibility can be taken as exogenous after controlling for income and the other included variables.}  Finally, it is worth noting that this similarity is not mechanical or otherwise built in to the procedure.  For example, applications in \citen{BellChenChernHans:nonGauss} and  \citen{BelloniChernozhukovHansen2011} use high-dimensional variable selection methods and produce sets of variables that differ substantially from intuitive baselines.

\appendix

\section{Notation}\label{subsec:notation}

\subsection{Overall Notation} We consider a random element $W=W_P$ taking values in the measure space $(\mathcal{W}, \mathcal{A}_{\mathcal{W}})$, with probability law $P \in \mathcal{P}$. Note that it is most convenient to think about $P$ as a parameter in a parameter set $\mathcal{P}$.  We shall also work with a bootstrap multiplier variable $\xi$ taking values in $(\mathbb{R}, \mathcal{A}_\mathbb{R})$ that is independent of $W_P$, having probability law $P_\xi$, which is fixed throughout.   We  consider $(W_{i})_{i=1}^\infty= (W_{i,P})_{i=1}^\infty$ and $(\xi_i)_{i=1}^\infty$ to be i.i.d. copies of $W$ and  $\xi$, which are also independent of each other. The data will be defined as some measurable function of $W_{i}$ for  $i = 1,..., n$,  where $n$ denotes the sample size.

 We require the sequences $(W_{i})_{i=1}^\infty$ and $(\xi_i)_{i=1}^\infty$ to live on a probability space $(\Omega, \mathcal{A}_\Omega, \Pr_P)$ for all $P \in \mathcal{P}$; note that other variables arising in the proofs do not need to live on the same space.
 It is important to keep track of the dependence on $P$ in the analysis since we want the results to hold uniformly in $P$ in some set $\mathcal{P}_n$ which may be dependent on $n$.  Typically, this set will increase with $n$; i.e. $\mathcal{P}_n \subseteq \mathcal{P}_{n+1}$.

 Throughout the paper we signify the dependence on $P$
 by mostly using $P$ as a subscript in $\Pr_P$, but in the proofs we sometimes use it as a subscript for variables as in $W_P$.  The operator $\Ep$ denotes a generic expectation operator with respect to a generic probability measure $\Pr$, while
 $\Ep_P$ denotes the expectation with respect to  $\Pr_P$. Note also that we use capital letters such as $W$ to denote random elements and use the corresponding lower case letters such as $w$ to denote fixed values that these random elements can take.

We denote by $\mathbb{P}_{n}$ the (random) empirical probability measure that assigns probability $n^{-1}$ to each $W_{i} \in (W_{i})_{i=1}^n$. $\En$ denotes the expectation with respect to the empirical measure, and
$\mathbb{G}_{n,P}$ denotes the empirical process $\sqrt{n}(\En - P)$, i.e.
{\small $$\mathbb{G}_{n,P}(f) = \mathbb{G}_{n,P}(f(W))  = n^{
-1/2} \sum_{i=1}^n \{f(W_i) - P[f(W)]\},  \ \  P [f (W)] := \int f(w) dP(w), $$}\!indexed by a  measurable class of functions $\mathcal{F}: \mathcal{W} \longmapsto \mathbb{R}$; see \citen[chap. 2.3]{vdV-W}. We shall often omit the index $P$ from $\mathbb{G}_{n,P}$ and simply write $\Gn$. In what follows, we use $\|\cdot\|_{\mathsf{P},q}$ to denote the $L^q(\mathsf{P})$ norm; for example, we use $\|f(W)\|_{P,q} = (\int |f(w)|^q d P(w))^{1/q}$ and $\| f(W)\|_{\Pn,q} = (n^{-1} \sum_{i=1}^n |f(W_i)|^q)^{1/q}$. For a vector $v = (v_1,\ldots,v_p)'\in \mathbb{R}^p$, $\|v\|_1 = |v_1| + \cdots + |v_p|$ denotes the $\ell_1$-norm of $v$, 
$\|v\| = \sqrt{v'v}$ denotes the Euclidean norm of $v$,
 and $\|v\|_0$ denotes the $\ell_0$-``norm" of $v$ which equals the number of non-zero components  of $v$.  For a positive integer $k$, $[k]$ denotes the set $\{1,\ldots, k\}.$
 For $x_n, y_n$ denoting sequences in $\Bbb{R}$, the statement $ x_n \lesssim y_n$ means that $x_n \leq A y_n$ for some constant $A$ that does not depend on $n$.

 We say that a collection of
random variables $ \mathcal{F}= \{f(W,t), t \in T\}$, where $f: \mathcal{W} \times T \to \mathbb{R}$, indexed by a set $T$ and viewed as functions of $W \in \mathcal{W}$, is \textit{suitably measurable } with respect to $W$ if it is image admissible Suslin class, as defined in \citen[p 186]{Dudley99}. In particular, $\mathcal{F}$ is suitably measurable if  $f: \mathcal{W}\times T \to \mathbb{R}$ is measurable and $T$ is a Polish space equipped with its Borel sigma algebra, see \citen[p 186]{Dudley99}.  This condition is a mild assumption satisfied in practical cases.

\subsection{ Notation for Stochastic Convergence Uniformly in $P$} All parameters, such as the law of the data, are indexed by $P$. This dependency is sometimes kept implicit.  We shall allow for the possibility that the probability measure $P=P_n$  can depend on $n$. We shall conduct our stochastic convergence analysis uniformly in $P$, where $P$ can vary within
some set $\mP_n$, which itself may vary with $n$.

The convergence analysis, namely the stochastic order relations and convergence in distribution,  uniformly in $P \in \mP_n$
and the analysis under all sequences $P_n \in \mP_n$ are equivalent.  Specifically,  consider  a sequence of stochastic processes $X_{n,P}$ and a random element $Y_P$, taking values in the normed space $\mathbb{D}$, defined on the probability space $(\Omega, \mathcal{A}_\Omega, \Pr_{P})$. Through most of the Appendix $\mathbb{D} = \ell^\infty(\mU)$, the space of uniformly bounded functions mapping an arbitrary index set $\mU$ to the real line, or $\mathbb{D} = UC(\mU)$, the space of uniformly continuous functions mapping an arbitrary index set $\mU$ to the real line.
 Consider also a sequence of deterministic positive constants $a_n$. We shall say that
\begin{itemize}
\item[(i)] $X_{n,P} = O_P(a_n)$  uniformly in $P \in \mP_n$,
if  $\lim_{K \to  \infty} \lim_{n \to \infty} \sup_{P \in \mP_n} \Pr_P^*(|X_{n,P}| > K a_n) =0 $,
 \item[(ii)] $X_{n,P} = o_P(a_n)$  uniformly in $P \in \mP_n$,
if  $\sup_{K>0} \lim_{n \to \infty} \sup_{P \in \mP_n} \Pr_P^*(|X_{n,P}| > K a_n) =0 $,
\item[(iii)] $X_{n,P} \rightsquigarrow Y_P$ uniformly in $P \in \mP_n$, if
$ \sup_{P \in \mP_n} \sup_{h \in \mathrm{BL}_1(\mathbb{D})} | \Ep^*_P h(X_{n,P}) -  \Ep_P h(Y_P)| \to 0.$
\end{itemize}
Here the symbol $\rightsquigarrow$ denotes weak convergence, i.e. convergence in distribution or law, $\mathrm{BL}_1(\mathbb{D})$ denotes the space of functions mapping $\mathbb{D}$ to $[0,1]$ with Lipschitz norm at most 1, and the outer probability and expectation, $\Pr_P^*$ and $\Ep^*_\Pr$, are invoked whenever  (non)-measurability arises.

\begin{lemma}\label{lemma: equivalences} The above notions (i), (ii) and (iii) are equivalent to the following notions (a), (b), and (c), each
holding for \textit{every} sequence $P_n \in \mP_n$:
\begin{itemize}
\item[(a)] $X_{n,P_n} = O_{P_n}(a_n)$,  i.e.
  $\lim_{K \to \infty} \lim_{n \to \infty}  \Pr^*_{P_n}(|X_{n,P_n}| > K a_n) =0 $;
 \item[(b)] $X_{n,P_n} = o_{P_n}(a_n)$, i.e.  $\sup_{K>0} \lim_{n \to \infty}  \Pr^*_{P_n}(|X_{n,P_n}| > K a_n) =0 $;
     \item[(c)]  $X_{n,P_n} \rightsquigarrow Y_{P_n}$, i.e.
    $
    \sup_{h \in \mathrm{BL}_1(\mathbb{D})} | \Ep^*_{P_n} h(X_{n,P_n}) -  \Ep_{P_n} h(Y_{P_n})| \to 0.$
   \end{itemize}
\end{lemma}
 The claims follow straightforwardly from the definitions,
so the proof is omitted.   We shall use this equivalence extensively in the proofs of the main results without explicit reference.

\section{Key Tools I: Uniform in $P$ Donsker Theorem, Multiplier Bootstrap, and Functional Delta Method }\label{app: B}
\subsection{Uniform in $P$ Donsker Property}

Let $ ( W_i)_{i=1}^\infty$ be a sequence of i.i.d. copies of the random element $ W$  taking values in the measure space $({\mathcal{W}}, \mathcal{A}_{{\mathcal{W}}})$ according to the probability law $P$ on that space. Let $\mathcal{F}_P= \{f_{t,P} : t \in T\}$ be a set  of suitably measurable functions $w \longmapsto f_{t,P}(w)$ mapping ${\mathcal{W}}$ to $\mathbb{R}$, equipped with a measurable envelope $F_P: \mathcal{W} \longmapsto \mathbb{R}$.  The class is indexed by $P \in \mathcal{P}$ and $t \in T$, where $T$ is a fixed, totally bounded semi-metric space equipped with a semi-metric $d_{T}$.     Let $N(\epsilon,\mathcal{F}_P, \| \cdot \|_{Q,2})$ denote the $\epsilon$-covering number of the class of functions $\mathcal{F}_P$ with respect to the $L^{2}(Q)$ seminorm $\| \cdot \|_{Q,2}$ for $Q$ a finitely-discrete measure on $({\mathcal{W}}, \mathcal{A}_{{\mathcal{W}}})$. We shall use the following result.

\begin{theorem}[\textbf{Uniform in $P$ Donsker Property}]\label{lemma: uniform Donsker} Work with the set-up above.  Suppose that for $q > 2$
\begin{eqnarray}\label{eq: characteristics1}
&& \sup_{P \in \mP} \|F_P\|_{P,q} \leq C  \text{ and }  \  \lim_{\delta \searrow 0} \sup_{P \in \mathcal{P}} \sup_{ d_{T}(t, \bar t) \leq \delta} \| f_{t,P} - f_{\bar t,P}\|_{P,2} = 0.
\end{eqnarray}
Furthermore, suppose that
\begin{eqnarray}\label{eq: characteristics2}
\lim_{\delta \searrow 0} \sup_{P \in \mathcal{P}}\int_0^\delta \sup_Q \sqrt{ \log N(\epsilon \|F_P\|_{Q,2}, \mathcal{F}_P, \| \cdot \|_{Q,2})} d \epsilon =0.
\end{eqnarray}
Let $\mathbb{G}_P$ denote the P-Brownian Bridge, and  consider $$Z_{n,P} := (Z_{n,P}(t))_{t \in T} := (\Gn(f_{t,P}))_{t \in T},  \ \  Z_P := (Z_P(t))_{t \in T} := (\mathbb{G}_P(f_{t,P}))_{t \in T}.$$

\noindent (a) Then, $Z_{n,P} \rightsquigarrow Z_P$ in $\ell^\infty(T)$ uniformly in $P \in \mathcal{P}$, namely
$$
\sup_{P \in \mP} \sup_{h \in \textrm{BL}_1(\ell^{\infty}(T))} | \Ep^*_P h (Z_{n,P}) - \Ep_P h(Z_P) | \to 0.
$$
(b) The process $Z_{n,P}$ is stochastically equicontinuous  uniformly in $P \in \mathcal{P}$, i.e.,  for every $\varepsilon>0$,
$$
\lim_{\delta \searrow 0} \limsup_{n \to \infty} \sup_{P \in \mathcal{P}} \Pr^*_P\left (  \sup_{d_T  (t, \bar t) \leq \delta} | Z_{n,P}(t) - Z_{n,P}(\bar t)| > \varepsilon \right ) =0.
$$
(c) The limit process $Z_P$ has the following continuity properties:
$$
\sup_{P \in \mP} \Ep_P \sup_{t \in T} |Z_P(t) | < \infty,  \ \  \lim_{\delta \searrow 0} \sup_{P \in \mathcal{P}} \Ep_P \sup_{d_T(t, \bar t) \leq \delta} |Z_{P}(t) - Z_{P}(\bar t)| = 0.
$$
(d) The paths $t \longmapsto  Z_P(t)$ are a.s. uniformly continuous on $(T,d_T)$ under each $P \in \mathcal P$.
\end{theorem}

\begin{remark}\textbf{[Important Feature of the Theorem]}
This is an extension of the uniform Donsker theorem stated in Theorem 2.8.2 in \citen{vdV-W}, which allows for the function classes $\mF$ to be \textbf{dependent on} $P$.  This generalization is crucial and is required in all of our problems.
\end{remark}

\subsection{Uniform in $P$ Validity of Multiplier Bootstrap}
Consider the setting of the preceding subsection. Let $(\xi_i)_{i=1}^n$ be i.i.d multipliers  whose distribution does not depend on $P$,
such that $\Ep \xi = 0$,  $\Ep \xi^2 = 1$, and $\Ep |\xi|^q \leq C$ for $q >2$. Consider the multiplier empirical process:
$$Z^*_{n,P} := (Z^*_{n,P}(t))_{t \in T} := (\Gn(\xi f_{t,P}))_{t \in T} := \left (\frac{1}{\sqrt{n}} \sum_{i=1}^n \xi_i f_{t,P} (W_i) \right)_{t \in T}.$$
Here $\mathbb{G}_n$ is taken to be an extended empirical processes defined by
the empirical measure that assigns mass $1/n$ to each point $(W_i, \xi_i)$ for $i=1,...,n$.
Let $Z_P = (Z_P(t))_{t \in T} = (\mathbb{G}_P(f_{t,P}))_{t \in T}$ as defined in Theorem \ref{lemma: uniform Donsker}.

\begin{theorem}[\textbf{Uniform in $P$ Validity of Multiplier Bootstrap}]\label{lemma: uniform Donsker for bootstrap} Assume
 the conditions of Theorem \ref{lemma: uniform Donsker} hold.  Then
 (a)  the  following  unconditional convergence takes place,  $Z^*_{n,P} \rightsquigarrow Z_P$
 in $\ell^{\infty}(T)$ uniformly in $P \in \mathcal{P}$, namely
 $$
\sup_{P \in \mP} \sup_{h \in \textrm{BL}_1(\ell^{\infty}(T))} | \Ep^*_P h (Z^*_{n,P}) - \Ep_P h(Z_P) | \to 0,
$$
and (b) the following conditional convergence takes place,  $Z^*_{n,P} \rightsquigarrow_B Z_P$
 in $\ell^{\infty}(T)$ uniformly in $P \in \mathcal{P}$, namely uniformly in $P \in \mathcal{P}$
$$
 \sup_{h \in \textrm{BL}_1(\ell^{\infty}(T))} | \Ep_{B_n} h (Z^*_{n,P}) - \Ep_P h(Z_P) | =  o_P^*(1),
$$
where $\Ep_{B_n}$ denotes the expectation over the multiplier weights $(\xi_{i})_{i=1}^n$ holding the data $(W_i)_{i=1}^n$ fixed.
\end{theorem}

\subsection{Uniform in $P$ Functional Delta Method and Bootstrap}
We shall use the functional delta method, as formulated in \citen[Chap. 3.9]{vdV-W}. Let $\mathbb{D%
}_{0}$, $\mathbb{D}$, and $\mathbb{E} $ be normed spaces, with $\mathbb{D}%
_{0} \subset\mathbb{D}$. A map $\phi: \mathbb{D}_{\phi} \subset\mathbb{D}
\longmapsto\mathbb{E}$ is called \textit{Hadamard-differentiable} at $\rho\in%
\mathbb{D}_{\phi}$ tangentially to $\mathbb{D}_{0}$ if there is a continuous
linear map $\phi_{\rho}^{\prime}: \mathbb{D}_{0} \longmapsto\mathbb{E}$ such
that
\begin{equation*}
\frac{\phi(\rho+ t_{n} h_{n}) - \phi(\rho)}{t_{n}} \rightarrow\phi
_{\rho}^{\prime}(h), \ \ \ n \rightarrow\infty,
\end{equation*}
for all sequences $t_{n} \rightarrow0$ in $\mathbb{R}$ and $h_{n} \rightarrow h \in \mathbb{D%
}_{0}$ in $\D$ such that $\rho+ t_{n} h_{n} \in\mathbb{D}_{\phi}$ for every $n$.

We now define the following notion of the uniform Hadamard differentiability:

\begin{definition}[\textbf{Uniform Hadamard Tangential Differentiability}]\label{def:uhd}
Consider a map $\phi: \mathbb{D}_\phi  \longmapsto\mathbb{E}$, where the domain of the map $\D_\phi $
is a subset of a normed space $ \D$ and the range is a subset of the normed space $\mathbb{E}$. Let
$\mathbb{D%
}_{0}$ be a normed space, with $\mathbb{D}%
_{0} \subset\mathbb{D}$, and  $\mathbb{D}_\rho$ be a compact metric space, a subset of $\mathbb{D}_\phi$. The map $\phi: \mathbb{D}_\phi \longmapsto\mathbb{E}$ is called \textit{Hadamard-differentiable uniformly} in $\rho \in \mathbb{D}_\rho$ tangentially to $\mathbb{D}_0$ with derivative map $h \longmapsto \phi'_\rho(h)$, if
\begin{equation*}
\Big |\frac{\phi(\rho_n+ t_{n} h_{n}) - \phi(\rho_n)}{t_{n}} -\phi_{\rho}^{\prime}(h)\Big|  \to 0,\ \   \Big| \phi'_{\rho_n} (h_n) -  \phi_{\rho}^{\prime}(h) \Big|    \to 0 , \ \ \ n \rightarrow\infty,
\end{equation*}
for all convergent sequences $\rho_n  \to \rho$ in $\mathbb{D}_\rho$, $t_{n} \rightarrow0$ in $\mathbb{R}$, and $h_{n} \rightarrow h \in \D_0$ in $\D$ such that $\rho_n+ t_{n} h_{n} \in\mathbb{D}_{\phi}$ for every $n$.  As a part of
the definition, we require
that the derivative map $h \longmapsto \phi_{\rho}^{\prime}(h)$ from $\mathbb{D}_0$
to $\mathbb{E}$ is linear for each $\rho \in \D_\rho$.\qed  \end{definition}

\begin{remark} Note that the definition requires that  the derivative map $ (\rho, h) \longmapsto  \phi_{\rho}^{\prime}(h)$, mapping $\mathbb{D}_\rho \times \mathbb{D}_0$ to $\mathbb{E}$,  is continuous at each $(\rho, h) \in \mathbb{D}_{\rho} \times \mathbb{D}_0$.  \qed \end{remark}

\begin{remark}[\textbf{Important Details of the Definition}]  Definition \ref{def:uhd} is different from the definition of  uniform differentiability given in
\citen[p. 379, eq. (3.9.12)]{vdV-W}, since our definition allows $\D_\rho$ to be much smaller than $\D_\phi$ and allows
$\D_\rho$ to be endowed with a much stronger metric  than the metric induced by the norm of $\D$.
These differences are essential for infinite-dimensional applications. For example, the quantile/inverse map is uniformly Hadamard
differentiable in the sense of Definition \ref{def:uhd} for a  suitable choice of $\D_\rho$:  Let $T = [\epsilon, 1- \epsilon]$,
$\D=\ell^\infty(T)$, $\D_\phi$= set of cadlag functions on $T$, $\D_0 = \mathrm{UC}(T)$, and  $\mathbb{D}_{\rho}$ be a compact subset of  ${C}^{1}(T)$ such that each $\rho \in \mathbb{D}_{\rho}$  obeys $\partial \rho(t)/\partial t \geq c>0$ on $t \in T$, where $c$ is a positive constant. However, the quantile/inverse map is  not  Hadamard differentiable uniformly on $\D_\rho$ if we set $\D_{\rho} =\D_\phi$  and hence is not uniformly differentiable
in the sense of the definition given in \citen{vdV-W} which requires $\D_\rho = \D_\phi$.  It is important and practical to keep the distinction between $\D_{\rho}$ and  $\D_\phi$ since the estimated values $\hat \rho$ may well be outside $\D_{\rho}$ unless explicitly imposed in estimation even though the population values of $\rho$ are in $\D_{\rho}$ by assumption. For example, the empirical cdf is in $\D_{\phi}$, but is outside $\D_{\rho}.$
 \qed
\end{remark}

\begin{theorem}[\textbf{Functional delta-method uniformly in $P \in \mP$}]\label{thm: delta-method} Let $\phi: \D_{\phi} \subset\D \longmapsto%
\mathbb{E}$ be Hadamard-differentiable uniformly in $ \rho \in \D_{\rho} \subset \D_{\phi}$ tangentially to $\mathbb{D%
}_{0}$ with derivative map  $\phi_{\rho}^{\prime}$. Let $\hat \rho_{n,P}$ be a sequence of stochastic processes taking values in $\D_{\phi}$, where each $\hat \rho_{n,P}$ is an estimator of the parameter $\rho_P \in \D_{\rho}$. Suppose there exists a sequence of constants $r_{n} \rightarrow\infty$ such that $Z_{n,P}= r_{n} (\hat \rho_{n,P}- \rho_P) \rightsquigarrow Z_P$  in $\D$ uniformly in $P \in \mP_n$. The limit process $Z_P$ is separable and takes its values in $\D_{0}$ for all $P \in \mP = \cup_{n \geq n_0} \mP_n$, where $n_0$ is fixed.  Moreover, the set of stochastic processes $\{Z_P: P \in \mP\}$  is relatively compact in the topology of weak convergence in $\D_0$, that is, every sequence  in this set can be split into weakly convergent subsequences. Then, $r_{n}\left( \phi(\hat \rho_{n,P}) - \phi(\rho_P) \right) \rightsquigarrow\phi_{\rho_P}^{\prime}(Z_P)$ in $\mathbb{E}$
 uniformly in $P \in \mP_n$. If $(\rho, h) \longmapsto \phi_{\rho}^{'}(h)$ is defined and continuous on the whole of $\D_\rho \times \D$, then the sequence $r_{n}\left( \phi(\hat \rho_{n,P}) - \phi(\rho_P) \right) -
\phi_{\rho_P}^{\prime}\left( r_{n} (\hat \rho_{n,P}- \rho_P)\right) $  converges to zero in outer probability  uniformly in $P \in \mP_n$.  Moreover, the set of stochastic processes $\{\phi'_{\rho_P}(Z_P): P \in \mP\}$  is relatively compact in the topology of weak convergence in $\mathbb{E}$.
\end{theorem}

The following result on the functional delta method applies to any bootstrap or other simulation method obeying certain conditions. Such methods include the multiplier bootstrap as a special case. Let $D_{n,P}= (W_{i,P})_{i=1}^n$ denote the data vector and $B_n= (\xi_i)_{i=1}^n$ be a vector of random variables used to generate bootstrap or simulation draws (the specifics may vary depending on the particular method employed). Consider sequences of stochastic processes
$\hat \rho_{n,P}= \hat \rho_{n,P}(D_{n,P})$, where  $Z_{n,P}=r_n(\hat \rho_{n,P}- \rho_{P}) \rightsquigarrow Z_P$ in the normed space $\D$ uniformly in $P \in \mP_n$. Also consider the bootstrap stochastic process $Z^{*}_{n,P} = Z_{n,P}(D_{n,P}, B_{n})$ in  $%
\D$, where $Z_{n,P}$ is a measurable function of $B_n$ for each value of $D_n$.
Suppose that  $Z^{*}_{n,P}$ converges conditionally given $D_{n}$ in distribution to $Z_P$
uniformly in $P \in \mP_n$, namely that
$$
 \sup_{h \in \mathrm{BL}_1(\D)} | \Ep_{B_n} [h(Z^*_{n,P})] -  \Ep_{P}  h(Z_P)| = o^*_P(1),
$$
uniformly in $P \in \mP_n$, where $\Ep_{B_n}$ denotes the expectation computed with respect
 to the law of $B_n$ holding the data $D_{n,P}$ fixed.  This is  denoted as  ``$Z^{*}_{n,P}
\rightsquigarrow_{B} Z_P $ uniformly in $P \in \mP_n$." Finally, let $\hat \rho^{*}_{n,P} = \hat \rho_{n,P}+ Z^{*}_{n,P}/r_n $ denote the bootstrap or simulation
draw of $\hat \rho_{n,P}$.

\begin{theorem}[\textbf{Uniform in $P$ functional delta-method for bootstrap and other simulation methods}]
\label{theorem:delta-method-bootstrap}
Assume the conditions of Theorem \ref{thm: delta-method} hold.  Let $\hat \rho_{n,P}$ and $\hat \rho^{*}_{n,P}$ be maps as
indicated previously taking values in $\D_{\phi}$ such that $r_n%
(\hat \rho_{n,P}- \rho_P) \rightsquigarrow Z_P$ and  $r_n(\hat \rho^{*}_{n,P} - \hat \rho_{n, P}) \rightsquigarrow_{B} Z_P$ in $\D$ uniformly in $P \in \mP_n$. Then, $X^*_{n,P}=r_n(\phi(\hat \rho ^{*}_{n,P}) -
\phi(\hat \rho_{n,P})) \rightsquigarrow_{B} X_P= \phi_{\rho_P}^{\prime }(Z_P) $ uniformly in $P \in \mP_n$.
\end{theorem}

\subsection{Proof of Theorem \ref{lemma: uniform Donsker}.}  Part (a) and (b) are a direct consequence of Lemma \ref{lemma: Donsker dep on n}. In particular,
Lemma \ref{lemma: Donsker dep on n}(a) implies stochastic equicontinuity under arbitrary subsequences $P_n \in \mathcal{P}$, which implies  part (b).  Part  (a) follows from Lemma \ref{lemma: Donsker dep on n}(b) by splitting an arbitrary sequence $n \in \mathbb{N}$ into subsequences $n \in \mathbb{N}'$ along each of which  the covariance function $
(t,s) \longmapsto c_{P_n}(t,s) := P_n f_{s,P_n} f_{t,P_n} - P_n f_{s,P_n} P_n f_{t,P_n}
$
converges uniformly and therefore also pointwise to a uniformly continuous function on $(T,d_T)$. This convergence is possible because $\{(t,s) \longmapsto c_{P}(t,s) : P \in \mathcal{P}\}$ is a relatively compact set in $\ell^\infty(T\times T)$ in view of the Arzela-Ascoli Theorem, the assumptions in equation (\ref{eq: characteristics1}), and total boundedness of $(T,d_T)$.  By Lemma \ref{lemma: Donsker dep on n}(b) pointwise convergence of the covariance function implies weak convergence to a tight Gaussian process which may depend on the identity $\mathbb{N}'$ of the subsequence.
Since this argument applies to each such subsequence that split the overall sequence, part (b) follows.

Part (c)
is immediate from the imposed uniform covering entropy condition and Dudley's metric entropy inequality for expectations of suprema of Gaussian processes (Corollary 2.2.8 in \citen{vdV-W}).  Claim (d) follows from claim (c) and a standard argument, based on the application of the Borel-Cantelli lemma. Indeed, let $m \in \N$ be a sequence and $\delta_m:= 2^{-m}\wedge \ \sup\left \{ \delta>0 : \sup_{P \in \mathcal{P}} \Ep_P \sup_{d_T(t, \bar t) \leq \delta} |Z_{P}(t) - Z_{P}(\bar t)| < 2^{-2m} \right \},$ then by the Markov inequality
$
\Pr_P \left ( \sup_{d_T(t, \bar t) \leq \delta_m} |Z_{P}(t) - Z_{P}(\bar t)| > 2^{-m}  \right ) \leq  2^{-2m+m} = 2^{-m}.
$
This sums to a finite number over $m \in \mathbb{N}$. Hence, by the Borel-Cantelli lemma, for almost all states $\omega \in \Omega$,
$|Z_{P}(t)(\omega) - Z_{P}(\bar t)(\omega)| \leq 2^{-m}$  for all $d_T(t, \bar t) \leq \delta_m \leq 2^{-m}$ and all $m$ sufficiently large. Hence
claim (d) follows.
\qed

\subsection{Proof  of Theorem \ref{lemma: uniform Donsker for bootstrap}}  Claim (a) is verified by  invoking Theorem \ref{lemma: uniform Donsker}.    We begin by showing that $Z^*_{P} =   (\mathbb{G}_{P}  \xi f_{t,P})_{t \in T}$ is equal in distribution to
 $Z_P = (\mathbb{G}_{P}  f_{t,P})_{t \in T}$, in particular, $Z^*_P$ and $Z_P$ share identical mean and covariance function, and thus they  share the continuity properties established in Theorem \ref{lemma: uniform Donsker}. This claim is immediate from the fact that multiplication by $\xi$ of  each $f \in \mF_P=\{f_{t,P}: t \in T\}$ yields a set $\xi \mF_P$ of measurable functions $\xi f:  (w,\xi) \longmapsto \xi f(w)$, mapping $\mathcal{W} \times \mathbb{R}$ to $\mathbb{R}$. Each such function has mean zero under $P \times P_\xi$,  i.e.  $\int s f(w) dP_{\xi}(s) dP(w) = 0$,  and covariance function $(\xi f,  \xi \tilde f) \longmapsto P f \tilde f - Pf P \tilde f$.  Hence the Gaussian process $(\mathbb{G}_P( \xi f))_{\xi f \in \xi \mF_P}$  shares the zero mean and the covariance function of $(\mathbb{G}_P(f))_{f \in \mF_P}$.

We are claiming that $Z^*_{n,P} \rightsquigarrow  Z^*_{P}$ in $\ell^\infty(T)$ uniformly in $P \in \mP$, where $Z^*_{n, P} :=   (\mathbb{G}_{n}   \xi f_{t,P})_{t \in T}$. We note that the function class $\mF_P$ and the corresponding envelope $F_P$ satisfy the conditions of Theorem \ref{lemma: uniform Donsker}.  The same is also true for the function class $\xi \mathcal{F}_P$ defined by $(w,\xi) \longmapsto \xi f_P(w)$, which maps $\mathcal{W}\times \mathbb{R}$ to $\mathbb{R}$ and its envelope $ |\xi| F_P$, since $\xi$ is independent of $W$.  Let $Q$ now denote a finitely discrete measure over $\mathcal{W}\times \mathbb{R}$. By Lemma \ref{lemma: andrews} multiplication by $\xi$ does not change qualitatively the uniform covering entropy bound:
$$
\log \sup_Q N( \epsilon \| |\xi| F_P\|_{Q,2} , \xi \mathcal{F}_P, \| \cdot \|_{Q,2}) \leq
 \log \sup_Q N( 2^{-1} \epsilon \| F_P\|_{Q,2},  \mathcal{F}_P, \| \cdot \|_{Q,2}).
$$
Moreover, multiplication by $\xi$ does not affect the norms, $\|\xi f_P(W)\|_{P\times P_\xi,2}= \|f_P(W)\|_{P,2}$, since $\xi$ is independent of $W$ by construction and $\Ep \xi^2 =1$.  The claim then follows.

Claim (b).  For each $\delta>0$ and $t \in T$, let $\pi_\delta (t)$ denote a closest element in a given, finite $\delta$-net
over $T$.   We begin by noting that
\begin{eqnarray*}
\Delta_P& := & \sup_{h \in \textrm{BL}_1} | \Ep_{B_n} h (Z^*_{n,P}) - \Ep_P h(Z_P) |  \\
& \leq &  I_P + II_P + III_P :=    \sup_{h  \in \textrm{BL}_1} | \Ep_P h(Z_P \circ \pi_\delta) -\Ep_P h(Z_P)| \\
 &  & +  \sup_{h  \in \textrm{BL}_1} | \Ep_{B_n} h( Z^*_{n,P} \circ \pi_\delta) - \Ep_P h(Z_P \circ \pi_\delta)|
 + \sup_{h  \in \textrm{BL}_1} | \Ep_{B_n} h (Z^*_{n,P}\circ \pi_\delta ) - \Ep_{B_n} h (Z^*_{n,P}) |,
\end{eqnarray*}
where  here and below $\textrm{BL}_1$ abbreviates $\textrm{BL}_1(\ell^{\infty}(T))$.

First,  we note that
$
I_P \leq  \Ep_P \left(\sup_{d_T(t, \bar t) \leq \delta}  | Z_{P}(t)  -  Z_{P}(\bar t)| \wedge 2 \right) =: \mu_P(\delta)$ and
 $\lim_{\delta \searrow 0} \sup_{P \in \mathcal{P}} \mu_P(\delta) = 0.
$
The first assertion follows from $$ I_P  \leq    \sup_{h  \in \textrm{BL}_1}  \Ep_P| h (Z^*_{n,P}\circ \pi_\delta ) - h (Z^*_{n,P}) |  \leq \Ep_P  \left( \sup_{t \in T}  | Z_{P}\circ \pi_\delta(t)  -  Z_{P}(t)| \wedge 2 \right)  \leq  \mu_P(\delta),$$ and the second assertion holds by Theorem \ref{lemma: uniform Donsker} (c).

Second,   $\Ep^*_PIII_P  \leq   \Ep^*_P  \left( \sup_{d_T(t, \bar t) \leq \delta}  | Z^*_{n,P}(t)  -  Z^*_{n,P}(\bar t)| \wedge 2 \right) =: \mu^*_P(\delta)$ and $ \lim_{n\to \infty} \sup_{P \in \mP} |\mu^*_P(\delta) - \mu_P(\delta) | = 0.$
The first assertion follows because $\Ep^*_PIII_P$ is bounded by
\begin{eqnarray*}
  \Ep^*_P \sup_{h  \in \textrm{BL}_1}  \Ep_{B_n}  | h (Z^*_{n,P}\circ \pi_\delta ) - h (Z^*_{n,P}) | \leq \Ep^*_P  \Ep_{B_n}  \left( \sup_{t \in T}  | Z^*_{n,P}\circ \pi_\delta(t)  -  Z^*_{n,P}(t)| \wedge 2 \right) \leq  \mu^*_P(\delta).
\end{eqnarray*}
The second assertion holds by part (a) of the present theorem.

Define $
\epsilon(\delta)  := \delta \vee \sup_{P \in \mP} \mu_P(\delta).
$
Then, by Markov's inequality, followed by $n \to \infty$,
$$
\limsup_{n \to \infty} \sup_{P \in \mP} \Pr^*_P \left (  III_P > \sqrt{ \epsilon (\delta) }\right) \leq \limsup_{n\to \infty} \frac{\sup_{P \in \mP} \mu^*_P(\delta)}{\sqrt{\epsilon(\delta)}} \leq  \frac{\sup_{P \in \mP} \mu_P(\delta)}{\sqrt{\epsilon (\delta)}} \leq \sqrt{\epsilon (\delta)}.
$$

Finally, by Lemma \ref{lemma: fidi mb}, 
 for each $\varepsilon>0$, $
\limsup_{n \to \infty} \sup_{P \in \mP} \Pr^*_P \left (  II_P > \varepsilon \right) = 0.
$

We can now conclude. Note that $\epsilon(\delta) \searrow 0$ if $\delta \searrow 0$, which holds by the definition of $\epsilon(\delta)$ and the property   $\sup_{P \in \mP} \mu_P(\delta) \searrow 0$  if $\delta \searrow 0$ noted above. Hence for each $\varepsilon>0$ and all $0<\delta< \delta_\varepsilon$ such that $3 \sqrt{ \epsilon (\delta)}< \varepsilon$,  \begin{eqnarray*}
\limsup_{n \to \infty} \sup_{P \in \mP} \Pr_P^*  \left (  \Delta_P >  \varepsilon  \right)
\leq \limsup_{n \to \infty} \sup_{P \in \mP} \Pr_P^* \left (  I_P + II_P + III_P >  3 \sqrt{ \epsilon (\delta) } \right) \leq  \sqrt{\epsilon (\delta)}.
\end{eqnarray*}
Sending $\delta \searrow 0$ gives the result. \qed

\subsection{Auxiliary Result: Conditional Multiplier CLT in $\mathbb{R}^d$ uniformly in $P \in \mathcal{P}$.}

We rely on the following lemma, which is apparently new.  An analogous result can be derived for almost sure convergence
from well-known non-uniform multiplier central limit theorems, but this strategy requires us to put all the variables indexed by $P$ on the single underlying probability space, which is much less convenient in applications.

\begin{lemma}[Conditional Multiplier Central Limit Theorem in $\mathbb{R}^d$ uniformly in $P \in \mathcal{P}$]  \label{lemma: fidi mb}   Let $(Z_{i,P})_{i=1}^\infty$ be i.i.d. random vectors on $\mathbb{R}^d$, indexed by  a parameter $P \in \mP$. The parameter $P$ represents probability laws on $\mathbb{R}^d$.  For each $P \in \mathcal{P}$, these vectors are assumed to be independent  of the i.i.d. sequence $ (\xi_i)_{i=1}^\infty$ with $\Ep \xi = 0$ and $\Ep \xi^2 =1$. There exist constants
$2< q< \infty $ and $0<M<  \infty$, such that $\Ep_P Z_{1,P} =0$ and $(\Ep_P\|Z_{1,P}\|^q)^{1/q} \leq M$ uniformly for all $P \in \mP$. Then,
for every $\varepsilon>0$
$$
\lim_{n \to \infty} \sup_{P \in \mathcal{P}} \Pr^*_P \left (  \sup_{h \in \mathrm{BL}_1(\mathbb{R}^d)} \Big | \Ep_{B_n} h\Big ( n^{-1/2} \sum_{i=1}^n \xi_i Z_{i,P}\Big ) - \Ep_Ph \Big ( N(0, \Ep_PZ_{1,P} Z_{1,P}') \Big )\Big | > \varepsilon \right ) = 0,
$$
where $\Ep_{B_n}$ denotes the expectation over $ (\xi_i)_{i=1}^n$  holding $(Z_{i,P})_{i=1}^n$ fixed.
\end{lemma}

\textbf{Proof of Lemma \ref{lemma: fidi mb}.}  Let $X$ and $Y$ be random variables in $\mathbb{R}^d$, then define
$
d_{BL}( X, Y) :=\sup_{h \in \mathrm{BL}_1(\mathbb{R}^d)} | \Ep h(X) - \Ep h ( Y)  |.$
  It suffices to show that for  any sequence $P_n \in \mathcal{P}$ and
$
N^* \sim  n^{-1/2} \sum_{i=1}^n \xi_i Z_{i,P_n} \mid  (Z_{i,P_n})_{i=1}^n, \ \  d_{BL}\left ( N^* , N(0, \Ep_{P_n} Z_{1,P_n} Z_{1,P_n}')  \right)   \to 0
$
in probability (under $\Pr_{P_n}$).

Following \citen{BickelFreedman:1981}, we shall rely on the Mallow's metric, written $m_r$, which  is a metric on the space of distribution functions on $\mathbb{R}^d$. For our purposes it suffices to recall that given a sequence of distribution functions $\{F_k\}$ and a distribution function $F$,  $m_r(F_k,F) \to 0$ if and only if $\int g d F_k \to \int g dF$ for each continuous and bounded $g: \mathbb{R}^d \to \mathbb{R}$, and $\int \| z\|^r dF_k(z) \to \int \|z\|^r d F(z)$. See \citen{BickelFreedman:1981} for the definition of $m_r$.

Under the assumptions of the lemma, we can split the sequence $n \in \mathbb{N}$ into subsequences $n \in \mathbb{N}'$, along each of which  the distribution function of $Z_{1,P_n}$ converges to  some distribution function $F'$ with respect to the Mallow's metric $m_{r}$, for some $2 < r <q$.  This also implies that $N(0, \Ep_{P_n} Z_{1,P_n} Z_{1,P_n}')$ converges weakly to a normal limit $N(0, Q')$  with  $Q' = \int z z' d F'(z)$
such that  $\|Q'\| \leq M$.   Both $Q'$ and $F'$ can depend on the subsequence $\mathbb{N}'$.

Let  $F_k$ be  the empirical distribution function of a  sequence $(z_i)_{i=1}^k$ of constant vectors in $\mathbb{R}^d$,
where $k  \in \N$. The law of  $N^*_{F_k}= k^{-1/2} \sum_{i=1}^k \xi_i z_i$ is completely determined by  $F_k$ and the law of $\xi$ (the latter is fixed, so it does not enter as the subscript in the definition of $N^*_{F_k}$).    If $m_r(F_k, F') \to 0$ as $k \to \infty$, then $d_{BL}(N^*_{F_k},  N(0,Q'))  \to 0$  by Lindeberg's central limit theorem.

Let $\mathbb{F}_n$ denote the empirical distribution function of $\(Z_{i,P_n}\)_{i=1}^n$.  Note that $N^* =  N^*_{\mathbb{F}_n} \sim  n^{-1/2} \sum_{i=1}^n \xi_i Z_{i,P_n} \mid  (Z_{i,P_n})_{i=1}^n$. By the law of large numbers for arrays,
$\int g d \mathbb{F}_n \to \int g d F'$ and $\int \|z\|^r d\mathbb{F}_n(z) \to  \int \|z\|^r dF'(z) $ in probability along the subsequence $n \in \N'$. Hence $m_r (\mathbb{F}_n, F') \to 0$ in probability  along the same subsequence.  We can conclude that   $d_{BL}(N^*_{\mathbb{F}_n},  N(0,Q')) \to 0 $  in probability along the same subsequence by the extended continuous mapping theorem \cite[Theorem 1.11.1]{vdV-W}.

The argument applies to every subsequence $\mathbb{N}'$ of the stated form. The claim in the first paragraph of the proof thus follows.  \qed

\subsection{Donsker Theorems  for Function Classes that depend on $n$ }
Let $ (W_i)_{i=1}^\infty$ be a sequence of i.i.d. copies of the random element $W$ taking values in the measure space $(\mathcal{W}, \mathcal{A}_{\mathcal{W}})$, whose law is determined by the probability measure $P$, and let $w \longmapsto f_{n,t}(w)$
be measurable functions  $f_{n,t}: \mathcal{W} \to \mathbb{R}$ indexed by $n \in \N$ and a fixed,
totally bounded semi-metric space $(T, d_T)$.  Consider the stochastic process
$$
(\mathbb{G}_n f_{n,t})_{t \in T} := \left \{ n^{-1/2} \sum_{i=1}^n (f_{n,t}(W_i) - Pf_{n,t} )\right \}_{t \in T}.
$$
This empirical process is indexed by a class of functions $\mathcal{F}_n = \{ f_{n,t}: t \in T\}$ with a measurable envelope
function $F_n$. It is important to note  here that the dependence on $n$ allows us to have \textit{the class itself} be possibly
dependent on the law $P_n$.

\begin{lemma}[\textbf{Donsker Theorem for Classes Changing with $n$}]\label{lemma: Donsker dep on n} Work with the set-up above. Suppose that for some fixed constant $q >2$ and
  every sequence $\delta_n \searrow 0$:
 \begin{eqnarray*}
  \|F_n\|_{P_n,q} = O(1),  \quad \sup_{d_T(s,t) \leq \delta_n} \| f_{n,s} - f_{n,t}\|_{P_n,2} \to 0,\quad  \int_0^{\delta_n} \sup_{Q } \sqrt{  \log N( \epsilon \|F_n\|_{Q,2}, \mathcal{F}_n, \| \cdot \|_{Q,2} ) } d \epsilon \to 0.
\end{eqnarray*}
(a) Then the empirical process $(\mathbb{G}_n f_{n,t})_{t \in T}$ is asymptotically tight in $\ell^\infty(T)$ i.e. stochastically equicontinuous.
(b) For any subsequence such that the covariance function $P_n f_{n,s} f_{n,t} - P_n f_{n,s} P_n f_{n,t}$ converges pointwise
 on $T\times T$, $(\mathbb{G}_n f_{n,t})_{t \in T}$ converges in $\ell^\infty(T)$ to a Gaussian process with covariance
 function given by the limit of the covariance function along that subsequence. \end{lemma}

\textbf{Proof.}  The use of Theorem 2.11.1  in \citen{vdV-W}, which does allow for  the probability space to depend on $n$, allows us to establish claim (a),
by repeating the proof (verbatim) of Theorem 2.11.22 in \citen[p. 220-221]{vdV-W}, except that the probability law is allowed to depend on $n$. (For the sake of completeness, the Supplementary Appendix, provides the complete proof). The proof of claim (b)
follows by a standard argument from the stochastic equicontinuity established in claim (a) and finite-dimensional convergence along the indicated subsequences. \qed

\subsection{Proof of Theorems \ref{thm: delta-method} and \ref{theorem:delta-method-bootstrap}.}
The proof consists of two parts, each proving the corresponding theorem.

Part 1.  We can split $\N$ into subsequences $\{ \N'\}$ along each of which
$
Z_{n, P_n} \rightsquigarrow Z' \in \D_0 \text{ in }  \D, \ \ \rho_{P_n} \to \rho'  \text{ in }  \D_{\rho}  \ \ (n \in \N'),
$
where $Z'$ and $\rho'$ can possibly depend on $\N'$.  It suffices to verify that for each $\N'$:
\begin{eqnarray}
\label{dclaim1} && r_n ( \phi( \hat \rho_{n,P_n}) - \phi(\rho_{P_n})) \rightsquigarrow \phi'_{\rho'}(Z')  \ \ (n \in \N') \\
\label{dclaim2} && r_n ( \phi( \hat \rho_{n,P_n}) - \phi(\rho_{P_n}))-  \phi'_{\rho_{P_n}} (  r_n ( \hat \rho_{n,P_n}- \rho_{P_n}))
 \rightsquigarrow 0  \ \ (n \in \N'), \\
 \label{dclaim3} && r_n ( \phi( \hat \rho_{n,P_n}) - \phi(\rho_{P_n}))-  \phi'_{\rho'} (  r_n ( \hat \rho_{n,P_n}- \rho_{P_n}))
 \rightsquigarrow 0  \ \ (n \in \N'),
\end{eqnarray}
where the last two claims hold provided that  $(\rho, h) \longmapsto \phi_{\rho}^{'}(h)$ is defined and continuous on the whole of $\D_\rho \times \D$.  The claim  (\ref{dclaim3}) is not needed in Part 1, but we need it for Part 2.

The map  $g_n(h) = r_n ( \phi(\rho_{P_n} + r_n^{-1}h ) - \phi(\rho_{P_n}))$, from $\D_n= \{ h \in \D: \rho_{P_n} + r_n^{-1} h \in \D_\phi\}$ to $\E$, satisfies $g_n(h_n) \to \phi'_{\rho'}(h)$ for every subsequence $h_n \to h \in \D_0$  (with $n \in \mathbb{N}')$. Application of the extended continuous mapping theorem \cite[Theorem 1.11.1]{vdV-W}  yields (\ref{dclaim1}).

Similarly, the map  $m_n(h) = r_n ( \phi(\rho_{P_n} + r_n^{-1}h ) - \phi(\rho_{P_n})) -   \phi'_{\rho_{P_n}}(h)$, from $\D_n= \{ h \in \D: \rho_{P_n} + r_n^{-1} h \in \D_\phi\}$ to $\E$, satisfies $m_n(h_n) \to \phi'_{\rho'}(h)- \phi'_{\rho'}(h) = 0$ for every subsequence $h_n \to h \in \D_0$  (with $n \in \mathbb{N}')$.  Application of the extended continuous mapping theorem \cite[Theorem 1.11.1]{vdV-W} yields (\ref{dclaim2}).   The proof of (\ref{dclaim3}) is completely analogous and is omitted.

To establish relative compactness, we work with each $\N'$. Then
$ \phi'_{\rho_{P_n}}(h)$ mapping $\D_0$ to $\E$ satisfies $ \phi'_{\rho_{P_n}}(h_n) \to \phi'_{\rho'}(h)$ for every subsequence $h_n \to h \in \D_0$  (with $n \in \N')$.  Application of the extended continuous mapping theorem \cite[Theorem 1.11.1]{vdV-W}  yields that  $\phi'_{\rho_{P_n}}(Z_{P}) \rightsquigarrow \phi'_{\rho'}(Z') $.

Part 2.  We can split $\N$ into subsequences $\{ \N'\}$ as above.  Along each $\N'$,
$$
r_n( \hat \rho^*_{n, P_n} - \rho_{P_n}) \rightsquigarrow Z'' \in \D_0 \text{ in }  \D, \  \  r_n( \hat \rho_{n, P_n} - \rho_{P_n}) \rightsquigarrow Z'  \in \D_0 \text{ in }  \D, \  \ \rho_{P_n} \to \rho'  \text{ in }  \D_{\rho}  \ \ (n \in \N'),
$$
where $Z''$ is a separable process in $\D_0$ (which is given by $Z'$ plus its independent copy $\bar Z'$). Indeed, note that $r_n( \hat \rho^*_{\rho_{n,P_n}} - \rho_{P_n})  = Z_{n,P_n}^* + Z_{n,P_n}$, and $(Z_{n,P_n}^*,Z_{n,P_n})$ converge weakly unconditionally to $(\bar Z', Z')$ by  a standard argument.

  Given each  $\N'$ the proof is similar to the proof of Theorem 3.9.15 of \citen{vdV-W}.    We can assume without loss of generality that the derivative
$\phi'_{\rho'}: \D \to \E$ is defined and continuous on the whole of $\D$. Otherwise, if $\phi'_{\rho'}$ is defined and continuous only on $\D_0$,  we can extend it to $\D$ by a Hahn-Banach extension such that $C= \|\phi'_{\rho'}\|_{\D_0 \to \mathbb{E}} =  \|\phi'_{\rho'}\|_{\D \to \mathbb{E}}< \infty$; see \citen[p. 380]{vdV-W} for details.   For each
$\N'$, by claim (\ref{dclaim3}), applied to $\hat \rho_{n,P_n}$ and to $\hat \rho^*_{n,P_n}$ replacing
$\hat \rho_{n,P_n}$,
\begin{eqnarray*}
& r_n ( \phi(\hat \rho_{n,P_n} ) - \phi (\rho_{P_n} )  ) = \phi'_{\rho'} ( r_n ( \hat \rho_{n,P_n} - \rho_{P_n} )  ) + o^*_{P_n}(1), \\
& r_n ( \phi(\hat \rho^*_{n,P_n} ) - \phi (\rho_{P_n} )  ) = \phi'_{\rho'} ( r_n ( \hat \rho^*_{n,P_n} - \rho_{P_n} )  ) + o^*_{P_n}(1).
\end{eqnarray*}
Subtracting these equations conclude that  for each $\eps>0$
\begin{equation}\label{good thing}
\Ep_{P_n} 1 \left (  \Big \|r_n ( \phi(\hat \rho^*_{n,P_n} ) - \phi (\hat \rho_{n,P_n}   ) )-   \phi'_{\rho'}  ( r_n ( \hat \rho^*_{n,P_n} - \hat \rho_{n,P_n}) )\Big \|^*_{\E} > \eps \right) \to 0   \ \ (n \in \N').
\end{equation}
For every $h \in \mathrm{BL}_1 (\E)$, the function $h \circ \phi'_{\rho'}$ is contained in $\mathrm{BL}_C(\D)$.  Moreover,
$r_n(\hat \rho^{*}_{n,P} - \hat \rho_{n, P}) \rightsquigarrow_{B} Z_{P}$ in $\D$ uniformly in $P \in \mP_n$
implies $r_n(\hat \rho^{*}_{n,P} - \hat \rho_{n, P}) \rightsquigarrow_{B} Z'$ along the subsequence  $n \in \N'$. These two facts  imply that
$$
\sup_{h \in \mathrm{BL}_1(\E)} \left |   \Ep_{B_n} h  \Big (\phi'_{\rho'}  ( r_n ( \hat \rho^*_{n,P_n} - \hat \rho_{n,P_n}) ) \Big ) -  \Ep h(\phi_{\rho'}(Z')) \right| = o^*_{P_n}(1)   \ \   (n \in \N').
$$
Next for each $\eps>0$ and along $n \in \N'$
\begin{eqnarray*}
\nonumber
&& \sup_{h \in \mathrm{BL}_1(\E)}  \left |  \Ep_{B_n} h \Big (  r_n ( \phi(\hat \rho^*_{n,P_n} ) - \phi (\hat \rho_{n,P_n} )  ) \Big )
-  \Ep_{B_n} h  \Big (\phi'_{\rho'}  ( r_n ( \hat \rho^*_{n,P_n} - \hat \rho_{n,P_n}) ) \Big )  \right | \\
&& \leq \eps +   2  \Ep_{B_n} 1\left  (  \Big \|r_n ( \phi(\hat \rho^*_{n,P_n} ) - \phi (\hat \rho_{n,P_n}   ) ) -   \phi'_{\rho'}  ( r_n ( \hat \rho^*_{n,P_n} - \hat \rho_{n,P_n})) \Big \|^*_{\E}     > \eps \right ) = o_{P_n} (1), \label{some ineq bs}
\end{eqnarray*}
where the $o_{P_n}(1)$ conclusion follows by the Markov inequality and by (\ref{good thing}).   Conclude that
\begin{eqnarray*}
\sup_{h \in \mathrm{BL}_1(\E)}  \left |  \Ep_{B_n} h \Big (  r_n ( \phi(\hat \rho^*_{n,P_n} ) - \phi (\hat \rho_{n,P_n} )  ) \Big )
 -   \Ep h (\phi_{\rho'}(Z')) \right |= o^*_{P_n}(1) \ \ (n \in \N'). \text{{\tiny$\qed$}}
\end{eqnarray*}

\section{Key Tools II: Probabilistic Inequalities}
Let $ ( W_i)_{i=1}^n$ be a sequence of i.i.d. copies of random element $ W$  taking values in the measure space $({\mathcal{W}}, \mathcal{A}_{{\mathcal{W}}})$ according to probability law $P$. Let $\mathcal{F}$ be a set  of suitably measurable functions $f:{\mathcal{W}} \longmapsto \mathbb{R}$, equipped with a measurable envelope $F: \mathcal{W} \longmapsto \mathbb{R}$.

The following maximal inequality is due to  \citen{CCK12}.

  \begin{lemma}[\textbf{A Maximal Inequality}]
\label{lemma:CCK}  Work with the setup above.  Suppose that $F\geq \sup_{f \in \mathcal{F}}|f|$ is a measurable envelope
with $\| F\|_{P,q} < \infty$ for some $q \geq 2$.  Let $M = \max_{i\leq n} F(W_i)$ and $\sigma^{2} > 0$ be any positive constant such that $\sup_{f \in \mF}  \| f \|_{P,2}^{2} \leq \sigma^{2} \leq \| F \|_{P,2}^{2}$. Suppose that there exist constants $a \geq e$ and $v \geq 1$ such that
$\log \sup_{Q} N(\epsilon \| F \|_{Q,2}, \mF,  \| \cdot \|_{Q,2}) \leq  v (\log a + \log(1/\epsilon)), \ 0 <  \epsilon \leq 1.
$
Then
\begin{equation*}
\Ep_P [ \| \bG_{n} \|_{\mF} ] \leq K  \left( \sqrt{v\sigma^{2} \log \left ( \frac{a \| F \|_{P,2}}{\sigma} \right ) } + \frac{v\| M \|_{\Pr_P, 2}}{\sqrt{n}} \log \left ( \frac{a \| F \|_{P,2}}{\sigma} \right ) \right),
\end{equation*}
where $K$ is an absolute constant.  Moreover, for every $t \geq 1$, with probability $> 1-t^{-q/2}$,
\begin{multline*}
\| \bG_{n} \|_{\mF} \leq (1+\alpha) \Ep_P [ \| \bG_{n} \|_{\mF} ] + K(q) \Big [ (\sigma + n^{-1/2} \| M \|_{\Pr_P,q}) \sqrt{t}
+  \alpha^{-1}  n^{-1/2} \| M \|_{\Pr_P,2}t \Big ], \ \forall \alpha > 0,
\end{multline*}
where $K(q) > 0$ is a constant depending only on $q$.  In particular, setting $a \geq n$ and $t = \log n$,
with probability $> 1- c(\log n)^{-1}$,
\begin{equation} \label{simple bound}
\| \bG_{n} \|_{\mF} \leq K(q,c) \left ( \sigma \sqrt{v \log \left ( \frac{a \| F \|_{P,2}}{\sigma} \right ) } + \frac{v
 \| M \|_{\Pr_P,q} } {\sqrt{n}}\log \left ( \frac{a \| F \|_{P,2}}{\sigma} \right ) \right),
\end{equation}
where $  \| M \|_{\Pr_P,q}  \leq n^{1/q} \| F\|_{P,q}$ and  $K(q,c) > 0$ is a constant depending only on $q$ and $c$.

\end{lemma}

%

\section{Proofs for Section 4}

These results follow from the application of results given in Section 5.  The details
are given in the Supplementary Appendix.

\section{Proofs for Section 5}\label{subsec:ProofSection5}

\subsection{Proof of Theorem \ref{theorem:semiparametric}} In the proof $a \lesssim b$ means that $a \leq A b$, where the constant
$A$ depends on the constants  in Assumptions \ref{ass: S1}--\ref{ass: AS}, but not on $n$ once $n \geq n_0$, and not on $P \in \mathcal{P}_n$.   Since the argument is asymptotic, we can assume that $n \geq n_0$ in what follows. In order to establish the result uniformly in $P \in  \mathcal{P}_n$, it suffices to establish the result under the probability measure induced by any sequence $P = P_n \in \mathcal{P}_n$.  In the proof we shall use $P$, suppressing the dependency of $P_n$ on the sample size $n$.  Also, let
\begin{eqnarray}\label{eq: B}
 B(W) & : = & \max_{ j \in [d_{\theta}], k \in [d_{\theta} + d_t]} \sup_{\nu \in \Theta_u \times T_u(Z_u), u \in \mathcal{U}} \Big | \partial_{\nu_k} \Ep_P[\psi_{uj}(W_u, \nu)\mid Z_u] \Big |,
\end{eqnarray}

Step 1. (A Preliminary Rate Result). In this step we claim that  wp $1- o(1)$, $\sup_{u \in \mathcal{U}}\| \hat \theta_u - \theta_u\| \lesssim \tau_n.$
By definition $$\| \En \psi_u(W_u, \hat \theta_u, \hat h_u(Z_u))\| \leq \inf_{\theta \in \Theta_u}\| \En \psi_u(W_u, \theta, \hat h_u(Z_u))\| + \epsilon_n \text{  for each $u \in \mathcal{U}$, }$$ which implies
via triangle inequality that uniformly in $u \in \mathcal{U}$ with probability $1-o(1)$
\begin{equation}\label{eq: rate proof}
\Big \| P [\psi_u(W_u, \hat \theta_u, h_u(Z_u))] \Big \| \leq \epsilon_n + 2 I_1+ 2 I_2 \lesssim  \tau_n,
 \end{equation}
for $I_1$ and $I_2$ defined in Step 2 below.   The $\lesssim$  bound in  (\ref{eq: rate proof}) follows
 from Step 2 and from the assumption $\epsilon_n = o(n^{-1/2})$.
 Since by Assumption \ref{ass: S1}(iv), $2^{-1} ( \| J_u(\hat \theta_u- \theta_u)\| \wedge c_0)$ does not exceed the left side of  (\ref{eq: rate proof}) and  $\inf_{u \in \mathcal{U}}\textrm{mineig}(J_u'J_u)$ is bounded away from zero uniformly in $n$, we conclude that
$\sup_{u \in \mathcal{U}}\| \hat \theta_u - \theta_u\|  \lesssim (\inf_{u \in \mathcal{U}} \textrm{mineig}(J_u'J_u))^{-1/2}  \tau_n \lesssim \tau_n.$

Step 2. (Define and bound $I_1$ and $I_2$)  We claim that with probability $1- o(1)$:
\begin{eqnarray*}
I_1 & :=&  \sup_{\theta \in \Theta_u, u \in \mathcal{U}} \Big \| \En \psi_u(W_u, \theta, \hat h_u(Z_u)) - \En \psi_u(W_u, \theta, h_u(Z_u) ) \Big \| \lesssim  \tau_n , \\
I_2 & := &   \sup_{\theta \in \Theta_u, u \in \mathcal{U}}  \Big \| \En \psi_u(W_u, \theta, h_u(Z_u)) - P \psi_u(W_u, \theta, h_u(Z_u) ) \Big \| \lesssim  \tau_n.
\end{eqnarray*}
To establish this, we can bound $I_1 \leq 2I_{1a} + I_{1b}$ and $I_2 \leq I_{1a}$, where
with probability $1- o(1)$,
\begin{eqnarray*}
I_{1a} & :=&  \sup_{\theta \in \Theta_u, u \in \mathcal{U}, h \in \mathcal{H}_{un} \cup \{ h_u \}}  \Big\| \En \psi_u(W_u, \theta, h(Z_u)) - P \psi_u(W_u, \theta, h(Z_u) )  \Big \| \lesssim  \tau_n , \\
I_{1b} & := &   \sup_{\theta \in \Theta_u, u \in \mathcal{U}, h \in \mathcal{H}_{un} \cup \{ h_u \} }  \Big \| P \psi_u(W_u, \theta, h(Z_u)) - P \psi_u(W_u, \theta, h_u(Z_u) )  \Big \| \lesssim  \tau_n.
\end{eqnarray*}
These bounds in turn hold by the following arguments. In order to bound $I_{1b}$ we employ Taylor's expansion and the triangle inequality. For $\bar h(Z,u, j, \theta)$ denoting a point on a line connecting vectors $h(Z_u)$ and $h_u(Z_u)$, and ${t_m}$ denoting the $m$th element of the vector $t$,
\begin{eqnarray*}
I_{1b} & \leq & \sum_{j=1}^{d_{\theta}} \sum_{m=1}^{d_t}  \sup_{\theta \in \Theta_u, u \in \mathcal{U}, h \in \mathcal{H}_{un}} \Big  |  P\left [ \partial_{t_m} P \left [ \psi_{uj}(W_u, \theta, \bar h(Z, u, j, \theta))|Z_u \right ] (h_m(Z_u) - h_{um}(Z_u))  \right]  \Big | \\
& \leq&  d_\theta d_t   \|B\|_{P,2} \max_{u \in \mU, h \in \mathcal{H}_{un}, m \in [d_t]} \| h_m - h_{um}\|_{P,2},
\end{eqnarray*}
where the last inequality holds by the definition of $B(W)$ given earlier and H\"{o}lder's inequality.
 By  Assumption \ref{ass: S2}(ii)(c),  $\|B\|_{P,2}\leq C$, and  by Assumption \ref{ass: AS},  $\sup_{u \in \mU,  h \in \mathcal{H}_{un}, m \in [d_t]} \| h_m - h_{um}\|_{P,2} \lesssim \tau_n$, hence we conclude that $I_{1b} \lesssim \tau_n$ since $d_\theta$ and $d_t$ are fixed.

In order to bound $I_{1a}$ we employ the maximal inequality of Lemma \ref{lemma:CCK}
to the class $$\mathcal{F}_{1} = \{ \psi_{uj}(W_u, \theta, h(Z_u)):  j \in [d_\theta], u \in \mathcal{U}, \theta \in \Theta_u, h \in
\mathcal{H}_{un} \cup \{ {h}_{u}\}  \},$$
defined in Assumption \ref{ass: AS} and equipped with an envelope  $F_1 \leq F_0$, to conclude that with probability $1-o(1)$,
\begin{eqnarray*}
I_{1a} & \lesssim   n^{-1/2}  \Big (    \sqrt{ s_n \log (a_n) } + n^{-1/2} s_n n^{\frac{1}{q}}  \log (a_n)  \Big)  \lesssim \tau_n.
\end{eqnarray*}
Here we use that $\log \sup_{Q} N(\epsilon \| F_1 \|_{Q,2}, \mF_1,  \| \cdot \|_{Q,2}) \leq s_n \log (a_n/\epsilon) \vee 0$ by Assumption  \ref{ass: AS}; $\|F_0\|_{P, q} \leq C$ and $ \sup_{f \in \mathcal{F}_1} \| f\|^2_{P,2} \leq \sigma^2 \leq \| F_0 \|^2_{P,2}$ for $c \leq \sigma \leq C$ by Assumption \ref{ass: S2}(i); $a_n \geq n$ and $s_n\geq 1$ by Assumption  \ref{ass: AS}.

Step 3. (Linearization)  By definition
$$\sqrt{n} \|  \En \psi_u(W_u, \hat \theta_u, \hat h_u(Z_u) ) \| \leq \inf_{\theta \in \Theta_u} \sqrt{n} \|  \En \psi_u(W_u, \theta, \hat h_u(Z_u) ) \|+ \epsilon_n n^{1/2}.$$
Application of Taylor's theorem give that for all $u \in \mathcal{U}$
\begin{eqnarray*}
\sqrt{n}  \En \psi_u(W_u, \hat \theta_u, \hat h_u(Z_u) ) &=&
 \sqrt{n}  \En \psi_u(W_u, \theta_u, h_u(Z_u))  \\
 &+ & J_u \sqrt{n} (\hat \theta_u - \theta_u) + \mathrm{D}_{u,0}(\hat h_u -h_u ) + II_1(u) + II_2(u),
\end{eqnarray*}
where the terms $II_1(u)$ and $II_2(u)$ are defined in Step 4 and $\mathrm{D}_{u,0}(\hat h_u -h_u )$ is treated in the next paragraph.
Then  by the triangle inequality for all $u \in \mathcal{U}$ and Steps 4 and 5 we have
\begin{eqnarray*}
&& \left \| \sqrt{n}  \En \psi_u(W_u, \theta_u, h_u(Z_u)) + J_u \sqrt{n} (\hat \theta_u - \theta_u) + \mathrm{D}_{u,0}(\hat h_u -h_u )\right \| \\
 & & \leq  \epsilon_n \sqrt{n}  + \sup_{u \in \mathcal {U}} \Bigg  ( \inf_{\theta \in \Theta_u} \sqrt{n} \|  \En \psi_u(W_u, \theta, \hat h_u(Z_u) ) \| + \|II_1(u)\| + \|II_2(u)\|  \Bigg ) = o_P(1),
\end{eqnarray*}  where
the $o_P(1)$ bound
follows from Step 4,  $\epsilon_n \sqrt{n} = o(1)$  by assumption, and Step 5.

Moreover, by the orthogonality condition:
$$
\mathrm{D}_{u,0}(\hat h_u -h_u ) :=  \Bigg ( \sum_{m = 1}^{d_t}  \sqrt{n}  P\Big [ \partial_{t_m} P [ \psi_{uj}(W_u, \theta_u, h_u(Z_u))|Z_u] (\hat h_m(Z_u) - h_{um}(Z_u)) \Big ] \Bigg )_{j=1}^{d_\theta} = 0.
$$
Conclude using Assumption \ref{ass: S1}(iv) that
{\small $$
\sup_{u \in \mathcal{U}}\left \| J_u^{-1}  \sqrt{n}  \En \psi_u(W_u, \theta_u, h_u(Z_u)) +  \sqrt{n} (\hat \theta_u - \theta_u) \right \| \leq o_P(1) \sup_{u \in \mathcal{U}} (  \textrm{mineg}(J_u' J_u)^{-1/2}) = o_P(1).
$$}\!

Furthermore, the empirical process $(-\sqrt{n}  \En J_u^{-1}   \psi_u(W_u, \theta_u, h_u(Z_u)))_{ u \in \mathcal{U}}$ is equivalent to
an empirical process $\Gn$ indexed by $ \mathcal{F}_{P} := \Big \{ \bar \psi_{uj} :  j \in [d_{\theta}],  u \in \mathcal{U}  \Big \}_{},$
where $\bar \psi_{uj}$ is the $j$-th element of  $ -J_u^{-1}   \psi_u(W_u, \theta_u, h_u(Z_u))$ and we make explicit the dependence of $\mF_P$ on $P$. Let $\mathcal{M} = \{ M_{ujk}: j,k \in [d_{\theta}], u \in \mU \}$, where $M_{ujk}$ is the $(j,k)$ element of the matrix $J_{u}^{-1}.$  $\mathcal{M}$ is a class of uniformly H\"{o}lder continuous functions on $(\mathcal{U}, d_{\mathcal{U}})$ with a uniform covering entropy bounded by $C\log(e/\epsilon) \vee 0$ and equipped with a constant envelope $C$, given the stated assumptions.  This result follows from the fact that by Assumption \ref{ass: S2}(ii)(b)
\begin{align}\label{eq: Jlip}
\max_{j,k \in [d_{\theta}]} | M_{ujk} - M_{\bar ujk} | & \leq \| J^{-1}_u - J^{-1}_{\bar u} \|
=\| J^{-1}_u ( J_u - J_{\bar u}) J^{-1}_{\bar u} \| \nonumber \\
& \leq  \| J_u - J_{\bar u}\| \sup_{\tilde u \in \mathcal{U}} \| J^{-1}_{\tilde u}\|^2 \lesssim
\|u - \bar u\|^{\alpha_2},
\end{align}
and the constant envelope follows by Assumption \ref{ass: S1}(iv). Since $ \mathcal{F}_{P}$ is generated as a finite sum of products of the elements of  $\mathcal{M}$ and the class $\mF_0$ defined in Assumption  \ref{ass: S2},  the properties of $\mathcal{M}$ and the conditions on $\mathcal{F}_0$ in Assumption \ref{ass: S2}(ii) imply that $\mathcal{F}_{P}$ has a uniformly well-behaved uniform covering entropy by Lemma \ref{lemma: andrews}, namely
$$
\sup_{P \in \mathcal{P} = \cup_{n \geq n_0} \mathcal{P}_n}\log \sup_{Q} N(\epsilon \| C F_0 \|_{Q,2}, \mF_P,  \| \cdot \|_{Q,2}) \lesssim \log (e/\epsilon) \vee 0,
$$
where $F_P = C F_0$ is an envelope for $\mF_P$ since $\sup_{f \in \mF_P} |f| \lesssim \sup_{u \in \mU} \|J_{u}^{-1}\|  \sup_{f \in \mF_0} |f| \leq C F_0$ by Assumption \ref{ass: S2}(i).  The class $\mF_P$ is therefore Donsker uniformly in $P$ because $\sup_{P \in \mathcal{P}} \| F_P \|_{P,q} \leq C \sup_{P \in \mathcal{P}} \| F_0\|_{P,q}$ is bounded by Assumption \ref{ass: S2}(ii), and $\sup_{P \in \mathcal{P}} \| \bar \psi_u - \bar \psi_{\bar u} \|_{P,2} \to 0$ as $d_{\mU}(u, \bar u) \to 0$ by Assumption  \ref{ass: S2}(b) and \eqref{eq: Jlip}.  Application of Theorem \ref{lemma: uniform Donsker} gives the results of the theorem.


Step 4. (Define and Bound $II_1(u)$ and $II_2(u)$).  Let  $II_1(u) := (II_{1j}(u))_{j=1}^{d_\theta}$ and  $II_2(u) = (II_{2j}(u))_{j=1}^{d_\theta}$, where{\small  \begin{eqnarray*}
&& II_{1j} (u) :=    \sum_{r,k = 1}^{{d_{\nu}}} \sqrt{n} P \left [  \partial_{\nu_k} \partial_{\nu_r}  P[\psi_{uj}(W_u, \bar \nu_u(Z_u,  j))|Z_u] \{ \hat \nu_{ur}(Z_u) - \nu_{ur}(Z_u)\}\{ \hat \nu_{uk}(Z_u) - \nu_{uk}(Z_u)\} \right],\\
&& II_{2j} (u)  :=  \Gn(    \psi_{uj}(W_u, \hat \theta_u, \hat h_u(Z_u))- {\psi_{uj}(W_u, \theta_u, h_u(Z_u)) )} ,
\end{eqnarray*}}
$\nu_u(Z_u): =  ( \nu_{uk}(Z_u))_{k=1}^{d_{\nu}} : = (\theta_u', h_u(Z_u)')',$   $\hat \nu_u(Z_u): =  ( \hat \nu_{uk}(Z_u))_{k=1}^{d_{\nu}}:=
(\hat \theta_u', \hat h_u(Z_u)')'$, $d_{\nu}= d_\theta+ d_t$, and $\bar\nu_u(Z_u, j )$ is a vector on the line connecting $\nu_u(Z_u)$ and $\hat \nu_u(Z_u)$.

First, by Assumptions  \ref{ass: S2}(ii)(d) and \ref{ass: AS}, the claim of Step 1, and the H\"{o}lder inequality,
\begin{eqnarray*}
\max_{j \in [d_\theta]} \sup_{u \in \mathcal{U}}|II_{1j}(u)| & \leq&    \sup_{u \in \mathcal{U}}\sum_{r,k = 1}^{{d_{\nu}}} \sqrt{n} P \left [  C | \hat \nu_{ur}(Z_u) - \nu_{ur}(Z_u)| | \hat \nu_{uk}(Z_u) - \nu_{uk}(Z_u)| \right] \\
& \leq & C \sqrt{n} {d_{\nu}^2}   \max_{k \in {[d_{\nu}]}} \sup_{u \in \mathcal{U}}\| \hat \nu_{uk} - \nu_{uk} \|^2_{P,2} \lesssim_P \sqrt{n}  \tau_n^2 = o(1).
\end{eqnarray*}

Second, we have that with probability $1-o(1)$,
$
\max_{j \in [d_\theta]} \sup_{u \in \mathcal{U}}|II_{2j}(u)|  \lesssim  \sup_{f \in \mathcal{F}_2} | \Gn(f)|,
$
where, for $\Theta_{un} := \{ \theta \in \Theta_u: \| \theta - \theta_u \| \leq C \tau_n \}$,
$$
\mathcal{F}_2 =  \Big \{ \psi_{uj}(W_u, \theta, h(Z_u)) - \psi_{uj}(W_u, \theta_u, h_u(Z_u)):  j \in [d_\theta], u \in \mathcal{U},
h \in \mathcal{H}_{un}, \theta \in \Theta_{un}  \Big \}_{}.$$
{Application of Lemma \ref{lemma:CCK} with an envelope $F_2 \lesssim F_0$ gives that with probability $1-o(1)$
\begin{eqnarray}
\sup_{f \in \mathcal{F}_2} | \Gn(f)| \lesssim  \tau_n^{\alpha/2} \sqrt{ s_n \log (a_n)}    +  n^{-1/2} s_n n^{\frac{1}{q}} \log (a_n) ,
\end{eqnarray}
since $ \sup_{f \in \mF_2} |f| \leq 2 \sup_{f \in \mF_1} |f| \leq 2  F_0$ by  Assumption \ref{ass: AS};   $\|F_0\|_{P,q} \leq C$ by Assumption \ref{ass: S2}(i);  $\log \sup_Q   N( \epsilon \|F_2\|_{Q,2},  \mathcal{F}_{2}, \|\cdot\|_{Q,2}) \lesssim (s_n \log a_n + s_n \log (a_n/\epsilon))\vee 0$ by Lemma \ref{lemma: andrews} because $\mF_2 = \mF_1 - \mF_0$ for the $\mF_0$ and $\mF_1$ defined in Assumptions  \ref{ass: S2}(i)  and  \ref{ass: AS}; and $\sigma$  can be chosen so that $
 \sup_{f \in \mathcal{F}_2} \| f\|_{P,2} \leq \sigma  \lesssim  \tau^{\alpha/2}_n$.  Indeed,
 \begin{eqnarray*}
  \sup_{f \in \mathcal{F}_2} \| f\|^2_{P,2} & \leq &    \sup_{j \in [d_\theta], u \in \mathcal{U}, \nu \in \Theta_{un} \times  \mathcal{H}_{un}}   P \left( P[  ( \psi_{uj}(W_u, \nu(Z_u)) - \psi_{uj}(W_u, \nu_u(Z_u)))^2 |Z_u]  \right) \\
 & \leq &  \sup_{u \in \mathcal{U}, \nu \in \Theta_{un} \times  \mathcal{H}_{un}}  P  \left ( C  \| \nu(Z_u) - \nu_u(Z_u) \|^{\alpha} \right)  \\
 & = &  \sup_{u \in \mathcal{U}, \nu \in \Theta_{un} \times  \mathcal{H}_{un}}  C \| \nu - \nu_u \|^{\alpha}_{P, \alpha}  \leq    \sup_{ u \in \mathcal{U}, \nu \in \Theta_{un} \times  \mathcal{H}_{un}} C  \| \nu - \nu_u \|_{P, 2}^{\alpha}  \lesssim \tau_n^{\alpha},
 \end{eqnarray*}
where the first inequality
follows by the law of iterated expectations;  the second inequality follows  by Assumption \ref{ass: S2}(ii)(a);
and the last inequality follows from $\alpha \in [1,2]$ by Assumption \ref{ass: S2},  the monotonicity of the norm
$\|\cdot\|_{P,\alpha}$ in $\alpha \in [1, \infty]$, and  Assumption \ref{ass: AS}.}

Conclude using the growth conditions of Assumption \ref{ass: AS} that with probability $1-o(1)$
\begin{eqnarray}
\max_{j \in [d_\theta]} \sup_{u \in \mathcal{U}}|II_{2j}(u)| \lesssim \tau^{\alpha/2}_n \sqrt{ s_n \log (a_n)}    +  n^{-1/2} s_n n^{\frac{1}{q}} \log (a_n) = o(1).
\end{eqnarray}

Step 5. In this step we show that
$
 \sup_{u \in \mathcal {U}} \inf_{\theta \in \Theta_u} \sqrt{n} \|  \En \psi_u(W_u, \theta, \hat h_u(Z_u) ) \| =o_P(1).
$
We have that with probability $1- o(1)$
$$
\inf_{\theta \in \Theta_u} \sqrt{n} \|  \En \psi_u(W_u, \theta, \hat h_u(Z_u) ) \| \leq \sqrt{n} \|  \En \psi_u(W_u, \bar  \theta_u, \hat h_u(Z_u) ) \|,
$$
where $\bar  \theta_u = \theta_u - J^{-1}_u \En \psi_u(W_u, \theta_u, h_u(Z_u))$, since $\bar  \theta_u \in \Theta_u$ for all $u \in \mathcal{U}$ with probability $1- o(1)$, and, in fact, $\sup_{u \in \mathcal{U}} \| \bar  \theta_u - \theta_u \| =O_P( 1/\sqrt{n})$
by the last paragraph of Step 3.

Then, arguing similarly to Step 3 and 4, we can conclude that uniformly in $u \in \mathcal{U}$:
{\small \begin{eqnarray*}
\sqrt{n} \|  \En \psi_u(W_u, \bar  \theta_u, \hat h_u(Z_u) ) \| & \leq &  \sqrt{n} \|  \En \psi_u(W_u, \theta_u, h_u(Z_u) ) + J_u (\bar  \theta_u - \theta_u) + {\mathrm{D}_{u,0}(\hat h_u - h_u)} \| + o_P(1)
\end{eqnarray*}}\!
where the first term on the right side is zero by definition of $\bar  \theta_u$ and  $\mathrm{D}_{u,0}(\hat h_u - h_u) = 0$.
\qed
\subsection{Proof of Theorem \ref{theorem: general bs}}

\textsc{Step 0.}  In the proof $a \lesssim b$ means that $a \leq A b$, where the constant
$A$ depends on the constants  in Assumptions  \ref{ass: S1}-- \ref{ass: AS}, but not on $n$ once $n \geq n_0$, and not on $P \in \mathcal{P}_n$. In Step 1, we consider a sequence $P_n$ in $\mathcal{P}_n$, but for simplicity, we write  $P =P_n$ throughout the proof,  suppressing the index $n$.  Since the argument is asymptotic, we can  assume that $n \geq n_0$ in what follows.

Let $\Pn$ denote the measure that puts mass $n^{-1}$ at the points $(\xi_i, W_i)$ for $i=1,...,n$.
Let $\En$ denote the expectation with respect to this measure, so that
$\En f = n^{-1} \sum_{i=1}^n f(\xi_i, W_i)$, and $\Gn$ denote the corresponding empirical process $\sqrt{n} ( \En - P)$,
i.e.
$$
\Gn f = \sqrt{n}(\En f - P f) = n^{-1/2} \sum_{i=1}^n \Bigg ( f(\xi_i, W_i) - \int f(s, w) d P_\xi (s) dP (w) \Bigg).
$$

Recall that we define the bootstrap draw as:
$$
Z^*_{n,P}:= \sqrt{n}(\hat \theta^*- \hat \theta)  =  \(\frac{1}{\sqrt{n}} \sum_{i=1}^n \xi_i \hat \psi_{u}(W_i) \)_{u \in \mathcal{U}}= \left (\Gn \xi \hat \psi_u \right)_{u \in \mathcal{U}},
$$
where
$
 \hat \psi_u(W) = -\hat J^{-1}_u \psi_u(W_{u}, \hat \theta_u, \hat h_u(Z_u)).
$

\textsc{Step 1.}({Linearization})  In this step we establish that
\begin{equation}
\ \  \zeta^*_{n,P}:  = Z^*_{n,P}-   G^*_{n,P} =   o_P(1) \ \  \text{ in $\D=\ell^\infty(\mU)^{d_\theta}$},
\end{equation}
where $G^*_{n, P} :=   (\mathbb{G}_{n}  \xi \bar \psi_{u})_{u \in \mathcal{U}},$
and  $\bar \psi_u(W) = -J^{-1}_u \psi_u(W_{u}, \theta_u, h_u(Z_u))$.

The claim would follow from demonstrating that (a)
\begin{equation}\label{REL1}
 Z^\star_{n,P}-   G^*_{n,P} =   o_P(1) \ \  \text{ in $\D=\ell^\infty(\mU)^{d_\theta}$},
\end{equation}
where $Z^\star_{n, P} :=   (\mathbb{G}_{n}  \xi \check \psi_{u})_{u \in \mathcal{U}},$
and  $\check \psi_u(W) = -J^{-1}_u \psi_u(W_{u}, \hat \theta_u, \hat h_u(Z_u))$ (not that
the hat from $J_u$ disappeared) ; and (b)
\begin{equation}\label{REL2}
 Z^\star_{n,P}-   Z^*_{n,P} =   o_P(1) \ \  \text{ in $\D=\ell^\infty(\mU)^{d_\theta}$}.
\end{equation}

To show claim (\ref{REL1}),  we note that {with probability $1-\delta_n$,  $\hat h_u \in \mathcal{H}_{un},$ $\hat \theta_u \in \Theta_{un} = \{ \theta \in \Theta_u: \| \theta - \theta_u \| \leq C \tau_n \}$,  so that
$
\| \zeta^\star_{n,P} \|_{\D} \lesssim \sup_{f \in \mF_3} |\Gn[\xi f]|,
$
where
$$
\mathcal{F}_{3} = \Big \{ \tilde \psi_{uj}(\bar \theta_u, \bar h_u) - \bar \psi_{uj} :
 j \in [d_{\theta}], u \in \mU,  \bar \theta_u \in \Theta_{un},  \bar h_u  \in \mathcal{H}_{un}\Big\},
$$
where $\tilde \psi_{uj}(\bar \theta_u, \bar h_u)$ is the $j$-th element of $-  J^{-1}_u \psi_u(W_u, \bar \theta_u, \bar h_u(Z_u))$, and $\bar \psi_{uj}$ is the $j$-th element of $-J^{-1}_u \psi_u(W_u,  \theta_u,  h_u(Z_u))$.
By the arguments similar to those employed in the proof of the previous theorem, $\mF_3$ obeys $$
\log \sup_Q   N( \epsilon \|F_3\|_{Q,2}, \mathcal{F}_{3}, \|\cdot\|_{Q,2}) \lesssim  (s_n \log a_n + s_n \log (a_n/\epsilon)) \vee 0,$$
for an envelope  $F_3 \lesssim F_0 $.
By Lemma \ref{lemma: andrews}, multiplication of this class by $\xi$ does not change the entropy bound modulo  an absolute constant, namely
$$
\log \sup_Q   N( \epsilon \||\xi| F_3\|_{Q,2}, \xi \mathcal{F}_{3}, \|\cdot\|_{Q,2}) \lesssim (s_n \log a_n + s_n \log (a_n/\epsilon))\vee 0.
$$
Also $\Ep[\exp(|\xi|)] < \infty$ implies $(\Ep [\max_{ i \leq n} |\xi_i|^2])^{1/2} \lesssim \log n$, so that, using independence
of $(\xi_i)_{i=1}^n$ from $(W_i)_{i=1}^n$ and Assumption \ref{ass: S2}(i),
 $$\|  \max_{i \leq n } \xi_i F_0(W_i)\|_{\Pr_P, 2}   \leq  \|   \max_{i \leq n } \xi_i \|_{\Pr_P, 2}
 \| \max_{i \leq n }F_0(W_i)\|_{\Pr_P,2} \lesssim n^{1/q}  \log n.$$
Applying Lemma \ref{lemma:CCK},
\begin{eqnarray*}
&&  \sup_{f \in \xi \mathcal{F}_3} | \Gn (f) |  = O_{P} \(  \tau_n^{\alpha/2} \sqrt{s_n \log (a_n)}  +  \frac{s_n n^{1/q}  \log n }{\sqrt{n}}  \log (a_n)  \)  = o_P(1),
\end{eqnarray*}
for $\sup_{f \in \xi \mathcal{F}_3}\|  f\|_{P,2}= \sup_{f \in  \mathcal{F}_3}\| f\|_{P,2} \lesssim \sigma_n \lesssim \tau_n^{\alpha/2}$, where the details of calculations are similar to those in the  proof of Theorem \ref{theorem:semiparametric}.
Indeed, with probability $1 - o(\delta_n),$
 \begin{eqnarray*}
  \sup_{f \in \mathcal{F}_3} \| f\|^2_{P,2} & \lesssim &   \sup_{u \in \mU} \| J_{u}^{-1} \|^{2}  \sup_{j \in [d_\theta], u \in \mathcal{U}, \nu \in \Theta_{un} \times  \mathcal{H}_{un}}   \ P \left( P[  ( \psi_{uj}(W_u, \nu(Z_u)) - \psi_{uj}(W_u, \nu_u(Z_u)))^2 |Z_u]  \right) \\
 & \lesssim &  \sup_{u \in \mathcal{U}, \nu \in \Theta_{un} \times  \mathcal{H}_{un}}   \| \nu - \nu_u \|^{\alpha}_{P, \alpha}   \lesssim    \sup_{ u \in \mathcal{U}, \nu \in \Theta_{un} \times  \mathcal{H}_{un}}   \| \nu - \nu_u \|_{P, 2}^{\alpha}  \lesssim \tau_n^{\alpha},
 \end{eqnarray*}
where the first inequality
follows from the triangle inequality and the law of iterated expectations;  the second inequality follows  by Assumption \ref{ass: S2}(ii)(a) and Assumption \ref{ass: S2}(i); the third inequality follows from $\alpha \in [1,2]$ by Assumption \ref{ass: S2},  the monotonicity of the norm
$\|\cdot\|_{P,\alpha}$ in $\alpha \in [1, \infty]$, and  Assumption \ref{ass: AS}; and the last inequality follows from $ \| \nu - \nu_u \|_{P, 2} \lesssim \tau_n$ by the definition of $ \Theta_{un}$ and $ \mathcal{H}_{un}$.}
The claim (\ref{REL1}) follows.

To show claim (\ref{REL2}),  bound
$$
\|Z^\star_{n,P}-   Z^*_{n,P}\|_{\D}  = \|  ( \hat J_u^{-1} J_u    Z^*_{n,P}(u) -  Z^*_{n,P}(u)  )_{u \in \mU}\|_{\D} \lesssim \sup_{u \in \mU}  \| \hat J_u^{-1} J_u - I \|
\|Z^*_{n,P} \|_{\D} = o_P(1),
$$
since $\sup_{u \in \mU}  \| \hat J_u^{-1} J_u - I \| = o_P(1)$ by the assumption of the theorem, and since $\|Z^*_{n,P} \|_{\D} = O_P(1)$ by
$\|Z^*_{n,P} \|_{\D} =  \| G^*_{n,P} + o_P(1) \|_{\D}   \rightsquigarrow_B
 \| Z_{P}\|_{\D}$, which  follows by claim (\ref{REL1}) and by  $G^*_{n,P} \rightsquigarrow_B  Z_{P}$ in $\D$ holding by Theorem \ref{lemma: uniform Donsker for bootstrap}.

\textsc{Step 2}. Here we are claiming that
$Z^*_{n,P} \rightsquigarrow_B  Z_{P}$  in $\D=\ell^\infty(\mU)^{d_\theta}$,  under any sequence $P =P_n \in \mP_n$, were $Z_{P} =   (\mathbb{G}_{P}  \bar \psi_{u})_{u \in \mathcal{U}}$. By the triangle inequality and Step 1,
\begin{eqnarray*}
&& \sup_{h \in \mathrm{BL}_1(\D) } \Big |   \Ep_{B_n}  h ( Z^*_{n,P} )  - \Ep_P h ( Z_{P})  \Big |\leq \sup_{h \in \mathrm{BL}_1(\D) } \Big |   \Ep_{B_n}  h (G^*_{n,P} )  - \Ep_P h ( Z_{P})  \Big |
+   \Ep_{B_n} ( \|  \zeta^*_{n,P} \|_{\D} \wedge 2 ),
\end{eqnarray*}
where the first term  is $o^*_P(1)$, since $
G^*_{n,P} \rightsquigarrow_B  Z_{P}$  by Theorem \ref{lemma: uniform Donsker for bootstrap}, and the second term is $o_P(1)$  because  $ \|\zeta^*_{n,P}\|_{\D}  = o_P(1) $ implies that  $ \Ep_P ( \|  \zeta^*_{n,P} \|_{\D} \wedge 2 ) =
\Ep_P \Ep_{B_n} ( \|  \zeta^*_{n,P} \|_{\D} \wedge 2 ) \to 0$, which in turn implies that  $\Ep_{B_n} ( \|  \zeta^*_{n,P} \|_{\D} \wedge 2 ) = o_P(1)$ by the Markov inequality.
 \qed

\subsection{Proof of Theorem \ref{theorem3}}  This is an immediate consequence of
Theorems  \ref{theorem:semiparametric}, \ref{theorem: general bs},  \ref{thm: delta-method}, and \ref{theorem:delta-method-bootstrap}.\qed

\section{Implementation Details}\label{sec:Implementation}

In this section, we provide details about how we implemented the methodology developed in the main body of the paper in the empirical example.  We first discuss estimation of local average treatment effects (LATE) and then extend this discussion to estimation of local quantile treatment effects (LQTE).  Estimation of all other quantities proceeds in a similar fashion and so is not discussed.

\subsection{Local Average Treatment Effects}

Recall that the LATE of treatment $D$ on outcome $Y$ is defined as
\begin{align*}
\Delta_{LATE} = \theta_{Y}(1)-\theta_{Y}(0) = \frac{\alpha_{1_1(D)Y}(1)-\alpha_{1_1(D)Y}(0)}{\alpha_{1_1(D)}(1) - \alpha_{1_1(D)}(0)} - \frac{\alpha_{1_0(D)Y}(1)-\alpha_{1_0(D)Y}(0)}{\alpha_{1_0(D)}(1) - \alpha_{1_0(D)}(0)}
\end{align*}
for $\alpha_V(z)$ and $\theta_Y(d)$ defined in equations (\ref{KeyRF1}) and (\ref{define:LASF}) respectively.
It then follows by plugging in the definition of $\alpha_V(z)$ that we can express the LATE as
\begin{align*}
\Delta_{LATE} &= \frac{\alpha_{Y}(1)-\alpha_{Y}(0)}{\alpha_{1_1(D)}(1) - \alpha_{1_1(D)}(0)}.
\end{align*}

To obtain an estimate of the LATE, we thus need estimates of $\alpha_{Y}(z)$ and $\alpha_{1_1(D)}(z)$.  Using the low-bias moment function given in equation (\ref{LowBiasMoment}), estimates of these key quantities can be constructed from estimates of
$\Ep_P[Y|Z=1,X]$, $\Ep_P[Y|Z=0,X]$, $\Ep_P[D|Z=1,X]$, $\Ep_P[D|Z=0,X]$, and $\Ep_P[Z|X]$ where $Z$ is the binary instrument (401(k) eligibility); $D$ is the binary treatment (401(k) participation); $X$ is the set of raw covariates discussed in the empirical section; and $Y$ is net financial assets.  In our application, we have $\Ep_P[D|Z=0,X] = 0$ since one cannot participate unless one is eligible.  We use Post-Lasso to estimate $\Ep_P[Y|Z=1,X]$ and $\Ep_P[Y|Z=0,X]$ and post-$\ell_1$-penalized logistic regression to estimate $\Ep_P[D|Z=1,X]$ and $\Ep_P[Z|X]$.

To estimate $\Ep_P[Y|Z=1,X]$, we postulate that $\Ep_P[Y|Z=1,X] \approx f(X)'\beta_Y(1)$, where $f(X)$ is one of the pre-specified sets of controls discussed in the empirical section with dimension $p$.  Let $\mathcal{I}_1$ denote the indices of observations that have $z_i = 1$.  To estimate the coefficients $\beta_Y(1)$, we apply the formulation of the Post-Lasso estimator given in \citen{BellChenChernHans:nonGauss} with outcomes $\{y_i\}_{i \in \mathcal{I}_1}$ and covariates $\{f(x_i)\}_{i \in \mathcal{I}_1}$.  We set $\lambda = 1.1\sqrt{n}\Phi^{-1}(1-(.1/\log(n))/(2(2p)))$ where $\Phi(\cdot)$ is the standard normal distribution function.  We calculate penalty loadings according to Algorithm A.1 of \citen{BellChenChernHans:nonGauss} using Post-Lasso coefficient estimates at each iteration and with a the maximum number of iterations set to 15.\footnote{Here and in all following instances, we stop iterating before reaching the maximum number of iterations if the $\ell_2$-norm of the difference in penalty loadings calculated across consecutive iterations is less than $10^{-6}$.}  Let $\widehat\beta_Y(1)$ denote the resulting Post-Lasso estimates of the coefficients using $\lambda$ given above and the final set of penalty loadings.  We then estimate $\Ep_P[Y|Z=1,X=x_i]$ as $f(x_i)'\widehat\beta_Y(1)$ for each $i = 1,...,n$.  We follow the same procedure to obtain estimates of $\Ep_P[Y|Z=0,X=x_i]$ as $f(x_i)'\widehat\beta_Y(0)$ for each $i = 1,...,n$ where $\widehat\beta_Y(0)$ are the Post-Lasso estimates using only the observations with $z_i = 0$.

Estimation of $\Ep_P[D|Z=1,X]$ and $\Ep_P[Z|X]$ proceed similarly replacing Post-Lasso estimation with post-$\ell_1$-penalized logistic regression.  Specifically, we assume that $\Ep_P[D|Z=1,X] \approx \Lambda_0(f(X)'\beta_D(1))$ where $\Lambda_0(\cdot)$ is the logistic link function.  We then obtain estimates of $\beta_D(1)$ by using the post-$\ell_1$-penalized estimator defined in equations (\ref{Def:LASSOmain2}) and (\ref{Def:PostLASSOmain2}) based on the logistic link function and with outcomes $\{d_i\}_{i \in \mathcal{I}_1}$ and covariates $\{f(x_i)\}_{i \in \mathcal{I}_1}$ for $\mathcal{I}_1$ defined as above.  We set $\lambda = 1.1\sqrt{n}\Phi^{-1}(1-(.1/\log(n))/(2(2p)))$ where $\Phi(\cdot)$ is the standard normal distribution function.  We calculate penalty loadings using Algorithm \ref{AlgFunc} of the main text with a maximum of 15 iterations. 
Let $\widehat\beta_D(1)$ denote the resulting post-$\ell_1$-penalized estimates of the coefficients using $\lambda$ given above and the final set of penalty loadings.  We estimate $\Ep_P[D|Z=1,X=x_i]$ as $\Lambda_0(f(x_i)'\widehat\beta_D(1))$ for each $i = 1,...,n$.  We follow this procedure to obtain estimates of $\Ep_P[Z|X=x_i]$ as $\Lambda_0(f(x_i)'\widehat\beta_Z)$ for each $i = 1,...,n$ where $\widehat\beta_Z$ are the post-$\ell_1$-penalized coefficient estimates obtained with $\{z_i\}_{i=1}^{n}$ as the outcome and $\{f(x_i)\}_{i=1}^{n}$ as covariates using $\lambda = 1.1\sqrt{n}\Phi^{-1}(1-(.1/\log(n))/(2p))$.

Using these baseline quantities, we obtain estimates
\begin{align*}
\widehat\alpha_{Y}(1) &= \frac{1}{n}\sum_{i=1}^{n} \left(\frac{z_i(y_i - f(x_i)'\widehat\beta_Y(1))}{\Lambda_0(f(x_i)'\widehat\beta_Z)} + f(x_i)'\widehat\beta_Y(1)\right) = \frac{1}{n}\sum_{i=1}^{n} \psi_{1,i} \\
\widehat\alpha_{Y}(0) &= \frac{1}{n}\sum_{i=1}^{n} \left(\frac{(1-z_i)(y_i - f(x_i)'\widehat\beta_Y(0))}{1-\Lambda_0(f(x_i)'\widehat\beta_Z)} + f(x_i)'\widehat\beta_Y(0)\right) = \frac{1}{n}\sum_{i=1}^{n} \psi_{0,i} \\
\widehat\alpha_{1_1(D)}(1) &= \frac{1}{n}\sum_{i=1}^{n} \left(\frac{z_i(d_i - \Lambda_0(f(x_i)'\widehat\beta_D(1)))}{\Lambda_0(f(x_i)'\widehat\beta_Z)} + \Lambda_0(f(x_i)'\widehat\beta_D(1))\right)  = \frac{1}{n}\sum_{i=1}^{n} \upsilon_{1,i} \\
\widehat\alpha_{1_1(D)}(0) &= \frac{1}{n}\sum_{i=1}^{n} \left(\frac{(1-z_i)d_i}{1-\Lambda_0(f(x_i)'\widehat\beta_Z)}\right)  = \frac{1}{n}\sum_{i=1}^{n} \upsilon_{0,i} = 0.
\end{align*}
We then plug these estimates in to obtain
$$
\widehat\Delta_{LATE} = \frac{\widehat\alpha_{Y}(1)-\widehat\alpha_{Y}(0)}{\widehat\alpha_{1_1(D)}(1) - \widehat\alpha_{1_1(D)}(0)}.
$$

We report both analytic and bootstrap standard error estimates for the LATE.  The analytic standard errors are calculated as
$$
\sqrt{\frac{1}{n-1}\sum_{i=1}^{n}\left(\frac{\psi_{1,i} - \psi_{0,i}}{\widehat\alpha_{1_1(D)}(1) - \widehat\alpha_{1_1(D)}(0)} - \widehat\Delta_{LATE}\right)^2/n}.
$$
We use wild bootstrap weights for obtaining the multiplier bootstrap estimates of the standard errors with 500 bootstrap replications.  Specifically, for each $b = 1,...,500$, we calculate a bootstrap estimate of the LATE as
$$
\widehat\Delta_{LATE}^{b} = \frac{\frac{1}{n}\sum_{i=1}^{n}(\psi_{1,i}-\psi_{0,i})\xi_i^b}{\frac{1}{n}\sum_{i=1}^{n}(\upsilon_{1,i}-\upsilon_{0,i})\xi_i^b}
$$
where $\xi_i^b = 1 + r_{1,i}^b/\sqrt{2}+((r_{2,i}^{b})^2-1)/2$ is the bootstrap draw for multiplier weight for observation $i$ in bootstrap repetition $b$ where $r_{1,i}^b$ and $r_{2,i}^b$ are random numbers generated as iid draws from two independent standard normal random variables.  The bootstrap standard error estimate is then the bootstrap interquartile range rescaled with the normal distribution: $[q_{LATE}(.75) - q_{LATE}(.25)]/[q_N(.75) - q_N(0.25)]$, where $q_{LATE}(p)$ is the $p$th quantile of $\{\widehat\Delta_{LATE}^{b}\}_{b=1}^{500}$ and $q_N(p)$ is the $p$th quantile of the $N(0,1)$.

\subsection{Local Quantile Treatment Effects}

Calculation and inference for LQTE is more cumbersome than for the LATE.  We begin by choosing the set over which we would like to look at the LQTE.  In our example, we chose to look at quantiles in the interval $[0.1,0.9]$.

To calculate the LQTE, we first calculate the local average structural function for outcomes $Y_u = 1(Y \le u)$ for a set of $u$ and then invert to obtain estimates of the LQTE.  In our example, we chose to look at $u \in [q_Y(.05),q_Y(.95)]$ where $q_Y(.05)$ and $q_Y(.95)$ are respectively the sample $5^{th}$ and $95^{th}$ percentiles of the outcome of interest $Y$.  Since looking at the continuum of values in this interval is infeasible, we discretize the interval and look at $Y_u = 1(Y \le u)$ for $u \in \{q_Y(.05),q_Y(.06),q_Y(.07),...,q_Y(.93),q_Y(.94),q_Y(.95)\}$. I.e. we set $u$ equal to each percentile of $Y$ between the $5^{th}$ and $95^{th}$ percentiles for a total of 91 different values of $u$ to be considered.  For each value of $u$, we need an estimate of the local average structural function defined in (\ref{define:LASF}) for $d \in \{0,1\}$:
\begin{align*}
\theta_{1(Y \le u)}(d) = \frac{\alpha_{1_d(D)1(Y \le u)}(1) - \alpha_{1_d(D)1(Y \le u)}(0)}{\alpha_{1_d(D)}(1) -  \alpha_{1_d(D)}(0)}.
\end{align*}

As with the LATE, we need estimates of $\Ep_P[D|Z=1,X]$ and $\Ep_P[Z|X]$.  We estimate these quantities as we did for the LATE but change the value of the penalty parameter used to reflect the fact that we are now interested in a large set, in theory a continuum, of model selection problems.  Specifically, we assume that $\Ep_P[D|Z=1,X] \approx \Lambda_0(f(X)'\beta_D(1))$ where $\Lambda_0(\cdot)$ is the logistic link function and $f(X)$ is one of the pre-specified sets of controls  discussed in the empirical section with dimension $p$.  We then obtain estimates of $\beta_D(1)$ by using the post-$\ell_1$-penalized estimator defined in equations (\ref{Def:LASSOmain2}) and (\ref{Def:PostLASSOmain2})  based on the logistic link function and with outcomes $\{d_i\}_{i \in \mathcal{I}_1}$ and covariates $\{f(x_i)\}_{i \in \mathcal{I}_1}$ for $\mathcal{I}_1$ defined as above.  We set $\lambda = 1.1\sqrt{n}\Phi^{-1}(1-(1/\log(n))/(2n(2p)))$ where $\Phi(\cdot)$ is the standard normal distribution function.  We calculate penalty loadings using Algorithm \ref{AlgFunc} with a maximum of 15 iterations. 
  Let $\widehat\beta_D(1)$ denote the resulting post-$\ell_1$-penalized estimates of the coefficients using $\lambda$ given above and the final set of penalty loadings.  We estimate $\Ep_P[D|Z=1,X=x_i]$ as $\Lambda_0(f(x_i)'\widehat\beta_D(1))$ for each $i = 1,...,n$.  We follow this procedure to obtain estimates of $\Ep_P[Z|X]$ as $\Lambda_0(f(x_i)'\widehat\beta_Z)$ for each $i = 1,...,n$ where $\widehat\beta_Z$ are the post-$\ell_1$-penalized coefficient estimates obtained with $\{z_i\}_{i=1}^{n}$ as the outcome and $\{f(x_i)\}_{i=1}^{n}$ as covariates and $\lambda = 1.1\sqrt{n}\Phi^{-1}(1-(1/\log(n))/(2np))$.  We also still have $\Ep_P[D|Z=0,X] = 0$ in our application since one cannot participate in a 401(k) unless one is eligible. We then plug-in these estimates to obtain
\begin{align*}
\widehat\alpha_{1_1(D)}(1) &= \frac{1}{n}\sum_{i=1}^{n} \left(\frac{z_i(d_i - \Lambda_0(f(x_i)'\widehat\beta_D(1)))}{\Lambda_0(f(x_i)'\widehat\beta_Z)} + \Lambda_0(f(x_i)'\widehat\beta_D(1))\right)  = \frac{1}{n}\sum_{i=1}^{n} \upsilon_{1,1,i} \\
\widehat\alpha_{1_1(D)}(0) &= \frac{1}{n}\sum_{i=1}^{n} \left(\frac{(1-z_i)d_i}{1-\Lambda_0(f(x_i)'\widehat\beta_Z)}\right)  = \frac{1}{n}\sum_{i=1}^{n} \upsilon_{1,0,i} = 0 \\
\widehat\alpha_{1_0(D)}(1) &= 1-\widehat\alpha_{1_1(D)}(1), \quad \widehat\alpha_{1_0(D)}(0) = 1-\widehat\alpha_{1_1(D)}(0).
\end{align*}

We also need to obtain estimates of $\alpha_{1_d(D)1(Y \le u)}(z)$ for each value of $u$ and for $$(z,d) \in \{(0,0),(0,1),(1,0),(1,1)\}.$$  These estimates will depend on the propensity score, $\Ep_P[Z|X]$, estimated above and quantities of the form $\Ep_P[1(D = d)1(Y \le u)|Z=z,X]$.  We again approximate this function with $\Ep_P[1(D = d)1(Y \le u)|Z=z,X] \approx \Lambda_0(X'\beta_{\mathbf{1}_d(D) Y_u}(z))$ and estimate the coefficients $\beta_{\mathbf{1}_d(D) Y_u}(z)$ for each combination of $d$ and $z$ and each $u$ using the post-$\ell_1$-penalized estimator defined in equations (\ref{Def:LASSOmain2}) and (\ref{Def:PostLASSOmain2}) based on the logistic link function.  We set $\lambda = 1.1\sqrt{n}\Phi^{-1}(1-(1/\log(n))/(2n(2p)))$ where $\Phi(\cdot)$ is the standard normal distribution function.  We calculate penalty loadings using Algorithm \ref{AlgFunc} of the main text with a maximum of 15 iterations.
 We follow this procedure for each $u$ with $\{1(y_i \le u) 1(d_i = 1)\}_{i \in \mathcal{I}_1}$ as the outcome and covariates $\{f(x_i)\}_{i \in \mathcal{I}_1}$, with $\{1(y_i \le u) 1(d_i = 0)\}_{i \in \mathcal{I}_1}$ as the outcome and covariates $\{f(x_i)\}_{i \in \mathcal{I}_1}$, and with $\{1(y_i \le u) 1(d_i = 0)\}_{i \in \mathcal{I}_0}$ as the outcome and covariates $\{f(x_i)\}_{i \in \mathcal{I}_0}$ for $\mathcal{I}_1$ and $\mathcal{I}_0$ defined as above to obtain point estimates  $\widehat \beta_{\mathbf{1}_1(D) Y_u}(1)$, $\widehat\beta_{\mathbf{1}_0(D) Y_u}(1)$, and $\widehat\beta_{\mathbf{1}_0(D) Y_u}(0)$ respectively.  We then estimate $\Ep_P[1(D = 1)1(Y \le u)|Z=1,X=x_i]$ as $\Lambda_0(f(x_i)'\widehat\beta_{\mathbf{1}_1(D) Y_u}(1))$ for each $i = 1,...,n$ and obtain estimates of $\Ep_P[1(D = 0)1(Y \le u)|Z=1,X=x_i]$, and $\Ep_P[1(D = 0)1(Y \le u)|Z=0,X=x_i]$ analogously.  As before, we have $\Ep_P[1(D = 1)1(Y \le u)|Z=0,X] = 0$ since one cannot participate unless one is eligible.
We then plug-in these estimates to obtain
\begin{footnotesize}\begin{align*}
\widehat\alpha_{1_1(D)1(Y \le u)}(1) &= \frac{1}{n}\sum_{i=1}^{n} \left(\frac{z_i(d_i 1(y_i \le u) - \Lambda_0(f(x_i)'\widehat\beta_{\mathbf{1}_1(D) Y_u}(1)))}{\Lambda_0(f(x_i)'\widehat\beta_Z)} + \Lambda_0(f(x_i)'\widehat\beta_{\mathbf{1}_1(D) Y_u}(1))\right)  = \frac{1}{n}\sum_{i=1}^{n} \kappa_{u,1,1,i} \\
\widehat\alpha_{1_1(D)1(Y \le u)}(0) &= \frac{1}{n}\sum_{i=1}^{n} \left(\frac{(1-z_i)(d_i 1(y_i \le u))}{1-\Lambda_0(f(x_i)'\widehat\beta_Z)}\right)  = \frac{1}{n}\sum_{i=1}^{n} \kappa_{u,1,0,i} = 0 \\
\widehat\alpha_{1_0(D)1(Y \le u)}(1) &=  \frac{1}{n}\sum_{i=1}^{n} \left(\frac{z_i((1-d_i) 1(y_i \le u) - \Lambda_0(f(x_i)'\widehat\beta_{\mathbf{1}_0(D) Y_u}(1)))}{\Lambda_0(f(x_i)'\widehat\beta_Z)} + \Lambda_0(f(x_i)'\widehat\beta_{\mathbf{1}_0(D) Y_u}(1))\right)  = \frac{1}{n}\sum_{i=1}^{n} \kappa_{u,0,1,i}\\
\widehat\alpha_{1_0(D)1(Y \le u)}(0) &= \frac{1}{n}\sum_{i=1}^{n} \left(\frac{(1-z_i)((1-d_i) 1(y_i \le u) - \Lambda_0(f(x_i)'\widehat\beta_{\mathbf{1}_0(D) Y_u}(0)))}{1-\Lambda_0(f(x_i)'\widehat\beta_Z)} + \Lambda_0(f(x_i)'\widehat\beta_{\mathbf{1}_0(D) Y_u}(0))\right)  = \frac{1}{n}\sum_{i=1}^{n} \kappa_{u,0,0,i}.
\end{align*}\end{footnotesize}

Estimates of the local average structural (distribution) functions are formed using the estimators defined in the previous two paragraphs as
\begin{align*}
\widehat\theta_{1(Y \le u)}(d) = \frac{\widehat\alpha_{1_d(D)1(Y \le u)}(1) - \widehat\alpha_{1_d(D)1(Y \le u)}(0)}{\widehat\alpha_{1_d(D)}(1) -  \widehat\alpha_{1_d(D)}(0)}.
\end{align*}
To obtain LQTE estimates, we then need to invert these local average structural functions.  Since we only have the estimated distribution for each $d$ evaluated on the finite grid of points $u \in \{q_Y(.05),q_Y(.06),q_Y(.07),...,q_Y(.93),q_Y(.94),q_Y(.95)\}$, we do this inversion by linearly interpolating the value of the distribution function between these points to find the value of the outcome associated with each quantile in the set $q \in [0.1,0.11,.0,12,...,0.89,.0.9]$ which we denote as $\widehat\theta^{\leftarrow}_Y(q,d)$.  The LQTE at point $q$ is then estimated as $\widehat\Delta(q) = \widehat\theta^{\leftarrow}_Y(q,1) - \widehat\theta^{\leftarrow}_Y(q,0)$.

For the LQTE, we only report inference based on the multiplier bootstrap using 500 bootstrap replications.
For each $b = 1,...,500$, we generate bootstrap weights as $\xi_i^b = 1 + r_{1,i}^b/\sqrt{2}+((r_{2,i}^{b})^2-1)/2$ for observation $i$ in bootstrap repetition $b$ where $r_{1,i}^b$ and $r_{2,i}^b$ are random numbers generated as iid draws from two independent standard normal random variables.  We then use these weights to form bootstrap estimates of the local average structural functions
\begin{align*}
\widehat\theta^{b}_{1(Y \le u)}(d) = \frac{\widehat\alpha^{b}_{1_d(D)1(Y \le u)}(1) - \widehat\alpha^{b}_{1_d(D)1(Y \le u)}(0)}{\widehat\alpha^{b}_{1_d(D)}(1) -  \widehat\alpha^{b}_{1_d(D)}(0)}
\end{align*}
where
\begin{align*}
\widehat\alpha^{b}_{1_1(D)}(1) &= \frac{1}{n}\sum_{i=1}^{n} \xi_{i}^b\upsilon_{1,1,i}, \ \widehat\alpha^{b}_{1_1(D)}(0) = \frac{1}{n}\sum_{i=1}^{n} \xi_{i}^b\upsilon_{1,0,i}, \\
\widehat\alpha^{b}_{1_0(D)}(1) &= 1-\widehat\alpha^{b}_{1_1(D)}(1),
\widehat\alpha^{b}_{1_0(D)}(0) = 1-\widehat\alpha^{b}_{1_1(D)}(0), \\
\widehat\alpha^{b}_{1_1(D)1(Y \le u)}(1) &= \frac{1}{n}\sum_{i=1}^{n} \xi_{i}^b\kappa_{u,1,1,i}, \
\widehat\alpha^{b}_{1_1(D)1(Y \le u)}(0) = \frac{1}{n}\sum_{i=1}^{n} \xi_{i}^b\kappa_{u,1,0,i} = 0, \\
\widehat\alpha^{b}_{1_0(D)1(Y \le u)}(1) &= \frac{1}{n}\sum_{i=1}^{n} \xi_{i}^b\kappa_{u,0,1,i}, \
\widehat\alpha^{b}_{1_0(D)1(Y \le u)}(0) = \frac{1}{n}\sum_{i=1}^{n} \xi_{i}^b\kappa_{u,0,0,i}.
\end{align*}
From these bootstrap estimates of the average structural distribution functions, we obtain bootstrap LQTE estimates as above through inversion by linearly interpolating the value of the distribution function between the finite set of points at which we have estimated values to find the value of the outcome associated with each quantile in the set $q \in [0.1,0.11,.0,12,...,0.89,.0.9]$, denoted $(\widehat\theta^{\leftarrow}_Y(q,d))^{b}$.  The bootstrap estimate of the LQTE for bootstrap replication $b$ at point $q$ is then $\widehat\Delta^b(q) = (\widehat\theta^{\leftarrow}_Y(q,1))^{b} - (\widehat\theta^{\leftarrow}_Y(q,0))^{b}$.
We form bootstrap standard error estimates for the LQTE at each quantile $q$ as
$$
s(q) = [q_{LQTE}(.75) - q_{LQTE}(.25)]/[q_N(.75) - q_N(.25)],
$$
where $q_{LQTE}(p)$ is the $p$th quantile of $\{\widehat\Delta^b(q)\}_{b=1}^{500}$ and $q_N(p)$ is the $p$th quantile of the $N(0,1)$.


We also use the bootstrap LQTE estimates to obtain the critical values we use when plotting the uniform confidence bands in our example.  We form bootstrap t-statistics for each quantile $q$ as $t^b(q) = (\widehat\Delta^b(q) - \widehat\Delta(q))/s(q)$.  We then take $t_{\max}^b = \max_q\{|t^b(q)|\}$ and use the 95$^{th}$ percentile of the bootstrap distribution of $t_{\max}^b$ as the critical value in constructing the confidence intervals for our figures following for example \citen{CFM}.

\bibliographystyle{ecta}
\bibliography{mybibVOLUME}

\newpage

\begin{center}
\Large{Supplement to ``Program Evaluation and Causal Inference with High-Dimensional Data''}
\end{center}

\medskip
\begin{center}
\normalsize 
A. Belloni, V. Chernozhukov,  I. Fern\'andez-Val, and C. Hansen
\end{center}
\medskip

\begin{footnotesize}
\textsc{Abstract.} The supplementary material contains 10 appendices with additional results and some omitted proofs.  Appendices \ref{sec:f}--\ref{sec:k} include additional results for Sections 2--7, respectively. Appendix \ref{sec:l} gathers auxiliary results on algebra of covering entropies. Appendices \ref{sec:m} and \ref{sec:n} contain the proofs of Sections 4 and 5 omitted from the main text. Appendix \ref{sec:o} contains the proofs of Sections 6 omitted from the main text, together with the proofs of the additional results for Section 6 in Appendix \ref{sec:j}. Appendix \ref{sec:p} reports the results of a simulation experiment.

\end{footnotesize}

\appendix

\setcounter{section}{5}

\section{Additional Results for Section 2}\label{sec:f}
\subsection{Causal Interpretations for Structural Parameters}\label{subsec: causal}
The quantities discussed in Sections \ref{subset: lates} and \ref{subset: latts} are well-defined and have causal interpretation under standard conditions.  We briefly recall these conditions, using the potential outcomes notation. Let  $Y_{u1}$ and $Y_{u0}$ denote
the potential outcomes  under the treatment states $1$ and $0$.  These outcomes are not observed jointly,
and we instead observe $Y_u = D Y_{u1} + (1-D) Y_{u0},$ where $D \in \mathcal{D} = \{0,1\}$
is the random variable indicating program participation or treatment state.  Under exogeneity, $D$ is assigned independently
of the potential outcomes conditional on covariates $X$, i.e.
$(Y_{u1}, Y_{u0}) \ci D \mid X$ a.s., where $\ci$ denotes statistical
independence.


Exogeneity fails when $D$ depends on the potential outcomes. For example, people may drop out of a program if they think the program will not benefit them. In this case, instrumental variables are useful in creating quasi-experimental fluctuations in $D$ that may identify useful effects. Let $Z$ be a binary instrument, such as an offer of participation,  that generates potential participation decisions $D_1$ and $D_0$ under the instrument states 1 and 0, respectively. As with the potential outcomes, the potential participation decisions under both instrument states are not observed jointly.  The realized participation decision is then given by $D = Z D_1 + (1-Z)D_0.$
 We assume that $Z$ is assigned randomly with respect to potential outcomes and participation decisions conditional on $X$, i.e., $(Y_{u0}, Y_{u1}, D_0, D_1) \ci Z \mid X$ a.s.


There are many causal quantities of interest for program evaluation.  Chief among these are various structural averages: $d \mapsto \Ep_{P}[Y_{ud}]$, the causal ASF;  $d \mapsto \Ep_P[ Y_{ud} \mid D = 1]$, the causal ASF-T;  $d \mapsto \Ep_P [ Y_{ud} \mid D_1 > D_0]$, the causal LASF; and $d \mapsto \Ep_P[ Y_{ud} \mid D_1 > D_0, D = 1],$ the causal LASF-T; as well as effects derived from them such as $\Ep_{P}[Y_{u1} - Y_{u0}]$, the causal ATE; $\Ep_P[ Y_{u1} - Y_{u0} \mid D = 1]$, the causal ATE-T;  $\Ep_P [ Y_{u1}- Y_{u0} \mid D_1 > D_0]$, the causal LATE; and  $\Ep_P[ Y_{u1} -  Y_{u0} \mid D_1 > D_0, D = 1],$ the causal LATE-T.
These causal quantities are the same as the structural parameters defined in Sections 2.2-2.3 under the following well-known sufficient condition.

\begin{assumption}[Assumptions for Causal/Structural Interpretability]\label{assumption: id}
The following conditions hold $P$-almost surely:  (Exogeneity) $((Y_{u1},Y_{u0})_{u \in \mU},D_1,D_0)  \ci Z \mid X$;
(First Stage) $\Ep_P[D_1 \mid X] \neq \Ep_P[D_0 \mid X]$; (Non-Degeneracy) $\Pr_P(Z=1 \mid X) \in (0,1)$;
(Monotonicity) $\Pr_P(D_1 \geq D_0 \mid X) = 1$. \end{assumption}

This condition due to \citen{imbens:angrist:94} and \citen{abadie:401k} is much-used in the program evaluation literature. It has an equivalent formulation in terms of a simultaneous equation model with a binary endogenous variable; see \citen{vytlacil:equiv} and \citen{heckman:vytlacil}.
For a thorough discussion of this assumption, we refer to \citen{imbens:angrist:94}.  Using this assumption, we present an identification lemma which follows from results of \citen{abadie:401k} and \citen{hong:nekipelov:2010} that both in turn build upon  \citen{imbens:angrist:94}.  The lemma shows that the parameters $\theta_{Y_u}$ and $\vartheta_{Y_u}$
defined earlier have a causal interpretation under Assumption \ref{assumption: id}.  Therefore, our referring to them as structural/causal is justified under this condition.

\begin{lemma}[Identification of Causal Effects ] Under Assumption \ref{assumption: id}, for each $d \in \mathcal{D}$,
$$
\Ep_P [ Y_{ud} \mid D_1 > D_0] = \theta_{Y_u}(d),   \ \  \Ep_P [ Y_{ud} \mid D_1 > D_0, D=1] = \vartheta_{Y_u}(d).
$$
Furthermore, if $D$ is exogenous, namely $D\equiv Z$ a.s., then
$$
\Ep_P [ Y_{ud} \mid D_1 > D_0]= \Ep_P [ Y_{ud} ], \ \  \Ep_P [ Y_{ud} \mid D_1 > D_0, D=1] = \Ep_P [ Y_{ud} \mid D=1].
$$
\end{lemma}

\section{Additional Results for Section \ref{EstimationSection}} \label{sec:g}

\begin{remark}[Another strategy for estimating $m_Z$ and $g_V$]\label{strategy2}
An alternative to the strategy for modeling and estimating $m_Z$ and $g_V$ is to treat $m_Z$ as in the text via (\ref{eq: approximations4})
while modeling $g_V$ through its disaggregation
\begin{equation}\label{relate g}
g_{V}(z,x) = \sum_{d=0}^1 e_{V}(d,z,x)  l_{D}(d,z,x),
\end{equation}
where the regression functions $e_{V}$ and $l_{D}$ map the support of $(D,Z,X)$, $\mathcal{DZX}$, to the real line and are defined by
\begin{eqnarray}
  && e_{V}(d,z,x): = \Ep_P[V|D=d, Z=z,X=x] \ \ \textrm{and} \\
  && l_{D}(d,z,x): =  \Pr_P[D=d|Z=z, X=x].
\end{eqnarray}
We will denote other potential values for the functions $e_{V}$ and $l_D$ by the parameters $e$ and $l$.  In this alternative approach, we can again use high-dimensional methods for modeling and estimating $e_{V}$ and $l_D$ using the same approach as in the main paper, and we can then use the relation (\ref{relate g}) to estimate $g_V$.\footnote{Upon conditioning on $D=d$ some parts become known; e.g.,
$e_{1_d(D) Y}(d',x,z) = 0$ if $d \neq d'$ and $e_{1_d(D)}(d',x,z) =1$ if $d = d'$.}  Specifically, we model the conditional expectation of $V$ given $D$, $Z$, and $X$ by
\begin{eqnarray}\label{eq: approximations-s2}
 && e_{V}(d,z,x) =:\Gamma_V [f(d,z,x) '\theta_{V} ] + \varrho_{V}(d,z,x), \\
 && f(d,z,x) :=  ( (1-d) f(z,x)', d f(z,x)')',   \\
 &&  \theta_V := (\theta_V(0,0)', \theta_V(0,1)', \theta_V(1,0)', \theta_V(1,1)')'.
\end{eqnarray}
We model the conditional probability of $D$ taking on 1 or $0$, given $Z$ and $X$ by
\begin{eqnarray}
  && l_{D}(1,z,x) =:   \Gamma_D [f(z,x) '\theta_{D}] +  \varrho_{D}(z,x), \\
  &&  l_D(0,z,x) = 1-  \Gamma_D [f(z,x) '\theta_{D}] -  \varrho_{D}(z,x),\\
 && f(z,x):=  ( (1-z) f(x)', z f(x)')',  \\
 &&  \theta_D := (\theta_D(0)', \theta_D(1)')'.\label{eq: approximations2-s2}
\end{eqnarray}
Here $\varrho_V(d,z,x)$ and $\varrho_D(z,x)$ are approximation errors, and the functions
 $\Gamma_V (f(d,z,x) '\theta_{V})$ and   $\Gamma_D (f(z,x) '\theta_{D})$
are generalized linear approximations to the target functions
 $e_{V}(d,z,x)$ and $l_{D}(1,z,x)$.  The functions  $\Gamma_V$ and $\Gamma_D$
 are taken again to be known link functions from the set $\mathcal{L}=\{ \Id, \Phi, 1-\Phi,  \Lambda_0, 1-\Lambda_0\}$ defined following equation (\ref{eq: approximations4}).

As in the strategy in the main text, we maintain approximate sparsity. We assume that
 there exist $\beta_Z$, $\theta_{V}$ and $\theta_{D}$ such that, for all $V \in \mV,$
\begin{equation}\label{eq: sparsity1-s2}
\|\theta_{V}\|_0  + \|\theta_{D}\|_0  + \|\beta_{Z}\|_0 \leq s.
\end{equation}
That is, there are at most $s=s_n \ll n$ components of $\theta_V$, $\theta_D$, and $\beta_Z$
 with nonzero values in the approximations to $e_V$, $l_D$ and $m_Z$.
The sparsity condition also requires the size of the approximation errors to be small compared to the conjectured size of the estimation error: For all $V \in \mV$, we assume
\begin{equation}\label{eq: sparsity2-s2}
\{\Ep_P[\varrho^{2}_{V}(D,Z,X)]\}^{1/2} + \{\Ep_P[\varrho^{2}_{D}(Z,X)]\}^{1/2}+ \{\Ep_P[r^{2}_{Z}(X)]\}^{1/2} \lesssim  \sqrt{s/n}.
\end{equation}
Note that the size of the approximating model $s=s_n$ can grow with $n$ just as in standard series estimation as long as
$s^2 \log^2 (p\vee n) \log^2(n) /n \to 0$.

We proceed with the estimation of $e_V$ and $l_D$ analogously to the approach outlined in the main text.  The Lasso estimator $\hat \theta_V$ and Post-Lasso estimator $\tilde \theta_V$ are defined analogously to $\hat \beta_V$ and $\tilde \beta_V$ using the data  $(\tilde Y_i, \tilde X_i)_{i=1}^n$= $( V_i, f(D_i,Z_i, X_i))_{i=1}^n$ and the link function $\Lambda = \Gamma_V$.  The estimator $\widehat e_V(D,Z,X) = \Gamma_V[f(D,Z,X)'\bar\theta_V]$, with $\bar \theta_V = \hat \theta_V$ or $\bar \theta_V = \tilde \theta_V$, has the near oracle rate of convergence $\sqrt{(s \log p)/n}$ and other desirable properties. The Lasso estimator $\hat \theta_D$ and Post-Lasso estimators $\tilde \theta_D$ are also defined analogously to $\hat \beta_V$ and $\tilde \beta_V$ using the data  $(\tilde Y_i, \tilde X_i)_{i=1}^n$= $( D_i, f(Z_i, X_i))_{i=1}^n$ and the link function $\Lambda = \Gamma_D$.  Again, the estimator $\widehat l_{D}(Z,X) =\Gamma_D[f(Z,X)'\bar \theta_D]$ of $l_{D}(Z,X)$,
 where $\bar \theta_D = \hat \theta_D$ or $\bar \theta_D = \tilde \theta_D$, has good theoretical properties including the near oracle rate of convergence, $\sqrt{(s \log p)/n}$.  The resulting estimator for $g_V$ is then
\begin{equation}\label{relate g hat}
\widehat g_{V}(z,x) = \sum_{d=0}^1 \widehat e_{V}(d,z,x)  \widehat l_{D}(d,z,x).
\end{equation}
The remaining estimation steps are the same as with the strategy given in the main text.
\end{remark}

\section{Additional Results for Section \ref{sec: asymptotics}} \label{sec:h}

\begin{assumption} [Approximate Sparsity for the Strategy of Section \ref{strategy2}]\label{ass: sparse2} Under each $P \in \mP_n$ and for each $n \geq n_0$, uniformly for all $V \in \mV $: (i)  The approximations (\ref{eq: approximations-s2})-(\ref{eq: approximations2-s2}) and (\ref{eq: approximations4}) apply with the link functions  $\Gamma_V$, $\Gamma_D$ and $\Lambda_Z$ belonging to the set $\mathcal{L}$, the sparsity condition $ \| \theta_{V}\|_0  + \|\theta_{D}\|_0  + \|\beta_{Z}\|_0 \leq s$ holding, the approximation errors satisfying $\|\varrho_{D}\|_{P,2}+ \|\varrho_{V}\|_{P,2} + \|r_{Z}\|_{P,2} \leq \delta_n n^{-1/4}$ and $\|\varrho_{D}\|_{P,\infty}+\|\varrho_{V}\|_{P,\infty} + \|r_{Z}\|_{P,\infty}  \leq \epsilon_n $, and the sparsity index $s$ and the number of terms $p$ in the vector $f(X)$ obeying  $ s^2 \log^2  (p \vee n)  \log^2 n  \leq \delta_n n$. (ii) There are estimators $\bar \theta_V$, $\bar \theta_D$,  and $\bar \beta_{Z}$ such that, with probability  no less than $1- \Delta_n$, the estimation errors satisfy $\|f(D,Z,X)'(\bar \theta_{V} - \theta_{V})\|_{\Pn,2} +
\|f(Z,X)'(\bar \theta_{D} - \theta_{D})\|_{\Pn,2} + \|f(X)'(\bar \beta_{Z} - \beta_{Z})\|_{\Pn,2} \leq \delta_n n^{-1/4}$ and $K_n \|\bar \theta_{V} - \theta_{V}\|_1 + K_n \|\bar \theta_{D} - \theta_{D}\|_1 + K_n \|\bar \beta_{Z} - \beta_{Z}\|_1 \leq \epsilon_n $;
the estimators  are sparse such that $ \|\bar \theta_{V}\|_0  + \|\bar \theta_{D}\|_0  + \|\bar \beta_{Z}\|_0  \leq Cs$; and the empirical and population norms induced by the Gram matrix formed by $(f(X_i))_{i=1}^n$ are equivalent on sparse subsets, $\sup_{\|\delta\|_{0} \leq \ell_n s} \left | \| f(X) '\delta\|_{\Pn,2}/\|f(X)'\delta\|_{P,2} -1 \right | \leq \epsilon_n$. (iii)  The following boundedness conditions hold: $\|\|f(X)\|_{\infty}||_{P, \infty} \leq K_n$ and  $\|V\|_{P,\infty} \leq C$. \end{assumption}

{Under the stated assumptions, the empirical reduced form process  $ \hat Z_{n,P} = \sqrt{n}(\hat \rho - \rho)$ defined by (\ref{define rho}), but constructed using the alternative strategy for estimating $m_Z$ and $g_V$ of Comment \ref{strategy2}, follows a functional central limit theorem and a functional central limit theorem for the multiplier bootstrap.  Theorem \ref{thm:str2} states these results. We omit the proof because it is analogous to the proofs of Theorems \ref{theorem1}--\ref{theorem2}.}


\begin{theorem} \label{thm:str2} Under Assumption \ref{ass: sparse2}  the results stated in Theorems \ref{theorem1}--\ref{theorem2} in the main text apply to the alternative strategy for estimating $m_Z$ and $g_V$ of Comment \ref{strategy2}.\end{theorem}








\section{Additional Results for Section \ref{FunctionalLassoSection}: Finite Sample Results of a Continuum of Lasso and Post-Lasso Estimators for Functional Responses}\label{subsec: lasso} \label{sec:j}

\subsection{Assumptions} We consider the following high level conditions which are implied by the primitive Assumptions \ref{ass: linear} and \ref{ass: logistic}. For each $n \geq 1$, there is a sequence of independent random variables $(W_i)_{i=1}^n$, defined on the probability space $(\Omega, \mathcal{A}_\Omega, \Pr_P)$ such that model (\ref{A:EqMainFunc}) holds with $\mathcal{U}\subset [0,1]^\dn$. Let $d_\mathcal{U}$ be a metric on $\mathcal{U}$ (and note that the results cover the case where $\dn$ is a function of $n$). Throughout this section we assume that the variables $( X_i, (Y_{ui},\zeta_{ui}:=Y_{ui}- \Ep_P[Y_{ui}\mid X_i])_{u \in \mathcal{U}})$ are generated as suitably measurable transformations of $W_i$ and $u\in\mathcal{U}$. Furthermore, this section uses the notation $\barEp_P[\cdot] = \frac{1}{n} \sum_{i=1}^n \Ep_P [ \cdot]$, because we allow for independent non-identically distributed (i.n.i.d.) data.

Consider fixed sequences of positive numbers $\delta_n \searrow 0$, $\epsilon_n \searrow 0$, and $\Delta_n \searrow 0$ at a speed at most polynomial in $n$, $\ell_n = \log n$, and $1 \leq K_n <\infty$; and positive constants $c$ and  $C$ which will not vary with $P$.

{\bf Condition WL.} {\it Suppose that for some $\epsilon>0$ there is a $N_n$ such that:
 (i) we have $\log N(\epsilon,\mathcal{U},d_\mathcal{U}) \leq N_n$;
(ii) uniformly over $u\in\mathcal{U}$, we have that ${\displaystyle \max_{j\leq p} } \frac{\{\barEp_P[|f_j(X)\zeta_{u}|^3]\}^{1/3}}{\{\barEp_P[|f_j(X)\zeta_{u}|^2]\}^{1/2}}\Phi^{-1}(1-1/\{2pN_nn\}) \leq \delta_n n^{1/6}$ and $0< c \leq \barEp_P[|f_j(X)\zeta_{u}|^2] \leq C$, $j=1,\ldots,p$;
and (iii) with probability $1-\Delta_n$, we have that   ${\displaystyle \sup_{u\in\mathcal{U}} \max_{j\leq p}} |(\En-\barEp_P)[f_j(X)^2\zeta_{u}^2]| \leq \delta_n$,
${\displaystyle \log(p \vee N_n \vee n) \sup_{d_\mathcal{U}(u,u')\leq \epsilon}  \max_{j\leq p}}\En[f_j(X)^2(\zeta_{u}-\zeta_{u'})^2] \leq \delta_n$,   ${\displaystyle\sup_{d_\mathcal{U}(u,u')\leq \epsilon}} \| \En[ f(X)(\zeta_{u}- \zeta_{u'}) ]\|_\infty\leq \delta_n n^{-\frac{1}{2}}$.
}

The following technical lemma justifies the choice of penalty level $\lambda$. It is based on self-normalized moderate deviation theory. In what follows, for $u\in\mathcal{U}$ we let $\hat  \Psi_{u0}$ denote a diagonal $p\times p$ matrix of ``ideal loadings" with diagonal elements given by $\hat  \Psi_{u0jj}=\{\En[f_j^2(X)\zeta_u^2]\}^{1/2}$ for $j=1,\ldots,p$.

\begin{lemma}[Choice of $\lambda$]\label{Thm:ChoiceLambda}
Suppose Condition WL holds, let $c'>c>1$ be constants, $\gamma \in [1/n,1/\log n]$, and $\lambda = c'\sqrt{n}\Phi^{-1}(1-\gamma/\{2pN_n\})$. Then for $n \geq n_0$ large enough depending only on Condition WL,
$$\Pr_P\left ( \lambda/n \geq c \sup_{u\in\mathcal{U}}\|\hat  \Psi^{-1}_{u0}\En[  f(X) \zeta_{u} ]\|_\infty \right ) \geq 1-\gamma-o(1).$$
\end{lemma}

We note that Condition WL(iii) contains high level conditions on the process $(Y_u,\zeta_u)_{u\in\mathcal{U}}$. The following lemma provides easy to verify sufficient conditions that imply Condition WL(iii).

\begin{lemma}\label{PrimitiveWL}
Suppose the i.i.d. sequence $((Y_{ui},\zeta_{ui})_{u\in\mathcal{U}},X_i), i=1,\ldots,n$, satisfies the following conditions: (i) $c\leq \max_{j\leq p}\Ep_P[f_j(X)^2] \leq C$, $\max_{j\leq p}|f_j(X)| \leq K_n$, $\sup_{u\in \mathcal{U}}\max_{i\leq n} |Y_{ui}| \leq B_n$, and $c\leq \sup_{u\in \mathcal{U}} \Ep_P[\zeta_u^2\mid X] \leq C$,  $P$-a.s.; (ii) for some random variable $Y$ we have $Y_u=G(Y,u)$ where $\{ G(\cdot, u) : u \in \mathcal{U}\}$ is a VC-class of functions with VC-index equal to $C'd_u$, (iii) For some fixed $\nu > 0$, we have $\Ep_P[|Y_{u}-Y_{u'}|^2\mid X]\leq L_n|u-u'|^\nu$ for any $u, u' \in \mathcal{U}$, $P$-a.s. For $\widetilde A := pnK_nB_nn^\nu/L_n$, we have with probability $1-\Delta_n$
{\small $$\begin{array}{rl}
 \displaystyle \sup_{d_{\mathcal{U}}(u,u')\leq 1/n} \|\En[ f(X)(\zeta_u-\zeta_{u'})]\|_\infty   & \lesssim \frac{1}{\sqrt{n}}\left\{\sqrt{\frac{(1+d_u) L_n\log(\widetilde A)}{n^\nu}} + \frac{(1+d_u)K_nB_n\log(\widetilde A)}{\sqrt{n}} \right\}\\
 \displaystyle  \sup_{d_{\mathcal{U}}(u,u')\leq 1/n} \max_{j\leq p}\En[f_j(X)^2(\zeta_u - \zeta_{u'})^2]
  &  \lesssim  L_nn^{-\nu} \left\{ 1 + \sqrt{\frac{K_n^2\log (p nK_n^2)}{n}} + \frac{K_n^2}{n}\log( pnK_n^2 )\right\}\\
 \displaystyle  \sup_{u\in\mathcal{U}}\max_{j\leq p}|(\En-\Ep_P)[f_j^2(X)\zeta_u^2]|
 & \lesssim \sqrt{\frac{(1+d_u) \log (npK_nB_n)}{n}}  + \frac{(1+d_u)K_n^2B_n^2}{n}\log(npB_nK_n)\end{array}$$}
where $\Delta_n$ is a fixed sequence going to zero.
\end{lemma}

Lemma \ref{PrimitiveWL} allows for several different cases including cases where $Y_u$ is generated by a non-smooth transformation of a random variable $Y$. For example, if $Y_u = 1\{ Y \leq u\}$ where $Y$ has bounded conditional probability density function, we have $d_u = 1$, $B_n = 1$, $\nu = 1$, $L_n=\sup_{y}f_{Y\mid X}(y\mid x)$. A similar result holds for independent non-identically distributed data.

In what follows for a  vector $\delta \in \RR^p$, and a set of
indices $T \subseteq \{1,\ldots,p\}$, we denote by $\delta_T \in \RR^p$ the vector such that $(\delta_{T})_{j} = \delta_j$ if $j\in T$ and $(\delta_{T})_{j}=0$ if $j \notin T$. For a set $T$, $|T|$ denotes the cardinality of $T$. Moreover, let $$\Delta_{\cc,u} := \{ \delta \in \RR^p : \|\delta_{T_u^c}\|_1\leq \cc\|\delta_{T_u}\|_1\}.$$

\subsection{Finite Sample Results: Linear Case}\label{Sec:FiniteLinear}

For the model described in (\ref{A:EqMainFunc}) with $\G(t)=t$ and $M(y,t) = \frac{1}{2}(y-t)^2$ we will study the finite sample properties of the associated Lasso and Post-Lasso estimators of $(\theta_u)_{u\in\mathcal{U}}$ defined in relations (\ref{Adef:LassoFunc}) and (\ref{Adef:PostFunc}).

The analysis relies on $T_u=\supp(\theta_u)$, $s_u:=\|\theta_u\|_0\leq s$, with $s \geq 1$, and on the restricted eigenvalues
\begin{equation}\label{DefKappa}\kappa_{\cc} = \inf_{u\in\mathcal{U}}\min_{\delta \in \Delta_{\cc,u}} \frac{\|f(X)'\delta\|_{\Pn,2}}{\|\delta_{T_u}\|}, \end{equation}
and maximum and minimum sparse eigenvalues
$$ \semin{m} = \min_{1\leq \|\delta\|_0\leq m}\frac{\|f(X)'\delta\|_{\Pn,2}^2}{\|\delta\|^2} \ \ \mbox{and} \ \ \semax{m} = \max_{1\leq \|\delta\|_0\leq m}\frac{\|f(X)'\delta\|_{\Pn,2}^2}{\|\delta\|^2}.$$

Next we present technical results on the performance of the estimators generated by Lasso that are used in the proof of Theorem \ref{Thm:RateEstimatedLassoLinear}.

\begin{lemma}[Rates of Convergence for Lasso]\label{Thm:RateEstimatedLasso}
The events $c_r \geq \sup_{u\in\mathcal{U}}\|r_u\|_{\Pn,2}$, $\ell\widehat\Psi_{u0}\leq \widehat\Psi_u \leq L \widehat \Psi_{u0}$, $u\in\mathcal{U}$, and $\lambda/n \geq c\sup_{u\in\mathcal{U}}\|\widehat\Psi_{u0}^{-1}\En[f(X)\zeta_u]\|_\infty$, for $c>1/\ell$, imply that
uniformly in $u\in\mathcal{U}$ $$\begin{array}{l}
 \displaystyle \| f(X)'(\hat\theta_u - \theta_{u})\|_{\Pn,2} \leq  2 c_r + \frac{2\lambda\sqrt{s}\left(L+\frac{1}{c}\right)}{n\kappa_{\tilde \cc}}\|\widehat\Psi_{u0}\|_\infty\\
 \displaystyle \|\hat\theta_u-\theta_{u}\|_1  \leq 2(1+2\tilde\cc)\left\{ \frac{\sqrt{s}c_r}{ \kappa_{2\tilde\cc}} + \frac{\lambda s \left(L+\frac{1}{c}\right)}{n\kappa_{\tilde \cc}\kappa_{2\tilde\cc}}\|\widehat\Psi_{u0}\|_\infty \right\}+ \left(1+\frac{1}{2\tilde\cc}\right)\frac{c\|\hat\Psi_{u 0}^{-1}\|_\infty}{\ell c-1}\frac{n}{\lambda} c_r^2\end{array}$$
where $\tilde \cc = \sup_{u\in\mathcal{U}}\|\hat \Psi_{u 0}^{-1}\|_\infty\|\hat \Psi_{u 0}\|_\infty(Lc+1)/(\ell c - 1)$
\end{lemma}

The following lemma summarizes sparsity properties of $(\hat\theta_u)_{u\in\mathcal{U}}$.

\begin{lemma}[Sparsity bound for Lasso]\label{Thm:Sparsity}
Consider the Lasso estimator $\widehat \theta_u$, its support $\widehat T_u=\supp(\widehat \theta_u)$, and let $\widehat s_u = \|\widehat \theta_u\|_0$. Assume that $c_r \geq \sup_{u\in\mathcal{U}}\|r_u\|_{\Pn,2}$, $\lambda/n \geq c\sup_{u\in\mathcal{U}}\|\widehat \Psi_{u 0}^{-1}\En[f(X)\zeta_{u}]\|_{\infty}$ and $\ell\widehat \Psi_{u0} \leq \widehat\Psi_u \leq L \widehat\Psi_{u0}$ for all $u\in\mathcal{U}$, with $L \geq 1 \geq \ell > 1/c$. Then, for $c_0=(Lc+1)/(\ell c - 1)$ and $\tilde \cc = c_0\sup_{u\in\mathcal{U}}\|\hat\Psi_{u 0}\|_\infty\|\hat\Psi_{u 0}^{-1}\|_\infty$ we have uniformly over $u\in\mathcal{U}$
$$ \hat s_u \leq  16c_0^2\left( \min_{m \in \mathcal{M}}\semax{m}\right)   \[\frac{n c_r}{\lambda} + \frac{\sqrt{s}}{\kappa_{\tilde \cc}}\|\widehat\Psi_{u0}\|_\infty\]^2\|\widehat \Psi_{u0}^{-1}\|_\infty^2$$
where $\mathcal{M}=\left\{ m \in \mathbb{N}:
m > 32c_0^2\semax{m} \sup_{u\in\mathcal{U}} \[\frac{n c_r}{\lambda} + \frac{\sqrt{s}}{\kappa_{\tilde \cc}}\|\widehat\Psi_{u0}\|_\infty\]^2\|\widehat \Psi_{u0}^{-1}\|_\infty^2 \right\}.$\end{lemma}

\begin{lemma}[Rate of Convergence of Post-Lasso]\label{Thm:2StepMain}
Under Conditions WL, let $\widetilde \theta_u$ be
the Post-Lasso estimator based on the support $\widetilde T_u$. Then, with probability $1-o(1)$, uniformly over $u\in\mathcal{U}$, we have for $\tilde s_u = |\widetilde T_u|$
$$
\displaystyle  \| \Ep_P[Y_u\mid X] - f(X)'\widetilde \theta_u\|_{\Pn,2} \leq C \frac{\sqrt{\tilde s_u \log (p\vee n^{\dn+1}) }}{\sqrt{n \ \semin{\tilde s_u}}}\|\widehat\Psi_{u0}\|_\infty  + \min_{\supp(\theta)\subseteq\widetilde T_u}  \|\Ep_P[Y_u\mid X]- f(X)'\theta\|_{\Pn,2}\\
$$
Moreover, if $\supp(\widehat \theta_u)\subseteq \widetilde T_u$ for every $u\in\mathcal{U}$, the following events $c_r \geq \sup_{u\in\mathcal{U}}\|r_u\|_{\Pn,2}$, $\ell\widehat\Psi_{u0}\leq \widehat\Psi_u \leq L \widehat \Psi_{u0}$, $u\in\mathcal{U}$, and $\lambda/n \geq c\sup_{u\in\mathcal{U}}\|\widehat\Psi^{-1}_{u0}\En[f(X)\zeta_u]\|_\infty$, for $c>1/\ell$, imply that
$$\sup_{u\in\mathcal{U}}\min_{\supp(\theta)\subseteq\widetilde T_u}  \|\Ep_P[Y_u\mid X]-f(X)'\theta\|_{\Pn,2} \leq 3 c_r +  \(L + \frac{1}{c}\) \frac{2\lambda \sqrt{s}}{n \kappa_{\tilde \cc}}\sup_{u\in\mathcal{U}}\|\widehat\Psi_{u0}\|_{\infty}.$$
\end{lemma}

\subsection{Finite Sample Results: Logistic Case}\label{Sec:ResultsLogistic}

For the model described in (\ref{A:EqMainFunc}) with $\G(t)=\exp(t)/\{1+\exp(t)\}$ and $M(y,t) = -\{1\{y=1\}\log(\G(t))+1\{y=0\}\log(1-\G(t))\}$ we will study finite the sample properties of the associated Lasso and Post-Lasso estimators of $(\theta_u)_{u\in\mathcal{U}}$  defined in relations (\ref{Adef:LassoFunc}) and (\ref{Adef:PostFunc}). In what follows we use the notation $$M_u(\theta) = \En[M(Y_u,f(X)'\theta)].$$

In the finite sample analysis we will consider not only the design matrix $\En[ f(X) f(X)']$ but also a weighted counterpart $\En[w_{u} f(X) f(X)']$ where $w_{ui}=\Ep_P[Y_{ui}\mid X_i](1-\Ep_P[Y_{ui}\mid X_i])$, $i=1,\ldots,n$, $u\in\mathcal{U}$, is the conditional variance of the outcome variable $Y_{ui}$.

For $T_u=\supp(\theta_u)$, $s_u = \|\theta_u\|_0 \leq s$, with $s \geq 1$, the (logistic) restricted eigenvalue is defined as \begin{equation}\label{DefKAppaLog} \bar\kappa_\cc := \inf_{u\in \mathcal{U}} \min_{\delta \in \Delta_{\cc,u}} \frac{\|\sqrt{w_{u}} f(X)'\delta\|_{\Pn,2}}{\|\delta_{T_u}\|}.\end{equation}
For a subset $A_u\subset \RR^p$, $u\in\mathcal{U}$, let the non-linear impact coefficient \cite{BC-SparseQR,BCY-honest} be defined as
 \begin{equation}\label{NIC:Logistic}\bar q_{A_u} := \inf_{ \delta \in A_u} \frac{\En\[w_{u}| f(X)'\delta|^2\]^{3/2}}{\En\[w_{u}| f(X)'\delta|^3\]}.\end{equation}

Note that $\bar q_{A_u}$ can be bounded as
$$ \bar q_{A_u} = \inf_{ \delta \in A_u} \frac{\En\[w_{u}| f(X)'\delta|^2\]^{3/2}}{\En\[w_{u}| f(X)'\delta|^3\]} \geq \inf_{ \delta \in A_u} \frac{\En\[w_{u}| f(X)'\delta|^2\]^{1/2}}{\max_{i\leq n}\| f(X_i)\|_\infty\|\delta\|_1}$$
which can lead to interesting bounds provided $A_u$ is appropriate (like the restrictive set $\Delta_{\cc,u}$ in the definition of restricted eigenvalues). In Lemma \ref{Lemma:LassoLogisticRateRaw} we have
$A_u = \Delta_{2\tilde\cc,u}\cup \{ \delta \in \RR^p :  \|\delta\|_1\leq \frac{6c\|\widehat\Psi_{u0}^{-1}\|_\infty}{\ell c - 1}\frac{n}{\lambda}\|\frac{r_{u}}{\sqrt{w_{u}}}\|_{\Pn,2}\|\sqrt{w_{u}} f(X)'\delta\|_{\Pn,2}\}$, for $u\in \mathcal{U}$. For this choice of sets, and provided that  with probability $1-o(1)$ we have
 $\ell c > c' > 1$, $\sup_{u\in\mathcal{U}}\|r_{u}/\sqrt{w_{u}}\|_{\Pn,2}\lesssim \sqrt{ s \log(p\vee n)/ n}$, $\sup_{u\in\mathcal{U}}\|\widehat\Psi_{u0}^{-1}\|_\infty\lesssim 1$ and $  \sqrt{n\log (p\vee n)}\lesssim \lambda$, we have that uniformly over $u\in\mathcal{U}$, with probability $1-o(1)$
 \begin{equation}\label{LowerBound:q}\bar q_{A_u} \geq \frac{1}{{\displaystyle \max_{i\leq n}}\| f(X)\|_\infty} \left( \frac{\bar\kappa_{2\tilde\cc}}{\sqrt{s_u}(1+2\tilde\cc)} \wedge \frac{(\lambda/n)(\ell c-1)}{6c \|\widehat\Psi_{u0}^{-1}\|_\infty\|r_{u}/\sqrt{w_{u}}\|_{\Pn,2}}\right) \gtrsim \frac{\bar\kappa_{2\tilde\cc}}{\sqrt{s}{\displaystyle \max_{i\leq n}}\| f(X_i)\|_\infty}.\end{equation}

The definitions above differ from their counterpart in the analysis of $\ell_1$-penalized least squares estimators by the weighting $0\leq w_{ui}\leq 1$. Thus it is relevant to understand their relations through the quantities $$\psi_{u}(A) := \min_{\delta \in A} \frac{\|\sqrt{w_{u}} f(X)'\delta\|_{\Pn,2}}{\| f(X)'\delta\|_{\Pn,2}}.$$
Many primitive conditions on the data generating process will imply $\psi_{u}(A)$ to be bounded away from zero for the relevant choices of $A$. We refer to \citen{BCY-honest} for bounds on $\psi_{u}$. For notational convenience we will also work with a rescaling of the approximation errors $\tilde r_{u}(X)$ defined as
 \begin{equation}\label{def:tilder}\tilde r_{ui}=\tilde r_{u}(X_i)=\G^{-1}( \ \G(f(X_i)'\theta_u)+r_{ui} \ ) - f(X_i)'\theta_u,\end{equation}
which is the unique solution to $ \G(f(X_i)'\theta_u+\tilde r_{u}(X_i))= \G(f(X_i)'\theta_u)+r_{u}(X_i)$.
It follows that $| r_{ui} | \leq |\tilde r_{ui}|$ and that\footnote{The last relation follows from noting that for the logistic function we have $\inf_{ 0\leq t \leq \tilde r_{ui}}\G'(f(X_i'\theta_u)+t) = \min\{\G'(f(X_i'\theta_u)+\tilde r_{ui}),\G'(f(X_i'\theta_u))\}$ since $\G'$ is unimodal. Moreover, $\G'(f(X_i'\theta_u)+\tilde r_{ui})=w_{ui}$ and  $\G'(f(X_i'\theta_u)) = \G(f(X_i'\theta_u))[1-\G(f(X_i'\theta_u))]=[\G(f(X_i'\theta_u))+r_{ui}-r_{ui}][1-\G(f(X_i'\theta_u))-r_{ui}+r_{ui}]\geq w_{ui}-2|r_{ui}|$ since $|r_{ui}|\leq 1$.} $|\tilde r_{ui}| \leq |r_{ui}| / \inf_{ 0\leq t \leq \tilde r_{ui}}\G'(f(X_i'\theta_u)+t) \leq |r_{ui}| / \{ w_{ui} - 2|r_{ui}|\}_+$.

Next we derive finite sample bounds provided some crucial events occur.

\begin{lemma}[Rates of Convergence for $\ell_1$-Logistic Estimator]\label{Lemma:LassoLogisticRateRaw}
Assume that $$\lambda/n \geq c \sup_{u\in \mathcal{U}}\|\widehat\Psi^{-1}_{u0}\En[f(X)\zeta_u] \|_\infty $$ for $c>1$. Further, let $\ell\widehat\Psi_{u0} \leq \widehat\Psi_u \leq L\widehat\Psi_{u0}$ for $L\geq  1 \geq \ell > 1/c$, uniformly over $u\in \mathcal{U}$, $\tilde \cc = (Lc+1)/(\ell c-1)\sup_{u\in\mathcal{U}}\|\widehat\Psi_{u0}\|_\infty\|\widehat\Psi_{u0}^{-1}\|_\infty$
and $$A_u = \Delta_{2\tilde\cc,u} \cup \{ \delta : \|\delta\|_1 \leq \frac{6c\|\widehat\Psi_{u0}^{-1}\|_\infty}{\ell c - 1}\frac{n}{\lambda}\|r_{u}/\sqrt{w_{u}}\|_{\Pn,2}\|\sqrt{w_{u}} f(X)'\delta\|_{\Pn,2}\}.$$
Provided that the nonlinear impact coefficient $\bar q_{A_u} > 3\left\{(L+\frac{1}{c})\|\widehat\Psi_{u0}\|_\infty\frac{\lambda\sqrt{s}}{n\bar\kappa_{2\tilde\cc}}+ 9\tilde \cc\|\tilde r_{u}/\sqrt{w_{u}}\|_{\Pn,2}\right\}$ for every $u\in \mathcal{U}$, we have uniformly over $u\in\mathcal{U}$
{\small $$ \|\sqrt{w_{u}} f(X)'(\hat\theta_u-\theta_u)\|_{\Pn,2} \leq 3\left\{(L+\frac{1}{c})\|\widehat\Psi_{u0}\|_\infty\frac{\lambda\sqrt{s}}{n\bar\kappa_{2\tilde\cc}}+ 9\tilde\cc\|\tilde r_{u}/\sqrt{w_{u}}\|_{\Pn,2}\right\} \ \ \ \mbox{and} \ \ \ $$ $$\|\hat \theta_u -\theta_u\|_1 \leq  3\left\{ \frac{(1+2\tilde\cc)\sqrt{s}}{\bar\kappa_{2\tilde \cc}}+ \frac{6c\|\widehat\Psi_{u0}^{-1}\|_\infty}{\ell c - 1}\frac{n}{\lambda}\left\|\frac{ r_{u}}{\sqrt{w_{u}}}\right\|_{\Pn,2}\right\}\left\{(L+\frac{1}{c})\|\widehat\Psi_{u0}\|_\infty\frac{\lambda\sqrt{s}}{n\bar\kappa_{2\tilde\cc}}+9\tilde\cc\left\|\frac{\tilde r_{u}}{\sqrt{w_{u}}}\right\|_{\Pn,2}\right\}$$}
\end{lemma}


The following result provides  bounds on the number of non-zero coefficients in the $\ell_1$-penalized estimator $\hat\theta_u$, uniformly over $u\in\mathcal{U}$.

\begin{lemma}[Sparsity of $\ell_1$-Logistic Estimator]\label{Lemma:LassoLogisticSparsity} Assume $\lambda/n \geq c \sup_{u\in \mathcal{U}}\|\widehat\Psi^{-1}_{u0}\En[f(X)\zeta_u] \|_\infty $ for $c>1$. Further, let $\ell\widehat\Psi_{u0} \leq \widehat\Psi_u \leq L\widehat\Psi_{u0}$ for $L\geq  1 \geq \ell > 1/c$, uniformly over $u\in \mathcal{U}$, $c_0=(Lc+1)/(\ell c-1)$, $\tilde \cc = c_0\sup_{u\in\mathcal{U}}\|\widehat\Psi_{u0}\|_\infty\|\widehat\Psi_{u0}^{-1}\|_\infty$
and $A_u = \Delta_{2\tilde\cc,u} \cup \{ \delta : \|\delta\|_1 \leq \frac{6c\|\widehat\Psi_{u0}^{-1}\|_\infty}{\ell c - 1}\frac{n}{\lambda}\|r_{u}/\sqrt{w_{u}}\|_{\Pn,2}\|\sqrt{w_{u}} f(X)'\delta\|_{\Pn,2}\}$, and $\bar q_{A_u} > 3\left\{(L+\frac{1}{c})\|\widehat\Psi_{u0}\|_\infty\frac{\lambda\sqrt{s}}{n\bar\kappa_{2\tilde\cc}}+ 9\tilde \cc\|\tilde r_{u}/\sqrt{w_{u}}\|_{\Pn,2}\right\}$ for every $u\in \mathcal{U}$.
Then for $\hat s_u = \|\hat\theta_u\|_0$, uniformly over $u\in \mathcal{U}$,
$$\hat s_u \leq  \left( \min_{m \in \mathcal{M}}\semax{m}\right)   \[\frac{c_0}{\psi(A_u)}\left\{ 3\|\hat\Psi_{u0}\|_\infty\frac{\sqrt{s}}{\bar\kappa_{2\tilde \cc}}   +28\tilde\cc\frac{n\|\tilde r_{u}/\sqrt{w_u}\|_{\Pn,2}}{\lambda}\right\}\]^2$$
where $\mathcal{M}=\left\{ m \in \mathbb{N}:
m > 2 \[\frac{c_0}{\psi(A_u)}\sup_{u\in\mathcal{U}}\left\{ 3\|\hat\Psi_{u0}\|_\infty\frac{\sqrt{s}}{\bar\kappa_{2\tilde \cc}}   +28\tilde\cc\frac{n\|\tilde r_{u}/\sqrt{w_u}\|_{\Pn,2}}{\lambda}\right\}\]^2 \right\}.$

Moreover, if $\sup_{u\in\mathcal{U}}\max_{i\leq n}| f(X_i)'(\hat\theta_u-\theta_u)-\tilde r_{ui}| \leq 1$ we have
$$\hat s_u \leq  \left( \min_{m \in \mathcal{M}}\semax{m}\right) 4c_0^2  \left\{ 3\|\hat\Psi_{u0}\|_\infty\frac{\sqrt{s}}{\bar\kappa_{2\tilde \cc}}   +28\tilde\cc\frac{n\|\tilde r_{u}/\sqrt{w_u}\|_{\Pn,2}}{\lambda}\right\}^2$$
where $\mathcal{M}=\left\{ m \in \mathbb{N}:
m > 8 c_0^2\sup_{u\in\mathcal{U}}\[ 3\|\hat\Psi_{u0}\|_\infty\frac{\sqrt{s}}{\bar\kappa_{2\tilde \cc}}   +28\tilde\cc\frac{n\|\tilde r_{u}/\sqrt{w_u}\|_{\Pn,2}}{\lambda}\]^2 \right\}.$
\end{lemma}

Next we turn to finite sample bounds for the logistic regression estimator where the support was selected based on $\ell_1$-penalized logistic regression. The results will hold uniformly over $u\in\mathcal{U}$ provided the side conditions also hold uniformly over $\mathcal{U}$.

\begin{lemma}[Rate of Convergence for Post-$\ell_1$-Logistic Estimator]\label{Lemma:PostLassoLogisticRateRaw} Consider $\widetilde \theta_u$ defined as the post model selection logistic regression with the support $\widetilde T_u$ and let $\tilde s_u := |\widetilde T_u|$. Uniformly over $u\in\mathcal{U}$ we have
{\small  $$ \|\sqrt{w_{u}} f(X)'(\widetilde \theta_u - \theta_u)\|_{\Pn,2} \leq \sqrt{3}\sqrt{0\vee\{M_u(\widetilde \theta_u) - M_u(\theta_u)\}} + 3\left\{\frac{\sqrt{\tilde s_u+s_u}\|\En[f(X)\zeta_u]\|_\infty}{\psi_{u}(A_u)\sqrt{\semin{\tilde s_u+s_u}}} + 3\left\|\frac{\tilde r_{u}}{\sqrt{w_{u}}}\right\|_{\Pn,2}\right\}$$}
\noindent provided that,  for every $u\in\mathcal{U}$ and $A_u = \{ \delta \in \RR^p : \|\delta\|_0 \leq \tilde s_u + s_u\}$,
 $$\bar q_{A_u}>6\left\{\frac{\sqrt{\tilde s_u+s_u}\|\En[f(X)\zeta_u]\|_\infty}{\psi_{u}(A_u)\sqrt{\semin{\tilde s_u+s_u}}} + 3\left\|\frac{\tilde r_{u}}{\sqrt{w_{u}}}\right\|_{\Pn,2}\right\} \text{ and } \bar q_{A_u}>6\sqrt{0\vee \{M_u(\widetilde \theta_u) - M_u(\theta_u)\}}.$$
\end{lemma}

\begin{remark}\label{L2nToL1rates} Since for a sparse vector $\delta$ such that $\|\delta\|_0=k$ we have $\|\delta\|_1\leq \sqrt{k}\|\delta\|\leq \sqrt{k}\|f(X)'\delta\|_{\Pn,2}/\sqrt{\semin{k}}$, the results above can directly establish bounds on the rate of convergence in the $\ell_1$-norm.\end{remark}

\section{Additional Results for Section \ref{sec: 401k}} \label{sec:k}

In this section, we report additional results to supplement those provided in the main text.  Specifically, we provide results with both total wealth and net total financial assets as the outcome variable.
We present detailed results for four different sets of controls $f(X)$.  The first set uses the indicators of marital status, two-earner status, defined benefit pension status, IRA participation status, and home ownership status, a linear term for family size, five categories for age, four categories for education, and seven categories for income (Indicator specification).  We use the same definitions of categories as in \citen{CH401k} and note that this is identical to the specification in \citen{CH401k} and \citen{benjamin}. The second through fourth specifications correspond to the Quadratic Spline specification, the Quadratic Spline Plus Interactions specification, and the Quadratic Spline Plus Many Interactions specification described in the main text.

Results for intention to treat effects based on using 401(k) eligibility as the treatment variable are given in Appendix Table 1.  In Appendix Table 2, we report results using 401(k) participation as the treatment variable instrumenting with 401(k) eligibility.  We plot the QTE and QTE-T, based on using 401(k) eligibility as the treatment variable, in  Figures 3-6.  Finally, the LQTE and LQTE-T, based on using 401(k) participation as the treatment variability and instrumenting with eligibility, are plotted in Appendix Figures 7-10.  The results are broadly consistent with the discussion provided in the main text with the selection and no selection results being similar in the low-dimensional cases and the selection results being substantially more regular in the high-dimensional cases.  We also see that the patterns of point estimates for total wealth and net total financial assets are similar, though the total wealth estimates have substantially larger estimated standard errors, especially for high quantiles.

\section{Auxiliary Results: Algebra of Covering Entropies}\label{sec:l}

\begin{lemma}[\textbf{{Algebra for Covering Entropies}}] \text{  } \label{lemma: andrews}Work with the setup described in Appendix C of the main text. \\
(1) Let $\F$ be a VC subgraph class with a finite VC index $k$ or any
other class whose entropy is bounded above by that of such a VC subgraph class, then
the covering entropy of $\mF$ obeys:
\begin{equation*}
 \sup_{Q} \log  N(\epsilon \|F\|_{Q,2}, \F,  \| \cdot \|_{Q,2}) \lesssim 1+ k \log (1/\epsilon)\vee 0
\newline
\end{equation*}
(2) For any measurable classes of functions $\F$ and $\F^{\prime
} $ mapping $\mathcal{W}$ to $\mathbb{R}$
$$\begin{array}{l}
\log N(\epsilon \Vert F+F^{\prime }\Vert _{Q,2},\F+\F^{\prime
}, \| \cdot \|_{Q,2})\leq \log   N\left(\mbox{$ \frac{\epsilon }{2}$}\Vert F\Vert _{Q,2},\F, \| \cdot \|_{Q,2}\right)
+ \log N\left( \mbox{$ \frac{\epsilon }{2}$}\Vert F^{\prime }\Vert _{Q,2},\F^{\prime
}, \| \cdot \|_{Q,2}\right), \\
\log  N(\epsilon \Vert F\cdot F^{\prime }\Vert _{Q,2},\F\cdot \F^{\prime
}, \| \cdot \|_{Q,2})\leq \log   N\left( \mbox{$ \frac{\epsilon }{2}$}\Vert F\Vert _{Q,2},\F, \| \cdot \|_{Q,2}\right)
+ \log N\left( \mbox{$ \frac{\epsilon }{2}$}\Vert F^{\prime }\Vert _{Q,2},\F^{\prime
}, \| \cdot \|_{Q,2}\right), \\
 N(\epsilon \Vert F\vee F^{\prime }\Vert _{Q,2},\F\cup \F^{\prime
}, \| \cdot \|_{Q,2})\leq   N\left(\epsilon\Vert F\Vert _{Q,2},\F, \| \cdot \|_{Q,2}\right)
+ N\left( \epsilon\Vert F^{\prime }\Vert _{Q,2},\F^{\prime
}, \| \cdot \|_{Q,2}\right).
\end{array}
$$
(3)  Given a measurable class $%
\mathcal{F}$ mapping $\mathcal{W}$ to $\mathbb{R}$ and a random variable $\xi$ taking values in $\mathbb{R}$,
\begin{equation*}
 \log \sup_{Q} N(\epsilon \Vert |\xi|F\Vert _{Q,2},\xi\F, \| \cdot \|_{Q,2})\leq \log \sup_{Q} N\left(
\epsilon /2\Vert F\Vert _{Q,2},\F, \| \cdot \|_{Q,2}\right)
\end{equation*}
(4)  Given measurable classes $\F_j$ and envelopes $F_j$, $j=1,\ldots,k$, mapping $\mathcal{W}$ to $\mathbb{R}$, a function $\phi:\mathbb{R}^k\to\mathbb{R}$ such that for $f_j,g_j\in\F_j$,
$ |\phi(f_1,\ldots,f_k) - \phi(g_1,\ldots,g_k) | \leq \sum_{j=1}^k L_j(x)|f_j(x)-g_j(x)|$, $L_j(x)\geq 0$, and fixed functions $\bar f_j \in \F_j$,  the class of functions $\mathcal{L}=\{\phi(f_1,\ldots,f_k)-\phi(\bar f_1,\ldots,\bar f_k):f_j \in\mathcal{F}_j, j=1,\ldots,k\}$ satisfies
\begin{equation*}
 \log \sup_Q N(\epsilon\|\sum_{j=1}^kL_jF_j\|_{Q,2},\mathcal{L}, \| \cdot \|_{Q,2})\leq \sum_{j=1}^k\log  \sup_Q  N\left(
\mbox{$\frac{\epsilon}{k}$}\|F_j\|_{Q,2},\F_j, \| \cdot \|_{Q,2}\right).
\end{equation*}
\end{lemma}
\textbf{Proof.} For the proof (1)-(2) see, e.g., \citen{andrews:emp} and (3) follows from (2). To show (4) let $f=(f_1,\ldots,f_k)$ and $g=(g_1,\ldots,g_k)$ where $f_j,g_j \in \mathcal{F}_j$, $j=1,\ldots,k$. Then, by the condition on $\phi$, we have
\begin{equation}\label{AuxEntropyL}\begin{array}{rl}
\| \phi(f)-\phi(g) \|_{Q,2} & \leq \| \sum_{j=1}^k L_j|f_j-g_j| \ \|_{Q,2} \\
& \leq \sum_{j=1}^k \|  L_j|f_j-g_j| \ \|_{Q,2} \\
\end{array}\end{equation}
Let $\hat{\mathcal{N}}_j$ be a $(\epsilon/k)$-net for $\mF_j$ with the measure $\widetilde Q_j$, where $d\widetilde Q_j(x) = L_j^2(x)dQ(x)$. Then the set $\{ \phi(f_1,\ldots,f_k)-\phi(\bar f_1,\ldots,\bar f_k): f_j \in \hat{\mathcal{N}}_j\}$ is an $\epsilon$-net for $\mathcal{L}$ with respect to the measure $Q$ by (\ref{AuxEntropyL}). Thus, for any $\epsilon>0$ we have that
$$ \log N( \epsilon, \mathcal{L}, \| \cdot \|_{Q,2}) \leq \sum_{j=1}^k \log N( \epsilon/k,\F_j, \| \cdot \|_{\widetilde Q_j,2}) $$
Therefore,
$$\begin{array}{rl}
 \log N( \epsilon \| \sum_{j=1}^k L_j F_j\|_{Q,2}, \mathcal{L}, \| \cdot \|_{Q,2}) & \leq \sum_{j=1}^k \log N( \mbox{$\frac{\epsilon}{k}$}\| \sum_{j=1}^k L_j F_j\|_{Q,2},\F_j, \| \cdot \|_{\widetilde Q_j,2})\\ &  \leq \sum_{j=1}^k \log N( \mbox{$\frac{\epsilon}{k}$}\| L_j F_j\|_{Q,2},\F_j, \| \cdot \|_{\widetilde Q_j,2})  \\
 &  = \sum_{j=1}^k \log N( \mbox{$\frac{\epsilon}{k}$}\| F_j\|_{\widetilde Q_j,2},\F_j, \| \cdot \|_{\widetilde Q_j,2})  \\
 &  \leq \sum_{j=1}^k \log \sup_{\bar Q} N( \mbox{$\frac{\epsilon}{k}$}\| F_j\|_{\bar Q,2},\F_j, \| \cdot \|_{\bar Q,2})  \\
 \end{array}$$
and the result follows since the right hand side no longer depends on $Q$.
\qed
\begin{lemma}[\textbf{Covering Entropy for Classes obtained as Conditional Expectations}]\label{Lemma:PartialOutCovering}
 Let $\mathcal F$ denote a class of measurable functions $f: \mathcal{W}\times \mathcal{Y} \mapsto \mathbb{R}$ with a measurable envelope $F$. For a given $f \in \mathcal{F}$, let $\bar f: \mathcal{W} \mapsto \mathbb{R}$ be the function $\bar f (w) := \int f(w,y) d\mu_{w}(y)$ where $\mu_{w}$ is a regular conditional probability distribution over $y \in \mathcal{Y}$ conditional on $w\in\mathcal{W}$. Set $\bar{\mathcal{F}} = \{ \bar f : f \in \mathcal{F}\}$ and let $\bar F(w):=\int F(w,y) d\mu_w(y)$ be an envelope for $\bar{\mathcal{F}}$. Then, for $r, s \geq 1$,
     $$
    \log \sup_{Q} N(\epsilon \| \bar F\Vert _{Q,r}, \bar{\mathcal{F}}, \| \cdot \|_{Q,r}) \leq \log \sup_{\widetilde Q} N((\epsilon/4)^r \| F\Vert _{\widetilde Q,s},  \mathcal \F , \| \cdot \|_{\widetilde Q,s}),
    $$ where $Q$ belongs to the set of finitely-discrete probability measures over $\mathcal{W}$ such that  $0<\| \bar F\Vert _{Q,r}< \infty$, and $\widetilde Q$ belongs to the set of finitely-discrete probability measures over $\mathcal{W}\times \mathcal{Y}$ such that $0<\|  F\Vert _{\widetilde Q,s}< \infty$. In particular, for every $\epsilon > 0$ and any $k\geq 1$
        $$
    \log \sup_{Q} N(\epsilon, \bar{\mathcal{F}}, \| \cdot \|_{Q,k}) \leq \log \sup_{\widetilde Q} N(\epsilon/2,  \mathcal \F , \| \cdot \|_{\widetilde Q,k} ).
    $$
\end{lemma}
\textbf{Proof.} The proof generalizes the proof of Lemma A.2 in \citen{GhosalSenVaart2000}. For $f, g \in \mathcal{F}$ and the corresponding $\bar f, \bar g \in \bar{\mathcal{F}}$, and any probability measure $Q$ on $\mathcal{W}$, by Jensen's inequality, for any $k\geq 1$,
$$ \Ep_Q[ | \bar f  -  \bar g|^k ] = \Ep_Q[ |\mbox{$\int$} (f - g)   d\mu_w(y)|^k ]\leq \Ep_Q[ \mbox{$\int$}|f - g|^k   d\mu_w(y) ]=\Ep_{\bar Q}[|f-g|^k]$$
where $d\bar Q(w,y) = dQ(w)d\mu_w(y)$. Therefore, for any $\epsilon>0$
$$ \sup_Q N(\epsilon, \bar{\mathcal{F}}, \| \cdot \|_{Q,k}) \leq \sup_{\bar Q} N(\epsilon, \mathcal{F}, \| \cdot \|_{\bar Q,k}) \leq \sup_{\widetilde Q} N(\epsilon/2, \mathcal{F}, \| \cdot \|_{\widetilde Q,k} ), $$
where we use Problems 2.5.1-2 of \citen{vdV-W}  to replace the supremum over $\bar Q$ with the supremum over finitely-discrete probability measures $\widetilde Q$.

Moreover,  $\|\bar F\|_{Q,1} = \Ep_Q[\bar F(w)] = \Ep_Q[\int F(w,y)d\mu_w(y)] = \Ep_{\bar Q}[F(w,y)] = \|F\|_{\bar Q,1}.$ Therefore taking $k=1$,
$$
\begin{array}{ll}
& \sup_Q N(\epsilon\|\bar F\|_{Q,1},\bar{\mathcal{F}}, \| \cdot \|_{Q,1} ) \leq \sup_{\bar Q}  N(\epsilon\| F\|_{\bar Q,1},{\mathcal{F}}, \| \cdot \|_{\bar Q,1} )\\
& \leq\sup_{\widetilde Q}  N((\epsilon/2)\| F\|_{\widetilde Q,1},{\mathcal{F}}, \| \cdot \|_{\widetilde Q,1} ) \leq \sup_{\widetilde Q}  N((\epsilon/2)\| F\|_{\widetilde Q,s},{\mathcal{F}}, \| \cdot \|_{\widetilde Q,s} )
\end{array}
$$
where we use Problems 2.5.1-2  of \citen{vdV-W}  to replace the supremum over $\bar Q$ with the supremum over finitely-discrete probability measures $\widetilde Q$,  and then Problem 2.10.4 of \citen{vdV-W} to argue that the last bound in weakly increasing in $s \geq 1$.

Also, by the second part of the proof of Theorem 2.6.7 of \citen{vdV-W}
$$  \sup_Q N(\epsilon\|F\|_{Q,r}, {\mathcal{F}}, \| \cdot \|_{Q,r} ) \leq \sup_Q N((\epsilon/2)^r\| F\|_{Q,1}, {\mathcal{F}}, \| \cdot \|_{Q,1}).$$
\qed

\begin{remark}
Lemma \ref{Lemma:PartialOutCovering} extends the result in Lemma A.2 in \citen{GhosalSenVaart2000} and Lemma 5 in \citen{Sherman1994} which considered integral classes with respect to a fixed measure $\mu$ on $\mathcal{Y}$. In our applications we need to allow the integration measure to vary with $w$, namely we allow for $\mu_w$ to be a conditional distribution. \qed
\end{remark}

\section{Proofs for Section 4}\label{sec:m}
\subsection{Proof of Theorem \ref{theorem1}}


\textsc{Step 0.} (Preparation).    In the proof $a \lesssim b$ means that $a \leq A b$, where the constant
$A$ depends on the constants  in Assumptions \ref{assumption: basic} and \ref{ass: sparse1} only, but not on $n$ once $n \geq n_0=\min\{j: \delta_j \leq 1/2\}$, and not on $P \in \mathcal{P}_n$. We consider a sequence $P_n$ in $\mathcal{P}_n$, but for simplicity, we write  $P =P_n$ throughout the proof, \textit{suppressing} the index $n$.  Since the argument is asymptotic, we can assume that $n \geq n_0$ in what follows.

To proceed with the presentation of the proofs, it might be convenient for the reader
 to have the notation collected in one place. The influence function and low-bias moment functions for $\alpha_V(z)$ for  $z \in \mathcal{Z} = \{0,1\}$  are given
respectively by
\begin{eqnarray*}
&&\psi^{\alpha}_{V,z} (W) = \psi^{\alpha}_{V, z, g_{V},m_Z} (W, \alpha_V(z)),  \ \ \psi^{\alpha}_{V, z, g, m}(W,\alpha) = \frac{1(Z=z) (V - g(z,X))}{m(z,X)}  + g(z,X) - \alpha.
 \end{eqnarray*}
 The influence function and the moment function for $\gamma_V$ are
$\psi^{\gamma}_{V} (W) = \psi^{\gamma}_{V} (W,\gamma_V)$ and $\psi^{\gamma}_{V} (W,\gamma) =V - \gamma.$
 Recall that the estimator of the reduced-form parameters $\alpha_V(z)$  and $\gamma_V$ are   solutions $\alpha = \hat \alpha_V(z)$ and $\gamma= \hat \gamma_V$
to the equations
\begin{eqnarray*}
\En[\psi^{\alpha}_{V, z, \hat g_{V}, \hat m_Z} (W, \alpha) ] = 0, \ \ \En[\psi^{\gamma}_{V} (W, \gamma) ] = 0,
\end{eqnarray*}
where $\hat g_V(z,x) = \Lambda_V(f(z,x)'\bar \beta_V)$, $\hat m_Z(1,x)= \Lambda_Z(f(x)'\bar \beta_Z)$, $\hat m_Z(0,x) = 1 - \hat m_Z(1,x)$, and $\bar \beta_V$ and $\bar \beta_Z$ are estimators as in Assumption \ref{ass: sparse1}.
For each  variable $V \in \mV_u$,
$$
\mV_u = (V_{uj})_{j=1}^5 = (Y_u,  \mathbf{1}_0(D)Y_u,  \mathbf{1}_0(D),  \mathbf{1}_1(D)Y_u,  \mathbf{1}_1(D)),
$$
 we obtain
the estimator
$\hat \rho_u = \big (\{\hat \alpha_{V} (0), \hat \alpha_{V} (1),  \hat \gamma_V\} \big )_{V \in \mV_u}$
$\text { of }$
$\rho_u := \big (\{ \alpha_{V} (0),  \alpha_{V} (1),   \gamma_V\} \big )_{V \in \mV_u}.$
The estimator and the estimand are vectors in $\mathbb{R}^{d_\rho}$ with a fixed finite dimension.
We stack these vectors into the processes $\hat \rho = (\hat \rho_u)_{u \in \mU}$ and $\rho = (\rho_u)_{u \in \mU}$.

\textsc{Step 1.}(Linearization)  In this step we establish the first claim, namely that
\begin{equation}
\sqrt{n}(\hat \rho- \rho)  = Z_{n,P} + o_P(1) \ \  \text{ in $\D = \ell^\infty(\mU)^{d_\rho}$},
\end{equation}
where $Z_{n, P} =   (\mathbb{G}_{n}   \psi^\rho_{u})_{u \in \mathcal{U}}$ and $\psi^\rho_{u}= (\{\psi^\alpha_{V,0}, \psi^\alpha_{V,1}, \psi^\gamma_{V}\})_{ V \in \mV_u}$.
The components $(\sqrt{n}(\hat \gamma_{V_{uj}} - \gamma_{V_{uj}}))_{u \in \mU}$ of
$\sqrt{n}(\hat \rho- \rho)$ trivially have the linear representation (with no error) for each
$j \in \mathcal{J}$. We only need to establish the claim for the empirical process
 $(\sqrt{n}(\hat \alpha_{V_{uj}}(z) - \alpha_{V_{uj}}(z)))_{u \in \mU}$ for $z \in \{0,1\}$ and each $j \in \mathcal{J}$,
 which we do in the steps below.

(a) We make some preliminary observations.
For $t= (t_1, t_2,t_3, t_4) \in \mathbb{R}^2\times (0,1)^2$, $v \in \mathbb{R}$, and $(z,\bar z) \in \{0,1\}^2$, we define the function
$(v,z,\bar z,t) \mapsto \varphi (v,z,\bar z,t)$ via:
 $$\varphi(v,z, 1,  t) =  \frac{1(z=1) (v -t_2) } {t_4}  +    t_2, \ \ \varphi(v,z, 0,  t) =  \frac{1(z=0) (v -t_1) } {t_3}  +    t_1.$$
The derivatives of this function with respect to $t$ obey for all $k =(k_j)_{j=1}^4 \in\mathbb{N}^4:  0 \leq |k| \leq 3$,
\begin{equation}\label{Lip}
|\partial^k_t \varphi (v, z,\bar z, t) | \leq L, \ \  \forall (v,\bar z,z,t):  |v| \leq C, |t_1|, |t_2|  \leq C,   c'/2 \leq |t_3| , |t_4| \leq 1-c'/2,
\end{equation}
where $L$ depends only on $c'$ and $C$,  $|k| = \sum_{j=1}^4 k_j,$ and $\partial^k_t := \partial^{k_1}_{t_1}\partial^{k_2}_{t_2}\partial^{k_3}_{t_3}\partial^{k_4}_{t_4}.$

 (b)  Let
 \begin{eqnarray*}
& & \hat h_V(X) := ( \hat g_V(0,X), \hat g_V(1,X),  1- \hat m_Z(1, X), \hat m_Z(1, X))', \\
& & h_V(X) :=
(  g_V(0,X),  g_V(1,X), 1 - m_Z(1, X), m_Z(1, X))',\\
&& f_{\hat h_V,V,z} (W) := \varphi(V,Z, z,  \hat h_V(X)) , \\
&&  f_{h_V,V,z}  (W) := \varphi(V,Z, z,  h_V(X)).
\end{eqnarray*}
We observe that with probability no less than $1- \Delta_n$,  $$\hat g_V(0,\cdot)  \in \mathcal{G}_V(0), \ \ \hat g_V(1,\cdot)  \in \mathcal{G}_V(1), \ \ \hat m_Z(1, \cdot)  \in \mathcal{M}(1),  \ \ \hat m_Z(0, \cdot)  \in \mathcal{M}(0) = 1- \mathcal{M}(1),$$
where
\begin{eqnarray*}
\mathcal{G}_V(z):=\left \{ \begin{array}{ll}
& x \mapsto \Lambda_V (f(z,x)'\beta) :  \|\beta\|_0 \leq s C  \\
& \|  \Lambda_V (f(z,X)'\beta)  - g_V(z,X)\|_{P,2} \lesssim \delta_n n^{-1/4}  \\
&  \| \Lambda_V (f(z,X)'\beta)  - g_V(z,X)\|_{P, \infty}
 \lesssim  \epsilon_n\end{array} \right \},
   \end{eqnarray*}
   \begin{eqnarray*}
\mathcal{M}(1):=\left \{ \begin{array}{ll}
& x \mapsto \Lambda_Z (f(x)'\beta) :  \|\beta\|_0 \leq s C  \\
& \|  \Lambda_Z (f(X)'\beta)  - m_Z(1,X)\|_{P,2} \lesssim \delta_n n^{-1/4}  \\
&  \| \Lambda_Z (f(X)'\beta)  - m_Z(1,X)\|_{P, \infty}
 \lesssim  \epsilon_n\end{array} \right \}.
   \end{eqnarray*}

To see this, note that under Assumption \ref{ass: sparse1} for all $n \geq \min\{j: \delta_j \leq 1/2\}$,
 \begin{eqnarray*}
 \| \Lambda_Z(f(X) '\beta) - m_Z(1, X) \|_{P,2} && \leq \| \Lambda_Z(f(X) '\beta)  - \Lambda_Z(f(X) '\beta_Z)  \|_{P,2}+ \| r_{Z}(X) \|_{P,2}  \\ && \lesssim \|\partial \Lambda_Z\|_{\infty} \| f(X)'(\beta - \beta_{Z})\|_{P,2}+ \| r_{Z} (X)\|_{P,2} \\
  && \lesssim \|\partial \Lambda_Z\|_{\infty} \| f(X)'(\beta - \beta_{Z})\|_{\Pn,2}+ \| r_{Z}(X) \|_{P,2}  \lesssim \delta_n n^{-1/4} \\
  \| \Lambda_Z(f(X) '\beta) - m_Z(1, X) \|_{P,\infty} &&  \leq \| \Lambda_Z(f(X) '\beta)  - \Lambda_Z(f(X) '\beta_Z)  \|_{P,\infty}+ \| r_{Z}(X) \|_{P,\infty}\\
&&   \leq \|\partial \Lambda_Z\|_{\infty}  \| f(X)'(\beta - \beta_{Z})\|_{P,\infty}+ \| r_{Z}(X) \|_{P,\infty} \\
 &&  \lesssim   K_n \|\beta- \beta_{Z}\|_{1} + \epsilon_n \leq 2 \epsilon_n,
\end{eqnarray*}
for $\beta=\bar \beta_{Z}$, with evaluation after computing the norms, and for  $\|\partial \Lambda \|_{\infty}$ denoting $\sup_{l \in \mathbb{R}}|\partial \Lambda(l)|$ here and below.  Similarly, under Assumption \ref{ass: sparse1},
 \begin{eqnarray*}
 \| \Lambda_V(f(Z,X) '\beta) - g_V(Z, X) \|_{P,2} && \lesssim \|\partial \Lambda_V\|_{\infty} \| f(Z,X)'(\beta - \beta_{V})\|_{\Pn,2}+ \| r_{V}(Z,X) \|_{P,2}  \lesssim \delta_n n^{-1/4} \\
 \| \Lambda_V(f(Z,X) '\beta) - g_V(Z, X) \|_{P,\infty}
 &&  \lesssim   K_n \|\beta- \beta_{V}\|_{1} + \epsilon_n \leq 2 \epsilon_n,
\end{eqnarray*}
for $\beta=\bar \beta_V$, with evaluation after computing the norms, and noting that for any $\beta$
{\small $$  \| \Lambda_V(f(0,X) '\beta) - g_V(0, X) \|_{P,2} \vee
   \| \Lambda_V(f(1,X) '\beta) - g_V(1, X) \|_{P,2} \lesssim
    \| \Lambda_V(f(Z,X) '\beta) - g_V(Z, X) \|_{P,2} $$}
 under condition (iii) of Assumption \ref{assumption: basic}, and
{\small  $$  \| \Lambda_V(f(0,X) '\beta) - g_V(0, X) \|_{P,\infty} \vee
   \| \Lambda_V(f(1,X) '\beta) - g_V(1, X) \|_{P,\infty} \leq
    \| \Lambda_V(f(Z,X) '\beta) - g_V(Z, X) \|_{P,\infty} $$}
  under condition (iii) of Assumption \ref{assumption: basic}.

Hence with probability at least $1- \Delta_n$,
   $$\hat h_V \in \mathcal{H}_{V,n}:=\{ h =  (\bar g(0,\cdot), \bar g(1,\cdot), \bar m_Z(0,\cdot), \bar m_Z(1,\cdot) ) \in \mathcal{G}_{V}(0) \times \mathcal{G}_V(1) \times \mathcal{M}(0)\times \mathcal{M}(1) \}.$$

(c) We have that $$\alpha_V(z) = \Ep_P[f_{h_V, V, z}] \text{
and } \hat \alpha(z) = \En[f_{\hat h_V, V, z}],$$ so that
\begin{eqnarray*}
\sqrt{n} (\hat \alpha_V(z) - \alpha_V(z)) =  \underbrace{\Gn[f_{h_V, V,z}]}_{I_{V}(z)} + \underbrace{ \Gn[f_{h, V, z}- f_{h_V, V, z}]}_{II_{V}(z)} +  \underbrace{ \sqrt{n} \ P[f_{h, V,z} - f_{h_V,V,z}]}_{III_{V}(z)},
\end{eqnarray*}
with $h$ evaluated at $h = \hat h_V$.

(d) Note that for $$\Delta_{V,i} := (\Delta_{1V,i}, \Delta_{2V,i}, \Delta_{3V,i}, \Delta_{4V,i}) = h( X_i) - h_V(X_i),  \ \ \Delta_{V,i}^k := \Delta_{1V,i}^{k_1} \Delta_{2V,i}^{k_2} \Delta_{3V,i}^{k_3} \Delta_{4V,i}^{k_4},$$
\begin{eqnarray*}
III_V(z) & = & \sqrt{n} \sum_{|k|=1} P[\partial^k_t \varphi(V_i,Z_i, z, h_V(X_i)) \Delta_{V,i} ^k]  \\
 & + & \sqrt{n} \sum_{|k|=2} 2^{-1}  P[\partial^k_t \varphi(V_i,Z_i, z, h_V(X_i)) \Delta_{V,i} ^k ] \\
 & + & \sqrt{n} \sum_{|k|=3} 6^{-1}  \int_0^1 P[\partial^k_t \varphi(V_i,Z_i, z, h_V(X_i) + \lambda\Delta_{V,i}) \Delta_{V,i} ^k ] d\lambda, \\
 & =: & III^a_V(z) + III^b_V(z) + III^c_V(z),
\end{eqnarray*}
with $h$ evaluated at $h = \hat h$ after computing the expectations under $P$.

By the law of iterated expectations and the orthogonality property of the moment condition for $\alpha_V$,
$$\Ep_P[\partial^k_t \varphi(V_i,Z_i, z, h_V(X_i))| X_i]= 0 \ \  \forall k \in \mathbb{N}^4 : |k| =1,   \implies III^a_V(z) =0.$$

Moreover, uniformly for any
$h \in \mathcal{H}_{V,n}$, in view of properties noted in Steps (a) and (b),
$$|III^b_V(z)| \lesssim \sqrt{n} \| h - h_V \|^2_{P,2} \lesssim \sqrt{n} (\delta_n n^{-1/4})^2 \leq \delta_n^2,$$
$$|III^c_V(z)|  \lesssim \sqrt{n} \|h - h_V\|^2_{P,2}\| h-h_V\|_{P,\infty} \lesssim \sqrt{n} (\delta_n n^{-1/4})^2 \epsilon_n \leq \delta_n^2 \epsilon_n.$$  Since $\hat h_V \in \mathcal{H}_{V,n}$ for all $V \in \mV=\{ V_{uj}: u \in \mathcal{U} , j \in \mathcal{J}\}$ with probability $1- \Delta_n$,
for $n \geq n_0$,
$$
\Pr_P\Big ( |III_V(z)| \lesssim \delta^2_n, \forall z \in \{0,1\}, \forall V \in \mV \Big ) \geq 1- \Delta_n.
$$

(e) Furthermore, with probability $1 - \Delta_n$
$$
\sup_{V \in \mV} \max_{z \in \{0,1\}}|II_V(z)| \leq \sup_{h \in \mathcal{H}_{V,n}, z \in \{0,1\},  V \in \mV } |\Gn[f_{h,V,z}]- \Gn[f_{h_V,V,z}]|.
 $$

 The classes of functions,
 \begin{equation}\label{define: mV} \mathcal{V}:=\{ V_{uj}: u \in \mathcal{U}, j \in \mathcal{J}\} \ \ \text{ and}  \  \  \mathcal{V}^*: = \{ g_{V_{uj}} (Z,X): u \in \mathcal{U}, j \in \mathcal{J}  \},
  \end{equation}
  viewed as maps from the sample space $\mathcal{W}$
 to the real line, are bounded by a constant envelope and obey $\log \sup_Q N( \epsilon, \mathcal{V}, \|\cdot\|_{Q,2}) \lesssim  \log(\mathrm{e}/\epsilon)\vee 0$,
which holds by Assumption \ref{assumption: basic}(ii),  and $\log \sup_Q N( \epsilon, \mathcal{V}^*, \|\cdot\|_{Q,2}) \lesssim  \log(\mathrm{e}/\epsilon)\vee 0$
which holds by Assumption \ref{assumption: basic}(ii) and  Lemma \ref{Lemma:PartialOutCovering}.  The uniform covering entropy of the function sets $$\mathcal{B} = \{1(Z=z): z \in \{0,1\}\} \text{ and  } \mathcal{M}^* = \{m_Z(z,X) : z \in \{0,1\}\}$$ are trivially bounded by $\log(\mathrm{e}/\epsilon)\vee 0$.

 The class of functions $$\mathcal{G}:=\{\mathcal{G}_V(z): V \in \mV, z \in \{0,1\}\}$$  has a constant
 envelope and is a subset of
 $$
 \{( x,z) \mapsto \Lambda (f(z,x)'\beta) :  \|\beta\|_0 \leq s C,  \Lambda \in \mathcal{L}=\{ \Id,  \Phi, 1- \Phi, \Lambda_0, 1- \Lambda_0 \}\},
 $$
which is  a union of 5 sets of the form
$$
 \{( x,z) \mapsto \Lambda (f(z,x)'\beta) :  \|\beta\|_0 \leq s C\}
$$
with $\Lambda \in \mathcal{L}$ a fixed monotone function for each of the 5 sets; each of these sets are the unions of at most $\binom{2p}{Cs}$ VC-subgraph classes of functions with VC indices bounded by $C's$. Note that a fixed monotone transformations $\Lambda$ preserves the VC-subgraph property \cite[Lemma 2.6.18]{vdV-W}.    Therefore
 $$
 \log \sup_Q N( \epsilon, \mathcal{G}, \|\cdot\|_{Q,2}) \lesssim  (s \log p + s \log (\mathrm{e}/\epsilon)) \vee 0.
 $$

 Similarly, the class of functions $
\mathcal{M} = (\mathcal{M}(1) \cup (1 - \mathcal{M}(1)))$ has a constant envelope,
 is  a union of at most 5 sets, which are themselves the unions of at most $\binom{p}{Cs}$ VC-subgraph classes of functions with VC indices bounded by $C's$ since a fixed monotone transformations $\Lambda$ preserves the VC-subgraph property.  Therefore,
$
\log \sup_Q N( \epsilon, \mathcal{M}, \|\cdot\|_{Q,2})  \lesssim  (s \log p + s \log (\mathrm{e}/\epsilon)) \vee 0.
$

Finally, the set of functions
$$
\mathcal{J}_{n} = \{f_{h,V,z} - f_{h_V, V, z}:  z \in \{0,1\},  V \in \mV,  h \in \mathcal{H}_{V,n}\},
$$
 is a Lipschitz transform of function sets $\mV$,  $\mV^*$, $\mathcal{B}$, $\mathcal{M}^*$,  $\mathcal{G}$, and $\mathcal{M}$,
 with bounded Lipschitz coefficients and with a constant envelope. Therefore,
\begin{equation*}
\log \sup_Q   N( \epsilon, \mathcal{J}_{n}, \|\cdot\|_{Q,2}) \lesssim  (s \log p + s \log (\mathrm{e}/\epsilon))\vee 0.
\end{equation*}

Applying Lemma \ref{lemma:CCK} with $ \sigma_n = C' \delta_n n^{-1/4} $ and the envelope $J_n = C'$,   with probability $1 - \Delta_n$ for some constant $K>e$
\begin{eqnarray*}
&& \sup_{V \in \mV} \max_{z \in \{0,1\}}|II_V(z)| \leq  \sup_{f \in \mathcal{J}_n} | \Gn (f) | \\
&& \lesssim \( \sqrt{s \sigma^2_n \log (p \vee K \vee \sigma_n^{-1})}   +   \frac{s}{\sqrt{n}} \log (p \vee K \vee \sigma_n^{-1}) \) \\
&&  \lesssim \(  \sqrt{s \delta_n^2 n^{-1/2} \log (p \vee n)}  +  \sqrt{s^2 n^{-1} \log^2 (p \vee n)}\) \\
&&  \lesssim \( \delta_n \delta_n^{1/4} + \delta_n^{1/2}\) \lesssim \delta_n^{1/2}.
\end{eqnarray*}
Here we have used some simple calculations, exploiting the boundedness condition in Assumptions \ref{assumption: basic} and \ref{ass: sparse1}, to deduce that
\begin{eqnarray*}
&&  \sup_{f \in \mathcal{J}_n}\|  f\|_{P,2} \lesssim   \sup_{h \in \mathcal{H}_{V,n}, V \in \mV}  \|  h - h_V\|_{P,2} \lesssim \delta_n n^{-1/4} \lesssim \sigma_n \leq \|J_n\|_{P,2},
\end{eqnarray*}
by definition of the set $\mathcal{H}_{V,n}$, so that  we can use Lemma \ref{lemma:CCK}. We also note that $\log (1/\delta_n) \lesssim \log (n)$ by the assumption on $\delta_n$ and that $s^2 \log^2 (p \vee n)\log^2 (n)/n \leq \delta_n$ by Assumption \ref{ass: sparse1}(i).


(f) The claim of Step 1 follows by collecting Steps (a)-(e).

\textsc{Step 2} (Uniform Donskerness).   Here we claim that Assumption \ref{assumption: basic} implies that the set of vectors of  functions $ (\psi^\rho_u)_{ u \in \mU}$ is $P$-Donsker uniformly in $\mP$,
namely that
\begin{equation*}
Z_{n,P} \rightsquigarrow  Z_{P}  \ \  \text{ in $\D =  \ell^\infty(\mU)^{d_\rho}$,  uniformly in $P \in  \mP$},
\end{equation*}
where $Z_{n, P} =   (\mathbb{G}_{n}   \psi^\rho_{u})_{u \in \mathcal{U}}$
 and $Z_{P} =   (\mathbb{G}_{P}   \psi^\rho_{u})_{u \in \mathcal{U}}$. Moreover,   $Z_P$  has bounded, uniformly continuous paths uniformly in $P \in   \mP$:
\begin{equation*}
\sup_{P \in  \mP} \Ep_P\sup_{u \in \mU} \|Z_P(u)\| < \infty, \ \ \lim_{\varepsilon \searrow 0} \sup_{P \in \mP}
\Ep_P \sup_{ d_{\mathcal{U}}(u, \tilde u) \leq \varepsilon}  \|Z_P(u) - Z_P(\tilde u)\| = 0.
\end{equation*}
To verify these claims we shall invoke Theorem \ref{lemma: uniform Donsker}.

To demonstrate the claim, it will suffice to consider the set of $\mathbb{R}$-valued functions $\Psi = (\psi_{uk}: u \in \mU, k \in [d_\rho])$. Further, we notice that
$\mathbb{G}_n \psi^{\alpha}_{V,z} = \mathbb{G}_n f$,  for $f \in \mathcal{F}_{z}$,
$$
\mathcal{F}_{z} = \left \{  \frac{1\{Z=z\} (V- g_V(z,X))}{m_Z(z,X)} +g_V(z,X),   V \in \mathcal{V} \right \}, \  z=0, 1,
$$
and that $\mathbb{G}_n  \psi^{\gamma}_{V} = \mathbb{G}_n f$,  for $f =V \in \mathcal{V}$. Hence
$\mathbb{G}_n( \psi_{uk}) = \mathbb{G}_n(f)$
for $f \in  \mathcal{F}_P = \mathcal{F}_{0} \cup  \mathcal{F}_{1}  \cup \mathcal{V}$.  We thus
need to check that the conditions of Theorem \ref{lemma: uniform Donsker} apply to
$\mathcal{F}_P$ uniformly in $P \in \mP$.

Observe that
$\mathcal{F}_z$ is formed as a uniform Lipschitz transform of the function sets $ \mathcal{B}$, $ \mV,$ $\mV^*$ and $\mathcal{M}^*$ defined in Step 1(e), where the validity of the Lipschitz property relies on Assumption \ref{assumption: basic}(iii) (to keep
the denominator away from zero) and on the boundedness conditions in Assumption \ref{assumption: basic}(iii) and Assumption \ref{ass: sparse1}(iii).  The  function sets $ \mathcal{B}$, $ \mV,$ $\mV^*$ and $\mathcal{M}^*$ are uniformly bounded classes that have uniform covering entropy bounded by $ \log(\mathrm{e}/\epsilon)\vee 0$ up to a multiplicative constant, and so $\mathcal{F}_z$, which is uniformly bounded under Assumption \ref{assumption: basic}, the uniform covering entropy bounded by $\log(\mathrm{e}/\epsilon)\vee 0$ up to a multiplicative constant (e.g. \citen{vdV-W}).  Since $\mathcal{F}_P$ is uniformly bounded and is a finite union of function sets with the uniform entropies obeying the said properties, it also follows that $\mathcal{F}_P$  has this property; namely,
$$
\sup_{P \in \mP}  \sup_Q \log N( \epsilon, \mathcal{F}_P, \|\cdot\|_{Q,2}) \lesssim \log(\mathrm{e}/\epsilon) \vee 0.
$$
Since $\int_0^\infty \sqrt{ \log(\mathrm{e}/\epsilon) \vee 0} d\epsilon = \mathrm{e} \sqrt{\pi}/2 < \infty$ and $\mathcal{F}_P$ is uniformly bounded, the first condition in (\ref{eq: characteristics1}) and the entropy  condition (\ref{eq: characteristics2}) in  Theorem \ref{lemma: uniform Donsker} hold.

We demonstrate the second condition in (\ref{eq: characteristics1}).  Consider a sequence of positive constants $\epsilon$ approaching zero, and note that
$$
 \sup_{ d_{\mathcal{U}}(u, \tilde u) \leq \epsilon}  \max_{k \leq d_\rho} \|\psi_{uk} - \psi_{\tilde uk}\|_{P,2}
 \lesssim  \sup_{ d_{\mathcal{U}}(u, \tilde u) \leq \epsilon}  \| f_u - f_{\tilde u}\|_{P,2}
$$
where $f_u$ and $f_{\tilde u}$ must be of the form:
$$
  \frac{1\{Z=z\} (U_u- g_{U_u}(z,X))}{m_Z(z,X)} +g_{U_u}(z,X),    \frac{1\{Z=z\} (U_{\tilde u}- g_{U_{\tilde u}}(z,X))}{m_Z(z,X)} +g_{U_{\tilde u}}(z,X),
$$
with $(U_u, U_{\tilde u})$ equal to  either  $(Y_u, Y_{\tilde u})$ or $(1_d(D) Y_u,  1_d(D) Y_{\tilde u})$, for  $d=0$ or $1$,
 and $z=0$ or $1$. Then
$$
\sup_{P \in \mathcal{P}}\|f_u - f_{\tilde u}\|_{P,2} \lesssim \sup_{P \in \mathcal{P}} \| Y_u - Y_{\tilde u}\|_{P,2} \to 0,
$$
as $d_{\mU}(u, \tilde u) \to 0$ by Assumption \ref{assumption: basic}(ii).  Indeed, $\sup_{P \in \mathcal{P}}\|f_u - f_{\tilde u}\|_{P,2} \lesssim \sup_{P \in \mathcal{P}} \| Y_u - Y_{\tilde u}\|_{P,2}$ follows from a sequence
of inequalities holding uniformly in $P \in \mP$: (1) $$\|f_u - f_{\tilde u}\|_{P,2} \lesssim \| U_u - U_{\tilde u}\|_{P,2}+  \| g_{U_u}(z,X) -
 g_{U_{\tilde u}}(z,X) \|_{P,2},$$ which we deduce using the triangle inequality and the fact that $m_Z(z,X)$ is bounded away
 from zero, (2)  $ \| U_u- U_{\tilde u}\|_{P,2}\leq  \|Y_u - Y_{\tilde u}\|_{P,2} $, which we deduced using the
 Holder inequality, and (3) $$ \| g_{U_u}(z,X) - g_{U_{\tilde u}}(z,X)\|_{P,2} \leq  \|U_u - U_{\tilde u}\|_{P,2} ,$$ which we deduce by the definition  of $g_{U_u}(z,X) = \Ep_P[U_u|X,Z=z]$ and
the contraction property of the conditional expectation. \qed

\subsection{Proof of Theorem \ref{theorem2}}

The proof will be similar to the  proof of Theorem \ref{theorem1}.

\textsc{Step 0.} (Preparation).    In the proof $a \lesssim b$ means that $a \leq A b$, where the constant
$A$ depends on the constants  in Assumptions \ref{assumption: basic} and \ref{ass: sparse1} only, but not on $n$ once $n \geq n_0=\min\{j: \delta_j \leq 1/2\}$, and not on $P \in \mathcal{P}_n$. We consider a sequence $P_n$ in $\mathcal{P}_n$, but for simplicity, we write  $P =P_n$ throughout the proof,  suppressing the index $n$.  Since the argument is asymptotic, we can assume that $n \geq n_0$ in what follows.
Let $\Pn$ denote the measure that puts mass $n^{-1}$ on points $(\xi_i, W_i)$ for $i=1,...,n$.
Let $\En$ denote the expectation with respect to this measure, so that
$\En f = n^{-1} \sum_{i=1}^n f(\xi_i, W_i)$, and $\Gn$ denote the corresponding empirical process $\sqrt{n} ( \En - P)$,
i.e.
$$\Gn f = \sqrt{n}(\En f - P f) = n^{-1/2} \sum_{i=1}^n \left ( f(\xi_i, W_i) - \int f(s, w) d P_\xi (s) dP (w) \right ).
$$

Recall that we define the bootstrap draw as:
$$
Z^*_{n,P} =\sqrt{n}(\hat \rho^*- \hat \rho)  =  \(\frac{1}{\sqrt{n}} \sum_{i=1}^n \xi_i \hat \psi^\rho_{u}(W_i) \)_{u \in \mathcal{U}}= \left( \Gn \xi \hat \psi^\rho_u \right)_{u \in \mathcal{U}},
$$
since $P[\xi \hat \psi_u^{\rho}] = 0$ because $\xi$ is independent of $W$ and has zero mean. Here $\hat \psi^\rho_{u} =    (\hat \psi^\rho_{V})_{ V \in \mV_u},$ where $
 \hat \psi^{\rho}_{V} (W) =\{\psi^{\alpha}_{V, 0, \hat g_{V}, \hat m_{Z}}(W,\hat \alpha_{V}(0)), \psi^{\alpha}_{V, 1, \hat g_{V}, \hat m_{Z}}(W,\hat \alpha_{V}(1)), \psi^{\gamma}_{V}(W, \hat \gamma_{V})\},$  is a plug-in estimator of the influence function $\psi^{\rho}_u$.

\textsc{Step 1.}(Linearization)  In this step we establish  that
\begin{equation}
\zeta_{n,P}^*:= Z^*_{n,P} - G^*_{n,P} = o_P(1),    \ \   \text{ for } G^*_{n, P} :=   (\mathbb{G}_{n}  \xi \psi^\rho_{u})_{u \in \mathcal{U}}, \ \  \ \  \text{ in $\D= \ell^\infty(\mU)^{d_\rho}$},
\end{equation}
where $\zeta_{n,P}^* = \zeta_{n,P}(D_n, B_n)$ is a linearization error, arising completely due to
estimation of the influence function; if the influence function were known, this term would
be zero.

For the components $(\sqrt{n}(\hat \gamma^*_{V} - \hat \gamma_{V}))_{V \in \mV}$ of
$\sqrt{n}(\hat \rho^*- \hat \rho)$  the linearization follows by  the
 representation,
$$
\sqrt{n}(\hat \gamma^*_{V} - \hat \gamma_{V}) =  \Gn \xi \psi_V^{\gamma} -     \underbrace{(\hat \gamma_{V} - \gamma_{V}  )\Gn\xi}_{I^*_{V}},
$$
for all  $V \in \mV$, and noting that $\sup_{V \in \mV} |I^*_V| = \sup_{V \in \mV}  | (\hat \gamma_{V} - \gamma_{V}  )| | \Gn\xi| = O_P( n^{-1/2} )$,  for $\mV$ defined in (\ref{define: mV}) by Theorem \ref{theorem1} and by  $| \Gn \xi| = O_P(1)$.

It remains to establish the claim for the empirical process
 $(\sqrt{n}(\hat \alpha^*_{V_{uj}}(z) - \hat \alpha_{V_{uj}}(z)))_{u \in \mU}$ for $z \in \{0,1\}$ and $j \in \mathcal{J}$. As in the proof
 of Theorem 4.1, we have that with probability at least $1- \Delta_n$,
   $$\hat h_V \in \mathcal{H}_{V,n}:=\{ h =  (\bar g_V(0,\cdot), \bar g_V(1,\cdot), \bar m_Z(0,\cdot), \bar m_Z(1,\cdot) ) \in \mathcal{G}_{V}(0) \times \mathcal{G}_V(1) \times \mathcal{M}(0)\times \mathcal{M}(1) \}.$$
We have the representation:
\begin{eqnarray*}
\sqrt{n} (\hat \alpha^*_V(z) - \hat \alpha_V(z)) = \Gn \xi \psi_{V,z}^{\alpha} + \underbrace{\Gn[\xi f_{\hat h_V, V, z}- \xi f_{h_V, V, z}] - (\hat \alpha_V(z) - \alpha_V(z)) \Gn \xi}_{II^*_V(z)},
\end{eqnarray*}
where $\sup_{V \in \mV, z \in \{0,1\}} (\hat \alpha_V(z) - \alpha_V(z)) = O_P(n^{-1/2})$ by Theorem \ref{theorem1}.

Hence to establish $\sup_{V \in \mV}|II^*_V(z)| = o_P(1)$,  it remains to show that with probability $1- \Delta_n$
$$
\sup_{ z \in \{0,1\},  V \in \mV } |\Gn[\xi f_{\hat h_V, V, z} - \xi f_{h_V, V, z}]|  \leq  \sup_{f \in \xi \mathcal{J}_n} | \Gn (f) | = o_P(1),
 $$
 where  $$
\mathcal{J}_{n} = \{f_{h,V,z} - f_{h_V, V, z}:  z \in \{0,1\},  V \in \mV,  h \in \mathcal{H}_{V,n}\}.
$$
By the calculations in Step 1(e) of the proof of Theorem \ref{theorem1}, $\mathcal{J}_{n}$ obeys $
\log \sup_Q   N( \epsilon, \mathcal{J}_{n}, \|\cdot\|_{Q,2}) \lesssim  (s \log p + s \log (e/\epsilon))\vee 0.
$ By Lemma \ref{lemma: andrews}, multiplication of this class by $\xi$ does not change the entropy bound modulo
 an absolute constant, namely
$$
\log \sup_Q   N( \epsilon \|J_n\|_{Q,2}, \xi \mathcal{J}_{n}, \|\cdot\|_{Q,2}) \lesssim  (s \log p + s \log (e /\epsilon))\vee 0,
$$
where the envelope $J_n$ for $\xi \mathcal{J}_{n}$  is $|\xi|$ times a constant.  Also, $\Ep[\exp(|\xi|)] < \infty$ implies that
$
(\Ep [\max_{ i \leq n} |\xi_i|^2])^{1/2} \lesssim \log n.$ Thus,  applying Lemma \ref{lemma:CCK} with  $\sigma = \sigma_n = C' \delta_n n^{-1/4}$ and the envelope $J_n = C' |\xi|$,   for some constant $K>e$
\begin{eqnarray*}
 \sup_{f \in \xi \mathcal{J}_n} | \Gn (f) |  && \lesssim \( \sqrt{s \sigma^2_n \log (p \vee K \vee \sigma_n^{-1})}   +   \frac{s \log n }{\sqrt{n}}  \log (p \vee K \vee \sigma_n^{-1}) \) \\
&& \lesssim \(  \sqrt{s \delta_n^2 n^{-1/2} \log (p \vee n)}  +  \sqrt{s^2 n^{-1} \log^2 (p \vee n) \log^2(n)}\) \\
&& \lesssim  \( \delta_n \delta_n^{1/4} + \delta_n^{1/2}\)  \lesssim (\delta_n^{1/2}) = o_P(1),
\end{eqnarray*}
for $\sup_{f \in \xi \mathcal{J}_n}\|  f\|_{P,2}= \sup_{f \in  \mathcal{J}_n}\| f\|_{P,2} \lesssim  \sigma_n$;  where the details of calculations are
the same as in Step 1(e) of the  proof of Theorem \ref{theorem1}.

Finally, we conclude that
$$
\| \zeta_{n,P}^* \|_\D \lesssim \sup_{V \in \mV} |I^*_V| + \sup_{V \in \mV, z \in \{0,1\}} |II^*_V|  = o_P(1).
$$

\textsc{Step 2}. Here we are claiming that
$Z^*_{n,P} \rightsquigarrow_B  Z_{P}$  in $\D$,  under any sequence $P =P_n \in \mP_n$, where $Z_{P} =   (\mathbb{G}_{P}  \psi^\rho_{u})_{u \in \mathcal{U}}$. We have that
\begin{eqnarray*}
&& \sup_{h \in \mathrm{BL}_1(\D) } \Big |   \Ep_{B_n}  h ( Z^*_{n,P} )  - \Ep_P h ( Z_{P})  \Big |\leq \sup_{h \in \mathrm{BL}_1(\D) } \Big |   \Ep_{B_n}  h (G^*_{n,P} )  - \Ep_P h ( Z_{P})  \Big |
+   \Ep_{B_n} ( \|  \zeta^*_{n,P} \|_{\D} \wedge 2 ),
\end{eqnarray*}
where the first term  is $o^*_P(1)$, since $
G^*_{n,P} \rightsquigarrow_B  Z_{P}$  by Theorem \ref{lemma: uniform Donsker for bootstrap}, and the second term is $o_P(1)$  because  $ \|\zeta^*_{n,P}\|_{\D}  = o_P(1) $ implies that  $ \Ep_P ( \|  \zeta^*_{n,P} \|_{\D} \wedge 2 ) =
\Ep_P \Ep_{B_n} ( \|  \zeta^*_{n,P} \|_{\D} \wedge 2 ) \to 0$, which in turn implies that  $\Ep_{B_n} ( \|  \zeta^*_{n,P} \|_{\D} \wedge 2 ) = o_P(1)$ by the Markov inequality. \qed

\subsection{Proof of Corollary \ref{corollary: bs LATE}}  This is an immediate consequence of
Theorems \ref{theorem1}, \ref{theorem2},  \ref{thm: delta-method}, and \ref{theorem:delta-method-bootstrap}.\qed

\section{Omited Proofs for Section 5}\label{sec:n}

\begin{lemma}[\textbf{Donsker Theorem for Classes Changing with $n$}]\label{lemma: Donsker dep on n} Work with the set-up described in Appendix B of the main text. Suppose that for some fixed constant $q >2$ and
  every sequence $\delta_n \searrow 0$:
 \begin{eqnarray*}
  \|F_n\|_{P_n,q} = O(1),  \quad \sup_{d_T(s,t) \leq \delta_n} \| f_{n,s} - f_{n,t}\|_{P_n,2} \to 0,\quad  \int_0^{\delta_n} \sup_{Q } \sqrt{  \log N( \epsilon \|F_n\|_{Q,2}, \mathcal{F}_n, \| \cdot \|_{Q,2} ) } d \epsilon \to 0.
\end{eqnarray*}
(a) Then the empirical process $(\mathbb{G}_n f_{n,t})_{t \in T}$ is asymptotically tight in $\ell^\infty(T)$.
(b) For any subsequence such that the covariance function $P_n f_{n,s} f_{n,t} - P_n f_{n,s} P_n f_{n,t}$ converges pointwise
 on $T\times T$, $(\mathbb{G}_n f_{n,t})_{t \in T}$ converges in $\ell^\infty(T)$ to a Gaussian process with covariance
 function given by the limit of the covariance function along that subsequence. \end{lemma}

\textbf{Proof.}  The proof follows is similar to the proof of Theorem 2.11.22 in \citen[p. 220-221]{vdV-W}, except that the probability law is allowed to depend on $n$. Indeed, the use of Theorem 2.11.1  in \citen{vdV-W}, which does allow for  the probability space to depend on $n$, allows us to establish claim (a), whereas the proof of claim (b) follows by a standard argument.

  The random distance given in  Theorem 2.11.1  in \citen{vdV-W} (Lemma \ref{theorem: 2.11.1} below) reduces to
$
d_n^2(s,t) = \frac{1}{n} \sum_{i=1}^n (f_{n,s} - f_{n,t})^2(W_i) = \bP_n(f_{n,s} - f_{n,t})^2.
$
It follows that $N(\varepsilon,T,d_n) = N(\varepsilon, \mathcal{F}_n, L_2(\bP_n))$, for every $\varepsilon > 0$. If $F_n$ is replaced by $F_n \vee 1$ , then the conditions of the lemma still hold. Hence, assume without loss of generality than $F_n \geq 1$. Insert the bound on the covering numbers and next make a change of variables to bound the entropy integral $\int_{0}^{\delta_n} \sqrt{\log N(\varepsilon, \mathcal{F}_n, d_n)} d \varepsilon$ in Lemma \ref{theorem: 2.11.1}  by
$ \int_{0}^{\delta_n} \sqrt{\log N(\varepsilon \|F_n \|_{\bP_n,2}, \mathcal{F}_n, L_2(\bP_n))}
d\varepsilon \| F_n\|_{\bP_n,2}.$
This converges to zero in probability for every $\delta_n \downarrow 0$ by the conditions of the lemma. Apply Lemma \ref{theorem: 2.11.1}  to obtain the result. \qed

\begin{lemma}[van der Vaart and Wellner (1996, Th. 2.11.1)] \label{theorem: 2.11.1} For each $n$, let $Z_{n1}, \ldots, Z_{n,m_n}$ be independent stochastic processes,
defined on the product probability space $\prod_{i=1}^{m_n} (\mathcal{W}_{ni},
\mathcal{A}_{ni}, P_{ni})$, with each $Z_{ni} = Z_{ni}(f,w)$ depending on the $i$th coordinate of $w = (w_1, \ldots, w_{m_n})$, and
 indexed by a totally bounded semimetric space $(T,\rho)$. Assume that the sums $\sum_{i=1}^{m_n} e_i Z_{ni}$ are measurable in the sense that every one of the maps
\begin{eqnarray*}
w &\longmapsto& \sup_{\rho(f,g) < \delta} \left| \sum_{i=1}^{m_n} e_i\left(Z_{ni}(f) - Z_{ni}(g)\right)\right|,  \ \
w \longmapsto \sup_{\rho(f,g) < \delta} \left| \sum_{i=1}^{m_n} e_i\left(Z_{ni}(f) - Z_{ni}(g)\right)^2\right|,\end{eqnarray*}
is measurable, for every $\delta > 0$, every vector $(e_1,\ldots, e_{m_n}) \in \{-1,0,1\}^{m_n}$, and every natural number $n$. Also, for every $\eta>0$ and every $\delta_n \downarrow 0$:
$$\sum_{i=1}^{m_n} \Ep^*\|Z_{ni}\|^2_{\mathcal{F}_n}\{\|Z_{ni}\|_{\mathcal{F}_n} > \eta \} +
\sup_{\rho(s,t) < \delta_n} \sum_{i=1}^{m_n} \Ep\left(Z_{ni}(f) - Z_{ni}(g) \right)^2   \to 0,$$
and $\int_{0}^{\delta_n} \sqrt{\log N(\varepsilon, \mathcal{F}_n, d_n)} d\varepsilon  \overset{\Pr^*}{\to} 0$, where $d_n$ is the random semimetric
$$
d_n^2(f,g) = \sum_{i=1}^{m_n} \left(Z_{ni}(f) - Z_{ni}(g) \right)^2.
$$
Then the sequence $\sum_{i=1}^{m_n}  (Z_{ni} - \Ep Z_{ni})$ is asymptotically $\rho$-equicontinuous.
\end{lemma}

\section{Proofs for Section 6 and Appendix \ref{sec:j}}\label{sec:o}

\begin{proof}[Proof of Theorem \ref{Thm:RateEstimatedLassoLinear}]

In order to establish the result uniformly in $P \in  \mathcal{P}_n$, it suffices to establish the result under the probability measure induced by any sequence $P = P_n \in \mathcal{P}_n$.  In the proof we shall use $P$, suppressing the dependency of $P_n$ on the sample size $n$. To prove this result we invoke Lemmas \ref{Thm:RateEstimatedLasso}-\ref{Thm:2StepMain} in Appendix \ref{subsec: lasso}. These lemmas rely on specific events (described below) and Condition WL which is also stated in Appendix \ref{subsec: lasso}. We will show that Assumption \ref{ass: linear} implies that the required events occur with probability $1-o(1)$ and also implies Condition WL.

Let $\widehat\Psi_{u0,jj}=\{\En[|f_j(X)\zeta_u|^2]\}^{1/2}$ denote the ideal penalty loadings. The three events required to occur with probability $1-o(1)$ are the following:  $E_1 := \{c_r \geq \sup_{u\in\mathcal{U}}\|r_u\|_{\Pn,2}\}$,  and where  $c_r := C\sqrt{s\log(p\vee n)/n}$;  $ E_2:= \{\lambda/n\geq \sqrt{c}\sup_{u\in\mathcal{U}}\|\widehat\Psi^{-1}_{u0}\En[\zeta_uf(X)]\|_\infty\}, \ \ E_3 := \{\ell \widehat\Psi_{u0}\leq \widehat\Psi_{u}\leq L \widehat\Psi_{u0}\},$ for some  $1/\sqrt{c}<1/\sqrt[4]{c}<\ell$  and $L$ uniformly bounded for the penalty loading $\widehat\Psi_u$ in all iterations $k\leq K$ for $n$ sufficiently large.

By Assumption \ref{ass: linear}(iv)(b) $E_1$ holds with probability $1-o(1)$.

Next we verify that Condition WL holds. Condition WL(i) is implied by the approximate sparsity condition in Assumption \ref{ass: linear}(i) and the covering condition in Assumption \ref{ass: linear}(ii). By Assumption \ref{ass: linear} we have that $\dn$ is fixed and the Algorithm sets $\gamma \in [1/n,\min\{\log^{-1} n, pn^{\dn-1}\}]$ so that $\gamma = o(1)$ and $\Phi^{-1}(1-\gamma/\{2pn^{\dn}\})\leq C\log^{1/2} (np) \leq C\delta_n n^{1/6}$ by Assumption \ref{ass: linear}(i). Since it is assumed that $\Ep_P[|f_j(X)\zeta_u|^2]\geq c$ and $\Ep_P[|f_j(X)\zeta_u|^3]\leq C$ uniformly in $j\leq p$ and $u\in \mathcal{U}$, Condition WL(ii) holds. Condition WL(iii) follows from  Assumption \ref{ass: linear}(iv).

Since Condition WL holds, by Lemma \ref{Thm:ChoiceLambda}, the event $E_2$  occurs with probability $1-o(1)$.

Next we proceed to verify  occurrence of $E_3$. In the first iteration, the penalty loadings are defined as $\widehat\Psi_{ujj}=\{\En[|f_j(X)Y_u|^2]\}^{1/2}$ for $j=1,\ldots,p$, $u\in\mathcal{U}$. By Assumption \ref{ass: linear}, $\underline{c}\leq \Ep_P[|f_j(X)\zeta_u|^2]\leq \Ep_P[|f_j(X)Y_u|^2]\leq C$ uniformly over $u\in\mathcal{U}$ and $j=1,\ldots,p$. Moreover, Assumption \ref{ass: linear}(iv)(b) yields $$\sup_{u\in\mathcal{U}}\max_{j\leq p}|(\En-\Ep_P)[|f_j(X)Y_u|^2]| \leq \delta_n \ \ \mbox{and} \ \ \sup_{u\in\mathcal{U}}\max_{j\leq p}|(\En-\Ep_P)[|f_j(X)\zeta_u|^2]| \leq \delta_n$$ with probability $1-\Delta_n$. In turn this shows that for $n$ large so that $\delta_n\leq \underline{c}/4$ we have\footnote{Indeed, using that $\underline{c} \leq \Ep_P[|f_j(X)\zeta_u|^2] \leq \Ep_P[|f_j(X)Y_u|^2] \leq C$, we have   $(1-2\delta_n/\underline{c})\En[|f_j(X)\zeta_u|^2] \leq (1-2\delta_n/\underline{c})\{\delta_n + \Ep_P[|f_j(X)\zeta_u|^2]\}\leq \Ep_P[|f_j(X)\zeta_u|^2]-\delta_n \leq \Ep_P[|f_j(X)Y_u|^2]-\delta_n \leq  \En[|f_j(X)Y_u|^2] $. Similarly, $ \En[|f_j(X)Y_u|^2] \leq \delta_n +  \Ep_P[|f_j(X)Y_u|^2] \leq  \delta_n + C \leq (\{\delta_n + C\}/\{\underline{c}-\delta_n\})\En[|f_j(X)\zeta_u|^2]$.}
$$\begin{array}{l}
 (1-2\delta_n/\underline{c})\En[|f_j(X)\zeta_u|^2]  \leq  \En[|f_j(X)Y_u|^2] \leq   (\{C+\delta_n\}/\{\underline{c}-\delta_n\})\En[|f_j(X)\zeta_u|^2]
\end{array}$$ with probability $1-\Delta_n$ so that $\ell\widehat\Psi_{u0}\leq \widehat \Psi_u \leq L \widehat\Psi_{u0}$ for some uniformly bounded $L$ and $\ell>1/\sqrt[4]{c}$. Moreover, $\tilde \cc=\{(L\sqrt{c}+1)/(\sqrt{c}\ell-1)\}\sup_{u\in\mathcal{U}}\|\widehat \Psi^{-1}_{u0}\|_\infty\|\widehat \Psi_{u0}\|_\infty$ is uniformly bounded for $n$ large enough which implies that $\kappa_{2\tilde \cc}$ as defined in (\ref{DefKappa}) in Appendix \ref{Sec:FiniteLinear} is bounded away from zero with probability $1-\Delta_n$ by the condition on sparse eigenvalues of order $s\ell_n$ (see \citen{BickelRitovTsybakov2009} Lemma 4.1(ii)).

By Lemma \ref{Thm:RateEstimatedLasso}, since  $\lambda \in  [cn^{1/2}\log^{1/2}(p\vee n), Cn^{1/2}\log^{1/2}(p\vee n)]$ by the choice of $\gamma$ and $\dn$ fixed, $c_r \leq C\sqrt{s\log(p\vee n)/n}$, $\sup_{u\in\mathcal{U}}\|\widehat \Psi_{u0}\|_\infty \leq C$,  we have
$$\begin{array}{l}
 \displaystyle\sup_{u\in\mathcal{U}} \| f(X)'(\hat\theta_u - \theta_{u})\|_{\Pn,2} \leq C' \sqrt{\frac{s\log (p\vee n)}{n}}\ \ \mbox{and} \  \
 \displaystyle \sup_{u\in\mathcal{U}}\|\hat\theta_u-\theta_{u}\|_1  \leq C' \sqrt{\frac{s^2\log (p\vee n)}{n}}.\\ \end{array}$$


In the application of Lemma \ref{Thm:Sparsity}, by Assumption \ref{ass: linear}(iv)(c), we have that
$\min_{m \in \mathcal{M}}\semax{m}$ is uniformly bounded for $n$ large enough with probability $1-o(1)$. Thus, with probability $1-o(1)$, by Lemma \ref{Thm:Sparsity} we have
$$ \sup_{u\in\mathcal{U}} \hat s_u \leq C   \[\frac{n c_r}{\lambda} + \sqrt{s}\]^2 \leq C's.$$
Therefore by Lemma \ref{Thm:2StepMain} the Post-Lasso estimators $(\widetilde \theta_u)_{u\in\mathcal{U}}$ satisfy with probability $1-o(1)$
$$\begin{array}{l}
 \displaystyle\sup_{u\in\mathcal{U}} \| f(X)'(\widetilde\theta_u - \theta_{u})\|_{\Pn,2} \leq \bar{C} \sqrt{\frac{s\log (p\vee n)}{n}}\ \ \mbox{and} \  \
 \displaystyle \sup_{u\in\mathcal{U}}\|\widetilde\theta_u-\theta_{u}\|_1  \leq \bar{C} \sqrt{\frac{s^2\log (p\vee n)}{n}}\\ \end{array}$$
for some $\bar{C}$ independent of $n$, since uniformly in $u\in \mathcal{U}$ we have a sparsity bound $\|(\widetilde\theta_u - \theta_{u})\|_0 \leq C'' s$ and that ensures that a bound on the prediction rate yields a bound on the $\ell_1$-norm rate through the relations $\|v\|_1\leq \sqrt{\|v\|_0}\|v\|\leq \sqrt{\|v\|_0}\|f(X)'v\|_{\Pn,2}/\sqrt{\semin{\|v\|_0}}$.

In the $k$th iteration, the penalty loadings are constructed based on $(\widetilde\theta_u^{(k)})_{u\in\mathcal{U}}$, defined as $\widehat\Psi_{ujj}=\{\En[|f_j(X)\{Y_u-f(X)'\widetilde\theta_u^{(k)}\}|^2]\}^{1/2}$ for $j=1,\ldots,p$, $u\in\mathcal{U}$. We assume $(\widetilde\theta_u^{(k)})_{u\in\mathcal{U}}$ satisfy the rates above uniformly in $u\in\mathcal{U}$. Then with probability $1-o(1)$ we have uniformly in $u\in\mathcal{U}$ and $j=1,\ldots,p$ 
$$\begin{array}{rl}
|\widehat\Psi_{ujj}- \widehat\Psi_{u0jj}| & \leq \{\En[|f_j(X)\{f(X)'(\widetilde\theta_u-\theta_u)\}|^2]\}^{1/2}+\{\En[|f_j(X)r_u|^2]\}^{1/2}\\
& \leq K_n \|f(X)'(\widetilde\theta_u-\theta_u)\|_{\Pn,2}+K_n\|r_u\|_{\Pn,2} \leq \bar{C} K_n \sqrt{\frac{s\log(p\vee n)}{n}} \\
&\leq \bar{C} \delta_n^{1/2} \leq \widehat\Psi_{u0jj} (2\bar{C} \delta_n^{1/2}/\underline{c})
\end{array}$$
where we used that $\max_{i\leq n,j\leq p}|f_j(X_i)|\leq K_n$ a.s., and $K_n^2s\log(p\vee n) \leq \delta_n n$ by Assumption \ref{ass: linear}(iv)(a), and that $\inf_{u\in\mathcal{U},j\leq p}\widehat\Psi_{u0jj}\geq \underline{c}/2$ with probability $1-o(1)$ for $n$ large so that $\delta_n \leq \underline{c}/2$. Further, for $n$ large so that $(2\bar{C} \delta_n^{1/2}/\underline{c}) < 1-1/\sqrt[4]{c}$, this establishes that the event of the penalty loadings for the $(k+1)$th iteration also satisfy $\ell \widehat \Psi^{-1}_{u0}\leq  \widehat \Psi^{-1}_{u} \leq L  \widehat \Psi^{-1}_{u0}$ for a uniformly bounded $L$ and some $\ell > 1/\sqrt[4]{c}$ with probability $1-o(1)$ uniformly in $u\in\mathcal{U}$. 

This leads to the stated rates of convergence and sparsity bound.
\end{proof}

\begin{proof}[Proof of Theorem \ref{Thm:RateEstimatedLassoLogistic}]

In order to establish the result uniformly in $P \in  \mathcal{P}_n$, it suffices to establish the result under the probability measure induced by any sequence $P = P_n \in \mathcal{P}_n$.  In the proof we shall use $P$, suppressing the dependency of $P_n$ on the sample size $n$. The proof is similar to the proof of Theorem \ref{Thm:RateEstimatedLassoLinear}. We invoke Lemmas \ref{Lemma:LassoLogisticRateRaw}, \ref{Lemma:LassoLogisticSparsity} and \ref{Lemma:PostLassoLogisticRateRaw} which require Condition WL and some events to occur. We show that Assumption \ref{ass: logistic} implies Condition WL and that the required events occur  with probability at least $1-o(1)$.

Let $\widehat\Psi_{u0,jj}=\{\En[|f_j(X)\zeta_u|^2]\}^{1/2}$ denote the ideal penalty loadings, $w_{ui}=\Ep_P[Y_{ui}\mid X_i](1-\Ep_P[Y_{ui}\mid X_i])$ the conditional variance of $Y_{ui}$ given $X_i$ and $\tilde r_{ui}=\tilde r_{u}(X_i)$ the rescaled approximation error as defined in (\ref{def:tilder}). The three events required to occur with probability $1-o(1)$ are as follows:  $E_1 := \{c_r \geq \sup_{u\in\mathcal{U}}\|\tilde r_u/\sqrt{w_u}\|_{\Pn,2}\}$  for  $c_r := C'\sqrt{s\log(p\vee n)/n}$ where $C'$ is large enough; $E_2 := \{\lambda/n\geq \sqrt{c}\sup_{u\in\mathcal{U}}\|\widehat\Psi^{-1}_{u0}\En[\zeta_uf(X)]\|_\infty\}$; and $E_3:= \{\ell \widehat\Psi_{u0}\leq \widehat\Psi_{u}\leq L \widehat\Psi_{u0}\}$, for $\ell > 1/\sqrt[4]{c}$ and $L$ uniformly bounded,  for the penalty loading $\widehat\Psi_u$ in all iterations $k\leq K$ for $n$ sufficiently large.

Regarding $E_1$, by Assumption \ref{ass: logistic}(iii), we have $\underline{c}(1-\underline{c}) \leq w_{ui}\leq 1/4$. Since $|r_{u}(X_i)|\leq \delta_n$ a.s. uniformly on $u\in\mathcal{U}$ for $i=1,\ldots,n$, we have that the rescaled approximation error defined in (\ref{def:tilder}) satisfies $|\tilde r_{u}(X_i)|\leq |r_{u}(X_i)|/\{\underline{c}(1-\underline{c}) - 2\delta_n \}_+ \leq \tilde C |r_{u}(X_i)|$ for $n$ large enough so that $\delta_n \leq \underline{c}(1-\underline{c})/4$. Thus $\|\tilde r_u/\sqrt{w_u}\|_{\Pn,2} \leq \tilde C \| r_u/\sqrt{w_u}\|_{\Pn,2}$. Assumption \ref{ass: logistic}(iv)(b) yields $ \sup_{u\in\mathcal{U}}\|r_u/\sqrt{w_u}\|_{\Pn,2} \leq  C\sqrt{s\log(p\vee n)/n}$ with probability $1-o(1)$, so $E_3$ occurs  with probability $1-o(1)$.

To apply Lemma \ref{Thm:ChoiceLambda} to show that $E_2$ occurs with probability $1-o(1)$ we need to verify Condition WL. Condition WL(i) is implied by the sparsity in Assumption \ref{ass: logistic}(i)  and the covering condition in Assumption \ref{ass: logistic}(ii). By Assumption \ref{ass: logistic} we have that $\dn$ is fixed and the Algorithm sets $\gamma \in [1/n,\min\{\log^{-1} n, pn^{\dn-1}\}]$ so that $\gamma = o(1)$ and $\Phi^{-1}(1-\gamma/\{2pn^{\dn}\})\leq C\log^{1/2} (np) \leq C\delta_n n^{1/6}$ by Assumption \ref{ass: logistic}(i). Since it is assumed that $\Ep_P[|f_j(X)\zeta_u|^2]\geq c$ and $\Ep_P[|f_j(X)\zeta_u|^3]\leq C$ uniformly in $j\leq p$ and $u\in \mathcal{U}$, Condition WL(ii) holds. Condition WL(iii) follows from  Assumption \ref{ass: linear}(iv). Then, by Lemma \ref{Thm:ChoiceLambda}, the event $E_2$ occurs with probability $1-o(1)$.

Next we verify the occurrence of $E_3$. In the initial iteration, the penalty loadings are  defined as $\widehat\Psi_{ujj}=\frac{1}{2}\{\En[|f_j(X)|^2]\}^{1/2}$ for $j=1,\ldots,p$, $u\in\mathcal{U}$.
Assumption \ref{ass: logistic}(iv)(c) for the sparse eigenvalues implies that for $n$ large enough, $c'\leq \En[|f_j(X)|^2] \leq C'$ for all $j=1,\ldots,p,$ with probability $1-o(1)$.

Moreover, Assumption \ref{ass: logistic}(iv)(b) yields \begin{equation}\label{AuxLogLoad} \ \sup_{u\in\mathcal{U}}\max_{j\leq p}|(\En-\Ep_P)[|f_j(X)\zeta_u|^2]| \leq \delta_n\end{equation} with probability $1-\Delta_n$, so that $\widehat \Psi_{u0jj}$ is bounded away from zero and from above uniformly over $j=1,\ldots,p$, $u\in\mathcal{U}$, with the same probability because $\Ep_P[|f_j(X)\zeta_u|^2]$ is bounded away from zero and above. By (\ref{AuxLogLoad}) and  $\Ep_P[|f_j(X)\zeta_u|^2] \leq \frac{1}{4}\Ep_P[|f_j(X)|^2]$, for $n$ large enough, we have
$\ell\widehat\Psi_{u0}\leq \widehat\Psi_{u}\leq L\widehat\Psi_{u0}$ for some uniformly bounded $L$ and  $\ell > 1/\sqrt[4]{c}$ with probability $1-\Delta_n$. 

Thus, $\tilde \cc=\{(L\sqrt{c}+1)/(\ell \sqrt{c}-1)\}\sup_{u\in\mathcal{U}}\|\widehat\Psi^{-1}_{u0}\|_{\infty}\|\widehat\Psi_{u0}\|_\infty$ is uniformly bounded. In turn, since $\inf_{u\in\mathcal{U}} \min_{i\leq n} w_{ui}\geq \underline{c}(1-\underline{c})$ is bounded away from zero,  we have $\bar\kappa_{2\tilde \cc}\geq \sqrt{\underline{c}(1-\underline{c})}\kappa_{2\tilde \cc}$ by their definitions in (\ref{DefKappa}) and (\ref{DefKAppaLog}). It follows that $\kappa_{2\tilde \cc}$ is bounded away from zero by the condition on $s\ell_n$ sparse eigenvalues  stated in Assumption \ref{ass: logistic}(iv)(c), see \citen{BickelRitovTsybakov2009} Lemma 4.1(ii).

By the choice of $\gamma$ and $d_u$ fixed,  $\lambda \in  [cn^{1/2}\log^{1/2}(p\vee n), Cn^{1/2}\log^{1/2}(p\vee n)]$. By relation (\ref{LowerBound:q}) and Assumption \ref{ass: logistic}(iv)(a), $\inf_{u\in\mathcal{U}}\bar q_{A_u}\geq c' \bar\kappa_{2\tilde\cc}/\{\sqrt{s}K_n\}$. Under the condition $K_n^2s^2\log^2(p\vee n) \leq \delta_n n$, the side condition in Lemma \ref{Lemma:LassoLogisticRateRaw} holds with probability $1-o(1)$, and the lemma yields
$$\begin{array}{l}
 \displaystyle\sup_{u\in\mathcal{U}} \| f(X)'(\hat\theta_u - \theta_{u})\|_{\Pn,2} \leq C' \sqrt{\frac{s\log (p\vee n)}{n}}\ \ \mbox{and} \  \
 \displaystyle \sup_{u\in\mathcal{U}}\|\hat\theta_u-\theta_{u}\|_1  \leq C' \sqrt{\frac{s^2\log (p\vee n)}{n}}\\ \end{array}$$
In turn, under Assumption \ref{ass: logistic}(iv)(c) and $K_n^2s^2\log^2(p\vee n) \leq \delta_n n$, with probability $1-o(1)$ Lemma \ref{Lemma:LassoLogisticSparsity} implies $$ \sup_{u\in\mathcal{U}} \hat s_u \leq C''   \[\frac{n c_r}{\lambda} + \sqrt{s}\]^2 \leq C'''s$$
since $\min_{m \in \mathcal{M}}\semax{m}$ is uniformly bounded. The rate of convergence for $\widetilde\theta_u$ is given by Lemma \ref{Lemma:PostLassoLogisticRateRaw}, namely with probability $1-o(1)$
$$\begin{array}{l}
 \displaystyle\sup_{u\in\mathcal{U}} \| f(X)'(\widetilde\theta_u - \theta_{u})\|_{\Pn,2} \leq \bar{C} \sqrt{\frac{s\log (p\vee n)}{n}}\ \ \mbox{and} \  \
 \displaystyle \sup_{u\in\mathcal{U}}\|\widetilde\theta_u-\theta_{u}\|_1  \leq \bar{C} \sqrt{\frac{s^2\log (p\vee n)}{n}}\\ \end{array}$$
for some $\bar{C}$ independent of $n$, since by (\ref{QQ:UpperBound}) we have uniformly in $u\in\mathcal{U}$
$$\begin{array}{rl}
M_u(\tilde\theta_u) - M_u(\theta_u) & \leq M_u(\hat\theta_u) - M_u(\theta_u) \leq \frac{\lambda}{n}\|\widehat\Psi_u\theta_{u}\|_1 - \frac{\lambda}{n}\|\widehat\Psi_u\hat\theta_{u}\|_1 \leq  \frac{\lambda}{n}\|\widehat\Psi_u(\hat\theta_{uT_u}-\theta_u)\|_1\\
&\leq \bar{C}' s\log(p\vee n)/n,\end{array}$$
$\sup_{u\in\mathcal{U}}\|\En[f(X)\zeta_u]\|_\infty \leq C \sqrt{\log(p\vee n)/n}$ by Lemma \ref{Thm:ChoiceLambda}, $\semin{\hat s_u + s_u}$ is bounded away from zero (by Assumption \ref{ass: logistic}(iv)(c) and $\hat s_u \leq C'''s$), $\inf_{u\in\mathcal{U}}\psi_u(\{\delta \in \mathbb{R}^p:\|\delta\|_0\leq \hat s_u + s_u\})$ is bounded away from zero (because $\inf_{u\in\mathcal{U}}\min_{i\leq n}w_{ui}\geq \underline{c}(1-\underline{c})$), and $\sup_{u\in\mathcal{U}}\|\widehat\Psi_{u0}\|_\infty\leq C$ with probability $1-o(1)$.

In the $k$th iteration, the penalty loadings are constructed based on $(\widetilde\theta_u^{(k)})_{u\in\mathcal{U}}$, defined as $\widehat\Psi_{ujj}=\En[|f_j(X)\{Y_u-\G( f(X)'\widetilde\theta_u^{(k)})\}|^2]\}^{1/2}$ for $j=1,\ldots,p$, $u\in\mathcal{U}$. We assume $(\widetilde\theta_u^{(k)})_{u\in\mathcal{U}}$ satisfy the rates above uniformly in $u\in\mathcal{U}$. Then $$\begin{array}{rl}
|\widehat\Psi_{ujj}- \widehat\Psi_{u0jj}|
& \leq \{\En[|f_j(X)\{\G(f(X)'\widetilde\theta_u^{(k)})-\G(f(X)'\theta_u)\}|^2]\}^{1/2}+\{\En[|f_j(X)r_u|^2]\}^{1/2}\\
& \leq \{\En[|f_j(X)\{f(X)'(\widetilde\theta_u^{(k)}-\theta_u)\}|^2]\}^{1/2}+\{\En[|f_j(X)r_u|^2]\}^{1/2}\\
& \leq K_n \|f(X)'(\widetilde\theta_u^{(k)}-\theta_u)\|_{\Pn,2}+K_n\|r_u\|_{\Pn,2} \lesssim_P K_n \sqrt{\frac{s\log(p\vee n)}{n}}\\ & \leq  C\delta_n \leq (2C\delta_n/\underline{c})\widehat\Psi_{u0jj}
\end{array}$$
and therefore, provided that $(2C\delta_n/\underline{c})<1-1/\sqrt[4]{c}$, uniformly in $u\in\mathcal{U}$, $\ell \widehat\Psi_{u0}\leq \widehat\Psi_u \leq L \widehat \Psi_{u0}$ for $\ell > 1/\sqrt[4]{c}$ and $L$ uniformly bounded with probability $1-o(1)$. Then the same proof for the initial penalty loading choice applies to the iterate $(k+1)$.\end{proof}

\subsection{Proofs for Lasso with Functional Response: Penalty Level }

\begin{proof}[Proof of Lemma \ref{Thm:ChoiceLambda}]

By the triangle inequality
$$ \begin{array}{rl}
\sup_{u\in\mathcal{U}} \| \hat \Psi^{-1}_{u 0} \En[ f(X) \zeta_{u} ]\|_\infty & \leq \max_{u\in\mathcal{U}^\epsilon} \| \hat \Psi^{-1}_{u 0} \En[ f(X) \zeta_{u} ]\|_\infty\\
&  +\sup_{u\in\mathcal{U}^\epsilon,u'\in\mathcal{U},d_\mathcal{U}(u,u')\leq \epsilon} \| \hat \Psi^{-1}_{u 0} \En[ f(X) \zeta_{u} ] - \hat \Psi^{-1}_{u'0} \En[ f(X) \zeta_{u'} ]\|_\infty  \end{array}  $$
where $\mathcal{U}^\epsilon$ is a minimal $\epsilon$-net of $\mathcal{U}$. We will set $\epsilon = 1/n$ so that $|\mathcal{U}^\epsilon|\leq n^\dn$.

The proofs  in this section rely on the following result  due to \citen{jing:etal}.

\begin{lemma}[\textbf{Moderate deviations for self-normalized sums}]\label{Lemma: MDSN} Let $Z_{1}$,$\ldots$, $Z_{n}$ be independent, zero-mean random variables and $\mu \in (0,1]$. Let
$S_{n,n} = \sum_{i=1}^n Z_i,  \ \ V^2_{n,n} = \sum_{i=1}^nZ^2_i,$ $$M_n= \left \{\frac{1}{n} \sum_{i=1}^n \Ep [ Z_i^2 ] \right \}^{1/2} \Big / \left \{\frac{1}{n} \sum_{i=1}^n \Ep[|Z_i|^{2+\mu}] \right\}^{1/\{2+\mu\}}>0$$
and $0< \ell_n \leq n^{\frac{\mu}{2(2+\mu)}}M_n$. Then for some absolute constant $A$,
$$
\left |\frac{\Pr(|S_{n,n}/V_{n,n}|  \geq x) }{ 2 (1-\Phi(x))} - 1 \right |  \leq
\frac{A}{ \ell_n^{2+\mu}},  \ \ 0 \leq  x \leq  n^{\frac{\mu}{2(2+\mu)}}\frac{M_n}{\ell_n}-1. \\
$$
\end{lemma}

For each $j=1,\ldots,p$, and each $u\in \mathcal{U}^\epsilon$, we will apply Lemma \ref{Lemma: MDSN} with $Z_i := f_j(X_i)\zeta_{ui}$, and $\mu = 1$. Then, by Lemma \ref{Lemma: MDSN}, the union bound, and $|\mathcal{U}^\epsilon|\leq N_n$,  we have
\begin{equation}\label{SNcontrol}
\begin{array}{rl}\displaystyle\Pr_P \left( {\displaystyle\sup_{u\in\mathcal{U}^\epsilon}\max_{j\leq p}} \left|
\frac{\sqrt{n}\En[ f_j(X)\zeta_{u} ]}{\sqrt{\En[f_j(X)^2 \zeta_{u}^2]}}\right| > \Phi^{-1}(1-\mbox{$\frac{\gamma}{2pN_n}$}) \right) & \leq 2pN_n ( \gamma/2p N_n)\{ 1 + o(1)\}\\
& \leq \gamma\{ 1 + o(1)\}\end{array} \end{equation}
provided that $\max_{u,j}\{\barEp_P[|f_j(X)\zeta_u|^3]^{1/3}/\barEp_P[|f_j(X)\zeta_u|^2]^{1/2}\}\Phi^{-1}(1-\gamma/2pN_n)\leq \delta_n n^{1/6}$, which holds by Condition WL since $\gamma \geq 1/n$ (under this condition there is $\ell_n \to \infty$ obeying conditions
of Lemma \ref{Lemma: MDSN})

Moreover, by triangle inequality we have
 \begin{equation}\label{EqTri2}\begin{array}{ll}
{\displaystyle \sup_{u\in\mathcal{U}^\epsilon,u'\in\mathcal{U},d_\mathcal{U}(u,u')\leq \epsilon}} \| \hat \Psi^{-1}_{u 0} \En[f(X)\zeta_{u} ] - \hat \Psi^{-1}_{u'0} \En[ f(X)\zeta_{u'} ]\|_\infty  \\ \leq {\displaystyle  \sup_{u\in\mathcal{U}^\epsilon,u'\in\mathcal{U},d_\mathcal{U}(u,u')\leq \epsilon}} \|( \hat \Psi^{-1}_{u 0} - \hat \Psi^{-1}_{u'0}) \hat \Psi_{u 0} \|_\infty \| \hat \Psi^{-1}_{u 0}\En[ f(X)\zeta_{u} ]\|_\infty \\ +  {\displaystyle \sup_{u,u'\in\mathcal{U},d_\mathcal{U}(u,u')\leq \epsilon}} \| \En[ f(X)(\zeta_{u}- \zeta_{u'}) ]\|_\infty\|\hat \Psi^{-1}_{u'0} \|_\infty.\end{array}\end{equation}

To control the first term in (\ref{EqTri2}) we note that by Condition WL, $\hat \Psi_{u 0jj}$ is bounded away from zero  with probability $1-o(1)$ uniformly over $u\in\mathcal{U}$ and $j=1,\ldots,p$. Thus we have uniformly over $u\in\mathcal{U}$ and $j=1,\ldots,p$ \begin{equation}\label{eqaux37a}|(\hat \Psi^{-1}_{u 0jj}-\hat \Psi^{-1}_{u'0jj})\hat \Psi_{u 0jj}|=|\hat \Psi_{u 0jj}-\hat \Psi_{u' 0jj}|/\hat \Psi_{u' 0jj} \leq C|\hat \Psi_{u 0jj}-\hat \Psi_{u' 0jj}|\end{equation} with the same probability.
Moreover, we have
\begin{equation}\label{eqAux37}\begin{array}{rl}
{\displaystyle \sup_{u,u'\in\mathcal{U},d_\mathcal{U}(u,u')\leq \epsilon}} \max_{j\leq p}| \{\En[f_j(X)^2\zeta_{u}^2]\}^{1/2}-\{\En[f_j(X)^2\zeta_{u'}^2]\}^{1/2}| &\\
\leq{\displaystyle \sup_{u,u'\in\mathcal{U},d_\mathcal{U}(u,u')\leq \epsilon}}\max_{j\leq p}  \{\En[f_j(X)^2(\zeta_{u}-\zeta_{u'})^2]\}^{1/2}.\\
\end{array}\end{equation}
 Thus, relations (\ref{eqaux37a}) and (\ref{eqAux37}) imply that with probability $1-o(1)$
$${\displaystyle  \sup_{u,u'\in\mathcal{U},d_\mathcal{U}(u,u')\leq \epsilon}} \| (\hat \Psi^{-1}_{u 0} - \hat \Psi^{-1}_{u'0})\hat\Psi_{u0} \|_\infty \lesssim {\displaystyle \sup_{u,u'\in\mathcal{U},d_\mathcal{U}(u,u')\leq \epsilon}}  \max_{j\leq p}\{\En[f_j(X)^2(\zeta_{u}-\zeta_{u'})^2]\}^{1/2}.$$

By (\ref{SNcontrol}) $$\sup_{u\in\mathcal{U}^\epsilon} \|\widehat\Psi^{-1}_{u0}\En[ f(X)\zeta_{u} ]\|_\infty\leq C'\sqrt{\log(p\vee N_n\vee n)/n}$$ with probability $1-o(1)$, so that with the same probability
$$ \begin{array}{ll} {\displaystyle  \sup_{u\in \mathcal{U}^\epsilon,u'\in\mathcal{U},d_\mathcal{U}(u,u')\leq \epsilon}} \|  (\hat\Psi^{-1}_{u 0} - \hat \Psi^{-1}_{u'0})\hat \Psi_{u0} \|_\infty \|\widehat\Psi^{-1}_{u0}\En[ f(X)\zeta_{u} ]\|_\infty \\ \leq    {\displaystyle  \sup_{u,u'\in\mathcal{U},d_\mathcal{U}(u,u')\leq \epsilon}} \max_{j\leq p}  \{\En[f_j(X)^2(\zeta_{u}-\zeta_{u'})^2]\}^{1/2} C' \sqrt{\frac{\log(p\vee N_n \vee n)}{n}}\leq \frac{o(1)}{\sqrt{n}}\end{array}$$
where the last inequality follows by Condition WL(iii).

The last term in (\ref{EqTri2}) is of the order $o(n^{-1/2})$ with probability $1-o(1)$ since by Condition WL, $$\sup_{u,u'\in\mathcal{U},d_\mathcal{U}(u,u')\leq \epsilon} \| \En[ f(X)(\zeta_{u}- \zeta_{u'}) ]\|_\infty\leq \delta_n n^{-1/2}$$ with probability $1-\Delta_n$, and noting that by Condition WL $\sup_{u\in\mathcal{U}}\|\hat \Psi^{-1}_{u 0}\|_\infty$ is uniformly bounded with probability at least $1-o(1)-\Delta_n$.

The results above imply that (\ref{EqTri2}) is bounded by $o(1)/\sqrt{n}$ with probability $1-o(1)$. Since $\frac{1}{2}\sqrt{ \log (2pN_n/\gamma)} \leq \Phi^{-1}(1-\gamma/\{2pN_n\})$ for $n$ large enough (since  $\gamma/\{2pN_n\} \to 0$ and standard tail bounds),  we have that with probability $1-o(1)$  {\small $$ \frac{(c'-c)}{\sqrt{n}}\Phi^{-1}(1-\gamma/\{2pN_n\}) \geq \sup_{u\in\mathcal{U}^\epsilon,u'\in\mathcal{U},d_\mathcal{U}(u,u')\leq \epsilon} \| \hat \Psi^{-1}_{u 0} \En[ f(X)\zeta_{u} ] - \hat \Psi^{-1}_{u'0} \En[ f(X)\zeta_{u'} ]\|_\infty$$} and the result follows.
\end{proof}

\begin{proof}[Proof of Lemma \ref{PrimitiveWL}]
We start with the last statement of the lemma since it is more difficult (others will use similar calculations). Consider the class of functions $\mathcal{F}=\{ Y_u : u \in \mathcal{U}\}$, $\mathcal{F}'=\{ \Ep_P[Y_u\mid X] : u \in \mathcal{U}\}$, and $\mathcal{G} = \{ \zeta_u^2=(Y_u - \Ep_P[Y_u\mid X])^2 : u \in \mathcal{U} \}$. Let $F$ be a measurable envelope for $\mathcal{F}$ which satisfies $F \leq B_n$.

Because $\mathcal{F}$ is a VC-class of functions with VC index $C'd_u$, by Lemma \ref{lemma: andrews}(1) we have \begin{equation}\label{CNvc} \log N( \epsilon \|F\|_{Q,2}, \mathcal{F}, \|\cdot\|_{Q,2}) \lesssim 1+[d_u\log (e/\epsilon)\vee 0].\end{equation}
To bound the covering number for $\mathcal{F}'$ we apply Lemma \ref{Lemma:PartialOutCovering}, and since $\Ep[F\mid X] \leq F$, we have \begin{equation}\label{CNb1} \log \sup_Q N( \epsilon \|F\|_{Q,2}, \mathcal{F}', \|\cdot\|_{Q,2})  \leq \log \sup_Q N( \mbox{$\frac{\epsilon}{2}$} \|F\|_{Q,2}, \mathcal{F}, \|\cdot\|_{Q,2}).\end{equation}
Since $\mathcal{G} \subset (\mathcal{F}-\mathcal{F}')^2$, $G=4F^2$ is an envelope for $\mathcal{G}$ and  the covering number for $\mathcal{G}$ satisfies \begin{equation}\label{Chain1G} \begin{array}{rl}
\log  N( \epsilon \|4F^2\|_{Q,2}, \mathcal{G}, \|\cdot\|_{Q,2}) & \overset{(i)}{\leq} 2 \log N( \frac{\epsilon}{2} \|2F\|_{Q,2}, \mathcal{F}-\mathcal{F}', \|\cdot\|_{Q,2})\\
&  \overset{(ii)}{\leq} 2 \log  N( \frac{\epsilon}{4} \|F\|_{Q,2}, \mathcal{F}, \|\cdot\|_{Q,2})+ 2 \log  N( \frac{\epsilon}{4} \|F\|_{Q,2}, \mathcal{F}', \|\cdot\|_{Q,2})\\
&  \overset{(iii)}{\leq} 4 \log \sup_Q N( \frac{\epsilon}{8} \|F\|_{Q,2}, \mathcal{F}, \|\cdot\|_{Q,2}),\\
\end{array}\end{equation}
where (i) and (ii) follow by Lemma \ref{lemma: andrews}(2), and (iii) follows from (\ref{CNb1}).

Hence, the entropy bound for the class $\mathcal{M}=\cup_{j\in[p]}\mathcal{M}_j$, where  $\mathcal{M}_j = \{ f_j^2(X) \mathcal{G}\}$, $j \in [p]$ and envelope $M = 4K^2_nF^2$, satisfies
$$
\begin{array}{rl}
\log N( \epsilon \|M\|_{Q,2}, \mathcal{M}, \|\cdot\|_{Q,2}) & \overset{(a)}{\leq} \log p + \max_{j\in[p]}\log N( \epsilon \|4K^2_nF^2\|_{Q,2}, \mathcal{M}_j, \|\cdot\|_{Q,2})\\
&  \overset{(b)}{\leq} \log p + \log N( \epsilon \|4F^2\|_{Q,2}, \mathcal{G}, \|\cdot\|_{Q,2})\\
&  \overset{(c)}{\leq}\log p +  4 \log \sup_Q N( \frac{\epsilon}{8} \|F\|_{Q,2}, \mathcal{F}, \|\cdot\|_{Q,2}) \\
&  \overset{(d)}{\lesssim}\log p + [(1+d_u) \log (e/\epsilon) \vee 0 ],
\end{array}$$
where (a) follows by Lemma \ref{lemma: andrews}(2) for union of classes, (b) holds by Lemma \ref{lemma: andrews}(2) when one class has only a single function, (c) by (\ref{Chain1G}) and (d) follows from (\ref{CNvc}) and $\epsilon \leq 1$. Therefore, since $\sup_{u\in\mathcal{U}}\max_{j\leq p}\Ep_P[f_j^2(X)\zeta_u^2]$ is bounded away from zero and from above, by Lemma \ref{lemma:CCK}  we have with probability $1-O(1/\log n)$ that
$$
\begin{array}{rl}
\sup_{u\in\mathcal{U}}\max_{j\leq p}|(\En-\Ep_P)[f_j^2(X)\zeta_u^2]| & \lesssim \sqrt{\frac{(1+d_u) \log (npK^2_nB^2_n)}{n}} + \frac{(1+d_u)K_n^2B_n^2}{n}\log(npB^2_nK^2_n) .\\
\end{array}
$$
using the envelope $M = 4K^2_nB^2_n$, $v= C'$, $a=pn$ and a constant $\sigma$.

Consider the first term. By Lemma \ref{lemma:CCK} we have with probability $1-O(1/\log n)$ that $$\begin{array}{rl}
\displaystyle  \sup_{d_{\mathcal{U}}(u,u')\leq 1/n}\|\En[f(X)(\zeta_u - \zeta_{u'})]\|_\infty & =  \displaystyle  \sup_{d_{\mathcal{U}}(u,u')\leq 1/n} \frac{1}{\sqrt{n}}\max_{j\leq p}|\Gn( f_j(X)(\zeta_u-\zeta_{u'}))|\\
 & \lesssim \frac{1}{\sqrt{n}}\sqrt{\frac{(1+d_u) L_n\log(pnK_nB_n \frac{n^\nu}{L_n})}{n^\nu}} + \frac{(1+d_u)K_nB_n\log(pnK_nB_n\frac{n^\nu}{L_n})}{n}  \end{array}
 $$
using the envelope $F = 2K_nB_n$, $v= C'$, $a=pn$, the entropy bound in Lemma \ref{Lemma:PartialOutCovering}, and $  \sigma^2 \propto L_nn^{-\nu} \leq F^2$ for all $n$ sufficiently large, because   $L_nn^{-\nu} \searrow 0$ and
 $$\begin{array}{rl}
\displaystyle \sup_{d_{\mathcal{U}}(u,u')\leq 1/n} \max_{j\leq p}\Ep_P[ f_j(X)^2(\zeta_u-\zeta_{u'})^2]&\displaystyle \leq \sup_{d_{\mathcal{U}}(u,u')\leq 1/n} \max_{j\leq p}\Ep_P[ f_j(X)^2(Y_u-Y_{u'})^2] \\
 & \displaystyle\leq  \sup_{d_{\mathcal{U}}(u,u')\leq 1/n}L_n|u-u'|^\nu \max_{j\leq p}\Ep_P[ f_j(X)^2]\leq CL_nn^{-\nu}.\end{array}$$

To bound the second term in the statement of the lemma, it follows that
\begin{equation}\label{Eq:Split1}\begin{array}{rl}
 \displaystyle \sup_{d_{\mathcal{U}}(u,u')\leq 1/n}\max_{j\leq p}\En[f_j(X)^2(\zeta_u - \zeta_{u'})^2] 
 &\displaystyle = \sup_{d_{\mathcal{U}}(u,u')\leq 1/n}\max_{j\leq p}\En[f_j(X)^2(\Ep_P[Y_u - Y_{u'}\mid X] )^2]  \\
& \displaystyle  \leq  \sup_{d_{\mathcal{U}}(u,u')\leq 1/n}\max_{j\leq p}\En[f_j(X)^2\Ep_P[|Y_u - Y_{u'}|^2\mid X]]\\
& \displaystyle \leq  \max_{j\leq p}\En[f_j(X)^2]\sup_{d_{\mathcal{U}}(u,u')\leq 1/n}L_n|u-u'|^\nu\\
\end{array}
\end{equation}
where the first inequality holds by Jensen's inequality, and the second inequality holds by assumption. Since  $c\leq \max_{j\leq p}\{\Ep_P[f_j(X)^2]\}^{1/2} \leq C$, the result follows by Lemma  \ref{lemma:CCK} which yields with probability $1-O(1/\log n)$
\begin{equation}\label{Boundfj2}\begin{array}{rl} \max_{j\leq p}|(\En-\Ep_P)[f_j(X)^2]| & \lesssim\displaystyle   \sqrt{\frac{\log (p nK_n^2)}{n}} + \frac{K_n^2}{n}\log( pnK_n^2 ),\\
\end{array} \end{equation}  where we used the choice $C \leq \sigma  = C' \leq F = K_n^2$, $v=C$, $a = pn$.\end{proof}

\subsection{Proofs for Lasso with Functional Response: Linear Case}

\begin{proof}[Proof of Lemma \ref{Thm:RateEstimatedLasso}]
Let $\hat\delta_u = \hat\theta_u-\theta_{u}$. Throughout the proof we assume that the events $ c_r^2 \geq \sup_{u\in \mathcal{U}}\En[r_{u}^2]$, $\lambda/n \geq c\sup_{u\in\mathcal{U}}\|\hat  \Psi^{-1}_{u0}\En[\zeta_uf(X)]\|_\infty$ and $\ell \widehat \Psi_{u0} \leq \widehat \Psi_u \leq L\widehat \Psi_{u0}$ occur.

By definition of $\hat\theta_u$,
$$\hat\theta_u \in \arg\min_{\theta\in\RR^p} \En[(Y_u-f(X)'\theta)^2]+\frac{2\lambda}{n}\|\hat  \Psi_u \theta\|_1,$$
and $\ell \widehat \Psi_{u0} \leq \widehat \Psi_u \leq L\widehat \Psi_{u0}$, we have
\begin{equation}\label{Eq:EstLassoV2}\begin{array}{ll}
& \En[(f(X)'\hat\delta_u)^2] - 2\En[(Y_u-f(X)'\theta_{u})f(X)]'\hat\delta_u \\ & = \En[(Y_u-f(X)'\hat\theta_u)^2] -\En[(Y_u-f(X)'\theta_u)^2] \\
& \leq \frac{2\lambda}{n}\|\hat  \Psi_u \theta_{u}\|_1 -\frac{2\lambda}{n}\|\hat  \Psi_u \hat \theta_u\|_1 \\
 & \leq \frac{2\lambda}{n}\|\hat  \Psi_u\hat\delta_{uT_u}\|_1-\frac{2\lambda}{n}\|\hat  \Psi_u \hat\delta_{uT_u^c}\|_1\\
& \leq \frac{2\lambda}{n}L\|\hat  \Psi_{u 0}\hat\delta_{uT_u}\|_1-\frac{2\lambda}{n}\ell\|\hat  \Psi_{u 0}\hat\delta_{uT_u^c}\|_1.\end{array}\end{equation}
Therefore, by  $ c_r^2 \geq \sup_{u\in \mathcal{U}}\En[r_{u}^2]$ and $\lambda/n \geq c\sup_{u\in\mathcal{U}}\|\hat  \Psi^{-1}_{u0}\En[\zeta_uf(X)]\|_\infty$, we have
 \begin{equation}\label{LassoNewEq}\begin{array}{rl} &\En[(f(X)'\hat\delta_u)^2] \\
 &\leq  2\En[r_{u}f(X)]'\hat\delta_u+2(\hat  \Psi^{-1}_{u0}\En[\zeta_{u}f(X)])'(\hat  \Psi_{u 0}\hat\delta_u)+\frac{2\lambda}{n}L\|\hat  \Psi_{u 0}\hat\delta_{uT_u}\|_1-\frac{2\lambda}{n}\ell\|\hat  \Psi_{u 0}\hat\delta_{uT_u^c}\|_1\\  
& \leq 2c_r \{\En[(f(X)'\hat\delta_u)^2]\}^{1/2}+2\|\hat  \Psi^{-1}_{u0}\En[\zeta_uf(X)]\|_\infty\|\hat  \Psi_{u 0}\hat\delta_u\|_1
 +\frac{2\lambda}{n}L\|\hat  \Psi_{u 0}\hat\delta_{uT_u}\|_1-\frac{2\lambda}{n}\ell\|\hat  \Psi_{u 0}\hat\delta_{uT_u^c}\|_1\\
& \leq 2 c_r \{\En[(f(X)'\hat\delta_u)^2]\}^{1/2}+\frac{2\lambda}{cn}\|\hat  \Psi_{u 0}\hat\delta_u\|_1+\frac{2\lambda}{n}L\|\hat  \Psi_{u 0}\hat\delta_{uT_u}\|_1-\frac{2\lambda}{n}\ell\|\hat  \Psi_{u 0}\hat\delta_{uT_u^c}\|_1\\
& \leq 2c_r \{\En[(f(X)'\hat\delta_u)^2]\}^{1/2}+\frac{2\lambda}{n}\left(L+\frac{1}{c}\right)\|\hat  \Psi_{u 0}\hat\delta_{uT_u}\|_1-\frac{2\lambda}{n}\left(\ell-\frac{1}{c}\right)\|\hat  \Psi_{u 0}\hat\delta_{uT_u^c}\|_1.\\
\end{array}\end{equation}
Let $$\tilde \cc := \frac{cL+1}{c\ell-1}\sup_{u\in\mathcal{U}}\|\hat\Psi_{u 0}\|_\infty\|\hat\Psi_{u 0}^{-1}\|_\infty.$$ Therefore if $\hat \delta_u \not\in \Delta_{\tilde \cc,u} = \{ \delta \in \RR^p : \|\delta_{T_u^c}\|_1\leq \tilde \cc\|\delta_{T_u}\|_1\}$, we have that $ \left(L+\frac{1}{c}\right)\|\hat  \Psi_{u 0}\hat\delta_{uT_u}\|_1\leq \left(\ell-\frac{1}{c}\right)\|\hat  \Psi_{u 0}\hat\delta_{uT_u^c}\|_1$ so that
$$ \{\En[(f(X)'\hat\delta_u)^2]\}^{1/2} \leq 2c_r.$$

Otherwise assume $\hat\delta_u \in \Delta_{\tilde \cc,u}$. In this case (\ref{LassoNewEq}), the definition of $
\kappa_{\tilde \cc}$, and $\|\hat\delta_{uT_u}\|_1 \leq \sqrt{s}\| \hat\delta_{uT_u}\|$, we have
{\small $$\begin{array}{rl}\En[(f(X)'\hat\delta_u)^2]
& \leq 2c_r \{\En[(f(X)'\hat\delta_u)^2]\}^{1/2}+\frac{2\lambda}{n}\left(L+\frac{1}{c}\right)\|\widehat\Psi_{u0}\|_\infty\sqrt{s}\{\En[(f(X)'\hat\delta_u)^2]\}^{1/2}/\kappa_{\tilde \cc}\\
\end{array}$$}
which implies
\begin{equation}\label{Implies}  \{\En[(f(X)'\hat\delta_u)^2]\}^{1/2} \leq 2  c_r + \frac{2\lambda\sqrt{s}}{n\kappa_{\tilde \cc}}\left(L+\frac{1}{c}\right)\|\widehat\Psi_{u0}\|_\infty.\end{equation}
To establish the $\ell_1$-bound, first assume that $\hat\delta_u \in \Delta_{2\tilde \cc,u}$. In that case
$$\begin{array}{rl}
 \|\hat\delta_u\|_1 & \leq (1+2\tilde\cc)\|\hat\delta_{uT_u}\|_1 \leq (1+2\tilde\cc)\sqrt{s}\{\En[(f(X)'\hat\delta_u)^2]\}^{1/2}/\kappa_{2\tilde\cc} \\
 & \leq (1+2\tilde\cc)\left\{ 2 \frac{\sqrt{s} \ c_r}{\kappa_{2\tilde\cc}} + \frac{2\lambda s }{n\kappa_{\tilde \cc}\kappa_{2\tilde\cc}}\left(L+\frac{1}{c}\right)\|\widehat\Psi_{u0}\|_\infty\right\}\end{array}$$
where we used that $\|\hat\delta_{uT_u}\|_1\leq \sqrt{s}\|\hat\delta_{uT_u}\|$, the definition of the restricted eigenvalue, and the prediction rate derived in (\ref{Implies}).

Otherwise note that $\hat\delta_u \not\in \Delta_{2\tilde \cc,u}$ implies that $ \left(L+\frac{1}{c}\right)\|\hat  \Psi_{u 0}\hat\delta_{uT_u}\|_1\leq \frac{1}{2} \left(\ell-\frac{1}{c}\right)\|\hat  \Psi_{u 0}\hat\delta_{uT_u^c}\|_1$ so that (\ref{LassoNewEq}) yields
$$\frac{1}{2}\frac{2\lambda}{n} \left(\ell-\frac{1}{c}\right)\|\hat  \Psi_{u 0}\hat\delta_{uT_u^c}\|_1 \leq \{\En[(f(X)'\hat\delta_u)^2]\}^{1/2} \left( 2 c_r-\{\En[(f(X)'\hat\delta_u)^2]\}^{1/2}\right) \leq c_r^2$$
where we used that $\max_t t(2c_r-t)\leq c_r^2$. Therefore
$$\|\hat\delta_u\|_1 \leq  \left(1+\frac{1}{2\tilde\cc}\right)\|\hat\delta_{uT_u^c}\|_1 \leq \left(1+\frac{1}{2\tilde\cc}\right)\|\hat\Psi_{u 0}^{-1}\|_\infty\|\hat\Psi_{u 0}\hat\delta_{uT_u^c}\|_1\leq \left(1+\frac{1}{2\tilde\cc}\right)\frac{c\|\hat\Psi_{u 0}^{-1}\|_\infty}{\ell c-1}\frac{n}{\lambda} c_r^2.$$
\end{proof}

\begin{proof}[Proof of Lemma \ref{Thm:Sparsity}]
Step 1. Let $L_u = 4c_0\|\widehat \Psi_{u0}^{-1}\|_\infty \[\frac{n c_r}{\lambda} + \frac{\sqrt{s}}{\kappa_{\tilde \cc}}\|\widehat\Psi_{u0}\|_\infty\].$
By Step 2 below and the definition of $L_u$ we have
\begin{equation}\label{Eq:Sparsity}\hat s_u \leq \semax{\hat s_u}L_u^2.\end{equation}
  Consider any $M \in \mathcal{M}=\{ m \in \mathbb{N}: m >  2\semax{m}\sup_{u\in\mathcal{U}}L_u^2\} $, and suppose $\widehat s_u > M$.

Next recall the sublinearity of the maximum sparse eigenvalue (for a proof see Lemma 3 in \citen{BC-PostLASSO}), namely, for any integer $k \geq 0$ and constant $\ell \geq 1$ we have $\semax{\ell k}  \leq  \lceil \ell \rceil \semax{k}, $
where $\lceil \ell \rceil$ denotes the  ceiling of $\ell$. Therefore
$$ \hat s_u  \leq \semax{M\hat s_u/M}L_u^2 \leq  \ceil{\frac{\hat s_u}{M}}\semax{M}L_u^2.$$ Thus, since $\ceil{k}\leq 2k$ for any $k\geq 1$  we have
$ M \leq   2\semax{M}L_u^2$
which violates the condition that $M \in \mathcal{M}$. Therefore, we have $\widehat s_u \leq M$.

In turn, applying (\ref{Eq:Sparsity}) once more with $\widehat s_u \leq M$ we obtain
 $ \hat s_u \leq   \semax{M}L_u^2.$ The result follows by minimizing the bound over $M \in \mathcal{M}$.

Step 2. In this step we establish that uniformly over $u\in\mathcal{U}$
$$\sqrt{\hat s_u} \leq 4\sqrt{\semax{\hat s_u}} \|\widehat \Psi_{u0}^{-1}\|_\infty c_0\[\frac{n c_r}{\lambda} + \frac{\sqrt{s}}{\kappa_{\tilde \cc}}\|\widehat\Psi_{u0}\|_\infty\].$$
Let  $R_u = (r_{u 1},\ldots, r_{u n})'$, $\bold{Y}_u = (Y_{u1},\ldots,Y_{un})'$, $\bar\zeta_u=(\zeta_{u1},\ldots,\zeta_{un})'$, and $F=[ f(X_1);\ldots; f(X_n)]'$. We have from the optimality conditions that the Lasso estimator $\widehat \theta_u $ satisfies  $$\En[ \widehat \Psi_{ujj}^{-1}f_j(X)(Y_{u}-f(X)'\hat\theta_u)] = \sign(\hat\theta_{u j})\lambda/n \ \text{ for each } \ j \in \widehat T_u.  $$

Therefore, noting that $\|\widehat \Psi^{-1}_u\widehat \Psi_{u0}\|_\infty \leq 1/\ell$, we have
 \begin{eqnarray*}
   &  &\sqrt{\hat s_u} \lambda    =   \| (\widehat \Psi^{-1}_u F'(\bold{Y}_u - F\hat \theta_u))_{\widehat T_u} \| \\
            &  & \leq   \| (\widehat \Psi_u^{-1}F'\bar\zeta_u)_{\widehat T_u} \| + \| (\widehat \Psi^{-1}_u F'R_u)_{\widehat T_u} \| + \| (\widehat \Psi_u^{-1}  F'F(\theta_{u}-\hat \theta_u))_{\widehat T_u} \| \\
            & &  \leq  \sqrt{\hat s_u}\ \|\widehat \Psi_u^{-1}\widehat \Psi_{u0}\|_\infty\|\widehat \Psi_{u 0}^{-1}F' \bar\zeta_u\|_{\infty} +  n\sqrt{\semax{\hat s_u}}\| \widehat \Psi_u^{-1}\|_\infty c_r+ \\
             & & n \sqrt{ \semax{\hat s_u}} \|\widehat \Psi^{-1}_u\|_\infty\| F(\hat \theta_u - \theta_{u}) \|_{\Pn,2},\\
            & & \leq \sqrt{\hat s_u}\ (1/\ell) \ \|\widehat \Psi_{u0}^{-1}F'\bar\zeta_u\|_{\infty} + n \sqrt{ \semax{\hat s_u}} \frac{\|\widehat \Psi^{-1}_{u0}\|_\infty}{\ell}\{ c_r+\|F(\hat \theta_u- \theta_{u}) \|_{\Pn,2}\},
\end{eqnarray*}
where we used that $\|v\|\leq \|v\|_0^{1/2}\|v\|_\infty$ and
\begin{align*}
& \| (F'F(\theta_u-\hat \theta_{u}))_{\widehat T_u} \| \\
& \leq \sup_{\|\delta\|_0\leq \hat s_u, \|\delta\|\leq 1}|
\delta' F'F(\theta_u-\hat \theta_{u})| \leq
\sup_{\|\delta\|_0\leq \hat s_u, \|\delta\|\leq 1}\| \delta'F'\|\| F(\theta_u-\hat \theta_{u})\| \\
& \leq  \sup_{\|\delta\|_0\leq \hat
s_u, \|\delta\|\leq 1}\{ \delta'F'F \delta\}^{1/2}\|F(\theta_u-\hat \theta_{u})\| \leq
n\sqrt{\semax{\hat s_u}}\|f(X)'(\theta_{u}-\hat \theta_u)\|_{\Pn,2}.
\end{align*}
Since $\lambda/c \geq \sup_{u\in\mathcal{U}} \|\widehat \Psi_{u 0}^{-1}F'\bar\zeta_u\|_{\infty}$, and by Lemma \ref{Thm:RateEstimatedLasso}, we have that the estimate $\widehat\theta_u$ satisfies $\|f(X)'(\widehat\theta_u-\theta_u)\|_{\Pn,2} \leq 2c_r+2\(L + \frac{1}{c}\) \frac{\lambda \sqrt{s}}{n \kappa_{\tilde \cc}}\|\widehat\Psi_{u0}\|_\infty$ so that
\begin{align*}
 \sqrt{\hat s_u} & \leq \frac{\sqrt{\semax{\hat s_u}}\frac{\|\widehat \Psi_{u0}^{-1}\|_\infty}{\ell}\[\frac{3n c_r}{\lambda} +3\(L+\frac{1}{c}\)\frac{\sqrt{s}}{\kappa_{\tilde \cc}}\|\widehat\Psi_{u0}\|_\infty\]}{\(1-\frac{1}{c\ell}\)}\\
& \leq 4\frac{\(L+\frac{1}{c}\)}{\(1-\frac{1}{c\ell}\)}\frac{1}{\ell}\sqrt{\semax{\hat s_u}}\|\widehat \Psi_{u0}^{-1}\|_\infty\[\frac{n c_r}{\lambda}+\frac{\sqrt{s}}{\kappa_{\tilde \cc}}\|\widehat\Psi_{u0}\|_\infty\].
 \end{align*}

The result follows by noting that $(L+[1/c])/(1-1/[\ell c]) = c_0 \ell$ by definition of $c_0$.

\end{proof}

\begin{proof}[Proof of Lemma \ref{Thm:2StepMain}]

 Define $m_{u}:=(\Ep[Y_{u1}\mid X_1],\ldots,\Ep[Y_{un}\mid X_n])'$, $\bar \zeta_u:=(\zeta_{u1},\ldots,\zeta_{un})'$, and the $n\times p$ matrix $F:=[ f(X_1);\ldots; f(X_n)]'$. For a set of indices $S \subset \{1,\ldots,p\}$ we define  $\widehat{P}_{S} = F[S](F[S]'F[S])^{-1}F[S]'$ denote the projection matrix on the columns associated with the indices in $S$ where we interpret $\widehat{P}_{S}$ as a null operator if $S$ is empty.

   Since $Y_{ui} = m_{ui} + \zeta_{ui}$ we have  $$m_u - F\widetilde\theta_u = ( I - \widehat{P}_{\widetilde T_u})m_u - \widehat {P}_{\widetilde T_u} \bar\zeta_u$$ where $I$ is the identity operator. Therefore
\begin{equation}\label{eqPL}
\| m_u - F\widetilde\theta_u \| \leq \|( I - \widehat{P}_{\widetilde T_u}) m_u \| +\|\widehat {P}_{\widetilde T_u}  \bar\zeta_u\|.
\end{equation}
Since $\|F[\widetilde T_u]/\sqrt{n}( F[\widetilde T_u]' F[\widetilde T_u]/n)^{-1} \| \leq \sqrt{1/\semin{\tilde s_u}}$,  the last term in (\ref{eqPL}) satisfies
$$\begin{array}{rl}
\|\widehat{P}_{\widetilde T_u}\bar\zeta_u \|  & \leq  \sqrt{1/\semin{\tilde s_u}} \|F[\widetilde T_u]'\bar\zeta_u/\sqrt{n}\| \\
 & \leq \sqrt{1/\semin{\tilde s_u}} \sqrt{\tilde s_u}\|F'\bar\zeta_u/\sqrt{n}\|_\infty.\end{array}$$

By Lemma \ref{Thm:ChoiceLambda} with $\gamma = 1/n$, we have that with probability $1-o(1)$, uniformly in $u\in\mathcal{U}$
$$  \|F'\bar\zeta_u/\sqrt{n}\|_\infty \leq C \sqrt{ \log (p\vee n^{\dn+1})} \max_{1\leq j\leq p} \sqrt{\En[f_j(X)^2\zeta_{u}^2]} = C \sqrt{ \log (p\vee n^{\dn+1})}\|\widehat\Psi_{u0}\|_\infty.$$
The result follows.

The last statement follows from noting that the mean square approximation error provides an upper bound to the best mean square approximation error based on the model $\widetilde T_u$ provided that the model include the Lasso's mode, i.e. $\widehat T_u\subseteq\widetilde T_u$. Indeed, we have
\begin{align*}\displaystyle \sup_{u\in\mathcal{U}}\min_{\supp(\theta)\subseteq\widetilde T_u}  \|\Ep_P[Y_u\mid X]-f(X)'\theta\|_{\Pn,2} & \leq \sup_{u\in\mathcal{U}}\min_{\supp(\theta)\subseteq\widehat T_u}  \|\Ep_P[Y_u\mid X]-f(X)'\theta\|_{\Pn,2}\\
& \leq \sup_{u\in\mathcal{U}}  \|\Ep_P[Y_u\mid X]-f(X)'\hat\theta_u\|_{\Pn,2}\\
& \leq c_r + \sup_{u\in\mathcal{U}}  \|f(X)'\theta_u-f(X)'\hat\theta_u\|_{\Pn,2}\\
 & \leq 3 c_r +  \(L + \frac{1}{c}\) \frac{2\lambda \sqrt{s}}{n \kappa_{\tilde \cc}}\sup_{u\in\mathcal{U}}\|\widehat\Psi_{u0}\|_{\infty}
 \end{align*}
where we invoked Lemma \ref{Thm:RateEstimatedLasso} to bound $\|f(X)'(\widehat\theta_u-\theta_u)\|_{\Pn,2}$.\end{proof}


\subsection{Proofs for Lasso with Functional Response: Logistic Case}

\begin{proof}[Proof of Lemma \ref{Lemma:LassoLogisticRateRaw}]
Let $\delta_u = \hat\theta_u - \theta_u$ and $S_u = -\En[f(X)\zeta_u]$. By definition of $\hat \theta_u$ we have
$
M_u(\hat \theta_u) + \frac{\lambda}{n}\|\widehat\Psi_u\hat\theta_u\|_1  \leq M_u(\theta_u) + \frac{\lambda}{n}\|\widehat\Psi_u\theta_u\|_1$. Thus,
\begin{equation}\label{QQ:UpperBound}
\begin{array}{rl}
M_u(\hat \theta_u) - M_u(\theta_u) & \leq \frac{\lambda}{n}\|\widehat\Psi_u\theta_u\|_1 - \frac{\lambda}{n}\|\widehat\Psi_u\hat\theta_u\|_1 \\
& \leq \frac{\lambda}{n}\|\widehat\Psi_u\delta_{u,T_u}\|_1 - \frac{\lambda}{n} \|\widehat\Psi_u\delta_{u,T_u^c}\|_1 \leq \frac{\lambda L}{n}\|\widehat\Psi_{u0}\delta_{u,T_u}\|_1 - \frac{\lambda \ell}{n} \|\widehat\Psi_{u0}\delta_{u,T_u^c}\|_1.\\
\end{array}
\end{equation}

Moreover, by convexity of $M_u(\cdot)$ and H\"older's inequality we have \begin{equation}\label{QQ:LowerBound}
\begin{array}{rl}
& M_u(\hat \theta_u) - M_u(\theta_u)  \geq \partial_{\theta} M_u(\theta_u) \geq -\frac{\lambda}{n}\frac{1}{c}\|\widehat\Psi_{u0}\delta_{u}\|_1 -\|r_{u}/\sqrt{w_{u}}\|_{\Pn,2}\|\sqrt{w_{u}} f(X)'\delta_u\|_{\Pn,2}
\end{array}
\end{equation}
because \begin{eqnarray}
| \partial_{\theta} M_u(\theta_u)'\delta_u|  & =&  |S_u'\delta_u + \{\partial_{\theta} M_u(\theta_u)-S_u\}'\delta_u|   \leq    |S_u'\delta_u| + |\{\partial_{\theta} M_u(\theta_u)-S_u\}'\delta_u|  \nonumber \\
& \leq &  \|\widehat\Psi_{u0}^{-1}S_u \|_\infty \|\widehat\Psi_{u0}\delta_u\|_1+ \|r_{u}/\sqrt{w_{u}}\|_{\Pn,2}\|\sqrt{w_{u}}  f(X)'\delta_u\|_{\Pn,2} \nonumber \\
& \leq & \frac{\lambda}{n}\frac{1}{c}\|\widehat\Psi_{u0}\delta_{u}\|_1 + \|r_{u}/\sqrt{w_{u}}\|_{\Pn,2}\|\sqrt{w_{u}} f(X)'\delta_u\|_{\Pn,2},
\label{useful}
\end{eqnarray}
where we used that $\lambda/n\geq c\sup_{u\in\mathcal{U}}\|\widehat\Psi_{u0}^{-1}S_u \|_\infty$ and that $\partial_\theta M_u(\theta_u)=\En[\{\zeta_u+r_u\}f(X)]$ so that
\begin{eqnarray}
& & |\{\partial_\theta M_u(\theta_u)-S_u\}'\delta_u| = |\En[r_uf(X)'\delta_u]|\leq \|r_{u}/\sqrt{w_{u}}\|_{\Pn,2}\|\sqrt{w_{u}} f(X)'\delta_u\|_{\Pn,2}.
\end{eqnarray}

Combining (\ref{QQ:UpperBound}) and (\ref{QQ:LowerBound}) we have
\begin{equation}\label{Eq:92b}\frac{\lambda}{n}\frac{c\ell-1}{c}\| \widehat\Psi_{u0} \delta_{u,T_u^c}\|_1 \leq \frac{\lambda}{n}\frac{Lc+1}{c}\| \widehat\Psi_{u0}\delta_{u,T_u}\|_1+\|r_{u}/\sqrt{w_{u}}\|_{\Pn,2}\|\sqrt{w_{u}} f(X)'\delta_u\|_{\Pn,2} \end{equation}
and for $\tilde \cc = \frac{Lc+1}{\ell c -1}\sup_{u\in\mathcal{U}}\|\widehat\Psi_{u0}\|_\infty\|\widehat\Psi_{u0}^{-1}\|_\infty \geq 1$ we have $$ \|\delta_{u,T_u^c}\|_1 \leq \tilde \cc \|\delta_{u,T_u}\|_1 + \frac{n}{\lambda}\frac{c\|\widehat\Psi_{u0}^{-1}\|_\infty}{\ell c - 1} \|r_{u}/\sqrt{w_{u}}\|_{\Pn,2}\|\sqrt{w_{u}} f(X)'\delta_u\|_{\Pn,2}.$$
Suppose $\delta_u \not\in \Delta_{2\tilde\cc,u}$, namely $\|\delta_{u,T_u^c}\|_1\geq 2\tilde \cc\|  \delta_{u,T_u}\|_1$. Thus,
$$
\begin{array}{rl}
\| \delta_u\|_1 & \leq ( 1 +
\{2\tilde \cc\}^{-1}) \| \delta_{u,T_u^c}\|_1\\
& \leq ( 1 + \{2\tilde \cc\}^{-1})\tilde \cc\| \delta_{u,T_u}\|_1+( 1 + \{2\tilde \cc\}^{-1})\frac{n}{\lambda}\frac{c\|\widehat\Psi_{u0}^{-1}\|_\infty}{\ell c - 1} \|r_{u}/\sqrt{w_{u}}\|_{\Pn,2}\|\sqrt{w_{u}} f(X)'\delta_u\|_{\Pn,2}\\
& \leq ( 1 +
\{2\tilde \cc\}^{-1})\frac{1}{2}\| \delta_{u,T_u^c}\|_1+( 1 + \{2\tilde \cc\}^{-1})\frac{n}{\lambda}\frac{c\|\widehat\Psi_{u0}^{-1}\|_\infty}{\ell c - 1} \|r_{u}/\sqrt{w_{u}}\|_{\Pn,2}\|\sqrt{w_{u}} f(X)'\delta_u\|_{\Pn,2}.\\
\end{array}
$$
The relation above implies that if $\delta_u \not\in \Delta_{2\tilde\cc,u}$
\begin{equation}\label{BoundL1Zero}
\begin{array}{rl}
\| \delta_u\|_1 & \leq \frac{4\tilde \cc}{2\tilde \cc-1}( 1 + \{2\tilde \cc\}^{-1})\frac{n}{\lambda}\frac{c\|\widehat\Psi_{u0}^{-1}\|_\infty}{\ell c - 1} \|r_{u}/\sqrt{w_{u}}\|_{\Pn,2}\|\sqrt{w_{u}} f(X)'\delta_u\|_{\Pn,2}\\
& \leq \frac{6c\|\widehat\Psi_{u0}^{-1}\|_\infty}{\ell c - 1}\frac{n}{\lambda}\|r_{u}/\sqrt{w_{u}}\|_{\Pn,2}\|\sqrt{w_{u}} f(X)'\delta_u\|_{\Pn,2}=: I_u,  \\
\end{array}
 \end{equation} where we used that $\frac{4\tilde \cc}{2\tilde \cc-1}( 1 + \{2\tilde \cc\}^{-1})\leq 6$ since $\tilde \cc \geq 1$.
Combining the bound with the bound $$\|\delta_{u,T_u}\|_1 \leq \frac{\sqrt{s}}{\bar\kappa_{2\tilde \cc}} \|\sqrt{w_{u}} f(X)'\delta_u\|_{\Pn,2} =: II_u,  \  \text{ if } \delta_u\in\Delta_{2\tilde\cc,u}, $$ we have that $\delta_u$ satisfies \begin{equation}\label{BoundFirstL1}\|\delta_{u,T_u}\|_1 \leq I_u + II_u. \end{equation}

For every $u\in\mathcal{U}$, since $A_u = \Delta_{2\tilde \cc,u} \cup \{ \delta : \|\delta\|_1 \leq \frac{6c\|\widehat\Psi_{u0}^{-1}\|_\infty}{\ell c - 1}\frac{n}{\lambda}\|r_{u}/\sqrt{w_{u}}\|_{\Pn,2}\|\sqrt{w_{u}} f(X)'\delta\|_{\Pn,2}\}$, it follows that $\delta_u\in A_u$, and we have
$$ \begin{array}{l}
\frac{1}{3}\|\sqrt{w_{u}} f(X)'\delta_u\|_{\Pn,2}^2 \wedge \left\{ \frac{\bar q_{A_u}}{3}\|\sqrt{w_{u}} f(X)'\delta_u\|_{\Pn,2}\right\} \\
 \leq_{(1)} M_u(\hat \theta_u) - M_u(\theta_u) -\partial_{\theta} M_u(\theta_u)'\delta_u + 2\|\tilde r_{u}/\sqrt{w_{u}}\|_{\Pn,2}\|\sqrt{w_{u}} f(X)'\delta_u\|_{\Pn,2} \\
 \leq_{(2)} (L+\frac{1}{c})\frac{\lambda}{n}\|\widehat\Psi_{u0}\delta_{u,T_u}\|_1  +3\|\tilde r_{u}/\sqrt{w_{u}}\|_{\Pn,2}\|\sqrt{w_{u}} f(X)'\delta_u\|_{\Pn,2} \\
 \leq_{(3)} (L+\frac{1}{c})\|\widehat\Psi_{u0}\|_\infty\frac{\lambda}{n}\left\{I_u + II_u \right\}+3\|\tilde r_{u}/\sqrt{w_{u}}\|_{\Pn,2}\|\sqrt{w_{u}} f(X)'\delta_u\|_{\Pn,2} \\
\leq_{(4)} \left\{(L+\frac{1}{c})\|\widehat\Psi_{u0}\|_\infty\frac{\lambda\sqrt{s}}{n\bar\kappa_{2\tilde \cc}}+ 9\tilde \cc\|\tilde r_{u}/\sqrt{w_{u}}\|_{\Pn,2}\right\}\|\sqrt{w_{u}} f(X)'\delta_u\|_{\Pn,2},\\
\end{array}$$
where $(1)$ follows by Lemma \ref{Lemma:Minoration} with $A_u$, $(2)$ follows from  (\ref{useful}) and $|r_{ui}|\leq |\tilde r_{ui}|$, $(3)$ follows by $\|\widehat\Psi_{u0}\delta_{u,T_u}\|_1\leq \|\widehat\Psi_{u0}\|_\infty \|\delta_{u,T_u}\|_1$ and  (\ref{BoundFirstL1}),  $(4)$ follows from simplifications and $|r_{ui}|\leq |\tilde r_{ui}|$.
Since  the inequality $(x^2 \wedge a x) \leq b x$ holding for $x>0$ and $b< a <0$ implies $x \leq b$, the above system of the inequalities, provided that for every $u\in\mathcal{U}$ $$\bar q_{A_u}>3\left\{(L+\frac{1}{c})\|\widehat\Psi_{u0}\|_\infty\frac{\lambda\sqrt{s}}{n\bar\kappa_{2\tilde \cc}}+ 9\tilde \cc\|\tilde r_{u}/\sqrt{w_{u}}\|_{\Pn,2}\right\},$$
 implies that
$$ \|\sqrt{w_{u}} f(X)'\delta_u\|_{\Pn,2} \leq 3\left\{(L+\frac{1}{c})\|\widehat\Psi_{u0}\|_\infty\frac{\lambda\sqrt{s}}{n\bar\kappa_{2\tilde \cc}}+ 9\tilde \cc\|\tilde r_{u}/\sqrt{w_{u}}\|_{\Pn,2}\right\} =: III_u \ \ \ \mbox{for every} \ \ u\in \mathcal{U}.$$
The second result follows from the definition of $\bar\kappa_{2\tilde \cc}$, (\ref{BoundL1Zero}) and the bound on $\|\sqrt{w_{u}} f(X)'\delta_u\|_{\Pn,2} $ just derived, namely for every $u\in\mathcal{U}$ we have
$$\begin{array}{rl}
\|\delta_u\|_1 & \leq 1\{ \delta_u \in \Delta_{2\tilde\cc,u} \}\|\delta_u\|_1 + 1\{ \delta_u \notin \Delta_{2\tilde\cc,u} \}\|\delta_u\|_1 \\
 & \leq (1+2\tilde\cc) II_u + I_u \leq   3\left\{ \frac{(1+2\tilde\cc)\sqrt{s}}{\bar\kappa_{2\tilde \cc}}+ \frac{6c\|\widehat\Psi_{u0}^{-1}\|_\infty}{\ell c - 1}\frac{n}{\lambda}\left\|\frac{ r_{u}}{\sqrt{w_{u}}}\right\|_{\Pn,2}\right\} III_u \end{array}$$

\end{proof}

\begin{proof}[Proof of Lemma \ref{Lemma:LassoLogisticSparsity}]
The proof of both bounds are similar to the proof of sparsity for the linear case (Lemma \ref{Thm:Sparsity}) differing only on the definition of $L_u$ which are a consequence of pre-sparsity bounds established in Step 2 and Step 3.

Step 1. To establish the first bound by Step 2 below, triangle inequality and the definition of $\psi(A_u)$ we have
$$\begin{array}{rl}
 \sqrt{\hat s_u} & \leq  \frac{c(n/\lambda)}{(c\ell-1)}\sqrt{\semax{\hat s_u}}\| f(X)'(\hat\theta_u-\theta_u)-r_{u}\|_{\Pn,2} \\
 &  \leq \frac{c(n/\lambda)}{(c\ell-1)}\sqrt{\semax{\hat s_u}}\left\{\frac{\| \sqrt{w_u}f(X)'(\hat\theta_u-\theta_u)\|_{\Pn,2}}{\psi(A_u)}+\|r_{u}\|_{\Pn,2}\right\}\end{array}$$
uniformly in $u\in\mathcal{U}$. By Lemma \ref{Lemma:LassoLogisticRateRaw}, $\psi(A_u)\leq 1$  and  $\|r_{u}\|_{\Pn,2}\leq \|\tilde r_{u}/\sqrt{w_{u}}\|_{\Pn,2}$ we have
$$\begin{array}{rl} \sqrt{\hat s_u}& \leq \sqrt{\semax{\hat s_u}}\frac{c(n/\lambda)}{(c\ell-1)\psi(A_u)}\left\{ 3(L+\frac{1}{c})\|\hat\Psi_{u0}\|_\infty\frac{(\lambda/n)\sqrt{s}}{\bar\kappa_{2\tilde \cc}}   +28\tilde\cc\|\tilde r_{u}/\sqrt{w_u}\|_{\Pn,2}\right\}\\
& \leq \sqrt{\semax{\hat s_u}}\frac{c_0}{\psi(A_u)}\left\{ 3\|\hat\Psi_{u0}\|_\infty\frac{\sqrt{s}}{\bar\kappa_{2\tilde \cc}}   +28\tilde\cc\frac{n\|\tilde r_{u}/\sqrt{w_u}\|_{\Pn,2}}{\lambda}\right\}\end{array}$$

Let $L_u = \frac{c_0}{\psi(A_u)}\left\{ 3\|\hat\Psi_{u0}\|_\infty\frac{\sqrt{s}}{\bar\kappa_{2\tilde \cc}}   +28\tilde\cc\frac{n\|\tilde r_{u}/\sqrt{w_u}\|_{\Pn,2}}{\lambda}\right\}.$ Thus
 we have
\begin{equation}\label{Eq:SparsityL}\hat s_u \leq \semax{\hat s_u}L_u^2.\end{equation}
 which has the same structure as (\ref{Eq:Sparsity}) in the Step 1 of the proof of Lemma \ref{Thm:Sparsity}.

 Consider any $M \in \mathcal{M}=\{ m \in \mathbb{N}: m >  2\semax{m}\sup_{u\in\mathcal{U}}L_u^2\} $, and suppose $\widehat s_u > M$. By the sublinearity of the maximum sparse eigenvalue (Lemma 3 in \citen{BC-PostLASSO}),  for any integer $k \geq 0$ and constant $\ell \geq 1$ we have $\semax{\ell k}  \leq  \lceil \ell \rceil \semax{k}, $
where $\lceil \ell \rceil$ denotes the  ceiling of $\ell$. Therefore
$$ \hat s_u  \leq \semax{M\hat s_u/M}L_u^2 \leq  \ceil{\frac{\hat s_u}{M}}\semax{M}L_u^2.$$ Thus, since $\ceil{k}\leq 2k$ for any $k\geq 1$  we have
$ M \leq   2\semax{M}L_u^2$
which violates the condition that $M \in \mathcal{M}$. Therefore, we have $\widehat s_u \leq M$.
In turn, applying (\ref{Eq:SparsityL}) once more with $\widehat s_u \leq M$ we obtain
 $ \hat s_u \leq   \semax{M}L_u^2.$ The result follows by minimizing the bound over $M \in \mathcal{M}$.

Next we establish the second bound. By Step 3 below we have
$$ \sqrt{\hat s_u} \leq \frac{2c(n/\lambda)}{(c\ell-1)}\sqrt{\semax{\hat s_u}}\|\sqrt{w_{u}}\{ f(X)'(\hat\theta_u-\theta_u)-\tilde r_{u}\}\|_{\Pn,2}$$
By Lemma \ref{Lemma:LassoLogisticRateRaw} and that $\|\sqrt{w_{u}}\tilde r_{u}\|_{\Pn,2}\leq \|\tilde r_{u}/\sqrt{w_{u}}\|_{\Pn,2}$ we have
$$\begin{array}{rl} \sqrt{\hat s_u}& \leq \sqrt{\semax{\hat s_u}}\frac{2c(n/\lambda)}{(c\ell-1)}\left\{ 3(L+\frac{1}{c})\|\hat\Psi_{u0}\|_\infty\frac{(\lambda/n)\sqrt{s}}{\bar\kappa_{2\tilde \cc}}   +28\tilde\cc\|\tilde r_{u}/\sqrt{w_u}\|_{\Pn,2}\right\}\\
& \leq \sqrt{\semax{\hat s_u}}2c_0\left\{ 3\|\hat\Psi_{u0}\|_\infty\frac{\sqrt{s}}{\bar\kappa_{2\tilde \cc}}   +28\tilde\cc\frac{n\|\tilde r_{u}/\sqrt{w_u}\|_{\Pn,2}}{\lambda}\right\}\end{array}$$

Let $L_u = 2c_0\left\{ 3\|\hat\Psi_{u0}\|_\infty\frac{\sqrt{s}}{\bar\kappa_{2\tilde \cc}}   +28\tilde\cc\frac{n\|\tilde r_{u}/\sqrt{w_u}\|_{\Pn,2}}{\lambda}\right\}.$ Thus
  again we obtained the relation (\ref{Eq:Sparsity}) and the proof follows similarly
  to the Step 1 in the proof of Lemma \ref{Thm:Sparsity}.

Step 2. In this step we show that uniformly over $u\in \mathcal{U}$,
 \begin{equation}\label{SparsityLogisticStep2} \sqrt{\hat s_u} \leq \frac{c(n/\lambda)}{(c\ell-1)}\sqrt{\semax{\hat s_u}}\| f(X)'(\hat\theta_u-\theta_u)-r_{u}\|_{\Pn,2}.\end{equation}

Let $\G_{ui} := \Ep_P[Y_{ui}\mid X_i]$ and $S_u = -\En[ f(X)\zeta_u] = -\En[(Y_{u}-\G_{u}) f(X)]$. Let $\hat T_u = \supp(\hat\theta_u)$, $\hat s_u = \|\hat \theta_u\|_0$, $\delta_u = \hat\theta_u-\theta_u$, and $\hat \G_{ui}= \exp( f(X_i)'\hat\theta_u)/\{1+\exp( f(X_i)'\hat\theta_u)\}$. For any $j\in \hat T_u$ we have $|\En[(Y_{u}-\hat \G_{u})f_j(X)]|=\widehat\Psi_{ujj}\lambda/n$.

Since $\ell \widehat \Psi_{u0} \leq \widehat \Psi_u$ implies $\|\widehat \Psi_u^{-1}\widehat\Psi_{u0}\|_\infty \leq 1/\ell$, the first relation follows from
{\small $$\begin{array}{rl}
\frac{\lambda}{n}\sqrt{\hat s_u} & = \| (\widehat\Psi_{u}^{-1}\En[(Y_{u}-\hat \G_{u})f(X))_{\hat T_u}] \| \\
& \leq \|\widehat\Psi_{u}^{-1}\widehat\Psi_{u0}\|_\infty \|\widehat\Psi_{u0}^{-1} \En[(Y_u- \G_{u})f_{\hat T_u}(X)] \| + \|\widehat\Psi_{u}^{-1}\widehat\Psi_{u0}\|_\infty\|\widehat\Psi_{u0}^{-1}\|_\infty\| \En[(\hat \G_{u}-\G_{u})f_{\hat T_u}(X)] \|\\
& \leq \sqrt{\hat s_u}(1/\ell)\| \widehat\Psi_{u0}^{-1} \En[\zeta_u f(X)] \|_\infty + (1/\ell)\|\widehat\Psi_{u0}^{-1}\|_\infty\sup_{\|\theta\|_0\leq \hat s_u,\|\theta\|=1} \En[ |\hat\G_{u}-\G_{u}| \ | f(X)'\theta|]\\
& \leq \frac{\lambda}{\ell cn}\sqrt{\hat s_u} + \sqrt{\semax{\hat s_u}}(1/\ell)\|\widehat\Psi_{u0}^{-1}\|_\infty\| f(X)'\delta_u-r_{u}\|_{\Pn,2}\\
\end{array} $$}
uniformly in $u\in \mathcal{U}$, where we used that $\G$ is 1-Lipschitz. This relation implies  (\ref{SparsityLogisticStep2}).

Step 3. In this step we show that if $\max_{i\leq n}| f(X_i)'(\hat\theta_u-\theta_u)-\tilde r_{ui}| \leq 1$ we have
 \begin{equation}\label{SparsityLogisticStep3}\sqrt{\hat s_u} \leq \frac{2c(n/\lambda)}{(c\ell-1)}\sqrt{\semax{\hat s_u}}\|\sqrt{w_{u}}\{ f(X)'(\hat\theta_u-\theta_u)-\tilde r_{u}\}\|_{\Pn,2}\end{equation}
Note that uniformly in $u\in\mathcal{U}$, Lemma \ref{Lemma:AuxSparsity}  establishes that $
|\hat \G_{ui} - \G_{ui}| \leq w_{ui} 2| f(X)'\delta_u-\tilde r_{ui}|$ since $\max_{i\leq n}| f(X_i)'\delta_u-\tilde r_{ui}|\leq 1$ is assumed. Thus, combining this bound with the calculations performed in Step 2 we obtain
$$\begin{array}{rl}
\frac{\lambda}{n}\sqrt{\hat s_u} 
& \leq \frac{\lambda}{\ell cn}\sqrt{\hat s_u} + (2/\ell)\|\widehat\Psi_{u0}^{-1}\|_\infty\sqrt{\semax{\hat s_u}}\|\sqrt{w_{u}}\{ f(X)'\delta_u-\tilde r_{u}\}\|_{\Pn,2}\end{array} $$
which implies (\ref{SparsityLogisticStep3}). \end{proof}

\begin{proof}[Proof of Lemma \ref{Lemma:PostLassoLogisticRateRaw}]
Let $\tilde \delta_u=\tilde \theta_u - \theta_u$ and $\tilde t_{u} = \|\sqrt{w_{u}} f(X)'\tilde \delta_u\|_{\Pn,2}$ and $S_u = -\En[f(X)\zeta_u]$.

By Lemma \ref{Lemma:Minoration} with $A_u = \{\delta \in \RR^p : \|\delta\|_0 \leq \tilde s_u+s_u\}$, we have
$$ \begin{array}{rl}
\frac{1}{3}\tilde t_u^2 \wedge \left\{ \frac{\bar q_{A_u}}{3}\tilde t_u\right\} & \leq M_u(\tilde \theta_u) - M_u(\theta_u) -\partial_{\theta} M_u(\theta_u)'\tilde\delta_u + 2\|\tilde r_{u}/\sqrt{w_{u}}\|_{\Pn,2}\tilde t_u \\
& \leq M_u(\tilde \theta_u) - M_u(\theta_u) + \|S_u\|_\infty \|\tilde\delta_u\|_1 + 3\|\tilde r_{u}/\sqrt{w_u}\|_{\Pn,2}\tilde t_u\\
& \leq M_u(\tilde \theta_u) - M_u(\theta_u) + \tilde t_u\left\{ \frac{\sqrt{\tilde s_u+s_u}\|S_u\|_\infty}{\psi_{u}(A_u)\sqrt{\semin{\tilde s_u+s_u}}} + 3\|\tilde r_{u}/\sqrt{w_{u}}\|_{\Pn,2}\right\}.\\
\end{array}$$
where the second inequality holds by  calculations as in (\ref{useful}) and H\"{o}lder's inequality, and the last inequality follows from $$\|\tilde\delta_u\|_1\leq \sqrt{\tilde s_u+s_u}\|\tilde\delta_u\|_1\leq \frac{\sqrt{\tilde s_u+s_u}}{\sqrt{\semin{\tilde s_u+s_u}}}\|f(X)'\tilde\delta_u\|_{\Pn,2}\leq \frac{\sqrt{\tilde s_u+s_u}}{\sqrt{\semin{\tilde s_u+s_u}}}\frac{\|\sqrt{w_u}f(X)'\tilde\delta_u\|_{\Pn,2}}{\psi_{u}(A_u)}$$  by the definition $\psi_{u}(A) := \min_{\delta \in A} \frac{\|\sqrt{w_{u}} f(X)'\delta\|_{\Pn,2}}{\| f(X)'\delta\|_{\Pn,2}}$.

Recall the assumed conditions $\bar q_{A_u}/6>\left\{\frac{\sqrt{\tilde s_u+s_u}\|S_u\|_\infty}{\psi_{u}(A_u)\sqrt{\semin{\tilde s_u+s_u}}} + 3\|\tilde r_{u}/\sqrt{w_{u}}\|_{\Pn,2}\right\}$ and $\bar q_{A_u}/6>\sqrt{M_u(\tilde \theta_u) - M_u(\theta_u)}$. If
$\frac{1}{3}\tilde t_u^2 > \left\{ \frac{\bar q_{A_u}}{3}\tilde t_u\right\} $, then
$$ \frac{\bar q_{A_u}}{3}\tilde t_u \leq \frac{\bar q_{A_u}}{6}\sqrt{M_u(\tilde \theta_u) - M_u(\theta_u)} +\frac{\bar q_{A_u}}{6}\tilde t_u, $$
so that $\tilde t_u\leq \sqrt{0\vee \{M_u(\tilde \theta_u) - M_u(\theta_u)\}}$ which implies the result. Otherwise, we have
$$ \frac{1}{3}\tilde t_u^2 \leq \{M_u(\tilde \theta_u) - M_u(\theta_u)\} + \tilde t_u\left\{ \frac{\sqrt{\tilde s_u+s_u}\|S_u\|_\infty}{\psi_{u}(A_u)\sqrt{\semin{\tilde s_u+s_u}}} + 3\|\tilde r_{u}/\sqrt{w_{u}}\|_{\Pn,2}\right\},$$
since for positive numbers  $a$, $b$, $c$, inequality $a^2 \leq b + ac$ implies $a\leq \sqrt{b} + c$, we have
$$ \tilde t_u \leq \sqrt{3}\sqrt{0\vee\{M_u(\tilde \theta_u) - M_u(\theta_u)\}} + 3\left\{\frac{\sqrt{\tilde s_u+s_u}\|S_u\|_\infty}{\psi_{u}(A_u)\sqrt{\semin{\tilde s_u+s_u}}} + 3\|\tilde r_{ui}/\sqrt{w_{ui}}\|_{\Pn,2}\right\}.$$
\end{proof}

\subsection{Technical Lemmas: Logistic Case}

The proof of the following lower bound builds upon ideas developed in \citen{BC-SparseQR} for high-dimensional quantile regressions.

\begin{lemma}[Minoration Lemma]\label{Lemma:Minoration}
For any $u\in \mathcal{U}$ and $\delta \in A_u\subset\mathbb{R}^p$, we have
$$\begin{array}{c}M_u(\theta_u + \delta) - M_u(\theta_u) -\partial_{\theta} M_u(\theta_u)'\delta +2\|\tilde r_{u}/\sqrt{w_{u}}\|_{\Pn,2}\|\sqrt{w_{u}} f(X)'\delta\|_{\Pn,2} \\ \geq \left\{\mbox{$\frac{1}{3}$}\|\sqrt{w_{u}} f(X)'\delta\|_{\Pn,2}^2\right\} \wedge \left\{ \frac{\bar q_{A_u}}{3}\|\sqrt{w_{u}} f(X)'\delta\|_{\Pn,2}\right\} \end{array}$$
where $$\bar q_{A_u}=\inf_{ \delta \in A_u} \frac{\En\[w_{u}| f(X)'\delta|^2\]^{3/2}}{\En\[w_{u}| f(X)'\delta|^3\]}.$$ \end{lemma}
\begin{proof}

Step 1. (Minoration).   Consider the following non-negative convex function
 $$ F_u(\delta) = M_u(\theta_u+ \delta) - M_u(\theta_u) -\partial_{\theta} M_u(\theta_u)'\delta +2\|\tilde r_u/\sqrt{w_u}\|_{\Pn,2}\|\sqrt{w_u} f(X)'\delta\|_{\Pn,2}.$$
Note that if $\bar q_{A_u}=0$ the statement is trivial since $F_u(\delta)\geq 0$. Thus we can assume $\bar q_{A_u}>0$.

Step 2 below shows that for any $\delta = t\tilde \delta\in\mathbb{R}^p$ where $t\in \mathbb{R}$ and $\tilde \delta \in A_u$ such that $\|\sqrt{w_u} f(X)'\delta\|_{\Pn,2}\leq \bar q_{A_u}$ we have
\begin{equation}\label{MainMinor}
F_u(\delta) \geq \frac{1}{3} \|\sqrt{w_u} f(X)'\delta\|_{\Pn,2}^2.
\end{equation}
Thus (\ref{MainMinor}) covers the case that $\delta \in A_u$ and $\|\sqrt{w_u} f(X)'\delta\|_{\Pn,2}\leq \bar q_{A_u}$.

In the case that $\delta\in A_u$ and $\|\sqrt{w_u} f(X)'\delta\|_{\Pn,2}> \bar q_{A_u}$,  by convexity\footnote{If $\phi$ is a convex function with $\phi(0)=0$,  for $\alpha \in (0,1)$ we have $\phi(t) \geq \phi(\alpha t)/\alpha$. Indeed, by convexity, $\phi(\alpha t + (1-\alpha)0) \leq (1-\alpha)\phi(0)+\alpha\phi(t)=\alpha\phi(t)$.} of $F_u$ and $F_u(0)=0$ we have
\begin{equation}\label{MainMinorII} F_u(\delta) \geq \frac{\|\sqrt{w_u} f(X)'\delta\|_{\Pn,2}}{\bar q_{A_u}}F_u\left( \delta \frac{\bar q_{A_u}}{\|\sqrt{w_u} f(X)'\delta\|_{\Pn,2}}\right)\geq \frac{\bar q_{A_u}\|\sqrt{w_u} f(X)'\delta\|_{\Pn,2}}{3},\end{equation}where the last step follows by (\ref{MainMinor}) since  $$\| \sqrt{w_u}f(X)'\bar\delta\|_{\Pn,2}=\bar q_{A_u} \text{ for } \bar \delta = \delta \frac{\bar q_{A_u}}{\|\sqrt{w_u} f(X)'\delta\|_{\Pn,2}}.$$

Combining (\ref{MainMinor}) and (\ref{MainMinorII}) we have
$$ F_u(\delta) \geq \left\{\frac{1}{3} \|\sqrt{w_u} f(X)'\delta\|_{\Pn,2}^2\right\} \wedge \left\{ \frac{\bar q_{A_u}}{3}\|\sqrt{w_u} f(X)'\delta\|_{\Pn,2}\right\}.$$

Step 2. (Proof of (\ref{MainMinor})) Let $\tilde r_{ui}$ be such that $\G(f(X_i)'\theta_u+\tilde r_{ui})=\G(f(X_i)'\theta_u)+r_{ui} =\Ep_P[Y_{ui}\mid X_i]$. Defining $g_{ui}(t) = \log\{1+ \exp(  f(X_i)'\theta_u+\tilde r_{ui}+t f(X_i)'\delta)\}$, $\tilde g_{ui}(t) = \log\{1+ \exp(  f(X_i)'\theta_u+t f(X_i)'\delta)\}$,  $\G_{ui} :=\Ep_P[Y_{ui}\mid X_i],$ $\tilde \G_{ui} :=\exp(  f(X_i)'\theta_u )/\{1+\exp(  f(X_i)'\theta_u  )\},$ we have
\begin{equation}\label{IdPart1}\begin{array}{rl}
 & M_u( \theta_u+\delta) - M_u(\theta_u) -\partial_{\theta} M_u(\theta_u)'\delta =\\ & = \En\left[ \log \{1+{\rm exp}( f(X)'\{\theta_u+\delta\})\}-Y_u  f(X)'(\theta_u+\delta)\right] \\
  & \ \ - \En\left[ \log \{1+{\rm exp}( f(X)'\theta_u)\}-Y_u  f(X)' \theta_u \right] - \En\left[ (\tilde \G_u-Y_u )  f(X)'\delta\right]\\
 & = \En\left[ \log \{1+{\rm exp}( f(X)'\{\theta_u+\delta\})\} - \log \{1+{\rm exp}( f(X)'\theta_u)\} -  \tilde \G_u   f(X)'\delta\right]\\
& = \En[ \tilde g_u(1) - \tilde g_u(0) -  \tilde g_u'(0) ] \\
& = \En[ g_u(1) - g_u(0) -  g_u'(0) ] + \En[ \{\tilde g_u(1)-g_u(1)\} - \{\tilde g_u(0) -g_u(0)\} - \{ \tilde g_u'(0)-g_u'(0)\} ] \\\end{array}\end{equation}
Note that the function $g_{ui}$ is three times differentiable and satisfies,  $$\begin{array}{c}
 g'_{ui}(t)  =   ( f(X_i)'\delta) \G_{ui}(t), \ \ g''_{ui}(t)  =  ( f(X_i)'\delta)^2 \G_{ui}(t)[1-\G_{ui}(t)], \ \ \mbox{ and} \\ \  g'''_{ui}(t) = ( f(X_i)'\delta)^3 \G_{ui}(t)[1-\G_{ui}(t)][1-2\G_{ui}(t)]\end{array}$$
where $\G_{ui}(t) :=\exp(  f(X_i)'\theta_u +\tilde r_{ui} + t f(X_i)'\delta )/\{1+\exp(  f(X_i)'\theta_u +\tilde r_{ui} + t  f(X)'\delta )\}$.
Thus we have $|g'''_{ui}(t)|\leq | f(X)'\delta| g''_{ui}(t)$. Therefore, by Lemmas \ref{Lemma:SC} and \ref{Lemma:Auxtis} given following the conclusion of this proof, we have
\begin{equation}\label{IdPart2}\begin{array}{rl}
  g_{ui}(1) - g_{ui}(0) -  g_{ui}'(0) &   \geq \frac{( f(X_i)'\delta)^2w_{ui}}{( f(X_i)'\delta)^2}\left\{ \exp(-| f(X_i)'\delta|) + | f(X_i)'\delta| -1 \right\}  \\
  & \geq w_{ui} \left\{ \frac{| f(X_i)'\delta|^2}{2} - \frac{| f(X_i)'\delta|^3}{6} \right\}\end{array}\end{equation}
Moreover, letting $\Upsilon_{ui}(t) = \tilde g_{ui}(t) - g_{ui}(t)$ we have $$|\Upsilon_{ui}'(t)| = |( f(X_i)'\delta)\{ \G_{ui}(t)-\tilde \G_{ui}(t)\}|\leq | f(X_i)'\delta| \ |\tilde r_{ui}|$$ where $\tilde \G_{ui}(t) :=\exp(  f(X_i)'\theta_u + t  f(X_i)'\delta )/\{1+\exp(  f(X_i)'\theta_u +  t  f(X_i)'\delta )\}$. Thus
\begin{equation}\label{IdPart3} \begin{array}{l}|\En[ \{\tilde g_u(1)-g_u(1)\} - \{\tilde g_u(0) -g_u(0)\} - \{ \tilde g_u'(0)-g_u'(0)\} ]| =\\
= |\En[ \Upsilon_u(1) - \Upsilon_u(0)-\{\tilde \G_{u}-\G_{u}\} f(X)'\delta]|\\
\leq 2\En[|\tilde r_u|\ | f(X)'\delta| ].\\
\end{array}\end{equation}
Therefore, combining (\ref{IdPart1}) with the bounds (\ref{IdPart2}) and (\ref{IdPart3}) we have
$$\begin{array}{rl}
M_u(\theta_u+\delta) - M_u(\theta_u) -\partial_{\theta} M_u(\theta_u)'\delta & \displaystyle \geq \mbox{$\frac{1}{2}$}\En\left[ w_u| f(X)'\delta|^2\right] - \mbox{$\frac{1}{6}$}\En\left[ w_u| f(X)'\delta|^3\right]\\
&-2\|\tilde r_u/\sqrt{w_u}\|_{\Pn,2}\|\sqrt{w_u} f(X)'\delta\|_{\Pn,2},\\ \end{array}$$
which holds for any $\delta \in \mathbb{R}^p$.


Take any $\delta = t\tilde \delta$, $t\in\mathbb{R}\setminus\{0\}$, $\tilde\delta\in A_u$ such that $ \|\sqrt{w_u} f(X)'\delta\|_{\Pn,2} \leq  \bar q_{A_u}$. (Note that the case of $\delta=0$ is trivial.) We have $$\begin{array}{rl}
 \En[w_u| f(X)'\delta|^2]^{1/2} = \|\sqrt{w_u} f(X)'\delta\|_{\Pn,2}  \leq  \bar q_{A_u} & \leq \En \[w_u |f(X)'\tilde \delta|^2\]^{3/2}/\En\[w_u| f(X)'\tilde \delta|^3\]\\
 & = \En \[w_u |f(X)' \delta|^2\]^{3/2}/\En\[w_u| f(X)' \delta|^3\],\end{array}$$
since the scalar $t$ cancels out. Thus, $\En[w_u| f(X)'\delta|^3] \leq \En[w_u| f(X)'\delta|^2]$. Therefore we have $$\mbox{$\frac{1}{2}$}\En\left[ w_u| f(X)'\delta|^2\right] - \mbox{$\frac{1}{6}$}\En\left[ w_u| f(X)'\delta|^3\right] \geq \frac{1}{3}\En\left[ w_u| f(X)'\delta|^2\right] \ \ \mbox{and}$$
$$\begin{array}{rl}
M_u(\theta_u+\delta) - M_u(\theta_u) -\partial_{\theta} M_u(\theta_u)'\delta
&\geq \frac{1}{3}\En\left[ w_u| f(X)'\delta|^2\right] -2\|\frac{\tilde r_u}{\sqrt{w_u}}\|_{\Pn,2}\|\sqrt{w_u} f(X)'\delta\|_{\Pn,2},\\ \end{array}$$
which establishes that $F_u(\delta):=M_u(\theta_u+\delta) - M_u(\theta_u) -\partial_{\theta} M_u(\theta_u)'\delta+2\|\frac{\tilde r_u}{\sqrt{w_u}}\|_{\Pn,2}\|\sqrt{w_u} f(X)'\delta\|_{\Pn,2}$ is larger than $\frac{1}{3}\En\left[ w_u| f(X)'\delta|^2\right]$ for any $\delta=t\tilde\delta$, $t\in\mathbb{R}$, $\tilde\delta \in A_u$ and $\|\sqrt{w_u} f(X)'\delta\|_{\Pn,2} \leq \bar{q}_{A_u}$.
\end{proof}

\begin{lemma}[Lemma 1 from \citen{Bach2010}]\label{Lemma:SC}
Let $g:\RR \to \RR$ be a three times differentiable convex function such that for all $t \in \RR$, $|g'''(t)| \leq M g''(t)$ for some $M\geq 0$. Then, for all $t \geq 0$ we have
$$ \frac{g''(0)}{M^2} \left\{ \exp(-Mt) + Mt - 1\right\} \leq g(t) - g(0) - g'(0)t \leq \frac{g''(0)}{M^2} \left\{ \exp(Mt) + Mt - 1\right\}.$$
\end{lemma}
\begin{lemma}\label{Lemma:Auxtis} For $t\geq 0$ we have
$ \exp(-t)+t - 1 \geq \frac{1}{2}t^2 - \frac{1}{6}t^3.$
\end{lemma}
\begin{proof}[Proof of Lemma \ref{Lemma:Auxtis}]
For $t\geq 0$, consider the function $f(t) = \exp(-t) + t^3/6 - t^2/2 + t - 1$. The statement is equivalent to $f(t)\geq 0$ for $t\geq 0$. It follows that $f(0)=0$, $f'(0) = 0$, and $f''(t)=\exp(-t)+t-1\geq 0$ so that $f$ is convex. Therefore $f(t) \geq f(0)+tf'(0) = 0$.
\end{proof}

\begin{lemma}\label{Lemma:AuxSparsity}
The logistic link function satisfies $|\G(t+t_0)-\G(t_0)|\leq \G'(t_0)\{\exp(|t|) -1\}$. If $|t|\leq 1$ we have $\exp(|t|) -1\leq 2|t|$.
\end{lemma}
\begin{proof}
Note that $|\G''(s)|\leq \G'(s)$ for all $s\in\RR$. So that $-1\leq \frac{d}{ds}\log(\G'(s)) = \frac{\G''(s)}{\G'(s)} \leq 1$. Suppose $s\geq 0$. Therefore
$$ - s \leq \log(\G'(s+t_0)) - \log(\G'(t_0)) \leq s. $$
In turn this implies $\G'(t_0)\exp(-s)\leq \G'(s+t_0) \leq \G'(t_0)\exp(s)$. For $t>0$, integrating one more time from $0$ to $t$,
$$ \G'(t_0)\{1-\exp(-t)\}\leq \G(t+t_0) - \G(t_0) \leq \G'(t_0)\{\exp(t)-1\}.$$
Similarly, for $t<0$, integrating from $t$ to $0$, we have
$$ \G'(t_0)\{1-\exp(t)\}\leq \G(t+t_0) - \G(t_0) \leq \G'(t_0)\{\exp(-t)-1\}.$$

The first result follows by noting that $1-\exp(-|t|) \leq \exp(|t|)-1$. The second follows by verification.
\end{proof}

\section{Simulation Experiment}\label{sec:p}

In this section, we present results from a brief simulation experiment.  The results illustrate the performance of our proposed treatment effect estimator that makes use of estimating equations satisfying the key orthogonality condition given in equation (2) in the main text and variable selection relative to an estimator that uses variable selection but is based on a ``naive'' estimating equation that does not satisfy the orthogonality condition.  We find that inference based on the naive estimator can suffer from substantial size distortions and that the performance of this estimator is strongly dependent on features of the data generating process (DGP).   We also find that tests based on the estimator constructed using our procedure have size close to the nominal level uniformly across all DGPs we consider consistent with the theory developed in the paper.

For simplicity, we consider the case where the treatment, $d_i$, is exogenous conditional on control variables $x_i$.  In this case, we can apply the results of the paper substituting $d_i$ for $z_i$ in each instance where instruments $z_i$ are used since $d_i$ is conditionally exogenous and thus a valid instrument for itself.   All of the simulation results are based on data generated as
\begin{align*}
d_i &= \mathbf{1}\left\{\frac{\exp\{x_i'(c_d \theta_0)\}}{1+\exp\{x_i'(c_d \theta_0)\}} > v_i\right\} \\
y_i &=  d_i [x_i'(c_y \theta_0)] + \zeta_i
\end{align*}
where $v_i \sim U(0,1)$, $\zeta_i \sim N(0,1)$, $v_i$ and $\zeta_i$ are independent, $p = \dim(x_i) = 250$, the covariates $x_i \sim N(0,\Sigma)$ with $\Sigma_{kj} = (0.5)^{|j-k|}$, and
the sample size $n = 200$.  $\theta_0$ is a $p \times 1$ vector with elements set as $\theta_{0,j} = (1/j)^2$ for $j = 1,...,p$.  $c_d$ and $c_y$ are scalars that control the strength of the relationship between the controls, the outcome, and the treatment variable.  We use several different combinations of $c_d$ and $c_y$, setting $c_d = \sqrt{\frac{(\pi^2/3)R^2_d}{(1-R^2_d)\theta_0'\Sigma\theta_0}}$ and $c_y =  \sqrt{\frac{R^2_d}{(1-R^2_d)\theta_0'\Sigma\theta_0}}$ for all combinations of $R^2_d \in \{0,0.1,0.2,0.3,0.4,0.5,0.6,0.7,0.8,0.9\}$ and  $R^2_y \in \{0,0.1,0.2,0.3,0.4,0.5,0.6,0.7,0.8,0.9\}$.

We report results for two different inference procedures in Figure 11.  The right panel of the figure shows size of 5\% level t-tests for the average treatment effect where the point estimate is formed using our proposed estimator based on model selection and orthogonal estimating equations and the standard error is estimated using a plug-in estimator of the asymptotic variance.  The left panel shows size of 5\% level t-tests for the average treatment effect estimated as
$$
\widehat\theta_{naive} = \frac{1}{n}\sum_{i=1}^{n} (\widehat g_{y}(1,x_i) - \widehat g_{y}(0,x_i))
$$
where $\widehat g_y(d,x_i)$ is a post-model-selection estimator of $\Ep[Y|D=d,X=x_i]$ and the standard error is estimated using a plug-in estimator of the asymptotic variance of $\widehat\theta_{naive}$.

Both procedures rely on post-model-selection estimates of the conditional expectations $\Ep[Y|D=d,X=x_i]$, and we use exactly the same estimator of this quantity in both cases.  Specifically, we apply the Square-Root Lasso of \citen{BCW-SqLASSO} with outcome $Y$ and covariates $(D,D*X_1,...,D*X_p,(1-D),(1-D)*X_1,...,(1-D)*X_p)$ to select variables.  We set the penalty level in the Square-Root Lasso using the ``exact'' option of \citen{BCW-SqLASSO} under the assumption of homoscedastic, Gaussian errors $\zeta_i$ with the tuning confidence level required in \citen{BCW-SqLASSO} set equal to 95\%.  After running the Square-Root Lasso, we then estimate regression coefficients by regressing $Y$ onto only those variables that were estimated to have non-zero coefficients by the Square-Root Lasso.  We then form estimates of $\Ep[Y|D=1,X=x_i]$ by plugging in $(1,x_i')'$ into the estimated model for $i=1,...,n$ and form estimates of $\Ep[Y|D=0,X=x_i]$ by plugging in $(0,x_i')'$ into the estimated model for $i=1,...,n$.

For our proposed method, we also need an estimate of the propensity score.  We obtain our estimates of the propensity score by using $\ell_1-$penalized logistic regression with $D$ as the outcome and $X$ as the covariates with penalty level set equal to $.5\sqrt{n}\Phi^{-1}(1-1/2p)/n$ where $\Phi(\cdot)$ is the standard normal distribution function using the MATLAB function ``glmlasso''.\footnote{This penalty level is equivalent to that discussed in the main paper since ``glmlasso'' scales the problem in a slightly different way.}  We standardize the variables in $X$ and set penalty loadings equal to 1.  After running the $\ell_1-$penalized logistic regression, we estimate the propensity score by taking fitted values from the conventional logistic regression of $D$ onto only those variables that had non-zero estimated coefficients in the $\ell_1-$penalized logistic regression.

Looking at the results, we see the behavior of the naive testing procedure depends heavily on the underlying coefficient sequence used to generate the data.  There are substantial size distortions for many of the coefficient designs considered with good performance, size close to the nominal level, only occurring in a handful of cases.  It is worth noting that in practice one does not know the underlying DGP and even estimation of the quantities necessary to know where one is in the figure may be infeasible even in this simple scenario.  Our proposed procedure does a much better job at delivering accurate inference, producing tests with size close to the nominal level across all designs considered.  That is, the simulation illustrates the uniformity derived in the theoretical development of our estimator illustrating that its performance is relatively good uniformly across a variety of coefficient sequences.  While simply illustrative, these simulation results reinforce the theoretical development of the main paper which prove that our proposed estimation and inference procedures have good properties uniformly across a variety of DGPs where approximate sparsity holds.

\bibliographystyle{econometrica}
\bibliography{mybibVOLUME}

\pagebreak

\pagebreak
\begin{figure}[h]
	\includegraphics[width=\textwidth]{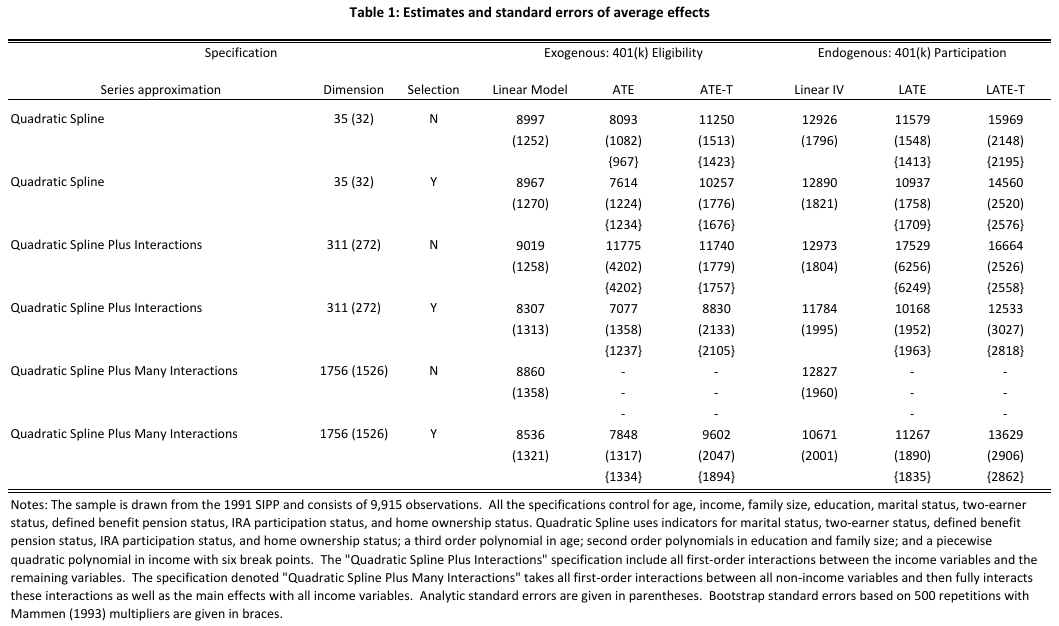}
	\label{fig:table1}
\end{figure}

\pagebreak
\begin{figure}[h]
	\includegraphics[width=\textwidth]{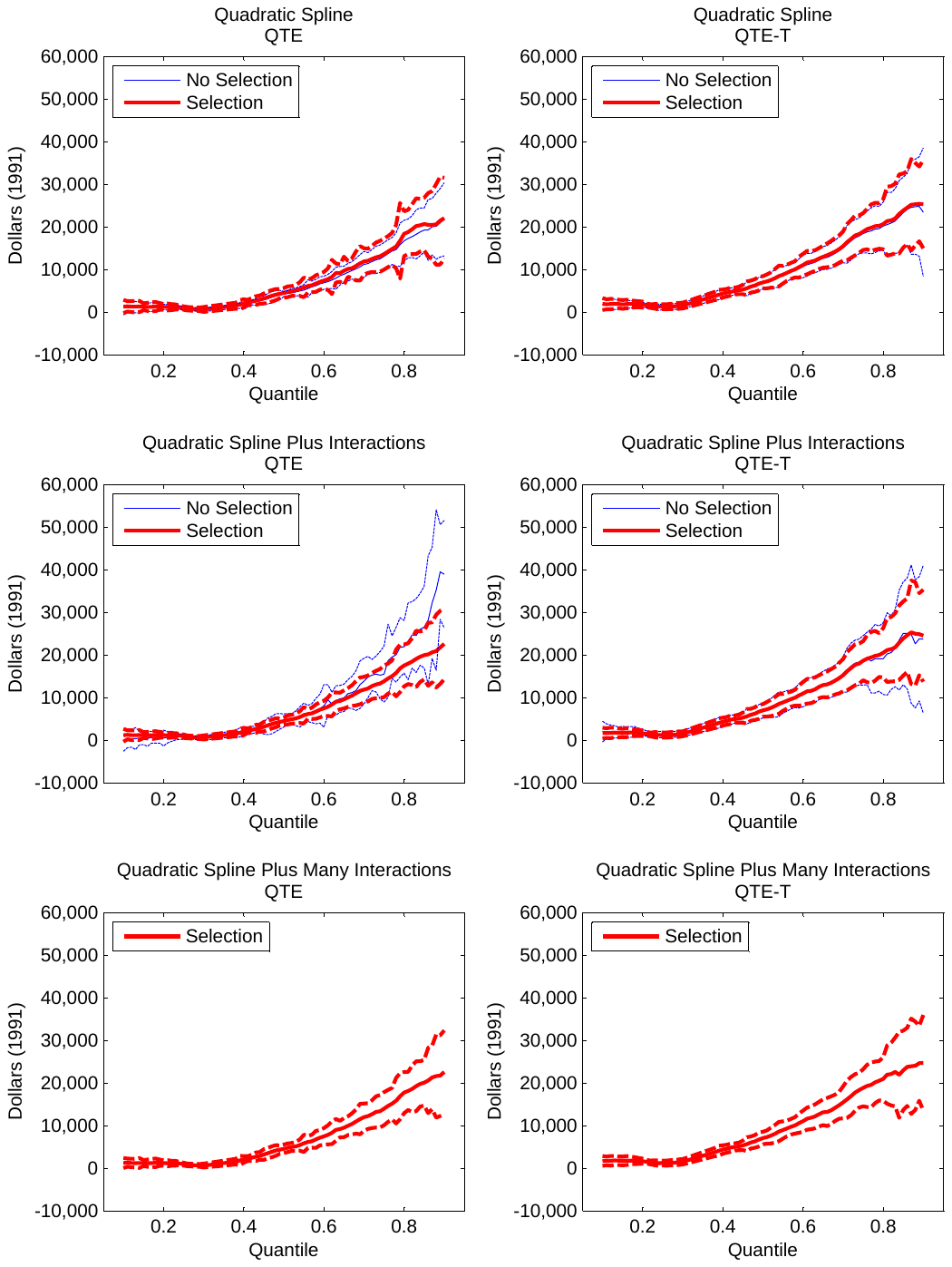}
	\label{fig:figure1}
	\caption{QTE and QTE-T estimates of the effect of 401(k) eligibility on net financial assets.}
\end{figure}

\pagebreak

\begin{figure}[h]
	\includegraphics[width=\textwidth]{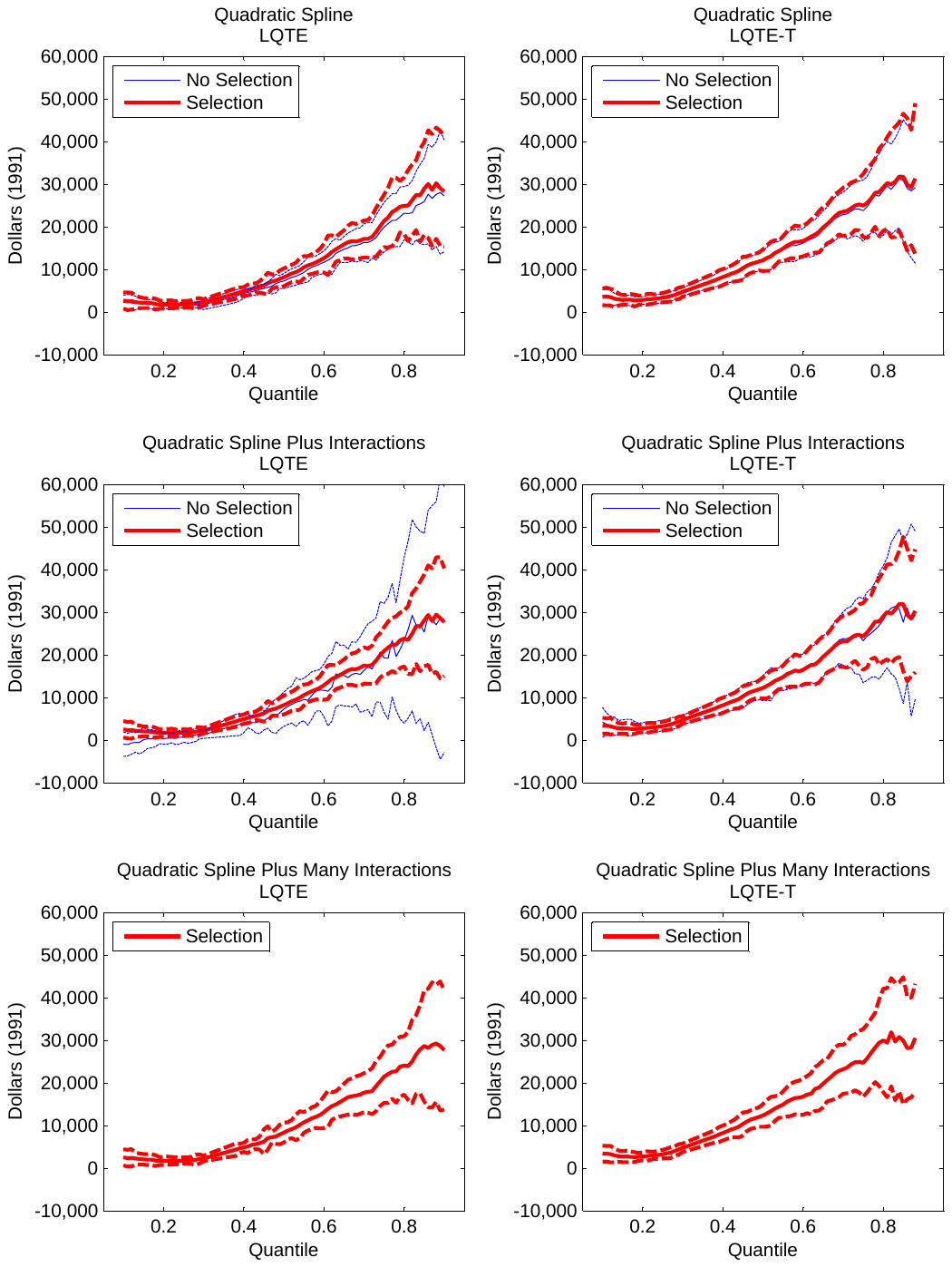}
	\label{fig:figure2}
	\caption{LQTE and LQTE-T estimates of the effect of 401(k) participation on net financial assets.}
\end{figure}

\pagebreak
\begin{figure}
	\includegraphics[width=\textwidth]{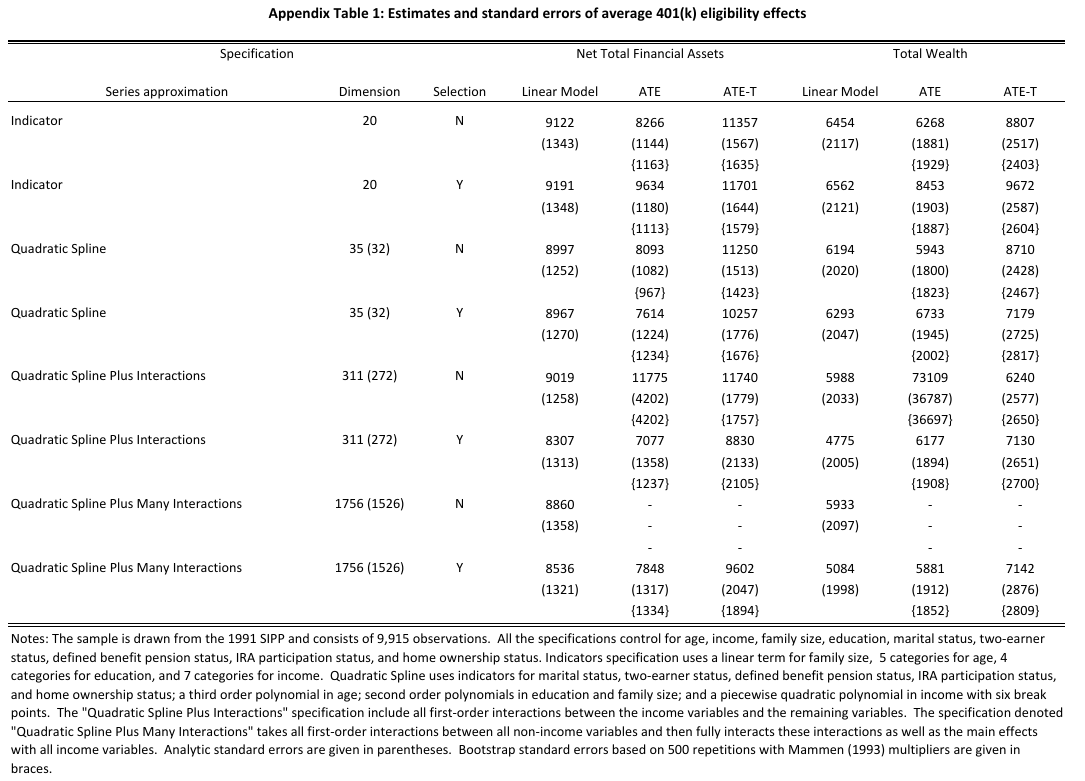}
	\label{fig:apptable1}
\end{figure}

\pagebreak
\begin{figure}
	\includegraphics[width=\textwidth]{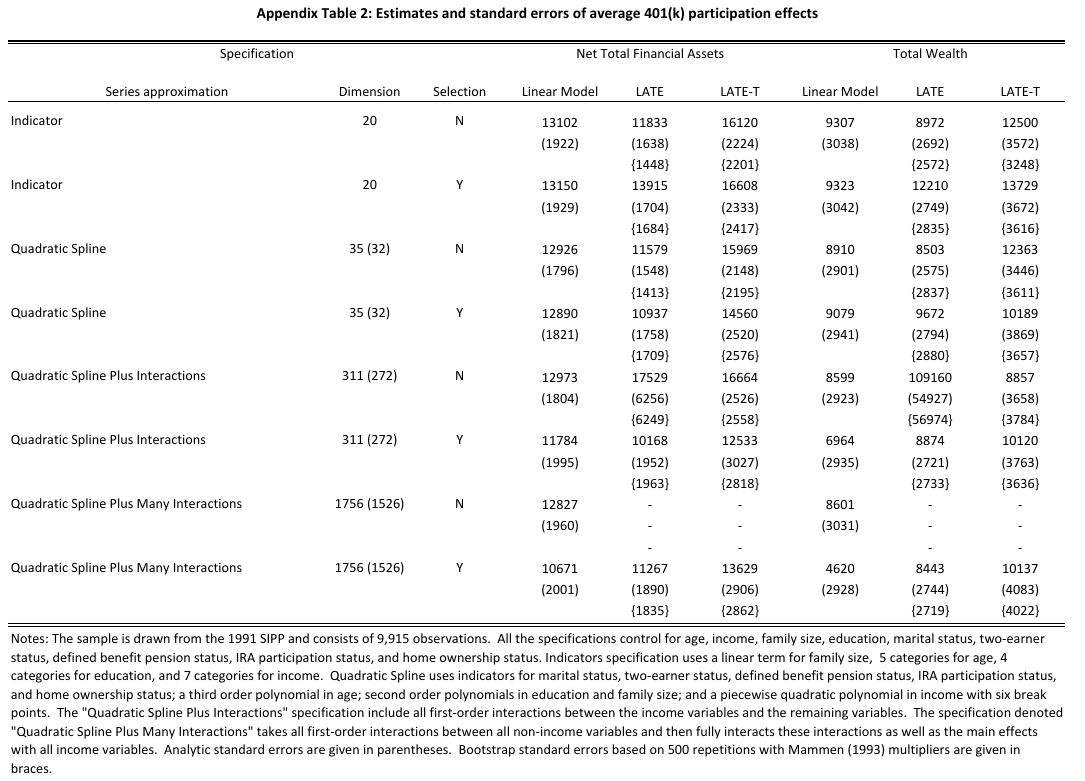}
	\label{fig:apptable2}
\end{figure}

\pagebreak
\begin{figure}
	\includegraphics[width=\textwidth]{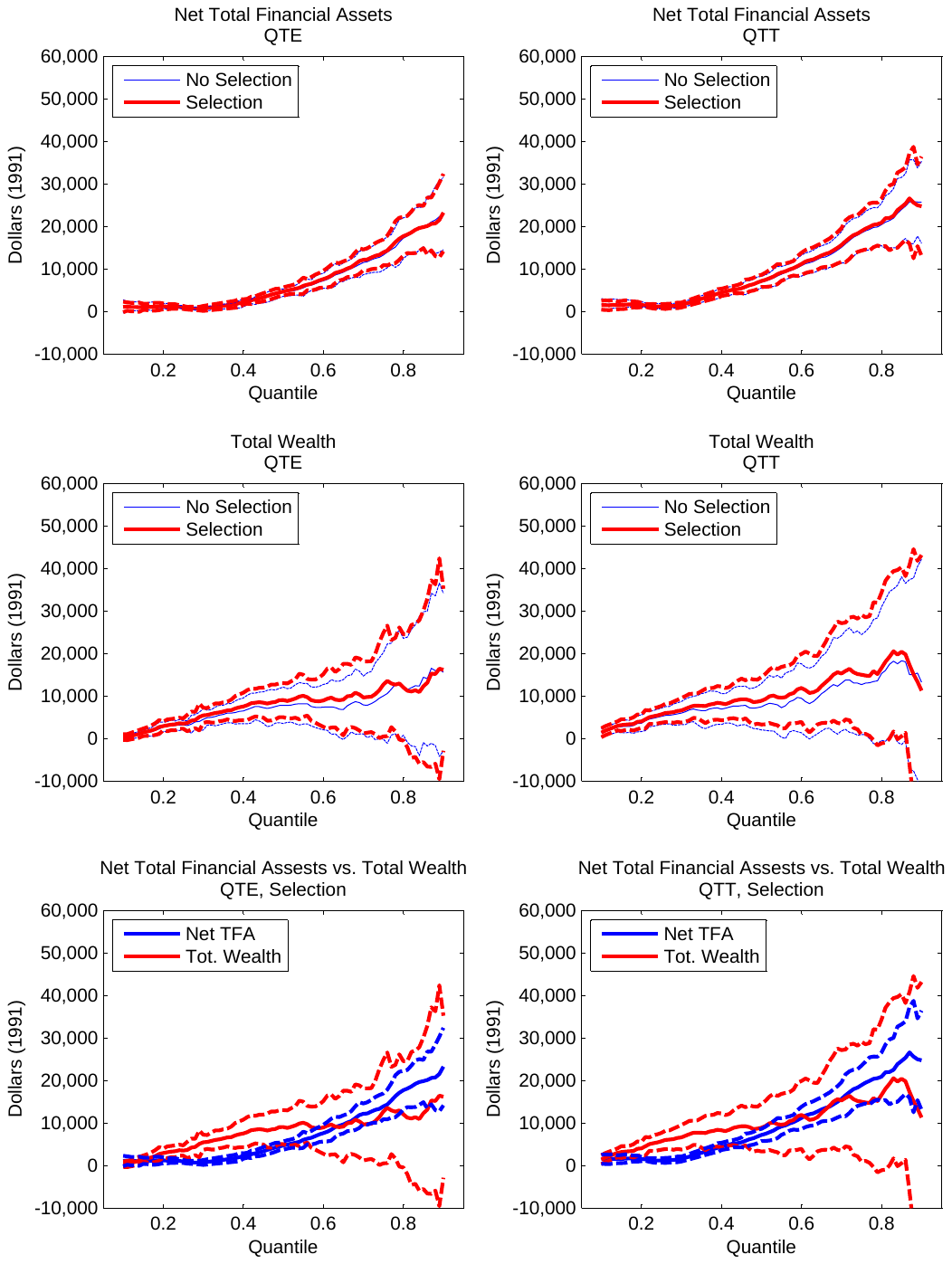}
	\label{fig:exogappfigure1}
	\caption{QTE and QTE-T estimates based on the Indicators specification.}
\end{figure}

\pagebreak
\begin{figure}
	\includegraphics[width=\textwidth]{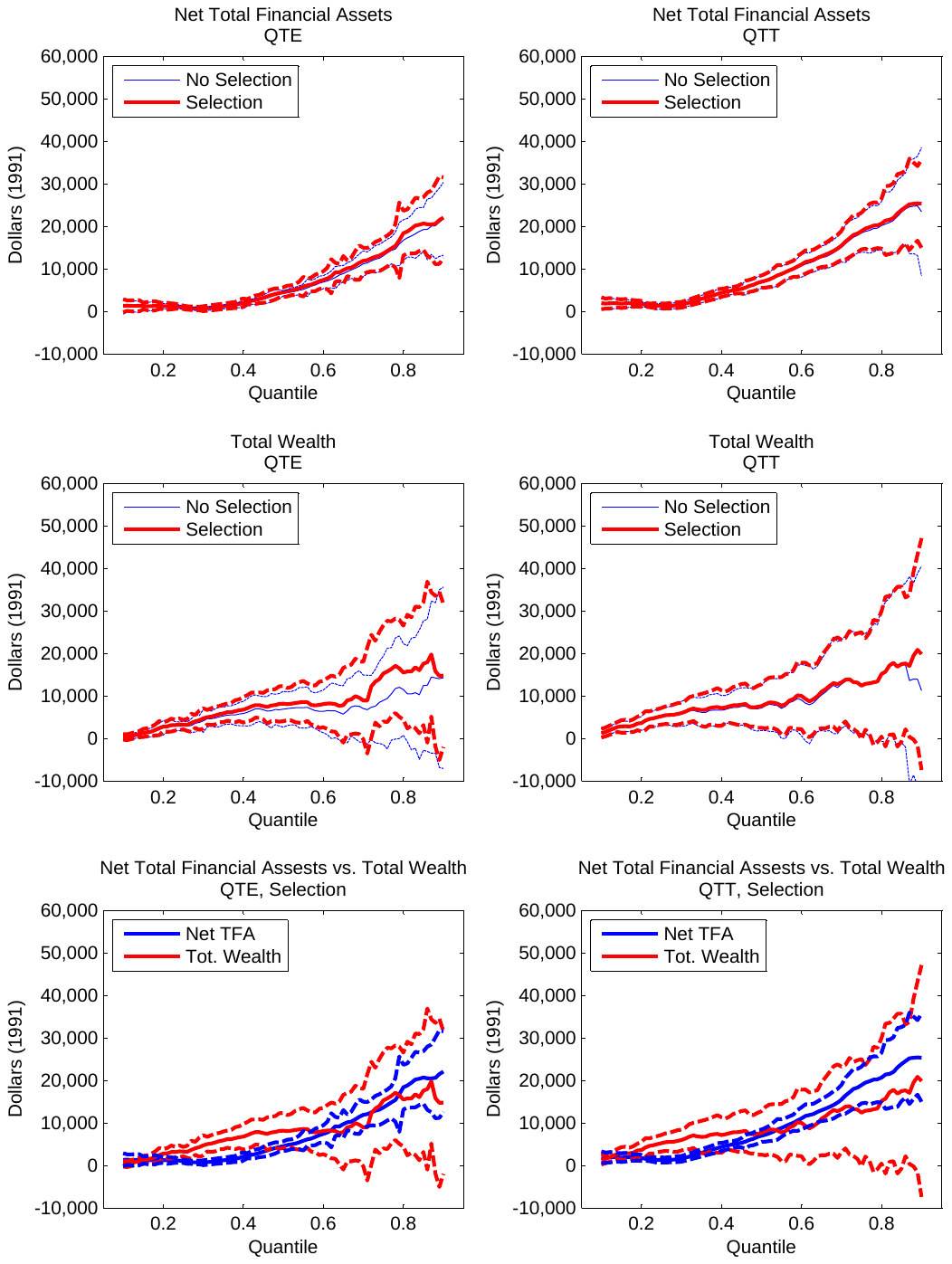}
	\label{fig:exogappfigure2}
	\caption{QTE and QTE-T estimates based on the Quadratic Spline specification.}
\end{figure}

\pagebreak
\begin{figure}
	\includegraphics[width=\textwidth]{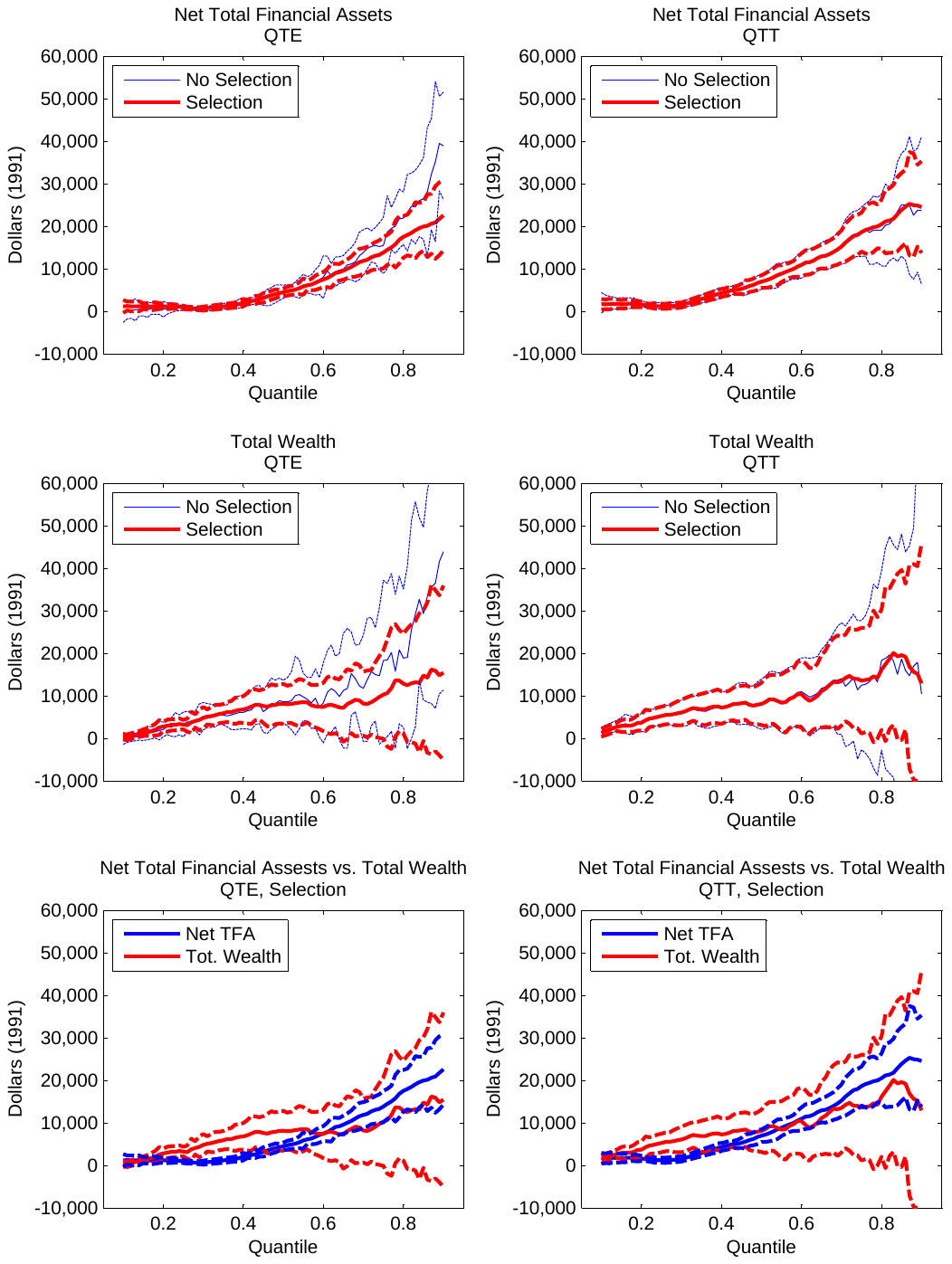}
	\label{fig:exogappfigure3}
	\caption{QTE and QTE-T estimates based on the Quadratic Spline Plus Interaction specification.}
\end{figure}

\pagebreak
\begin{figure}
	\includegraphics[width=\textwidth]{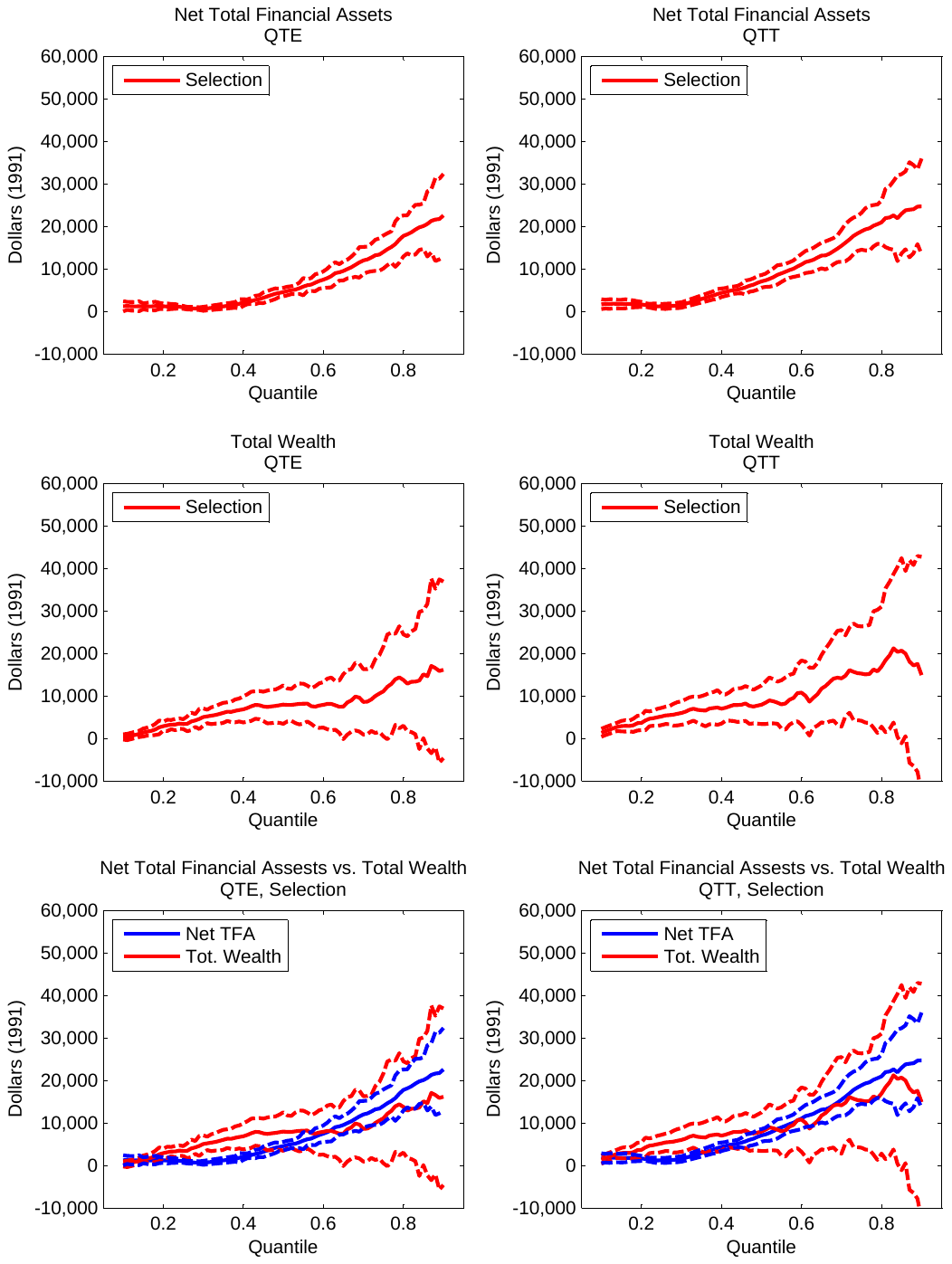}
	\label{fig:exogappfigure4}
	\caption{QTE and QTE-T estimates based on the Quadratic Spline Plus Many Interaction specification.}
\end{figure}

\pagebreak
\begin{figure}
	\includegraphics[width=\textwidth]{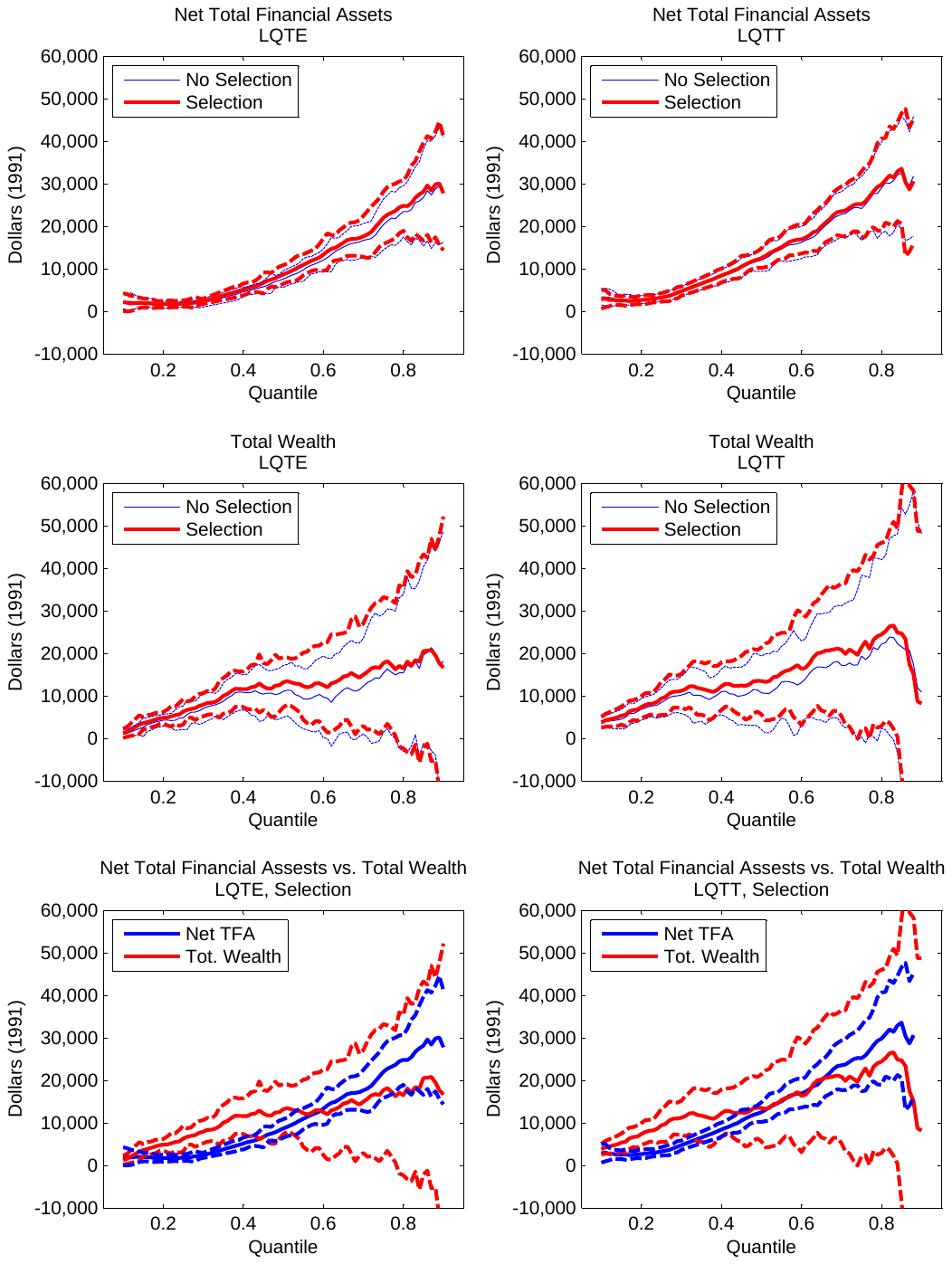}
	\label{fig:appfigure1}
	\caption{LQTE and LQTE-T estimates based on the Indicators specification.}
\end{figure}

\pagebreak
\begin{figure}
	\includegraphics[width=\textwidth]{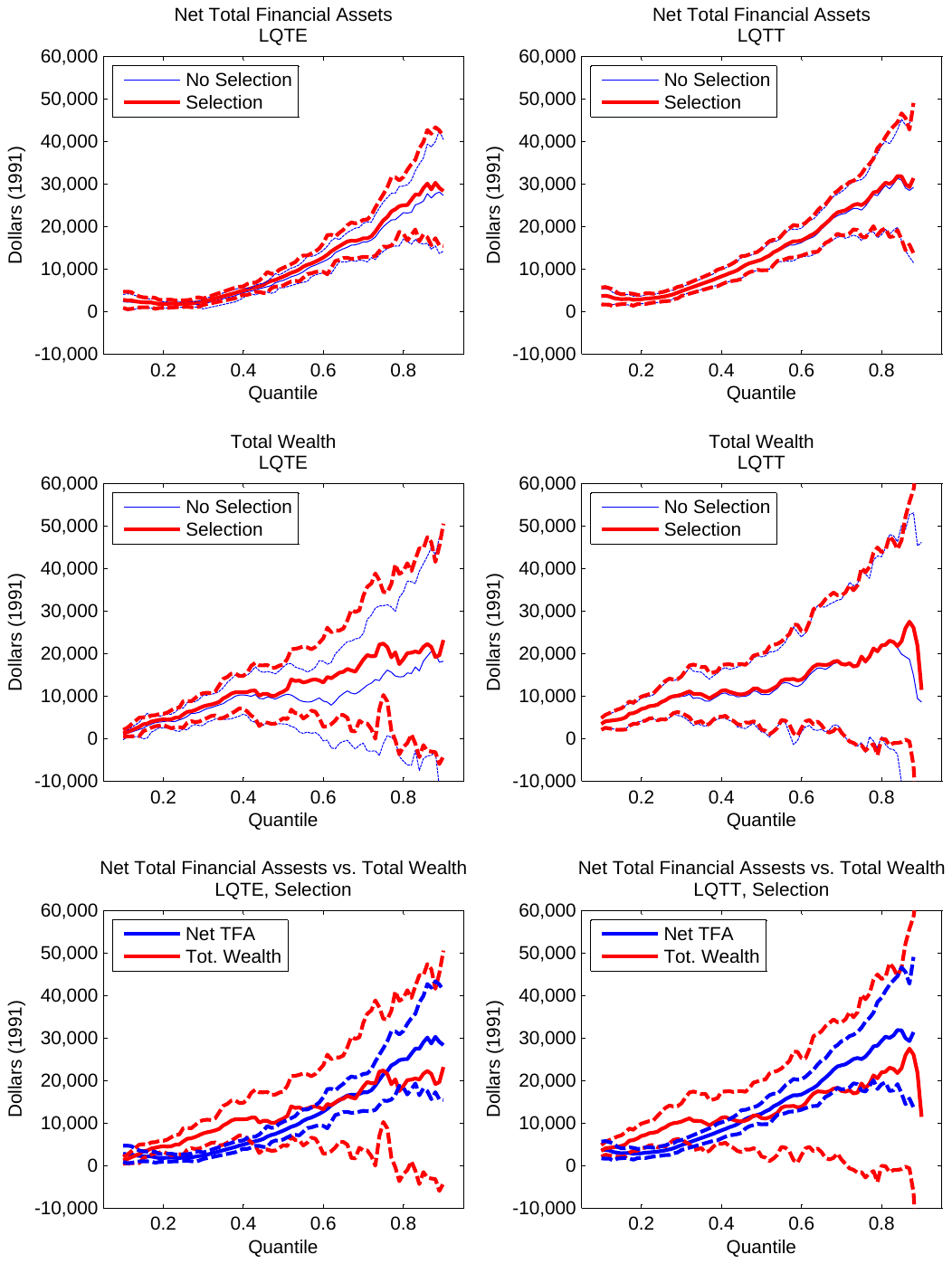}
	\label{fig:appfigure2}
	\caption{LQTE and LQTE-T estimates based on the Quadratic Spline specification.}
\end{figure}

\pagebreak
\begin{figure}
	\includegraphics[width=\textwidth]{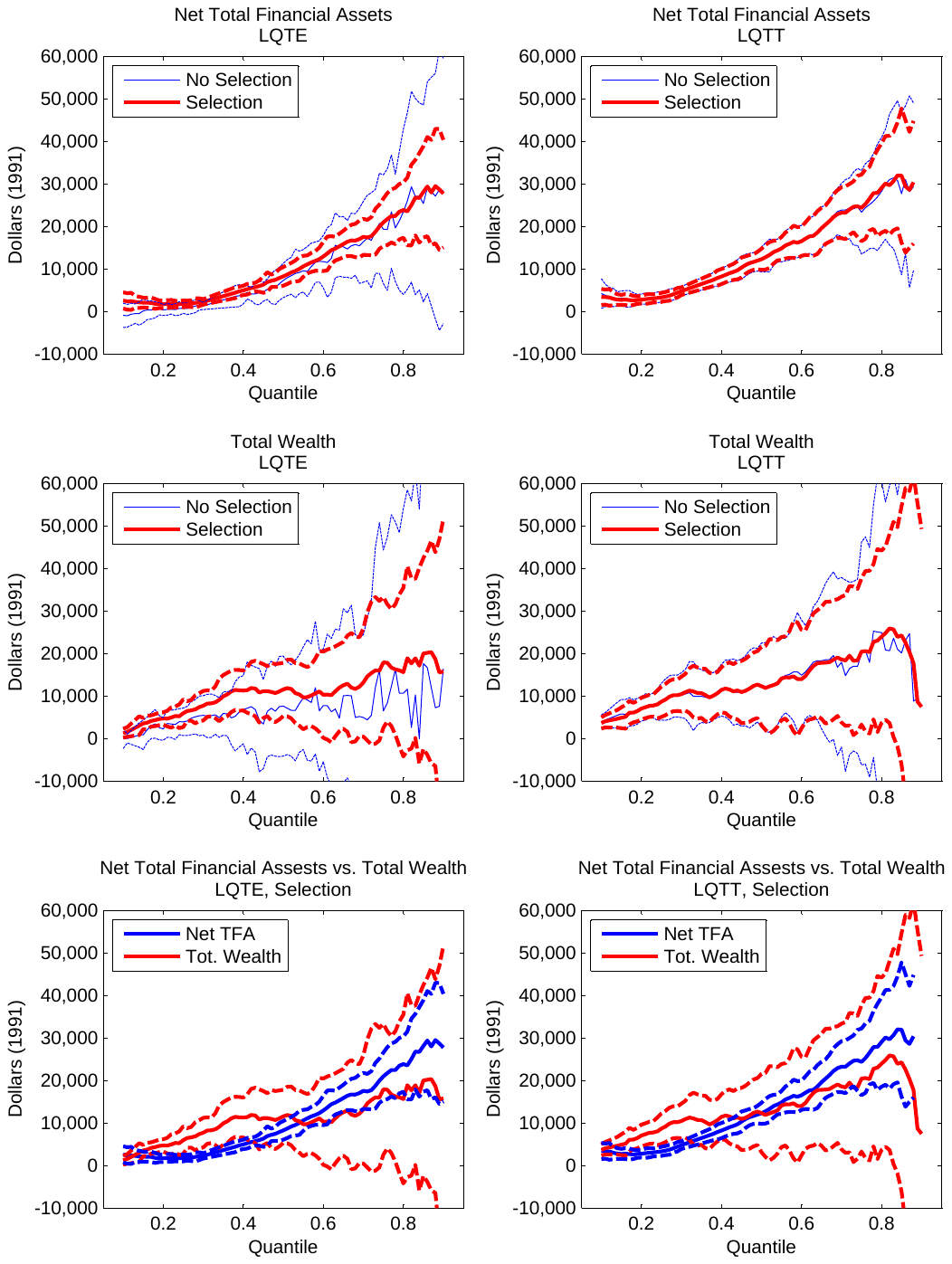}
	\label{fig:appfigure3}
	\caption{LQTE and LQTE-T estimates based on the Quadratic Spline Plus Interaction specification.}
\end{figure}

\pagebreak
\begin{figure}
	\includegraphics[width=\textwidth]{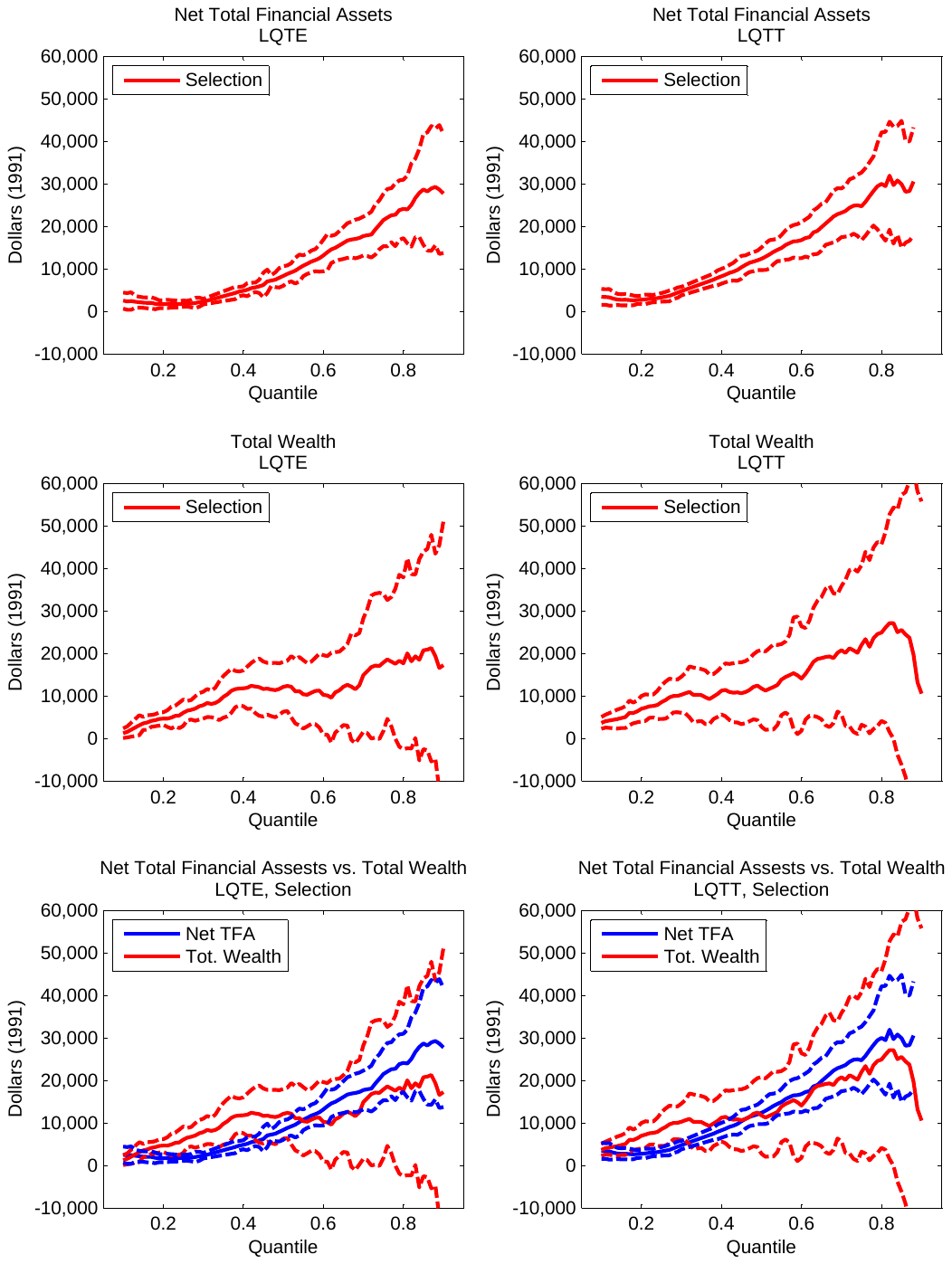}
	\label{fig:appfigure4}
	\caption{LQTE and LQTE-T estimates based on the Quadratic Spline Plus Many Interaction specification.}
\end{figure}

\pagebreak
\begin{figure}\label{SuppFigure1}
	\includegraphics[width=0.49\textwidth]{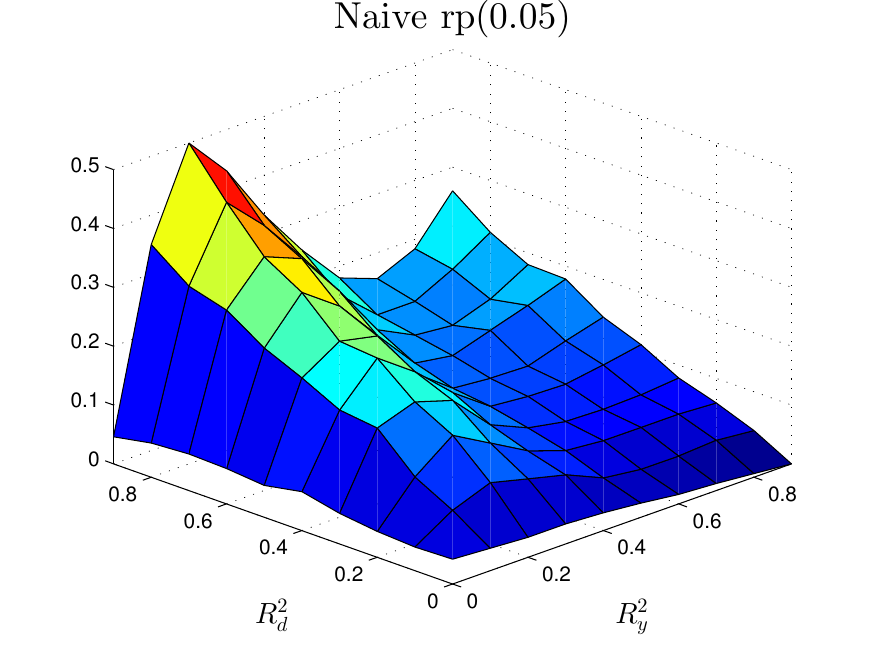}
	\includegraphics[width=0.49\textwidth]{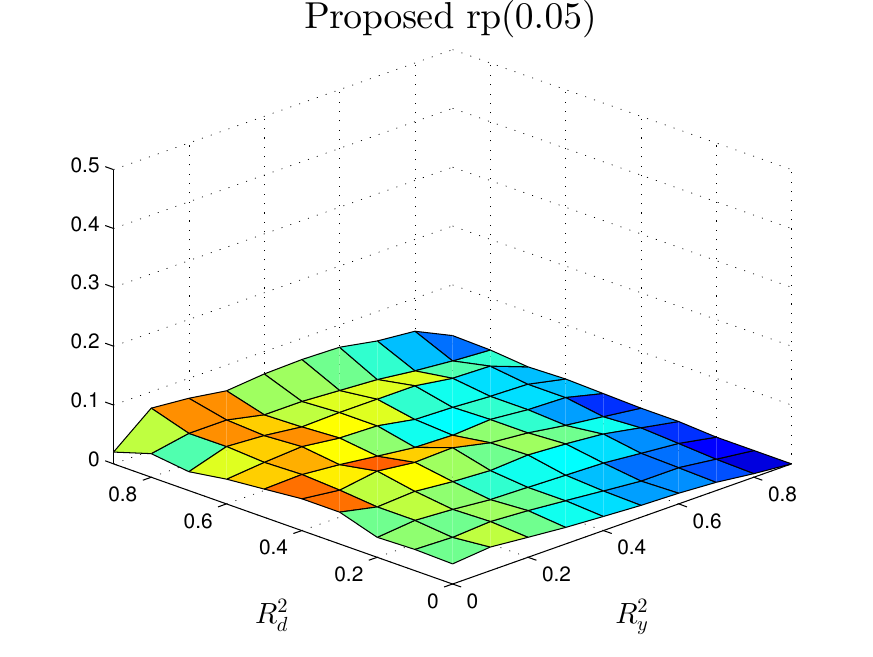}
	\caption{Rejection frequencies of 5\% level tests for average treatment effect estimators following model selection.  The left panel shows size of a test based on a ``naive'' estimator (Naive rp(0.05)), and the right panel shows size of a test based on our proposed procedure (Proposed rp(0.05)).}
\end{figure}

\end{document}